\documentclass[a4paper,oneside,12pt]{book}
\pdfoutput=1
\usepackage{graphics}               %figure insertion with \resizebox{\includegraphics{file.eps}
\usepackage{graphicx}               %figure insertion with \includegraphics[width=x.cm]{file.eps}
\usepackage{subfigure}              %subfigures with \mbox and \subfigure[]
\usepackage{amsmath}                %more maths options, e.g. {split} to align multi-line equationsl
\usepackage{amssymb}                %more maths symbols e.g. hbar and plot symbols
\usepackage{setspace}               %enables use of \doublespace \onehalfspace \singlespace
\usepackage{fancyhdr}               %customises header and allows chapter, section or page number
\usepackage{acronym}		        %for the glossary section
\usepackage{cite}
\usepackage[final]{pdfpages}
\usepackage{algpseudocode}
\usepackage[english]{babel}
\usepackage{multirow}

\usepackage{caption}
\usepackage{minitoc}
\usepackage{eso-pic}
\usepackage{epsfig}
\usepackage{epigraph}
\usepackage[outer=25mm, inner=41mm, top=36mm, bottom=29mm]{geometry}	
\usepackage{epstopdf}
\usepackage{amsthm}
\theoremstyle{definition}
\newtheorem{theorem}{Theorem}[section]
\newtheorem{definition}{Definition}[section]
\newtheorem{assumption}{Assumption}[section]
\newtheorem{lemma}{Lemma}[section]
\newtheorem{remark}{Remark}[section]

\begin{document}

\dominitoc
\pagenumbering{roman}   %Roman numeral numbering for initial section of report
%==========================DEFINE TITLE PAGE===========================================================================
% This isn't a very elegant way of defining your own title page but it sure beats learning how to program in perl
% which is what I think you need to know to configure your own latex .sty file or .cls file or something
\begin{titlepage}

\pagestyle{empty}
\begin{center}
\vfill
{\huge {\bf \LARGE Decentralized Control of Three Dimensional Mobile Robotic Sensor Networks\\}}
\vspace{1.5cm}
\vspace{1.5cm}
 by\\

\vspace{2ex}
{\Large {\bf \Large Valimohammad  Nazarzehi had}}\\

\vspace{1.5cm}

%\begin{figure}[!h]
%\begin{center}
%\includegraphics[width=3cm]{unsw}
%\label{fig:unswlogo}
%\end{center}
%\end{figure}
%
%
%\vspace{2ex}
%{\bf The School of Electrical Engineering and Telecommunications}\\
%{\bf The University of New South Wales}\\
%\vspace{1.5cm}
\vspace{1.5cm}
\vspace{1.5cm}
\vspace{1.5cm}
\vspace{1.5cm}
2016
\end{center} 

\end{titlepage}

%BY vALI
%\cleardoublepage
\let\cleardoublepage\clearpage
                  %Inserts empty page after title page
\singlespace
\singlespacing

\chapter*{Abstract}
\mtcaddchapter[Abstract]
%\addcontentsline{toc}{chapter}{Abstract}
\onehalfspace
\singlespacing
Decentralized control of mobile robotic sensor networks is a fundamental problem in robotics that has attracted intensive research in recent decades. Most of the existing works dealt with two-dimensional spaces. This report is concerned with the problem of decentralized self-deployment of mobile robotic sensor networks for coverage, search, and formation building in three-dimensional environments.\\
The first part of the report investigates the problem of complete sensing coverage in three-dimensional spaces. We propose a decentralized random algorithm to drive mobile robotic sensors on the vertices of a truncated octahedral grid for complete sensing coverage of a bounded 3D area. Then, we develop a decentralized random algorithm for self deployment of mobile robotic sensors to form a desired geometric shape on the vertices of the truncated octahedral grid. The second part of this report studies the problem of search in 3D spaces. We present a distributed random algorithm for search in bounded three dimensional environments. The proposed algorithm utilizes an optimal three dimensional grid for the search task.\\
Third, we study the problem of locating static and mobile targets in a bounded 3D space by a network of mobile robotic sensors. We introduce a novel decentralized bio-inspired random search algorithm for finding static and  mobile objects in 3D areas. This algorithm combines the Levy flight random search mechanism with a 3D covering grid. Using this algorithm, the mobile robotic sensors randomly move on the vertices of the covering grid with the length of the movements follow a Levy fight distribution.\\
This report studies the problem of 3D formation building in 3D spaces by a network of mobile robotic sensors. Decentralized consensus-based control law for the multi-robot system which results in forming a given geometric configuration from any initial positions in 3D environments is proposed. Then, a decentralized random motion coordination law for the multi-robot system for the case when the mobile robots are unaware of their positions in the configuration in three dimensional environments is presented. The proposed algorithms use some simple consensus rules for motion coordination and building desired geometric patterns. Convergence of the mobile robotic sensors to the given configurations are shown by extensive simulations. Moreover, a mathematically rigorous proof of convergence of the proposed algorithms to the given configurations are given.

\singlespacing

%==========================CONTENTS====================================================================================
\singlespace    %switch to single spacing for table of contents etc
\setcounter{tocdepth}{2}
\tableofcontents
\decrementmtc
\listoftables
\listoffigures

\adjustmtc
\mainmatter

%===========================MAIN BODY==================================================================================
\newpage
\pagenumbering{arabic}
%                           PART I
%___________________________REVIEW_____________________________________________________

\singlespacing
\chapter{Introduction}\label{chap:introduction}
\minitoc
Wireless sensor networks, which are networks without fixed infrastructures, consist of a number of small, low-power, low-cost nodes called sensor nodes which can vary from few to thousands \cite{rawat2014wireless}. These nodes have the ability to interact with their environment through sensors, processing information and communicating this information with their neighbors and communicating via wireless channels over a short distance \cite{gkikopouli2012survey}. The wireless sensor networks have the task of monitoring their environment by detecting the event occurring in the target area \cite{sangwan2015survey}.\\
Nowadays, sensor and robotic technology are developed enough to inspire the idea of mobile robotic sensor networks. In mobile robotic senor networks, sensors attached to autonomous underwater robots (AUVs) or unmanned aerial vehicles (UAVs), can move vertically and horizontally in the ocean and air, and  each robot can be viewed as a sensing node \cite{wang2012three}.  Static sensors deployment may not be possible in some applications where the deployment area is not reachable due to the aggressive environment or existence of mines. Moreover, mobile sensor networks are more flexible in terms of deployment and exploration abilities over static sensor networks \cite{zhang2012marine, khedo2010wireless,mansour2014wireless,6412670}. Besides, mobile sensor networks can maximize coverage with limited hardware. Also, they reduce the cost of operation, especially in the three-dimensional environments, in that areas of interest are often large \cite{survey, wu2011wireless,qi2015multi}.\\
Mobile wireless sensor networks have been used to help or replace human in different applications, for instance, sensing, monitoring \cite{gu2006data,susca2008monitoring}, surveillance search and rescue \cite{baxter2007multi} operations as well as exploration in hazardous environments.  Furthermore, they  are used for  intrusion detection, oceanographic data collection \cite{wang2012spatiotemporal}, ocean sampling, pollution monitoring in coastal areas \cite{cortes2005coordination}, detection of terrorist threats to ships in ports \cite{wang2008intrusion}, mine reconnaissance \cite{lin2010novel}, disaster prevention like detection of tsunamis and sending warnings \cite{chinnadurai2015underwater}, ocean resource exploration \cite{pompili2009three, akkaya2009self}, detection of the origin of the fire \cite{penders2011robot}, its location \cite{casbeer2005forest}, temperature, and direction also, they can be used for pollution estimation in urban areas  \cite{jalalkamali2013distributed,clark2005cooperative}.\\
In the literature, most of the works on mobile wireless sensor networks dealt with two-dimensional settings where all sensor nodes are dispersed in a two dimensional plane. However, in the real world, most of the sensor networks are deployed in 3D spaces such as wireless sensor networks used on the trees of different heights in a forest, or in a building with multiple floors, or underwater applications for ocean sampling and tsunami warnings, aerial defence systems as well as atmosphere pollution monitoring. Deploying sensor networks in 3D environments introduces new challenges in terms of connectivity, coverage, mobility, communication and coordination \cite{dong2012algorithm}. Moreover, there are significant physical, technological, and economic differences between terrestrial and air/underwater sensor networks. Compared to those for 2D sensor networks, 3D deployments are more expensive, higher mobility and areas of interest in ocean/air environments are often large \cite{wang2012three}. As a result, deploying three-dimensional sensor networks  entail extra computational complexity, and many problems cannot be solved by extensions or generalizations of 2D methods \cite{huang2007coverage}. In this report, we study three-dimensional mobile sensor networks to develop new approaches for deployment of mobile sensor networks for coverage, search, and formation building in 3D spaces.\\
The centralized and decentralized approach are two commonly accepted methods for controlling a network of mobile sensors. The centralized approach is based on the assumption that a central station is available to control a whole group of agents \cite{bakule2008decentralized,kalantar2007distributed}. In contrast, the distributed approach does not require a central station for control. The distributed approach is believed more promising due to many unavoidable physical constraints such as limited resources and energy, short wireless communication ranges, narrow bandwidths, and large sizes of mobile sensors to manage and control \cite{yang2008multi}. Furthermore, the distributed method has many advantages in realizing cooperative group performances. Also, it is scalable and robust \cite{bullo2009distributed}.\\
As a distributed solution to multi-agent coordination, consensus problems was introduced in \cite{ vicsek1995novel, tsitsiklis1984distributed, lynch1996distributed}. In \cite{vicsek1995novel, ren2005consensus}, the authors  defined consensus as reaching  an agreement on a common value of  a certain quantity of interest means by  a team of agents through negotiating with their neighbours. For more details and progress reports about consensus see survey paper \cite{olfati2007consensus,ren2005survey} and references therein. In this report, we propose decentralized control algorithm for coverage, search and formation building in 3D spaces.\\
Mobile sensor networks are deployed in an area of interest either deterministically or randomly. The choice of the deployment scheme depends on the type of sensors, applications and the environment that the sensors will operate in \cite{younis2008strategies}. Deterministic node deployment is necessary when sensors are expensive or when their operation is significantly affected by their position or in a controlled and human-friendly environment where the shape of the network, the location of the sensor node is known in advance. In this type of deployment, mobile sensor nodes are deployed in a predefined manner and the placement of sensor nodes for coverage pattern is much easier than random coverage \cite{4146943}. On the other hand, for many cases, the randomized sensors placement becomes the only feasible option \cite{7052338}. This is particularly true for the cases where the location and topology change from time to time. Also, for large, dangerous and inaccessible region where human intervention is not possible such as a battle field or a disaster region, intrusion detection, disaster recovery and forest fire detection. In the random deployment scenarios, mobile sensors may be dropped by a low-flying airplane, helicopter or an unmanned aerial or underwater vehicle \cite{wang2015coverage}. As a result, a self-deployment is needed after being distributed in an area of interest, to organize sensor nodes and to place nodes in the appropriate positions \cite{boufares2015three}.\\ The autonomous deployment of mobile wireless sensor networks is a process in which the mobile sensor nodes adjust their positions dynamically according to a certain algorithm, until the predefined requirement is achieved \cite{miao20143d}. In this report, we assume the mobile sensor networks are randomly deployed in a given three-dimensional area. This initial deployment does achieve neither area coverage, nor network connectivity. Thus, we propose distributed self-deployment algorithms to achieve the predefined coverage, search and formation building requirements. Note that the relocation process should take into account the energy cost of mechanical movement and the communication messages involved in directing the motion \cite{younis2008strategies}.\\
The following section presents a literature review on the relevant topics in this report.
\section{Distributed 3D Complete Sensing Coverage and Forming Specific Shapes}
Coverage is an important issue in a sensor network, and is usually treated as a measure of the quality of service. Three types of coverage are defined by Gage: blanket coverage, sweep coverage and barrier coverage \cite {gage,cheng2009distributed, chen2010local}. In \cite{cheng2011decentralized,faizan,6496149,cheng2013decentralized}, a distributed law was presented for a group of self-deployed  mobile sensors, to perform  coverage in a two-dimensional space. A distributed algorithm for blanket coverage of a bounded two-dimensional area was studied in \cite{main} that drives the network of mobile wireless sensors to form a grid consisting of  equilateral triangles. All of the mentioned algorithms are based on the two-dimensional space assumption where all mobile sensors are dispersed in a two dimensional plane \cite{cheng2011decentralized,cortes2002coverage,savkin2015decentralized,bretl2013robust,cheng2012self,faizan,main,cheng2011decentralized1,5210103}. However, in real life, most of the mobile robotic sensor networks are deployed in three-dimensional spaces such as wireless sensor networks deployed on the trees of different heights in a forest, or in a building with multiple floors, or underwater applications for ocean sampling and tsunami warnings, aerial defense systems as well as atmosphere pollution monitoring. For two-dimensional sensor networks coverage means arranging mobile sensors in a line or in a plane, while in three-dimensional cases, in order to detect any penetrating object entering in specific region or gather scientific data from a given three-dimensional space, mobile  sensors should cover a three-dimensional space. \\
Deploying mobile sensor networks in three-dimensional environments introduce new challenges in terms of connectivity, coverage, communication and coordination \cite {dong2012algorithm}. Moreover, there are significant physical, technological, and economic differences between terrestrial and air/underwater sensor networks. three-dimensional networks require extra computational complexity, and many problems cannot be solved by extensions or generalizations of 2D methods \cite{ huang2007coverage}. Furthermore, the problem of self-deployment of a three-dimensional robotic sensor network is more challenging than a two-dimensional network in the sense that there is not any mathematically proved optimal way to tessellate three-dimensional spaces by the minimum number of robotic sensors \cite{huang2007coverage, liu2013dynamic}.\\
A centralized approximate solution to the problem of assigning sensors to grid positions in three-dimensional environments was presented in \cite{barrier2008,barrier2011}. In \cite{bai2009full, bai2009low}, the authors studied the problem of constructing connected and full covered optimal three-dimensional static wireless sensor networks. In \cite{akkaya2009self}, a distributed technique for self-reconfiguration of underwater wireless sensor networks was proposed. However, based on that method, the sensors can only move vertically to adjust depth for maximum coverage. The coverage problem for 3D sensor networks with high deployment density was studied by Watfa and Commuri \cite{watfa20063}. Their proposed scheme chooses a subset of sensors to be alive and provide full coverage to the whole 3D region. Their method covers the target area, but the node redundancy is the major flaws of this scheme. In \cite{pompili2006deployment}, the idea is to adjust the depths of sensors to provide coverage after their initial random deployment at the bottom of the ocean. The deployment is controlled by a central station as a result; the approach is not fully distributed.\\ Another common method of mobile sensor network deployment for coverage is virtual force algorithm which is based on virtual forces between robots \cite{ boufares2015three, zou2003sensor}. Using virtual repulsive and attractive forces the mobile robotic sensors are forced to move away or towards each other so that full coverage is achieved. This method has several undesirable limitations. Firstly, using this approach, all robots remain in constant motion and slowly converge to a steady-state. As a result, it depends on mobility, which is a high power consumption task.  Moreover, using this method equilibrium steady-state cannot always be predicted, or the system may converge to local-minima \cite{stirling2010energy}. In \cite{bartolini2014vulnerabilities}, Bartolini et al. presented the security vulnerabilities of deployment algorithms based on virtual force approach. For more about coverage consider the survey papers and references therein \cite{feng2014coverage, zhu2012survey, watfa2006coverage, wang2014overview}.\\
%%%%%%%%%%%%%%%%%%%%%%%%%%%%%%%%%%%%%%%%%%%%%%%%%%%%%%%%%%%%%%%%%%%%%%%%%%%%%%%%%%%%%%%%%%%%%%%%%%%%%%%%%%%%%%%%%%%%%%%%%%%%%%%%%%%%%%%%%%%%%%%%%%%%%%%%%%
In \cite {nazarzehi2015distributed}, we proposed a decentralized dynamic search coverage algorithm for coverage of 3D spaces. Using that algorithm, a few number of mobile wireless sensors should perform the coverage task for exploring large 3D spaces. As a result, that method is not a good option for the cases where we need to monitor a whole area at the same time or for time-sensitive coverage tasks such as radiation, chemical, and biological agent’s detection where using that method; the target region is covered only part of the time.\\
Our coverage problem is to deploy a network of mobile sensors for complete sensing coverage such that every point of a given bounded three-dimensional region is sensed by at least one mobile sensor. In this report, we present a distributed method for self-deployment of a team of mobile robotic sensors for complete sensing coverage in a bounded three-dimensional region. The mobile sensors use the proposed distributed control approach to communicating with their neighboring mobile robotic sensors to coordinate their motions and to archive a common goal. To cover a bounded three-dimensional area with the minimum number of mobile wireless sensors, we take advantage of vertices of a three-dimensional truncated octahedral grid proposed in \cite{alam2006coverage}. Unlike \cite{alam2006coverage}, here we assume that there is not a common coordinate system among all mobile sensors, therefore, we use consensus variables to build a common coordinate system for the mobile wireless sensors.\\
To verify the effectiveness of the proposed truncated octahedral grid, we  compare four different three-dimensional space filling polyhedral grids and we show that a truncated octahedral grid provides better 3\texttt{D} covering grid than other 3\texttt{D} grids because it can cover a bounded 3\texttt{D} area by a fewer number of vertices compare to the other grids. As a result, using this grid can minimize the time of coverage and the number of mobile sensors.\\
Without connectivity the mobile sensor nodes cannot exchange information with their neighbours. As a result, connectivity is an important issue in wireless sensor networks. A network is considered to be connected if there is, at least, one path between any two sensor nodes in the target area \cite{mansoor2014coverage}. Here, maintaining connectivity is essential for continuous data gathering from all mobile sensors. As a result, we propose a decentralized control algorithm to deploy the minimum number of mobile sensors for full coverage and connectivity in 3D areas.\\
The problem of self-deployment of mobile wireless sensors in three-dimensional spaces for either sampling or sensing  coverage of desired phenomena has enormous potential applications. For example, it can be applied to detect intruders  entering a protected region. Also, it can be used for environmental monitoring of disposal sites on the deep ocean floor \cite{chemical},  sea floor surveying for hydrocarbon exploration \cite{underwater}, sensing  biological and chemical phenomena in different depths, equipment monitoring and leak detection \cite{2012underwater}. Furthermore, the results of this part are also applicable to monitoring of any three-dimensional spaces, such as airborne applications, space exploration, greenhouse gas monitoring and storm tracking study \cite{alam2014coverage, al2013efficient}.\\
\section{Distributed 3D Dynamic Search Coverage}
The problem of search by mobile wireless sensor networks is an emerging field of research  which is very close to the dynamic coverage because in both cases the goal is to monitor a wide area by a few number of mobile sensors \cite{liu2013dynamic}. To accomplish the search task especially for exploring large, unknown environments, it is necessary to use  effective search algorithms \cite{couceiro2015overview}.
A majority of existing search methods relies on a single mobile sensor. Using a single mobile sensor for search of large 3D environments is not effective especially for the time sensitive search applications. In contrast, a network of mobile sensors  can considerably minimize the search time as well as it improves the robustness of the search process.
In recent years, several strategies have been presented for search in the literature. However, most of them are restricted to 2D spaces and not applicable for 3D environments  \cite{cheng2009distributed,baranzadeh2013decentralized}.
In \cite{bbb}, the authors investigated two search algorithms which are spiral search and informed random search. Both of the mentioned algorithms are based on 2D assumption also, they are heuristic. Furthermore, they assumed that the search area is known to the robots a priori. In \cite{farmani2015tracking}, Farmani et al. proposed a multi-target localization and tracking algorithm for multiple UAVs.\\
Some research focuses on partially unknown environments based on the assumption that the mobile sensors knows  the target locations a priori \cite{ijaz2006using,datta2004line} but here we assume that the mobile sensor network is not aware of the locations of targets. In \cite{keeter2012cooperative}, a random search algorithm based on Levy flight, for the search of targets in a bounded 3D environment was proposed. While \cite{bbb} is based on the assumption that the mobile sensor network knows the search environment and the targets' locations, in \cite{cai2013improved} it is assumed that the sensor network knows the number of the targets' but it is unaware of the targets' locations. Here, we consider the case where the deployed mobile sensor network knows neither targets' locations nor the number of targets and the targets are randomly placed around the search area. In this report, we introduce a distributed grid-based random algorithm for search in 3D environments by a mobile sensor network where the search environment is unknown to the mobile sensors a priori. To minimize the time of the search the mobile sensors should avoid repeated exploration . As a result, the mobile sensors communicate with their neighbours to share their information while doing search task. Based on this algorithm, each mobile sensor builds a map of the explored area and shares it with other sensors passing within its communication range to minimize the time of the search. The proposed algorithm is based on the decentralized communication method which is an appropriate option for the case where the mobile sensors have limited communication ranges. Here, the mobile sensor network accomplishes the search task using vertices of a cubic and a truncated octahedral grid. The presented search algorithm uses the vertices of the mentioned 3D grids for the search of a given 3D environment. We demonstrate that the truncated octahedral grid provides better 3D covering than a cubic grid, as a result using this grid can minimize the time of the search. A practical application of the proposed search algorithm is target searching in a three-dimensional aerial and aquatic environment. For example, it can be used for the search of sea mines, sources of pollution, black boxes from downed aircraft.
\section{Distributed Bio-inspired Random Search for Locating Static Targets}
In many coverage applications, it is not necessary to monitor the whole area of interest. Monitoring only some specific points in the area of interest is sufficient and acceptable. For such cases, we assume only a limited number of discrete objects or target points need to be monitored. As a result, full coverage of a given area is not required. In such cases, the deployment cost will decrease because of the smaller number of mobile sensors needed compared to the number required to cover the entire area. The problem of locating targets in both two and three-dimensional spaces is a common task in biological and engineering systems. To do the search task, a team of mobile robots equipped with sensing peripherals capability is deployed in the search area to locate targets like sea mines, black boxes from downed aircraft or ships, hazardous chemicals, fire spots in the jungle or to measure a concentration of dangerous substances \cite{sutantyo2013collective,nurzaman2009biologically,bernard2011autonomous}. Finding targets in large environments through the use of a single robot is inefficient, and using a team of mobile robots with local communication capability can succeed more quickly. Moreover, the success of robot teams searching for targets where targets availability is unknown and the condition of the environment is unpredictable depends on the efficiency of the search strategy \cite{fricke2013microbiology}. In environments where the distribution of targets is unknown, a priori or change over time randomized search strategies are more effective than deterministic search \cite{stephens1986foraging,acar2003path}.\\
Animals perform foraging activities when searching for food sources in nature with no knowledge of the environment and a limited sensory range. Biological creatures, marine predators, fruit flies, and honey bees who search for sources of food, for a new site and a mate are  believed to use the Levy flight random motion \cite{viswanathan2002levy,yang2009cuckoo,hereford2010bio}. Levy flight is a random walk mechanism that has the Levy probability distribution function in determining the length of the walk. Levy walk random search approach is easy to implement also it does not require communication between searchers. Using this random search method, a searcher performs the search task by searching  a small area of space, and then jump to a zone that is likely previously unexplored. This behaviour is similar to that of a foraging animal who searches for food in a given area with a series of small movements, then travels a larger distance to another area to search again. In fact, by performing Levy flight random walk mechanism, forager optimizes the number of targets encountered versus the travelled distance. The idea is that the probability of returning to the previous site is smaller compared to another random walk mechanism. The process of evolution and natural selection has led animals to optimize their foraging strategies. As a result, researchers in this new emerging area are finding much inspiration from biology.\\
In \cite{fricke2013microbiology,fricke2015distinguishing}, the authors proposed a search strategy for a team of mobile robotic sensors based on the T cell movement in lymph nodes and they showed that the distributions of the step-sizes taken by T cells are best described by a random walk with the Levy-like distribution. In \cite{sutantyo2013collective}, a bio-inspired random search strategy for the multi-robot system to efficiently localize targets in underwater search scenarios was proposed. Also, the authors showed that how a novel adaptation strategy based on the firefly optimization algorithm can improve the Levy flight performance particularly when targets are clustered. In \cite{sutantyo2010multi}, the authors combined the Levy flight bio-inspired search algorithm to an artificial potential field to improve robots dispersion in the environment to optimize search task by a team of mobile robots \cite{sutantyo2010multi}. In \cite{keeter2012cooperative}, a randomized algorithm, based on the Levy flight, for locating sparse targets in a bounded three-dimensional  aquatic or air environments environments was proposed. Yang et al. developed a decentralized control algorithm of swarm robot inspired by bacteria chemotaxis for target search and trapping \cite{ yang2015self}.\\
In all of the mentioned works, mobile robotic sensors search an area of interest using a random walk with a Levy-flight distribution. Most of the previous works were based on the assumption that targets are sparsely distributed in the search areas. For many real-world search applications, targets are clustered in a group such as schools of fish, underwater mines clustered in fields and human victims in disaster areas. Furthermore, some of the previous studies were restricted to two dimensions, and some of them did not consider cooperation and communication amongst mobile robotic sensors \cite{ghaemi2009computer}. In this report, we propose a novel decentralized grid-based Levy flight random search algorithm. Here, we combine the bio-inspired Levy flight random search mechanism with a 3D covering grid proposed in \cite{nazarzehi2015distributed} to optimize the search procedure.\\
In the former works, the Levy flight random walks took place in a continuous space, but here the random walk occurs on a discrete grid. Unlike \cite{nazarzehi2015distributed}, in this study the mobile robotic sensors move randomly on the vertices of the covering grid with the length of motion follow a Levy flight distribution. Based on this method, the Levy flight distribution will generate the length of the movement, while the vertices of the covering truncated octahedral grid will improve the dispersion of the deployed mobile robotic sensors. As a result, it optimizes the search task by reducing the search time. To stop the search and reduce the cost of the operation, each mobile sensors communicates to other mobile sensors within its communication range to broadcast information regarding detected targets \cite{vali}. Performance of this grid-based random search method is verified by extensive simulations also to evaluate the performance of the proposed grid-based Levy flight algorithm we compare it to other random search strategies for locating sparse and clustered distributions of targets. Furthermore, we give a mathematically rigorous proof of the convergence of the proposed algorithm with probability 1 for any number of mobile robotic sensors and targets.
\section{Distributed Bio-inspired Algorithm for Search of Moving Targets}
The success of a network of mobile robotic sensors searching for targets in an unknown three-dimensional environment depends on the efficiency of the employed search strategy \cite{stevens2013autonomous}. There are two target search approaches named random search and deterministic search. Random search is just opposite to the determined search in that no predefined information is available about the location of moving targets, their distribution and movement pattern. In environments where the distribution of targets is unknown a priori or changes over time randomized search strategies are more effective \cite{stephens1986foraging,acar2003path}.\\
Viswanathan et al. \cite{viswanathan2000levy} discovered that the foraging behaviour of animals follows a certain random search strategy that is an optimal Markovian search strategy based on Levy processes. Markovian search strategies have many advantages as they are easy to implement and do not require communication between independent searchers. The Levy walk (Levy flight) is the trajectory of the searcher containing piecewise linear pieces of Levy distributed lengths, and has been exhibited in several natural systems. Based on the Levy walk method, a mobile robotic sensor searches a small area of space and then jumps to a region that is not previously explored and starts again. This behavior is similar to that of a foraging animal who searches for food in a given area with a series of small movements, then travels a larger distance to another area to search again.\\
In the bio-inspired Levy flight search approach, the length of the movement can be determined by the distribution of the targets. As a result, this method is flexible and suitable for the scenario where the animal has to adapt to the environmental changes when resources are sparsely distributed within unpredictable environments \cite{liu2013dynamic}. The process of evolution and natural selection has led animals to optimize their foraging strategies. As a result, in recent years, many search algorithms have been inspired by the Levy walk random search method \cite{fricke2013microbiology, fricke2015distinguishing}. In \cite{sutantyo2013collective}, the authors proposed a bio-inspired random search strategy for the multi-robot system to efficiently localize targets in underwater search scenarios. Also, they showed how a novel adaptation strategy based on the firefly optimization algorithm can improve the Levy flight performance particularly when targets are clustered. In \cite{sutantyo2010multi}, the authors combined the Levy flight bio-inspired search algorithm to an artificial potential field scheme to improve robots dispersion in the environment to optimize search task by a team of mobile robots \cite{sutantyo2010multi}. In \cite{keeter2012cooperative}, a randomized algorithm, based on the Levy flight, for locating sparse targets in a three-dimensional bounded aquatic or air environments was proposed. In \cite{stevens2013autonomous}, the authors show that using the Levy distribution exhibits advantages over a continuous sweeping search for moving targets in 2D environments. Most of the previous works are restricted to two dimensional spaces, and some of them do not consider cooperation and communication amongst robots \cite{ghaemi2009computer}. Furthermore, all of the former studies are based on the assumption that targets are stationary. Therefore, their results are not  applicable for detecting mobile targets in three dimensional spaces. There are many examples of moving targets such as a running vehicle or an intruder entering a protected area. In this report, we study the problem of detecting mobile  targets moving randomly in a bounded 3D environment by a network of mobile sensors \cite{nazarzehi2015distributed1}.
\section{Decentralized Three-dimensional Formation Building}
The problem of 3D formation building of a network of mobile robotic sensors capable of local communication is to coordinate that network of mobile robots such that they can realize some desired formations and keep it over time \cite{cao2013overview, barca2013swarm}. A formation of mobile robots can be used to inspect oil pipelines in underwater environments where a team of mobile robots is required to build a desired formation and hover above the pipeline at different depths to monitor pipelines and detect any abnormalities \cite{ismail2012dynamic}. In such missions, a formation of autonomous robots is preferred over a single robot as it can provide better redundancy and flexibility. Moreover, it can accomplish the monitoring task more rapidly and cost-effectively. Furthermore, a group of mobile robots moving together in a given configuration can form an efficient data acquisition network for gas and oil exploration, environmental monitoring in both aerial and underwater environments, security patrols, mapping of the seabed, search and rescue in hazardous environments \cite{borhaug2006cross,jouvencel20123d,yintao2012path,jouvencel2010coordinated,turpin2012decentralized}. Besides, such formation can be used in airborne applications such as space exploration, storm tracking, greenhouse gas monitoring and air quality monitoring \cite{alam2014coverage}.\\
There are two approaches to the problem of formation control of a group of mobile robots \cite{belkhouche2011modeling,ou2014finite,chen2005formation}.  The centralized approach is based on the assumption that a leader is available to control the whole group of mobile robots. As a result, the leader should have access to all robots information for coordination of their motions and other robots follow the leader.  In contrast, the distributed approach does not require a central station for formation building. Decentralized scheme is preferable to the centralized ones for many reasons. First of all, it provides robustness with respect to mobile robots failures also they are scalable. Secondly, in practice, the sensing and communication ranges of autonomous robots are limited. Consequently,  mobile robots do not have access to all information of every robot in the workspace but only of the robots within their neighborhood. Therefore, the design of the controller for each robot has to be based on the local information. In recent years, consensus strategies have been applied to achieve the formation control of multiple mobile robots, which focus on driving the kinematics of all mobile robots to a common value. Using this method, a team of mobile robots communicate to their neighbors to reach an agreement on a common value \cite {peng2013distributed}.\\  Decentralized formation control has been studied in a large number of recent publications \cite{jadbabaie2003coordination,savkin2010decentralized,savkin2004coordinated,savkin2010bearings,dimarogonas2007decentralized,listmann2009consensus,peng2013distributed,savkin2016distributed}. In \cite{savkin2013method, baranzadeh2015decentralized, baran}, a distributed formation building algorithm for formation building in a given geometric pattern has been proposed. Most of the above-cited papers on formation control are based on the 2D assumption. This assumption is justified for applications where mobile robots are deployed on the earth surface. Therefore, this 2D space hypothesis may no longer be valid in the sea and air where mobile robots are distributed over a 3D space.\\ Recently, the problem of formation  building in a 3D space has been  analysed in \cite{borhaug2007straight,scardovi2007stabilization,ercan2010regular,zeng2011artifical}. However, in most of the previous works the mobile robots were modelled as simple linear models, which  have not taken into account constraints on the robot's linear and angular velocities. Distributed Geodesic Control Laws for flocking of nonholonomic agents in 3D environments was proposed in \cite{moshtagh2007distributed}. By using that algorithm, all mobile robots will finally move in the same direction without any predefined configuration. In \cite{robio}, we presented preliminary results of decentralized 2D formation building in three-dimensional environments by some simulations. In this note, we extend the results in \cite{robio} by extension from a 2D formation building to a 3D formation building in 3D environments. In particular, here we introduce a random formation building algorithm for anonymous robots. In this report, we present a decentralized consensus-based control algorithm for the problem of 3D formation building for nonholonomic mobile robots in 3D environments. Based on the proposed algorithm, the mobile robots only communicate with their neighbors located in their communication range to build and maintain a desired 3D configuration. As a result, this method is a good option for the cases where the mobile robots have limited communication range due to the economic reasons or because of physical constraints (underwater environments). Unlike previous works \cite{roussos20103d}, using this algorithm, the mobile robots move in not only the same direction in 3D spaces but also they finally build a given 3D geometric configuration and move with the same speed. Furthermore, in this study the motion of each robot is modeled by a nonlinear model. Also, we take into account the standard constraints on robots' angular acceleration and linear velocity \cite{matveev20143d}. The performance of the proposed three dimensional formation building algorithm is validated  by numerical simulations also, we give mathematically rigorous proof of convergence of the proposed algorithms to the given geometric configurations.\\
 Here, we study the problem of formation building for the case when the mobile robots are unaware of their positions in the configuration. We present a decentralized random motion coordination law for a network of mobile robots so that the mobile robots reach consensus on their positions and form a desired three dimensional geometric pattern from any initial position.\\

%%%%%%%%%%%%%%%%%%%%%%%%%%%%%%%%%%%%%%%%%%%%%%%%%%%%%%%%%%%%%%%%%%%%%%%%%%%%%%%%%%%%%%%%%%%%%%%%%%%%%%%%%%%%%%%%%%%%%%%%%%%%%%%%%%%%%%%%%%%%%%

\vspace{-6mm}
\section{Report Contributions}\label{sec:report_contributions}
\vspace{-2mm}
The contributions of the report are described as follows:
\begin{enumerate}
\item  This report introduces a novel distributed control algorithm for self-deployment of mobile robotic sensor networks for complete sensing coverage of bounded three-dimensional spaces. Moreover, it develops a distributed control law for coordination of the mobile robotic sensors such that they form a given 3D shape at vertices of a truncated octahedral grid from any initial positions. Simulation results show that the 3D truncated octahedral grid outperforms other 3D grids in terms of complete sensing coverage time and a minimum number of mobile robotic sensors required to cover the given 3D region. The proposed control algorithms are based on the consensus approach that is simply implemented and computationally effective. The control algorithms are distributed, and the control action of each mobile sensor is based on the local information of its neighbouring mobile sensors. The effectiveness of the proposed control algorithms are confirmed by extensive simulations. Also, we give mathematically rigorous proof of convergence with probability 1 of the proposed algorithms.
\vspace{-2mm}
  \item The report proposes a set of distributed random laws for search coverage in unknown bounded three dimensional environments. The mobile sensors utilize a truncated octahedral grid for the search process. Furthermore, they exchange information regarding detected targets with their neighbouring mobile sensors to minimize the search time. The performance of the proposed algorithm is confirmed by extensive simulations and the convergence with probability 1 for any number of mobile sensors has been mathematically proved.
	\vspace{-2mm}
  \item This report presents a novel decentralized bio-inspired grid-based search algorithm to drive a network of mobile sensors for locating clustered and sparsely distributed targets in bounded 3D areas. The proposed method combines the bio-inspired Levy flight random search mechanism for determining the length of the walk with vertices of a covering truncated octahedral grid to optimize the search procedure by improving dispersion of the mobile robotic sensors. The proposed approach has the advantage that it does not need centralized control system also it is scalable. Moreover, this report compares the proposed bio-inspired grid-based random search scheme with other random search methods.
	\vspace{-2mm}
  \item This report considers the problem of detecting mobile targets moving randomly in a 3D space by a network of mobile sensors. It proposes a bio-inspired random search mechanism supplemented with a grid-based distribution strategy for the search in bounded 3D areas. Based on this algorithm, the mobile robots do the search task by randomly moving to the vertices of a common 3D covering grid from any initial position. The proposed algorithm uses some simple consensus rules for building a common covering grid. Also, it uses a random walk search pattern drawn by a Levy flight probability distribution to do the search task. The performance of the proposed algorithm is evaluated by simulations and comparison to the Levy walk random search method. Also, we give mathematically rigorous proof of convergence with probability 1 of the proposed algorithm.
		\vspace{-2mm}
  \item This report presents a decentralized control law for formation building in three dimensional environments. It  uses nonlinear standard kinematics equations with hard constraints on the robot angular and linear velocity to describe the robot motion in 3D spaces. Then, it proposes a random formation building algorithm for the problem of formation building with anonymous robots in 3D environments. This algorithm is an appropriate option for the case where robots do not know a priori its position in the configuration. The proposed control laws are based on the consensus approach that is simply implemented and computationally effective. The proposed control algorithms are decentralized and the control action of each robot is based on the local information of its neighbouring robots. Using this algorithm, all mobile robots eventually converge to the desired geometric configuration with the same direction and the same speed. The performance of the proposed decentralized control laws is confirmed by extensive simulations. Furthermore, it gives a mathematically rigorous analysis of convergence of the mobile robots to the given configurations.
	\vspace{-2mm}
  \item In this report, we have developed an effective, flexible and robust search scheme for target searching tasks, where the environment model and target locations are unknown to the mobile sensor network.
	\vspace{-2mm}
	\item Unlike many other studies in this area which present heuristic based strategies, in this report mathematically rigorous analysis is available for the proposed coverage, search and formation building algorithms.
	
	\end{enumerate}

\section{Report Organization} \label{sec:report_organizations}
The rest of this report is organized as follows:
\begin{itemize}
	\item \textbf{Chapter (2)} introduces a distributed random algorithm to drive a network of mobile robotic sensors on the vertices of a truncated octahedral grid for complete sensing coverage of a bounded 3D area. Moreover, it provides a decentralized algorithm for self-deployment of mobile robotic  sensors to form a desired 3D geometric shape on the vertices of a truncated octahedral grid.
	\item \textbf{Chapter (3)} presents a distributed random algorithm for search in bounded three dimensional environments by a network of mobile sensors. The proposed algorithm utilizes an optimal three dimensional grid pattern for the search task.
	\item \textbf{Chapter (4)} extends the distributed algorithm  proposed in the previous chapters by introducing a novel bio-inspired decentralized grid-based random search algorithm for finding randomly located objects in 3D areas by a mobile sensor network.
	\item \textbf{Chapter (5)} presents a distributed bio-inspired Levy flight algorithm for search of moving targets in three dimensional spaces. Using this algorithm, the mobile robots do the search task by randomly moving to the vertices of a common 3D covering grid from any initial position.
	\item \textbf{Chapter (6)} proposes a decentralized consensus-based control law for a team of mobile robots which result in forming a given geometric configuration from any initial positions in 3D environments. Then, it presents a decentralized random motion coordination law for the multi-robot system for the case when the mobile robots are unaware of their positions in the configuration in three dimensional environments.
	\item \textbf{Chapter (7)} summarizes the work presented in this report and gives directions for the future research.
\end{itemize}

%\singlespace
\singlespacing

\chapter{Distributed Complete Sensing Coverage and Forming Specific Shapes}\label{chap:1ValiCH1}
\minitoc
In this chapter, we present a decentralized control algorithm for self-deployment of mobile sensors for complete sensing coverage in a bounded three-dimensional region. The proposed decentralized control scheme uses distributed communication and control mechanisms for the cooperation of the mobile sensors. Based on this method, the mobile sensors exchange information using a wireless network \cite{new2}.  In recent years, distributed coordination of networks of mobile robotic sensors has attracted many researchers because based on this approach sensor networks are not required to access information from all other agents also the computational complexity of decentralized systems is independent of the number of agents in the network.  Furthermore, this approach is not sensitive to the loss of central leaders, and it scales well with increasing the number of mobile sensors. Moreover, it is flexible, robust and reliable. Besides, the decentralized control method is a good option for the cases where the mobile wireless sensors have limited communication ranges due to the economic reasons or because of physical constraints (underwater environments). Based on this scheme, the mobile sensors communicate with their neighboring mobile sensors to coordinate their motions and to archive a common goal .\\
To cover a bounded three-dimensional area with the minimum number of mobile wireless sensors, we take advantage of vertices of a three-dimensional truncated octahedral grid proposed in \cite{alam2006coverage}. Unlike \cite{alam2006coverage}, here we assume that there is not a common coordinate system among all mobile sensors, therefore we use consensus variables to build a common coordinate system for the mobile sensor network. To verify the effectiveness of the proposed truncated octahedral grid, we  compare four different three dimensional space filling polyhedral grids and we show that the truncated octahedral grid provides better covering grid than other three-dimensional grids because it can cover a bounded three-dimensional area by the fewer number of vertices compare to the other grids. As a result, using this grid we can cover the target area quicker and by less number of mobile sensors compare to the other grids.\\
At first part of this chapter, we introduce a distributed random algorithm for complete coverage of an unknown bounded three-dimensional area. At second part, we develop a distributed control law for coordination of motion of the mobile sensors such that they form a given three-dimensional shape at vertices of a truncated octahedral grid from any initial positions. The performance of the proposed algorithms are verified by simulations also we give a mathematically rigorous proof of the convergence of the proposed algorithm with probability 1 for any number of mobile sensors and any sensor initial conditions \cite{nazarzehi2015decentralized}.\\
The rest of this chapter is organized as follows: The problem is formulated in Section \ref{2.1}. In Section \ref{2.2}, a distributed random  algorithm for self-deployment of a mobile robotic sensor network on vertices of a truncated octahedral grid in a bounded region is presented. In Section \ref{2.4}, the results of section \ref{2.2} are developed for self-deployment of mobile wireless sensors in a given three-dimensional shape. Numerical simulation results that validate the proposed algorithms are presented in Section \ref{2.3}, Section \ref{2.5}. Conclusion is summarized in Section \ref{2.6}.

\begin{figure}
\centering
{\includegraphics[scale=0.7]{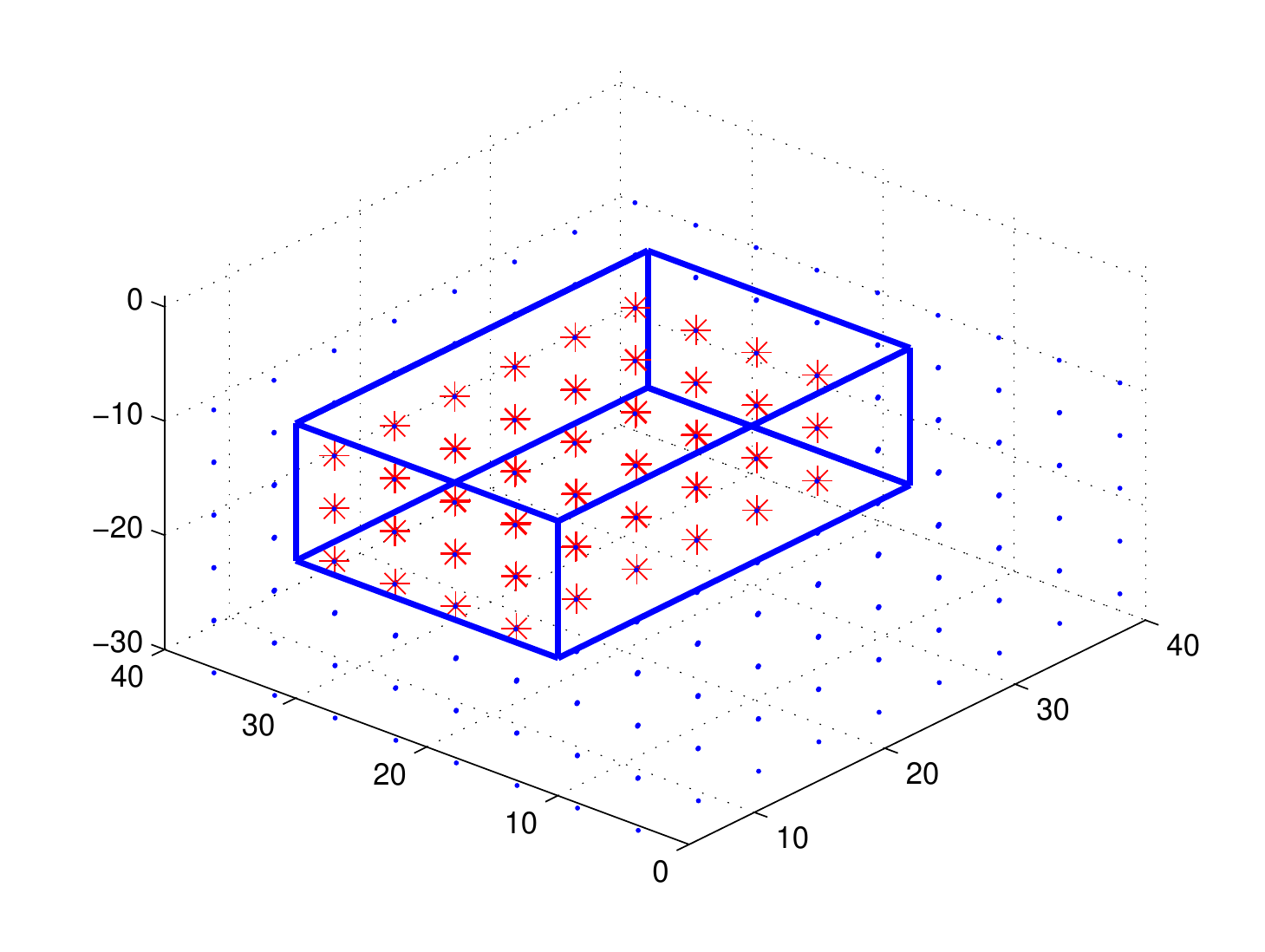}}
\caption{Target region($\texttt{M}$) represented by a cube, the set $\texttt{V}$ denoted by $’.’$ and the set $\hat{\texttt{V}}$ by $*$.}
 \label{MonitoringRegion}
\end{figure}

\begin{figure}
\centering
{\includegraphics[scale=0.85]{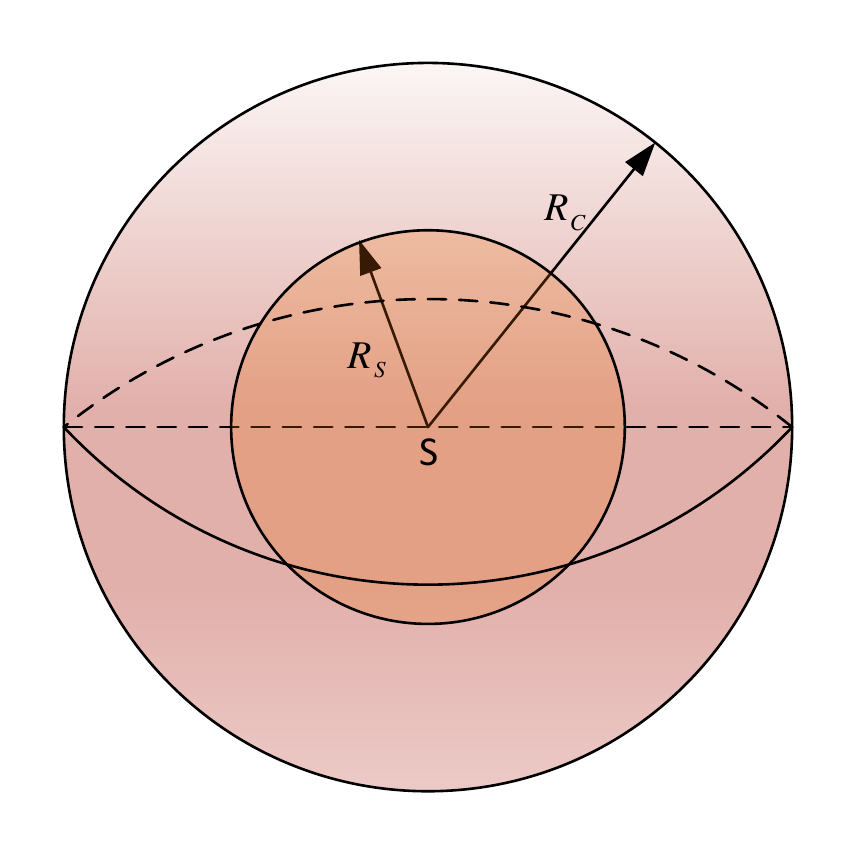}}
\caption{Sensing range denoted by $R_s$ and communication range by $R_c$}
 \label{sensor}
\end{figure}

%SAMPLE FROM HAS
\begin{figure*}[t!]
\begin{center}
\mbox{
\subfigure[]{
{\includegraphics[width=0.48\textwidth]{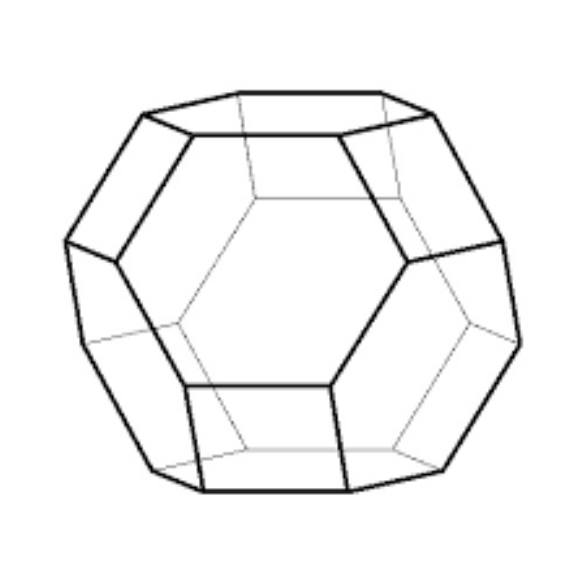}}\quad
\label{trunc1}
}
\hspace{-2mm}
\subfigure[]{
{\includegraphics[width=0.48\textwidth]{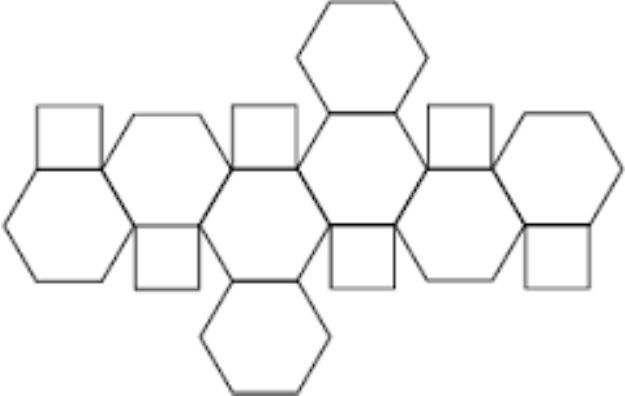}}\quad
\label{trunc2}
}
}
\vspace{-4mm}
\caption{A truncated octahedron.}
\vspace{-8mm}
\label{trunc}
\end{center}
\end{figure*}
%%%%%%%%%%%%%%%%%%%%%%%%%%%%%%%%%%%%%%%%%%%%%%%%%%%%%%
\begin{figure*}[t!]
\begin{center}
\mbox{
\subfigure[]{
{\includegraphics[width=0.48\textwidth]{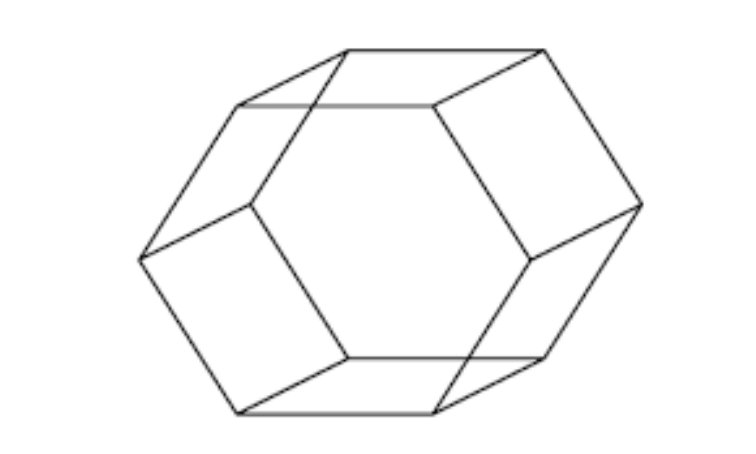}}\quad
\label{hex1}
}
\hspace{-2mm}
\subfigure[]{
{\includegraphics[width=0.48\textwidth]{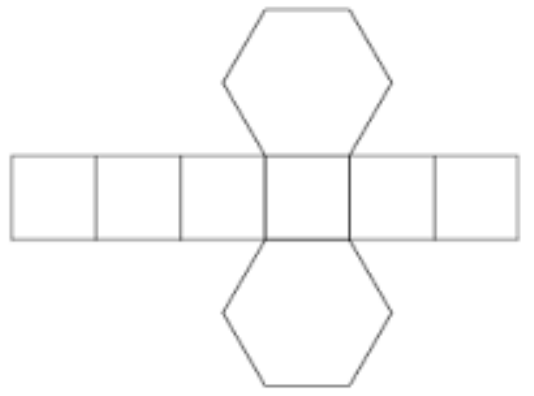}}\quad
\label{hex2}
}
}
\vspace{-4mm}
\caption{A hexagonal prism}
\vspace{-8mm}
\label{hex}
\end{center}
\end{figure*}

%%%%%%%%%%%%%%%%%%%%%%%%%%%%%%%%%%%%%%%%%%%%%%%%%%%%%%%%%%%%%%%%%%%%%%
\begin{figure*}[t!]
\begin{center}
\mbox{
\subfigure[]{
{\includegraphics[width=0.48\textwidth]{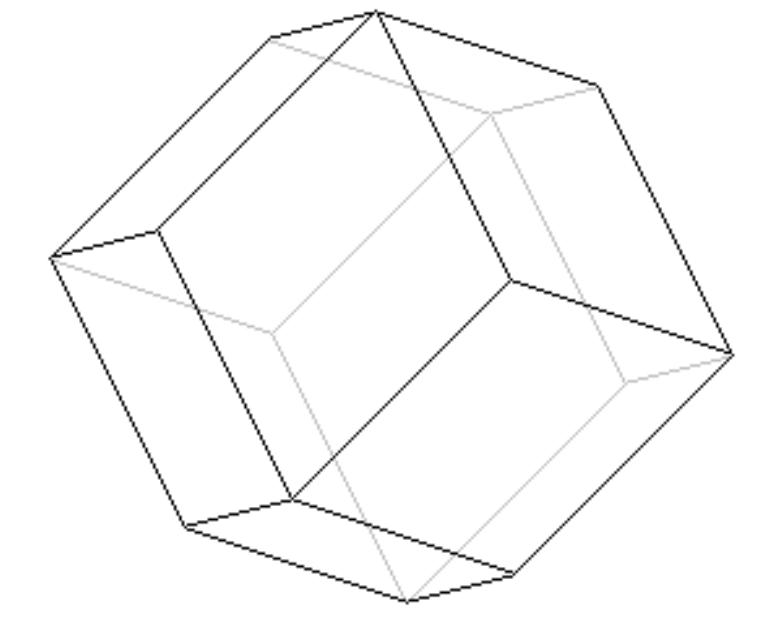}}\quad
\label{rhomb1}
}
\hspace{-2mm}
\subfigure[]{
{\includegraphics[width=0.48\textwidth]{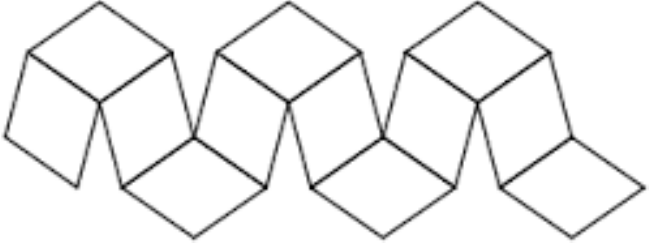}}\quad
\label{rhomb2}
}
}
\vspace{-4mm}
\caption{A rhombic dodecahedron}
\vspace{-8mm}
\label{rhomb}
\end{center}
\end{figure*}

%%%%%%%%%%%%%%%%%%%%%%%%%%%%%%%%%%%%%%%%%%%%%%%%%%%%%%%

%\begin{figure}[h]
%\centering
%\mbox{\subfigure[]{\includegraphics[width=3.6cm,height=2.8cm]{FigVali/coverage/hp1}}} \quad
%\subfigure[]{\includegraphics[width=3.5cm,height=2.2cm]{FigVali/coverage/hp2} }
%\caption{A hexagonal prism}
 %\label{hex}
%\end{figure}
%\begin{figure}[h]
%\centering
%\mbox{\subfigure[]{\includegraphics[width=3.4cm,height=2.8cm]{FigVali/coverage/rd1}}} \quad
%\subfigure[]{\includegraphics[width=3.8cm,height=2cm]{FigVali/coverage/rd2} }
%\caption{A Rhombic dodecahedron}
 %\label{rhomb}
%\end{figure}

\section{Problem Statement} \label{2.1}
Our coverage problem is to deploy a network of mobile sensors for complete sensing coverage such that every point inside a given bounded three-dimensional region is sensed by at least one mobile sensor. The objective of this chapter is to design a distributed random control algorithm to drive a team of mobile sensors for complete sensing coverage of a given bounded three-dimensional area.
The three-dimensional bounded regions $\texttt{M} \subset \texttt{R}^3$ to be covered is shown in Fig.\ref{MonitoringRegion}. We assume the mobile wireless sensor network consisting of n mobile sensors. Let $ p_i(.)\in \texttt{R}^3 $ be the Cartesian coordinates of the sensor $i$. We assume every mobile sensor has a limited sensing range, denoted by $R_s$. Sensing is omnidirectional and the sensing region of each mobile sensor can be represented by a sphere of radius $R_s$, having the mobile sensor at its centre, see Fig.\ref{sensor}.
\begin{assumption}
We assume a spherical sensing model such that each mobile sensor has a sensing range of $\texttt{R}_s>0 $ can reliably detect any object located within a distance of $\texttt{R}_s $ from the sensor. In other words, mobile sensor $i$ has the ability to identify objects in a sphere of radius $\texttt{R}_s $ defined by:
\begin{equation}
\texttt{S}_{i,\texttt{R}_s}=\{p\in \texttt{R}^3;\|(p-p_i (\texttt{K}))\| \leq \texttt{R}_s\}
\end{equation}
\end{assumption}
We assume every mobile sensor has a limited communication range, denoted by $R_c$. Communication is omnidirectional, and the communication region of each mobile sensor is represented by a sphere of radius $R_c$, having the mobile sensor at its centre, see Fig.\ref{sensor}.
\begin{assumption}
We assume a spherical communication model where each mobile sensor has a communication range of $\texttt{R}_c>0 $ can reliably communicate with any mobile sensors located within a distance of $\texttt{R}_c $ from the mobile sensor as shown in Fig.\ref{sensor}. Mathematically speaking, mobile sensor $i$ can obtain information on its neighbours in a sphere of radius  $\texttt{R}_c$ defined by:
\begin{equation}
\texttt{S}_{i,\texttt{R}_c}=\{p\in \texttt{R}^3;\|(p-p_i (\texttt{K}))\| \leq \texttt{R}_c\}
\end{equation}
\end{assumption}
\begin{definition}
The mobile sensors are called homogenous if they have the identical sensing ability, computational ability, and the capacity to communicate.
\end{definition}
Here, we assume the mobile sensors are homogenous. Also, we assume that the mobile sensors can detect the boundaries of the region $\texttt{M}$. For a given time $ \texttt{K} > 0 $, each mobile sensor communicates with its surrounding neighbours in the communication range at the discrete time instance $\texttt{K} = 0, 1, 2,..$ for the coordination of their motions \cite{kottege2011underwater}.
\begin{definition}
Space-filling polyhedron is a polyhedron that can be used to fill a three-dimensional space without overlaps or gaps.
\end{definition}
The figures of three space-filling polyhedral are shown in Fig.\ref{trunc}, Fig.\ref{hex} and Fig.\ref{rhomb}. As shown in Fig.\ref{trunc2}, a truncated octahedron has 14 faces (6 squares and 8 regular hexagonal), and the length of the edges of the squares and the hexagons are the same. Fig.\ref{hex2} shows that a hexagonal prism is a prism composed of six rectangular sides and two hexagonal bases. Moreover, as shown in Fig.\ref{rhomb2}, a rhombic dodecahedron is a convex polyhedron with 12 congruent rhombic faces.
As mentioned before, we assumed that the mobile sensors' communication and sensing models are spherical also it is clear that spheres do not tessellate three-dimensional spaces without overlap or gap.
\begin{figure*}[t!]
\begin{center}
\mbox{
\subfigure[Truncated octahedron and its circumsphere]{
{\includegraphics[width=0.45\textwidth]{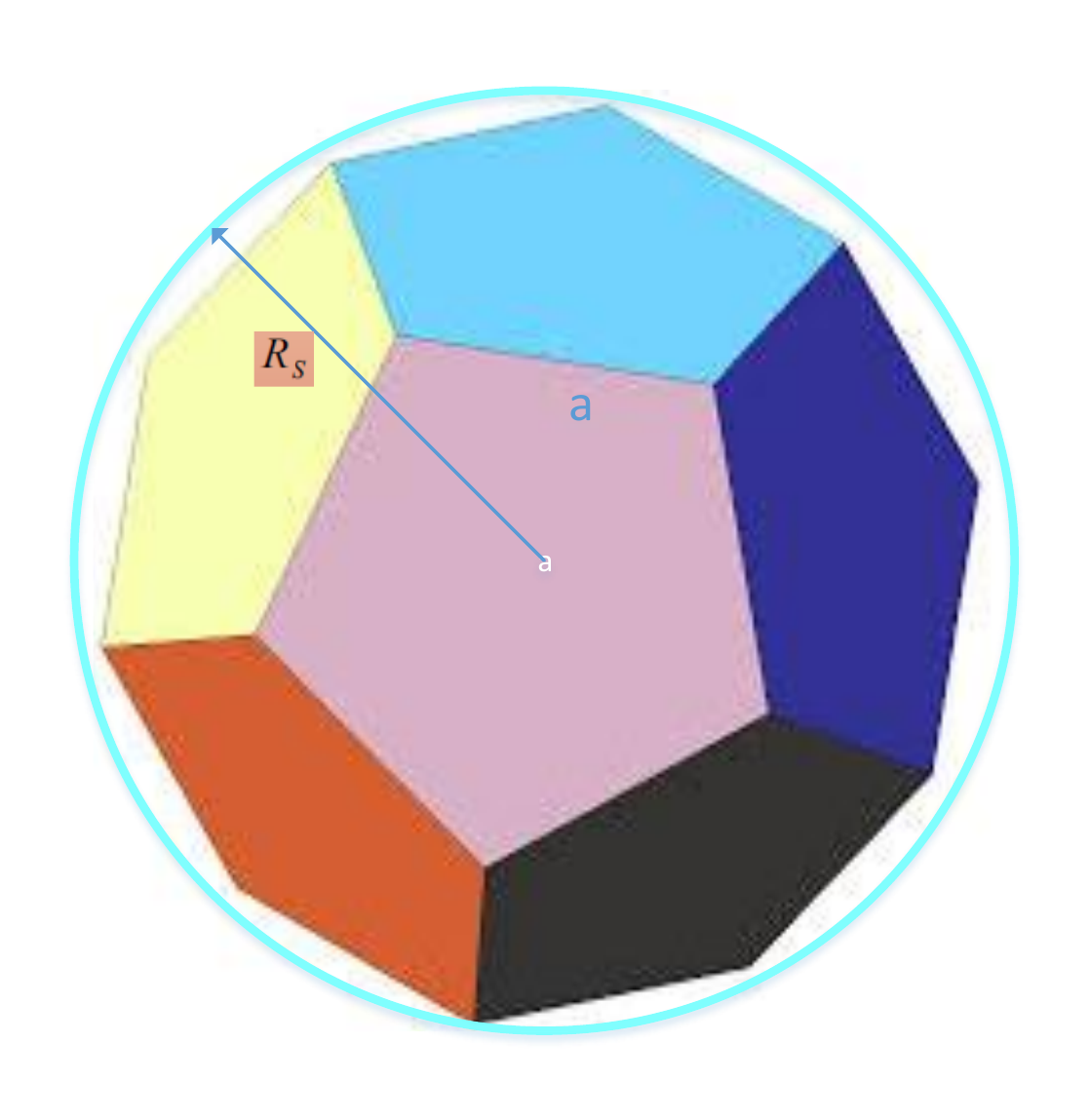}}\quad
\label{nn1}
}
\subfigure[Cube and its circumsphere]{
{\includegraphics[width=0.5\textwidth]{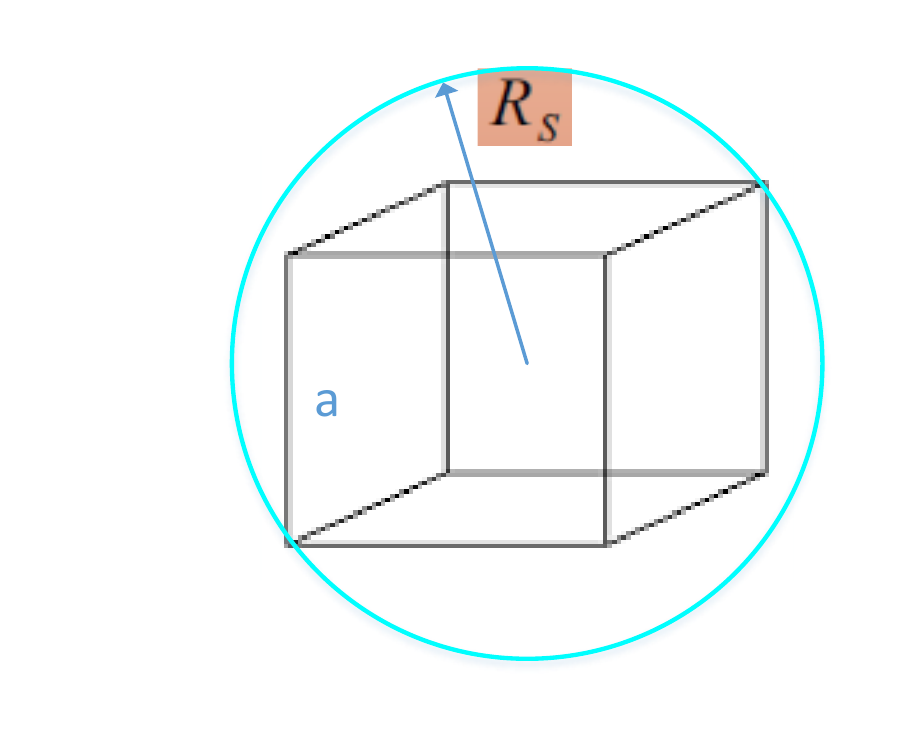}}\quad
\label{nn2}
}
}
\mbox{
\subfigure[Hexagonal prism and its circumsphere]{
{\includegraphics[width=0.5\textwidth]{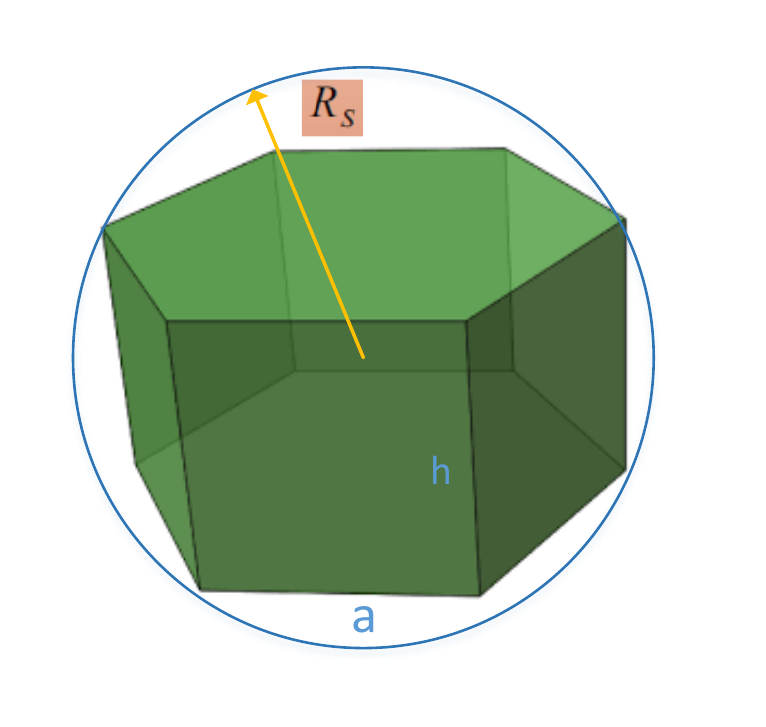}}\quad
\label{nn3}
}
\subfigure[Rhombic dodecahedron and its circumsphere]{
{\includegraphics[width=0.5\textwidth]{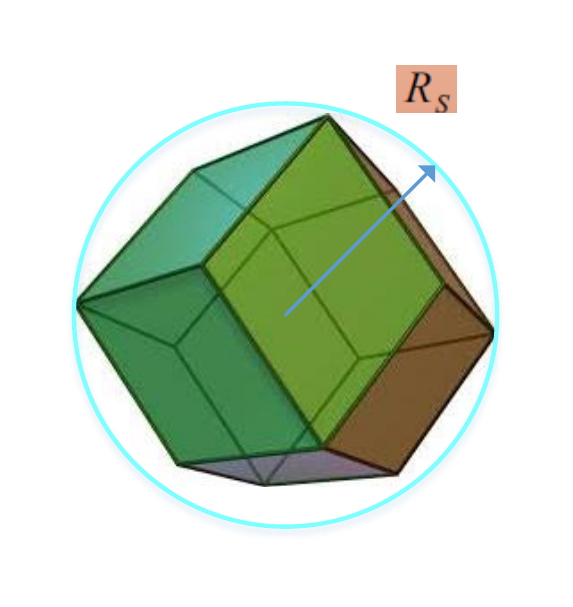}}\quad
\label{nn4}
}
}
\caption{Four space filling polyhedra with their circumspheres}
\vspace{-6mm}
\end{center}
\end{figure*}
\begin{table}
 \caption{Space filling polyhedrons' volumetric quotient}
 \label{vq}
 \begin{center}
 {\small
\begin{tabular*}{\columnwidth}[t]{cp{100pt}}\hline\hline
Space filling polyhedron & Volumetric quotient \\ \hline
$Truncated-octahedron$ & {\raggedright 0.68 } \\
$Cube$ & {\raggedright 0.36 } \\
$Hexagonal-prism$ & {\raggedright  0.47} \\
$ Rhombic-dodecahedron$ & {\raggedright 0.47 } \\
\hline\hline
  \end{tabular*}
 }%end of small
 \end{center}
\end{table}
Therefore, we introduce  a space-filling polyhedron that best approximates a sphere such that if each  mobile sensor is modelled by that polyhedron, then the number of mobile sensors required to cover a three-dimensional volume is minimized. Fig.\ref{nn1}, Fig.\ref{nn2},Fig.\ref{nn3} and Fig.\ref{nn4} show four space filling polyhera and their circumspher. Finding such optimal space filling polyhedrons is still an important open mathematical problem. This problem is similar to the century old  Lord Kelvin problem \cite{kelvin}. In 1887, Kelvin  proposed the truncated octahedrons as the solution to the problem of finding a space-filling arrangement of similar cells of equal volume with minimal surface area.
Different space filling polyhedrons can be compared using a criterion called volumetric quotient. The volumetric quotient is defined as \cite{alam}:
\begin{equation}
 V_q=\frac{V_p}{V_s}
\end{equation}
Where $V_p$ is the volume of a polyhedron and $V_s$ is the volume of its circumsphere. The volumetric quotient of a truncated octahedron is bigger than a cube, a hexagonal prism and  a rhombic dodecahedron as shown in Table.\ref{vq}. As a result, we expect that using a truncated octahedral grid for coverage leads to a fewer number of vertices to be covered and  minimizes the time of coverage and the number of mobile sensors for covering a three-dimensional space \cite{alam2008coverage}.
In the following, we assume that vertices of the truncated octahedral grid are the centres of truncated octahedrons of that grid.
%\textit{Definition 2.1}(Cubic covering set): Let $\texttt{V}$ be the infinite set of all vertices of a cubic grid cutting the monitoring region($\texttt{M}$) into equal cubes with the sides of $\frac{2}{\sqrt{3}}\texttt{R}_s>0$. The finite set of $\hat{\texttt{V}} =\texttt{V}\cap \texttt{M}$ is said a cubic covering set of $\texttt{M}$. (see Fig. 1).\\
\begin{definition}
Consider a truncated octahedral grid cutting $\texttt{M} $ into equal truncated octahedrons with the sides of $\frac{2 \texttt{R}_s}{\sqrt{10}}$. Let $\texttt{V}$ be the infinite set of centres of all the truncated octahedrons of this grid. The set ${\hat{\texttt{V}}}={\texttt{V}} \cap \texttt{M} $ is called a truncated octahedral covering set of $\texttt{M} $ (Fig.\ref{MonitoringRegion}).
\end{definition}

\begin{definition}
A sensor network is said to be connected if any sensor node can communicate with other sensor nodes directly or indirectly.
\end{definition}
Connectivity is a fundamental issue in wireless sensor networks because without connectivity the mobile sensor nodes cannot exchange information with their neighbours. For our coverage algorithm, connectivity is crucial as without connectivity the mobile sensors cannot reach consensus on the vertices of the common grid. As a result, they cannot build a common covering grid. For the truncated octahedron grid, the communication range must be at least 1.78 times the sensing range to maintain connectivity among nodes \cite{alam2008coverage}. Mathematically speaking:
\begin{equation}
 \frac{R_c}{Rs}\geq \frac{4}{\sqrt{5}}.
\end{equation}
It is worth to note that for the cases where the communication range is between 1.41 and 1.78 times the sensing range, then a hexagonal prism placement strategy or a rhombic dodecahedron placement strategy should be used to maintain connectivity in the network.\\
In this chapter, we present a decentralized control law to drive the mobile sensors to the vertices of the covering truncated octahedral  grid. The aim of this algorithm is to deploy the minimum number of mobile wireless sensors such that any point inside the bounded three-dimensional space is within the sensing range of at least one mobile sensor. In other words:
\begin{equation}
\forall \quad p \in \texttt{M},\quad  \exists \quad \texttt{v} \in \hat{\texttt{V}} \quad ; \|p-\texttt{v}\|\leq \texttt{R}_s
\end{equation}
It means that the monitoring region is covered completely if for any point inside the monitoring region $\texttt{M}$ there exists a mobile sensor in the sensing range of that point. Also, all mobile sensor can communicate with each other. First, we introduce a distributed random algorithm to drive the mobile sensors to the vertices of a truncated octahedral grid for complete coverage of an unknown bounded 3D area. Then, we propose a three-stage algorithm for coordination of motion of the mobile sensors, so that drive the mobile sensors on the vertices of a common truncated octahedral grid to form a desired 3D geometric shapes from any initial positions.\\
Let $N_i(\texttt{K})$ be the set of all mobile sensors $j$, $j\neq  i$ and $j\in {1, 2, . . ., n}$ that at time $\texttt{K}$ belong to the sphere $\texttt{S}_{i,\texttt{R}_c (\texttt{K})}$ also it is supposed that mobile sensor  $i$ has $|N_i(\texttt{K})|$ number of neighbours at time $\texttt{K}$. We employ the notion of graph to describe the relationships between neighbours, because  it  can change over time. For any time $\texttt{K}>0$, the relationships between neighbours are described by a simple undirected graph $\texttt{G}(\texttt{K})$ with vertex set ${1, 2, . . ., n}$ where $i$ corresponds to sensor $i$. The vertices $i$ and $j$ of the graph $\texttt{G}(\texttt{K})$, where $j\neq  i$, are connected by an edge if and only if, the mobile wireless sensors $i$ and $j$ are neighbours at time $\texttt{K}$. Here, we impose the following condition on the connectivity of the graph \cite{main, jadbabaie2003coordination}.
\begin{assumption} \label{assump:connec}
There exists an infinite sequence of contiguous, non-empty, bounded time intervals $[\texttt{K}_j, \texttt{K}_j+1)$, $ j = 0, 1, 2, . . .$  starting at $\texttt{K}_0 = 0$ such that across each $[\texttt{K}_j, \texttt{K}_j+1)$, the union of the collection $\{\texttt{G}(\texttt{K}) : \texttt{K} \in [\texttt{K}_j, \texttt{K}_j+1)\} $ is a connected graph \cite{main}.
\end{assumption}
 \section{Decentralized Algorithm for Self deployment of Mobile sensors for Complete coverage} \label{2.2}
In this section, we present an algorithm for self deployment of the mobile sensor network for complete sensing coverage of a bounded three-dimensional area. We use a node placement strategy proposed in \cite{alam2006coverage} to build a covering truncated octahedral grid's vertices and to find the locations where the mobile sensors should be placed.
\begin{equation}
V_1=( \texttt{x}+(2\alpha_1+\alpha_3)\frac{2 \texttt{R}_s}{\sqrt{5}},  \texttt{y}+(2\alpha_2+\alpha_3)\frac{2 \texttt{R}_s}{\sqrt{5}}, \texttt{z}+\alpha_3\frac{2 \texttt{R}_s}{\sqrt{5}})
\end{equation}

Here, $\alpha_1,\alpha_2$ and $\alpha_3 \in Z$ and $Z$ is the set of all integers. The distance between two neighbouring nodes is $\frac{2\texttt{R}_s}{\sqrt{5}}$. The inputs to our algorithm are sensing range $\texttt{R}_s$ and the coordinates of the point $(\texttt{X}, \texttt{Y}, \texttt{Z})$, which act as a seed for the growing the truncated octahedral grid. We use consensus variables to build a common coordinates $(\texttt{X},\texttt{Y},\texttt{Z})$ for all mobile sensors. As a result, we will build a common truncated octahedral covering grid for all mobile sensors without the need to know a common point a priori. Initially, the mobile sensors  do not have a common coordinate system. Therefore, we assume each mobile sensor has consensus variables $\texttt{X}_i(\texttt{K})$,$\texttt{Y}_i(\texttt{K})$ and $\texttt{Z}_i(\texttt{K})$ in its own coordinate system. Here, nodes (the mobile sensors) can reach a consensus on the frame of reference $(\texttt{X},\texttt{Y},\texttt{Z})$ by local estimation of distances to neighbours.\\
 We suppose that, for any mobile sensors, the vector $V_1=( \texttt{x}+(2\alpha_1+\alpha_3)\frac{2 \texttt{R}_s}{\sqrt{5}},  \texttt{y}+(2\alpha_2+\alpha_3)\frac{2 \texttt{R}_s}{\sqrt{5}}, \texttt{z}+\alpha_3\frac{2 \texttt{R}_s}{\sqrt{5}})) $ determines the location where nodes should be placed. It is obvious that any truncated octahedral grid is uniquely defined by a point $q_i(\texttt{K})=(\texttt{X}_i(\texttt{K}),\texttt{Y}_i(\texttt{K}),\texttt{Z}_i(\texttt{K}))$  and $\texttt{R}_s$. Thus, any  $q_i(\texttt{K})=(\texttt{X}_i(\texttt{K}),\texttt{Y}_i(\texttt{K}),\texttt{Z}_i(\texttt{K}))$ and $\texttt{R}_s$  uniquely define a truncated octahedral covering set of  $\texttt{M}$, which will be denoted as $\hat{\texttt{V}}[q,\texttt{R}_s]$. The scalar parameters  $ \alpha_1 $, $ \alpha_2 $,$ \alpha_3$ and $ \texttt{R}_s $ along with the three-dimensional consensus variables $q_i(\texttt{K})=[\texttt{X}_i(\texttt{K})\quad  \texttt{Y}_i(\texttt{K})\quad   \texttt{Z}_i(\texttt{K})]$ characterize the coordinates of a vertex of the grid. The mobile sensors will start with different values of the coordination variables  $\texttt{X}_i(0)$,$\texttt{X}_i(0)$ and $\texttt{Z}_i(0)$, then eventually converge to some consensus values $\texttt{X}_0$, $\texttt{Y}_0$ and $\texttt{Z}_0$ which define a common coordinate system for all mobile sensors.\\
Let  p  be a mobile sensor position, and $C[q](p)$ defines the closest vertices of the truncated octahedral grid $\hat{\texttt{V}}[q,\texttt{R}_s]$ to p. The following rules for updating the consensus variables and the mobile sensors’ coordinates are proposed:

\begin{equation}\label{eq:cons1}
\texttt{X}_i(\texttt{K}+1)=\frac{\texttt{X}_i(\texttt{K})+\sum_{\substack{
   j\in N_i(\texttt{K})
}}\texttt{X}_j(\texttt{K}) }
 {1+| N_i(\texttt{K}) |}
\end{equation}

\begin{equation*}
\texttt{Y}_i(\texttt{K}+1)=\frac{\texttt{Y}_i(\texttt{K})+\sum_{\substack{j\in N_i(\texttt{K})}}\texttt{Y}_j(\texttt{K})}{1+|N_i(\texttt{K})|}
\end{equation*}

\begin{equation*}
\texttt{Z}_i(\texttt{K}+1)=\frac{\texttt{Z}_i(\texttt{K})+\sum_{\substack{j\in N_i(\texttt{K})}}\texttt{Z}_j(\texttt{K})}{1+|N_i(\texttt{K})|}
\end{equation*}
\begin{equation}\label{eq:cons2}
p_i(\texttt{K}+1)=C[q_i(\texttt{K}),\texttt{R}_s](p_i(\texttt{K}))
\end{equation}
%\begin{figure}
%\centering
%\mbox{\subfigure[Horizontal barrier]{\includegraphics[scale=0.3]{hor3}}\quad
%\subfigure[Vertical barrier ]{\includegraphics[scale=0.3]{ver3} }}
%\label{figor10}
%
%\caption{Consensus over a triangular grid after 8 steps}
 %\label{figor7}
%
%\end{figure}
%\begin{figure}
%\centering
%
%\subfigure[ Control law 4,5: Sensor’s movement to the  vertices of the cubic grid]{\includegraphics[scale=0.55]{cube1} }
%
%
%\subfigure[ Control law 6: Sensors’ movement to cover unoccupied  vertices of the cubic grid]{\includegraphics[scale=0.55]{cube2} }
%
%\caption{Sensors’ movement to cover the monitoring region, sensors denoted by *, vertices of the cubic grid
 %represented by o
%}
 %\label{figor8}
%
%\end{figure}

%
The rules (~\ref{eq:cons1}) and (~\ref{eq:cons2}) are local rules for updating the consensus variables of each mobile wireless sensor  based on the average of its own and the consensus variables values of its neighbours. In other words,  mobile sensors use the consensus variables to achieve the consensus on the vertices of the truncated octahedral grid. Finally, the mobile wireless sensors converge to the vertices of a truncated octahedral grid.\\
\begin{theorem}\label{TH1}: Suppose that assumption ~\ref{assump:connec} hold and the mobile sensors move according to the distributed control algorithm (~\ref{eq:cons1}), (~\ref{eq:cons2}). Then, there exists a truncated octahedral grid $\hat{\texttt{V}}$ such that for any of the mobile wireless sensors there exists a $\hat{\texttt{v}}\in\hat{\texttt{V}} $ such that:
\begin{equation}
lim_{\texttt{K}\rightarrow\infty} p_i\left(\texttt{K}\right)= \hat{\texttt{v}}
 \end{equation}
\end{theorem}
Based on the law (~\ref{eq:cons1}) and the assumption ~\ref{assump:connec}, the consensus variables converges to some constant values as follows:\\$ \texttt{X}_i\left(\texttt{K}\right)\rightarrow \texttt{X}_0$, $ \texttt{Y}_i\left(\texttt{K}\right)\rightarrow \texttt{Y}_0$ and  $ \texttt{Z}_i\left(\texttt{K}\right)\rightarrow \texttt{Z}_0$.\\(See \cite{jadbabaie2003coordination,cao} for the proof of convergence of consensus variables to some constant values). Since $\texttt{R}_s$ ,$\alpha_1$ ,$\alpha_2$ and $\alpha_3$ are common to all mobile sensors, therefore:
\begin{equation*}
( \texttt{x}_i+(2\alpha_1+\alpha_3)\frac{2 \texttt{R}_s}{\sqrt{5}},  \texttt{y}_i+(2\alpha_2+\alpha_3)\frac{2 \texttt{R}_s}{\sqrt{5}}, \texttt{z}_i+\alpha_3\frac{2 \texttt{R}_s}{\sqrt{5}})\rightarrow
\end{equation*}

%\begin{figure*}[t!]
%\begin{center}
%\mbox{
%\subfigure[Mobile sensors' initial deployment: mobile sensors denoted by *, monitoring region represented by a cube]{
%{\includegraphics[width=0.5\textwidth]{FigVali/coverage/initial}}\quad
%\label{init}
%}
%\subfigure[Mobile sensor’s movement to the  vertices of the Truncated octahedron based  grid]{
%{\includegraphics[width=0.5\textwidth]{FigVali/coverage/consensus}}\quad
%\label{cons}
%}
%}
%\mbox{
%\subfigure[Mobile sensors' moving towards vacant vertices, after  7 steps]{
%{\includegraphics[width=0.5\textwidth]{FigVali/coverage/spreading1}}\quad
%\label{spread1}
%}
%\subfigure[Complete sensing coverage after 17 steps]{
%{\includegraphics[width=0.5\textwidth]{FigVali/coverage/spreading}}\quad
%\label{spread}
%}
%}
%\caption{Mobile sensors’ movement to cover monitoring region (Truncated octahedron based grid).}
%\label{coverage}
%\vspace{-6mm}
%\end{center}
%\end{figure*}
\begin{equation}
( \texttt{x}_0+(2\alpha_1+\alpha_3)\frac{2 \texttt{R}_s}{\sqrt{5}},  \texttt{y}_0+(2\alpha_2+\alpha_3)\frac{2 \texttt{R}_s}{\sqrt{5}}, \texttt{z}_0+\alpha_3\frac{2 \texttt{R}_s}{\sqrt{5}}).
\end{equation}
 It means that, after convergence of the consensus variables to some constant values, all mobile sensors will have a common truncated octahedral grid.\\
Then, based on the rule (~\ref{eq:cons2}), all mobile sensors move to the vertices of the common truncated octahedral grid. As a result, at the end of this stage, all mobile sensors are derived to some vertices of the truncated octahedral grid. This completes the proof of Theorem ~\ref{TH1}. \\
By using this algorithm, many of the vertices of the truncated octahedral grid remain unoccupied. To occupy all vacant vertices of the grid, we use the following random algorithm for spreading mobile sensors in the whole area \cite{main}

\begin{equation}\label{eq: random coverage}
\texttt{p}((K+1),i)=\begin{cases}
s\quad with\quad probability\quad\frac{1}{\left\|\texttt{S}\left(p_i\left(\texttt{K}\right)\right)\right\|}\\
\quad    \forall \quad s \in \texttt{S}\left(p_i\left(\texttt{K}\right)\right)\\
\end{cases}
\end{equation}
Where $\texttt{S}\left(p_i\left(\texttt{K}\right)\right)$ is defined as a set consisting of $\texttt{V} = p_i(\texttt{K})$ and all its unoccupied neighbours belong to $\hat{\texttt{V}}$. Furthermore, let $\left|\texttt{S}\left(p_i\left(\texttt{K}\right)\right)\right|$ represents the number of elements in $\texttt{S}\left(p_i\left(\texttt{K}\right)\right)$. It is clear that $1 \leq \left|\texttt{S}\left(p_i\left(\texttt{K}\right)\right)\right| \leq 27$.\\  The relation (~\ref {eq: random coverage}) implies that the mobile wireless sensors move randomly to occupy unoccupied vertices of the truncated octahedral grid. This continuous until  all vertices of the grid are occupied. It is obvious that all unoccupied vertices will becomes covered if the number of mobile sensors be more than  the number of vertices of the covering truncated octahedral grid.
\begin{theorem}\label{th} Suppose $N_m$ and $N_{\hat{\texttt{v}}}$ be the number of mobile sensors and the number of vertices in $\hat{\texttt{V}}$, respectively and $N_m\geq N_{\hat{\texttt{v}}}$. Also, assume the mobile sensors move according to the  law (~\ref {eq: random coverage}). Then with probability 1 there exists a time $\texttt{K}_0\geq 0$ such that for any $\hat{\texttt{v}}\in \hat{\texttt{V}}$ the relationship $\hat{\texttt{v}} = p_i(\texttt{K})$ holds for some $i = 1, 2, . . . , n$ and all $\texttt{K}\geq \texttt{K}_0$.
\end{theorem}
Here, the random movement of the mobile sensors is Markov process as the mobile sensors follow the transition rule proposed in (~\ref {eq: random coverage}) that only depends on the current state of the mobile sensors, does not depend on their history. Moreover, the proposed control law defines an absorbing Markov chain. An absorbing Markov chain is a Markov chain in which every state can reach an absorbing state. An absorbing state is a state that, once entered, cannot be left. In this case, the state when different mobile sensors are located in different vertices of the covering grid is an absorbing state. These absorbing states can be reached from any initial state with a non-zero probability. This implies that with probability 1, one of the absorbing states will be achieved. This completes the proof of Theorem ~\ref{th}.
\begin{figure*}[t!]
\begin{center}
\mbox{
\subfigure[Mobile sensors' initial deployment: mobile sensors denoted by *, monitoring region represented by a cube]{
{\includegraphics[width=0.8\textwidth,height=0.65\textwidth]{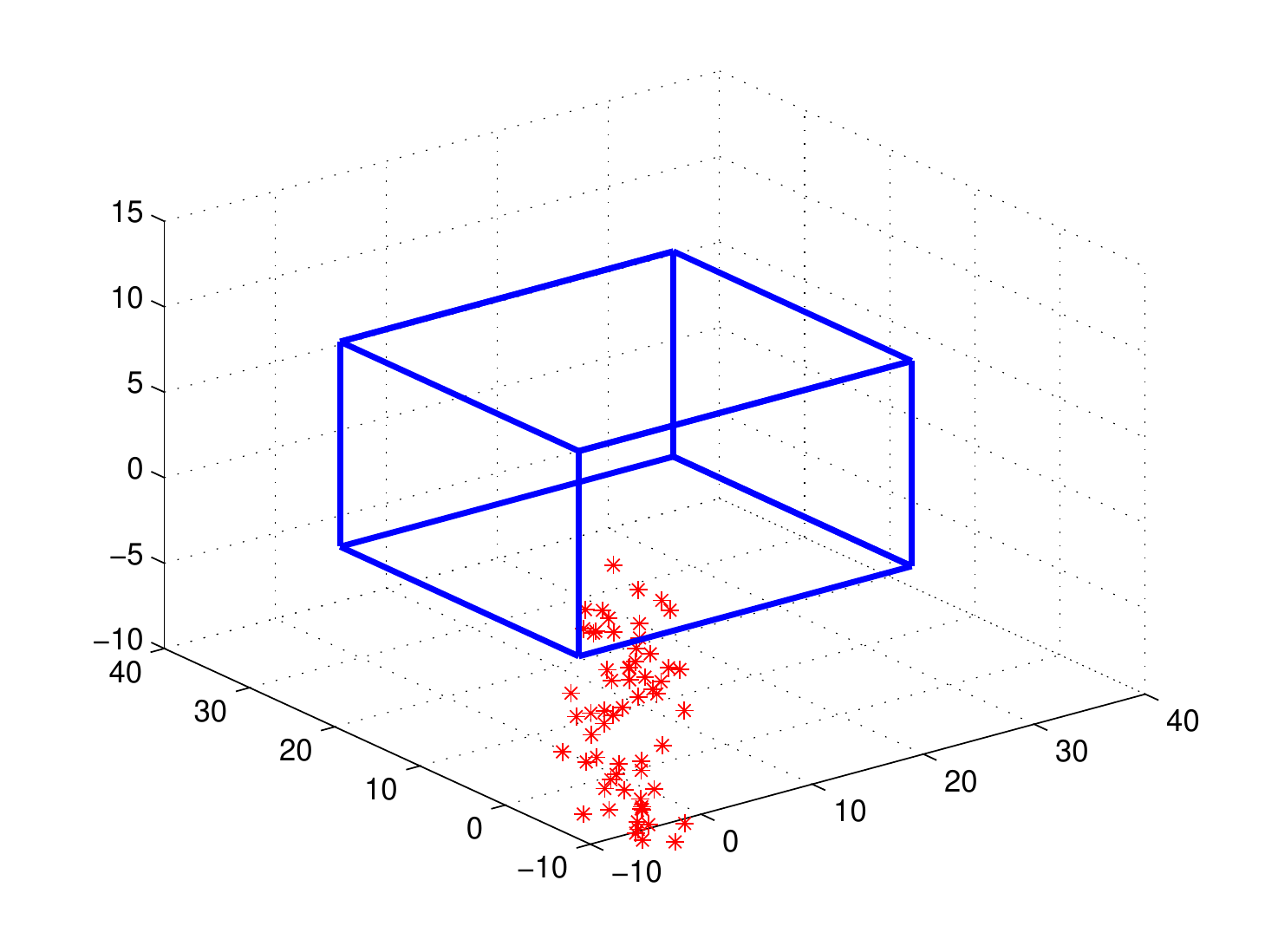}}\quad
\label{init}
}
}
\mbox{
\subfigure[Mobile sensor’s movement to the  vertices of the Truncated octahedron based  grid]{
{\includegraphics[width=0.8\textwidth,height=0.65\textwidth]{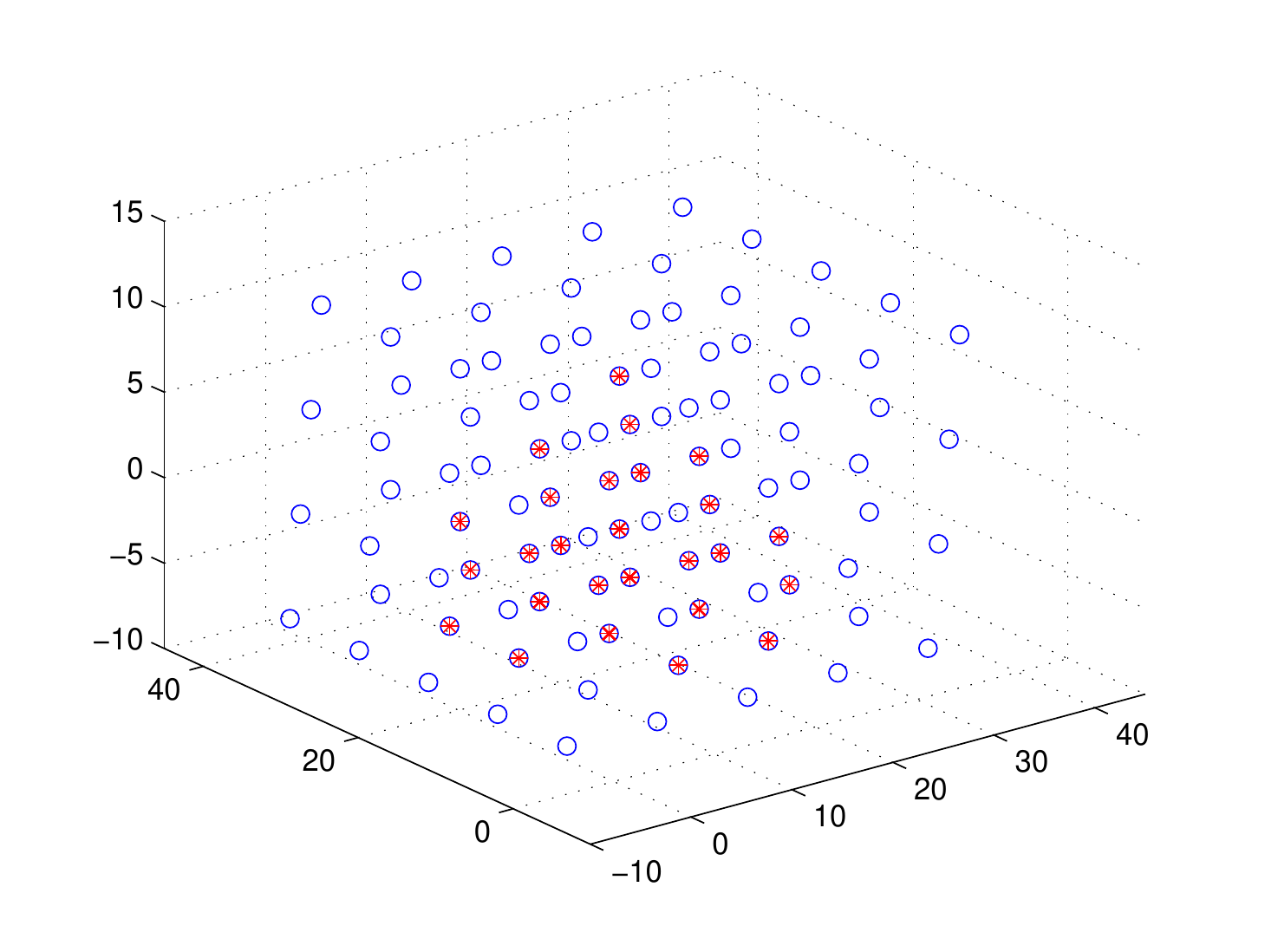}}\quad
\label{cons}
}
}
\caption{Consensus to the vertices of the common grid}
\label{coverage1}
\end{center}
\end{figure*}

\begin{figure*}[t!]
\begin{center}
\mbox{
\subfigure[Mobile sensors' moving towards vacant vertices, after  7 steps]{
{\includegraphics[width=0.8\textwidth,height=0.65\textwidth]{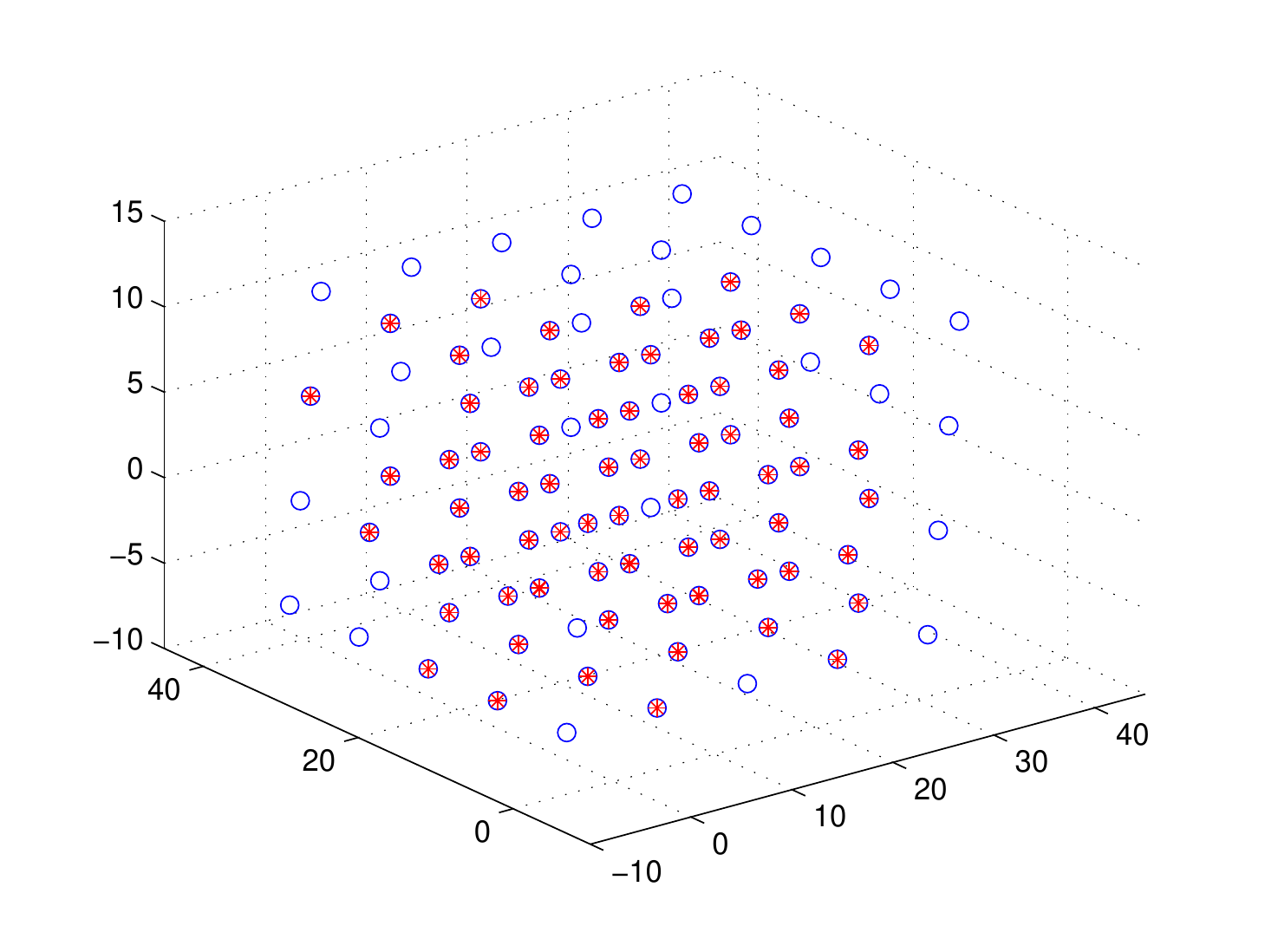}}\quad
\label{spread1}
}
}
\mbox{
\subfigure[Complete sensing coverage after 17 steps]{
{\includegraphics[width=0.8\textwidth,height=0.65\textwidth]{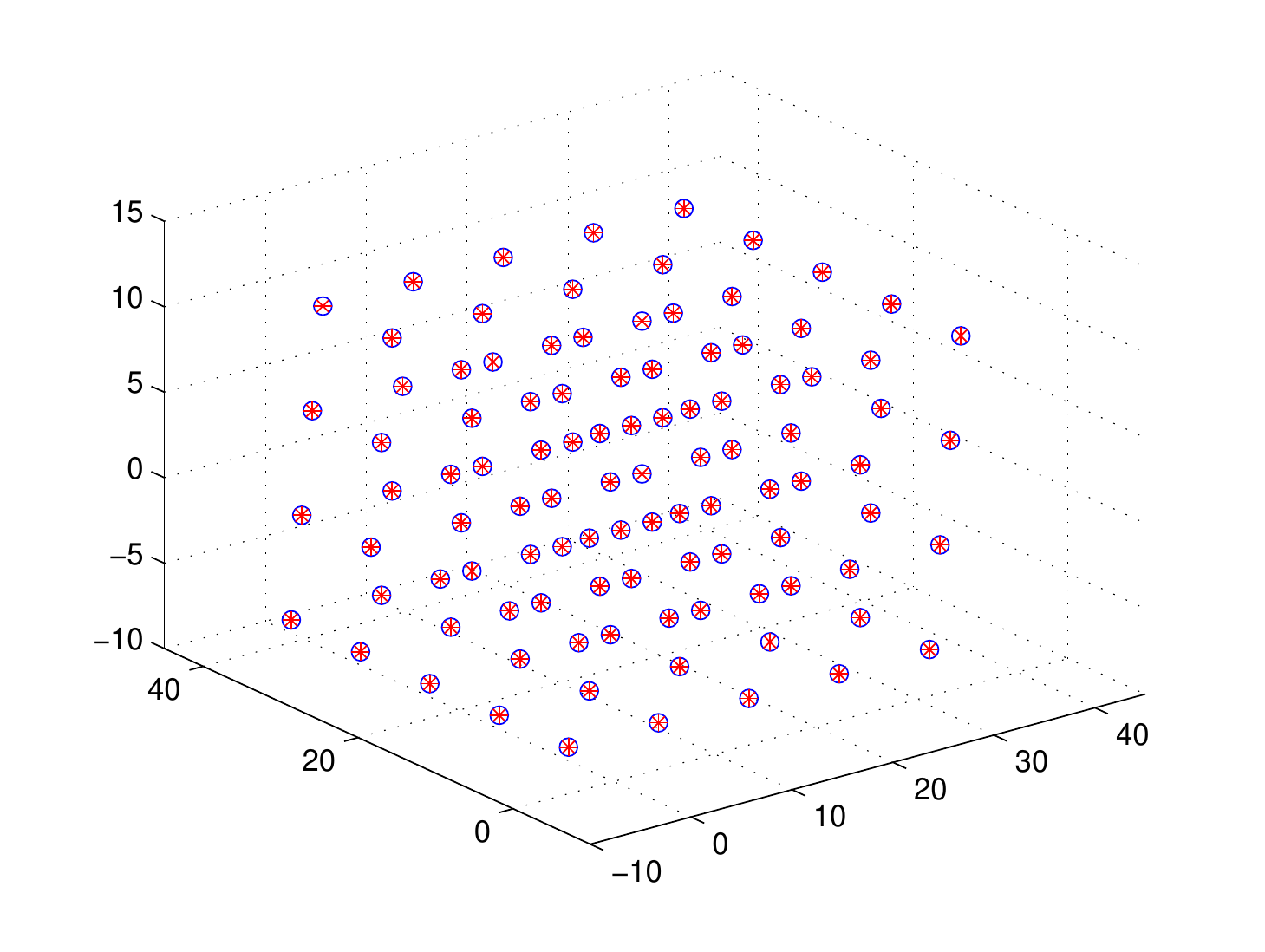}}\quad
\label{spread}
}
}
\caption{Mobile sensors’ movement to cover the monitoring region (Truncated octahedron based grid)}
\label{coverage2}
\end{center}
\end{figure*}
\subsection{Simulation Results} \label{2.3}
In this section, we present some examples of computer simulations to illustrate the proposed decentralized coverage algorithm. The simulations were carried out using MATLAB R2014b. In the first set of simulations, our objective is to obtain complete sensing coverage within a bounded three-dimensional space ($\texttt{M}$). Fig.\ref{coverage1}-Fig.\ref{coverage2} illustrate the evolution of the mobile wireless sensors' positions over time. We assume that 100 mobile wireless sensors are spread randomly around an area of interest as shown in Fig.\ref{init}, represented by $*$. For the initial configuration, it is assumed that the locations of these mobile sensors are randomly and independently distributed in a cube with a side of $10\texttt{R}_s$. It is clear that this initial random deployment doesn't achieve optimal coverage of the area of interest. Therefore, the mobile sensors use our proposed deployment algorithm to ensure the area of interest is covered, and the mobile sensor network is connected.\\
At first stage, the mobile wireless sensors move to the monitoring region according to the algorithms (~\ref{eq:cons1}),(~\ref{eq:cons2}) to form a  covering grid as shown in Fig.\ref{cons}. At the end of this stage, all mobile wireless sensors are at the vertices of the truncated octahedral grid. Next, the random algorithm (~\ref {eq: random coverage}) has been applied to spread mobile wireless sensors to obtain complete sensing coverage. Fig.\ref{spread1} shows the mobile wireless sensors’ locations after 7 steps. Finally, Fig.\ref{spread} represents the sensors location after 17 steps when all vertices of the truncated octahedral grid have been occupied, and the region $\texttt{M}$ has been completely covered by the mobile wireless sensors.\\
\begin{figure*}
\centering
{\includegraphics[scale=0.7]{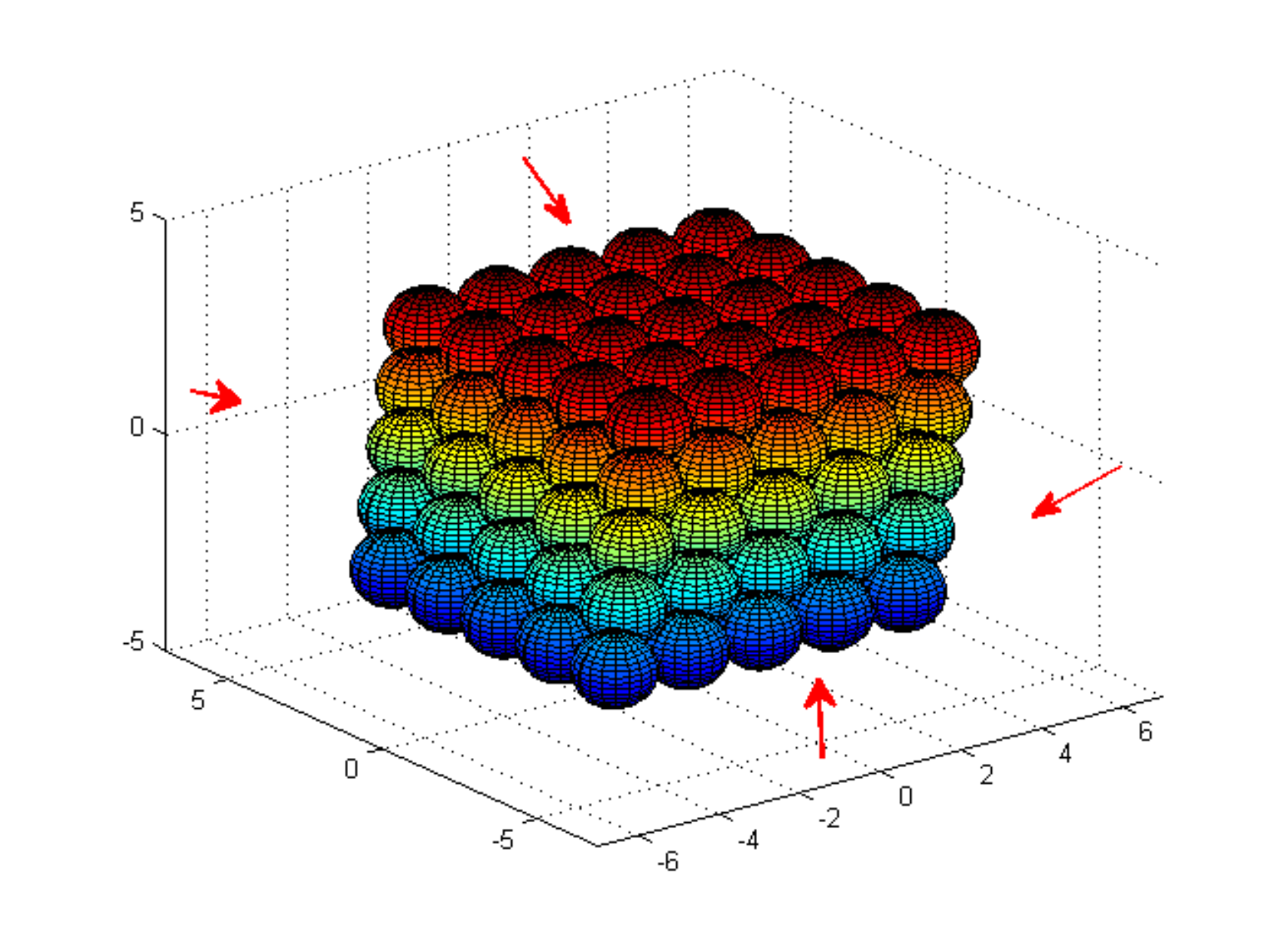}}
\caption{Intruder detection: the sensing ranges of the mobile sensors shown by syphers and intruders represented by arrows}
 \label{new}
\end{figure*}
Next, we consider four intruders are moving to the target area and employ them to evaluate the performance the proposed sensing coverage algorithm. Fig.\ref{new} shows intruders' position at the initial deployment and the arrows indicate their headings. Our simulation demonstrates the deployed mobile sensor network have detected the intruder after entering the target region.\\
 In the following, we carried out several simulations to compare the performance of four space filling polyhedra grids (Truncated octahedron, Hexagonal prism, Cube and  Rhombic dodecahedron) for complete sensing coverage of the same area. As mentioned before, we expected that a polyhedron with the higher volumetric quotient needs, the smaller number of vertices to cover a three-dimensional space. The volumetric quotient of a truncated octahedron is bigger than aforementioned space filling polyhedrons. It means that the number of mobile sensors require  to cover a three-dimensional area by a truncated octahedral grid is less than  the number of mobile sensors required by other grids.\\ Table.\ref{comp1} shows the average complete sensing coverage steps, and the number of mobile wireless sensors need for full sensing coverage of the same area using Truncated octahedral, Hexagonal prism, Cubic, and  Rhombic dodecahedron grids. As shown in table.\ref{comp1}, the average coverage time (step) of the truncated octahedral grid is less than that of the other three-dimensional grids. Because the number of vertices to be covered is less than other grids. Moreover, truncated octahedral grid has the smaller overlapping area, hence this grid requires the less number of mobile sensors for covering the same area. In conclusion, simulation results show that the truncated octahedron grid is the best among the three kinds of grids as it has the smallest overlapping area. Hexagonal prism and Rhombic dodecahedron provide relatively good performance while the cubic grid is the worst among all since it has the biggest overlapping area.

%\begin{figure*}
%\centering
%\mbox{\subfigure[Sensors denoted by *, vertices of the grid represented by o, the desired formation given by the sphere]{\includegraphics[scale=0.5]{FigVali/coverage/S1}}}\quad
%\subfigure[Sensor’s moved inside the sphere]{\includegraphics[scale=0.5]{FigVali/coverage/S2} } \quad
%
%\mbox{\subfigure[Sensors' movement to cover vacant vertices inside the sphere]{\includegraphics[scale=0.5]{FigVali/coverage/S3}}}\quad
%\subfigure[complete sensing coverage]{\includegraphics[scale=0.5]{FigVali/coverage/S4} }
%
%\caption{mobile sensors’ movement to form the spherical pattern
%}
 %\label{fig6}
%
%\end{figure*}
\begin{table}
 \caption{Comparison of four three-dimensional covering grids}
 \label{comp1}
 \begin{center}
 {\small
%\begin{tabular*}{\columnwidth}[t]{ | l | l | p{2cm}|}\hline\hline

%{\columnwidth}[t]{cp{72pt}p{96pt}}
\begin{tabular*}{\columnwidth}[t]{cp{72pt}p{47pt}}\hline
Space filling polyhedron & Number of vertices &  complete sensing coverage  (steps)\\ \hline
$Truncated-octahedron$ & {\raggedright 100 } & {\raggedright 17 }\\
$Cube$ & {\raggedright 172 } & {\raggedright 28}\\
$Hexagonal-prism$ & {\raggedright  140} & {\raggedright 23}\\
$ Rhombic-dodecahedron$ & {\raggedright 142 } & {\raggedright 22 }\\
\hline\hline
  \end{tabular*}
 }%end of small
 \end{center}
\end{table}
 \begin{figure*}[t!]
\begin{center}
\mbox{
\subfigure[Sensors denoted by *, vertices of the grid represented by o, the desired formation given by the sphere]{
{\includegraphics[width=0.47\textwidth,height=0.5\textwidth]{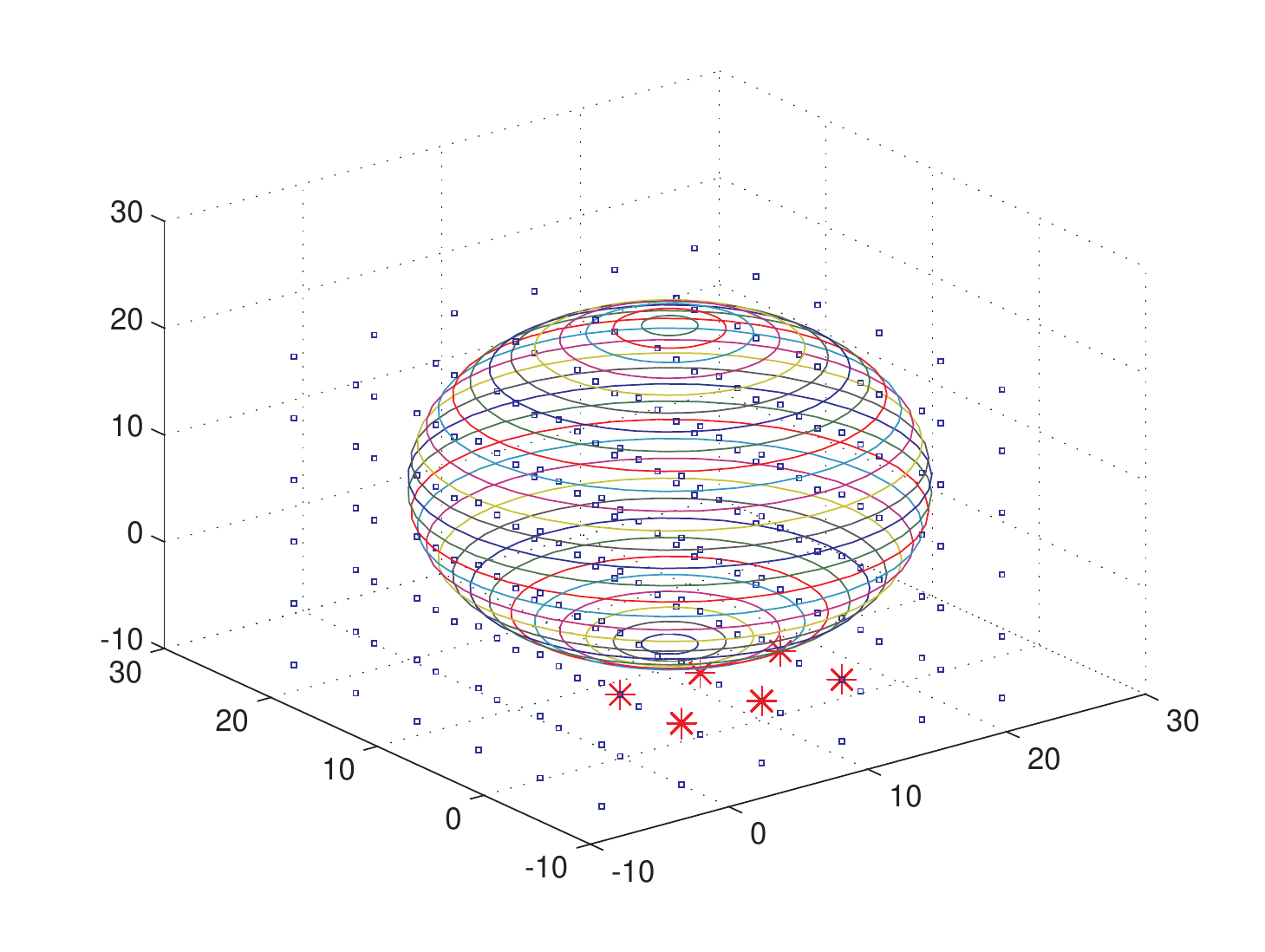}}\quad
\label{s1}
}
\subfigure[Sensors moved inside the sphere]{
{\includegraphics[width=0.47\textwidth,height=0.5\textwidth]{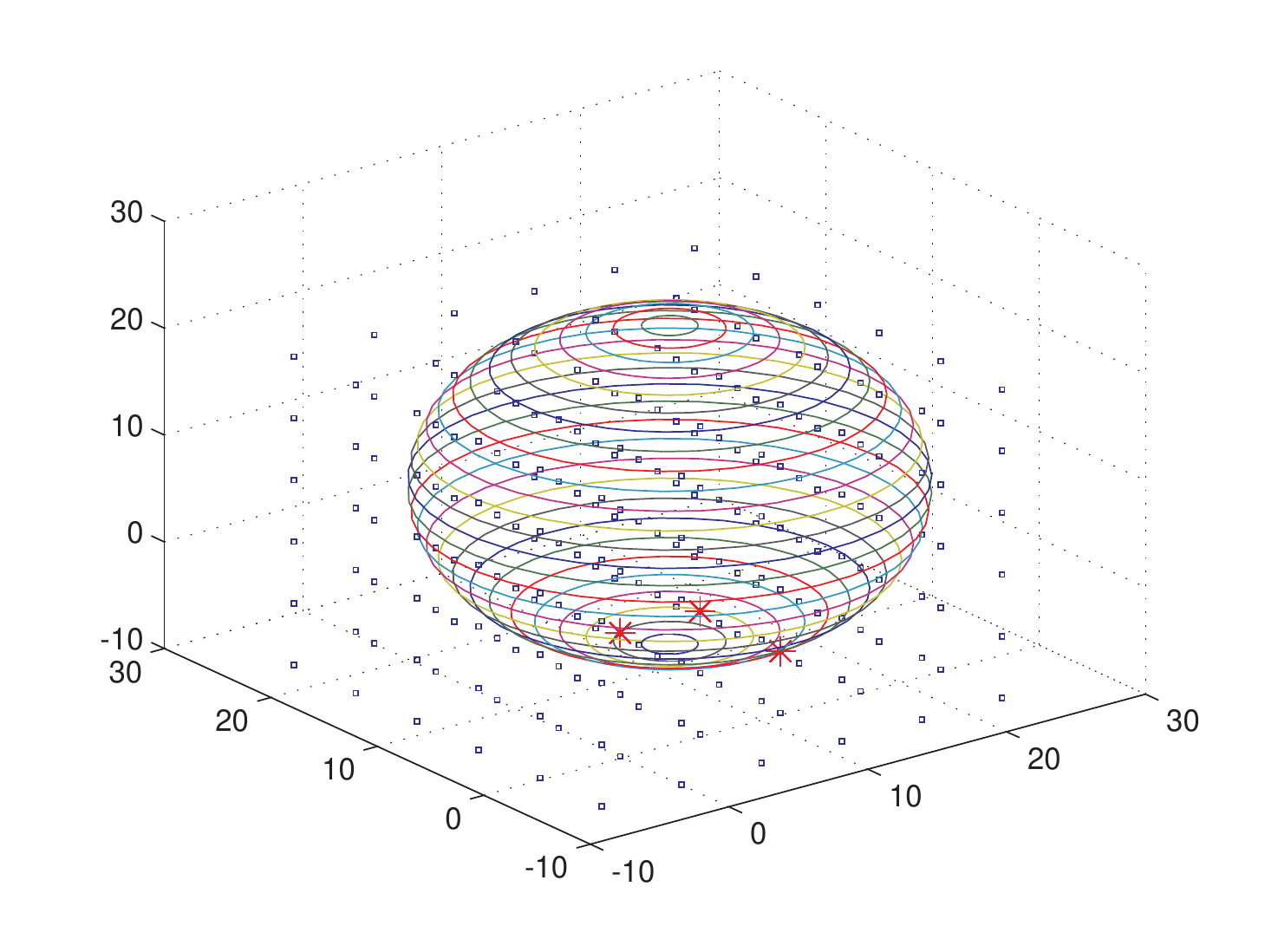}}\quad
\label{s2}
}
}
\mbox{
\subfigure[Sensors' movement to cover vacant vertices inside the sphere]{
{\includegraphics[width=0.47\textwidth,height=0.5\textwidth]{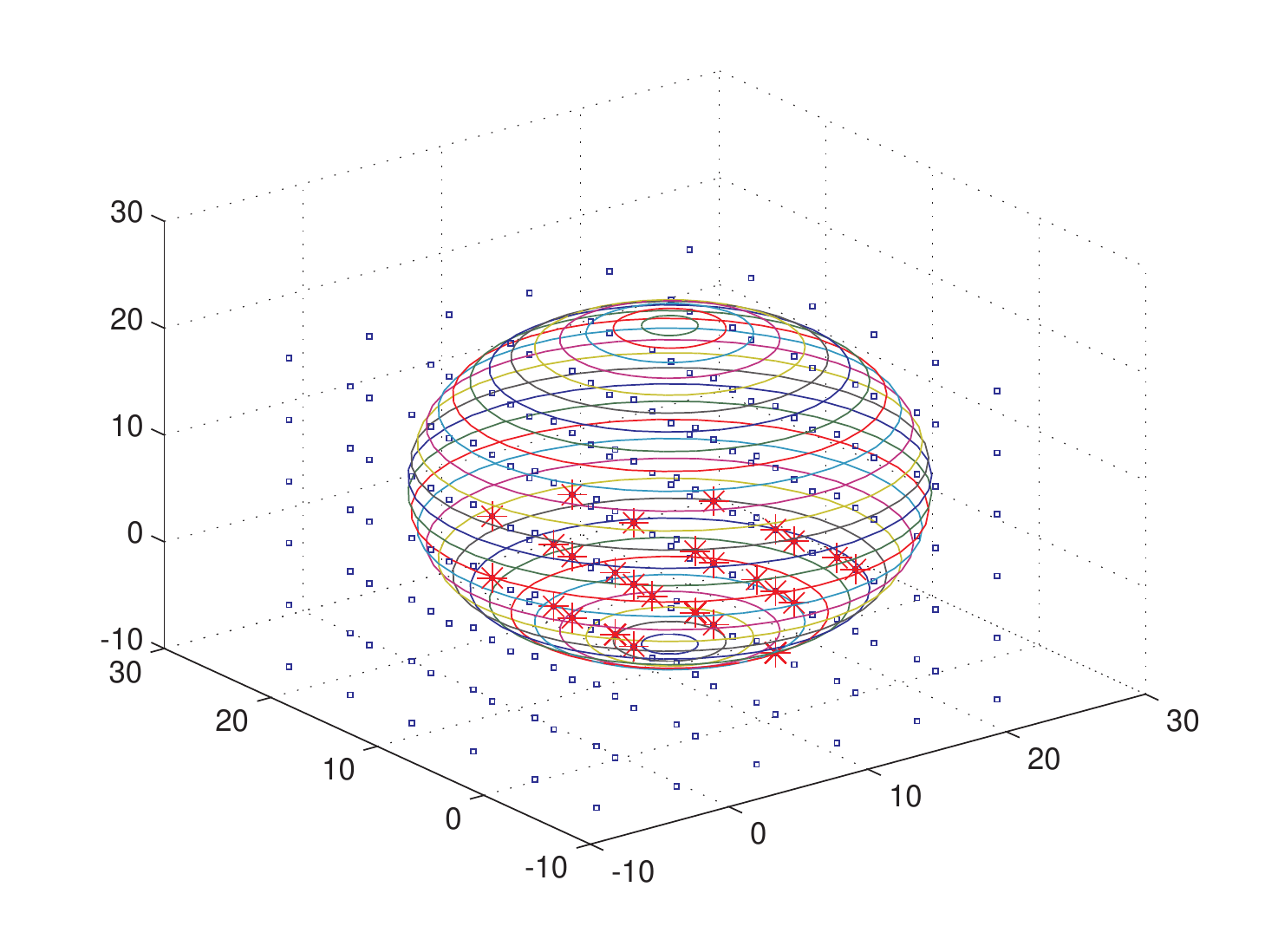}}\quad
\label{s3}
}
\subfigure[complete sensing coverage]{
{\includegraphics[width=0.47\textwidth,height=0.5\textwidth]{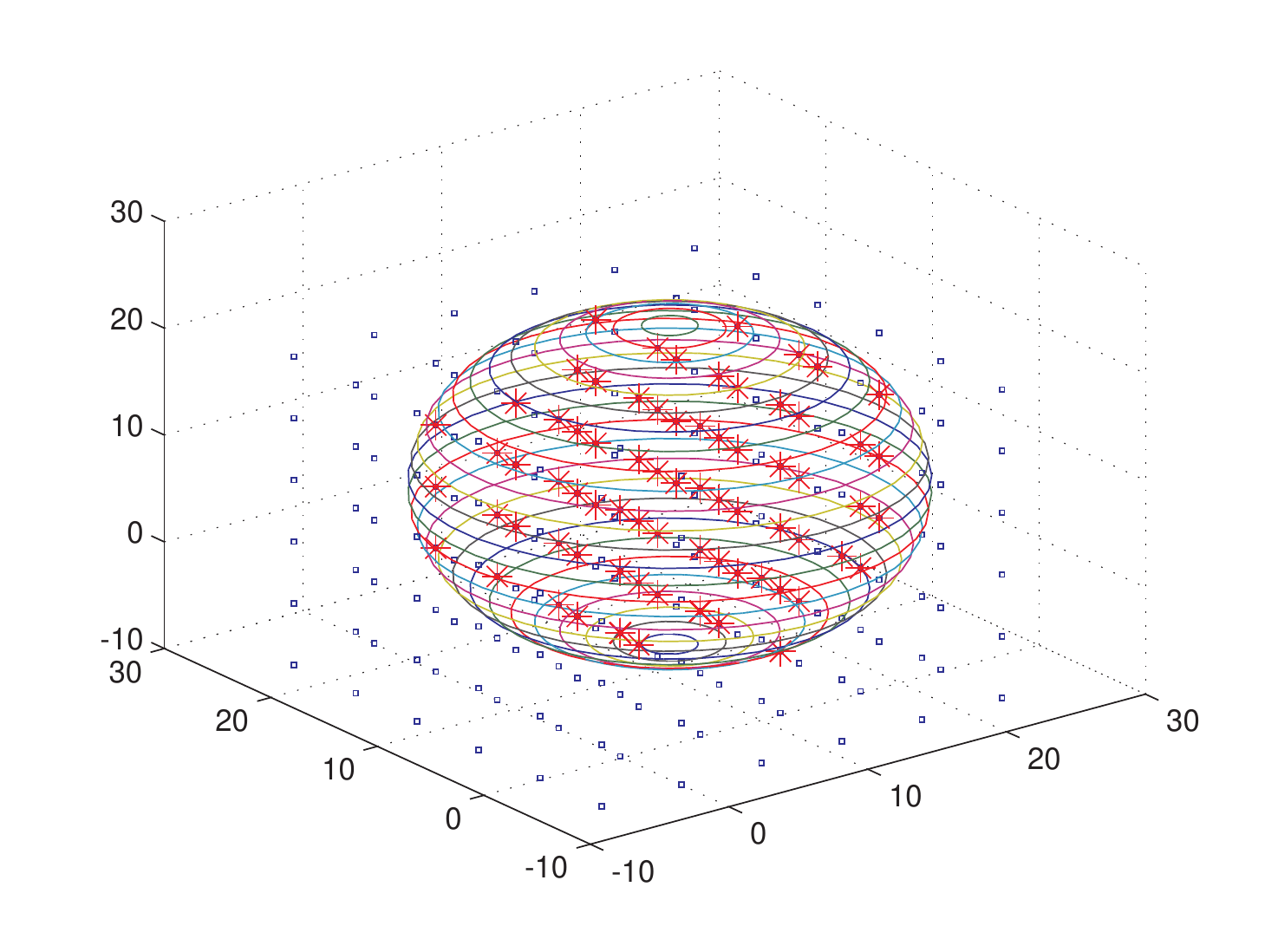}}\quad
\label{s4}
}
}
\caption{mobile sensors’ movement to form the spherical pattern}
\vspace{-6mm}
\end{center}
\end{figure*}
\section{Self-deployment in a Given Three-dimensional Formation}\label{2.4}
In this section, we develop a randomized distributed control law for the mobile sensor network which results in forming a given geometric formation on vertices of a truncated octahedral grid. Using this algorithm, the mobile sensors coordinate their motions to achieve a desired geometric pattern. To do that, we apply the control law (~\ref{eq:cons1}),(~\ref{eq:cons2}) which drive all mobile wireless sensor on vertices of a common covering truncated octahedral grid. As mentioned before, after applying control law (~\ref{eq:cons1}),(~\ref{eq:cons2}) the mobile sensors converge to some consensus values ($x_0$, $y_0$, $z_0$) which define the origin of the common coordinate system for the mobile sensor network.  We introduce the following rules for the mobile wireless sensors to achieve a consensus on the directions of the common coordinate system as follows:

\begin{equation}\label{eq: common}
\theta_i(\texttt{K}+1)=\frac{\theta_i(\texttt{K})+\sum_{\substack{
   j\in N_j(\texttt{K})
}}\theta_j(\texttt{K}) }
 {1+| N_j(\texttt{K}) |}
\end{equation}
\begin{equation*}
\psi_i(\texttt{K}+1)=\frac{\psi_i(\texttt{K})+\sum_{\substack{j\in N_i(\texttt{K})}}\psi_j(\texttt{K})}{1+|N_j(\texttt{K})|}
\end{equation*}
Where $\theta_i(t)$ and $\psi_i(t)$ denote the pitch and the yaw angles of the mobile wireless sensors, respectively. Based on this control law, the mobile wireless sensors start with random initial values for $\theta_i(k)$ and $\psi_i(k)$ and eventually converge to some constant values $\theta_0$ and $\psi_0$. It is obvious that the consensus values  $x_0$,$y_0$,$z_0$, $\theta_0$ and $\psi_0$ define a specific orthogonal coordinate system (x,y,z) on the three-dimensional space. Using this three dimensional coordinate system, we can define three dimensional geometric patterns by some relationship of the type  $f(x,y,z) \in R^3$ as follows:\\
(1)- Sphere: $x^2+y^2+z^2<r$;\\
(2)- Cuboid: $c_{x_1}<x<c_{x_2}, c_{y_1}<y<c_{y_2}, c_{z_1}<z<c_{z_2}$;\\
(3)- Torus: $x=(c+acos(\theta))cos(\varphi)$, $y=(c+acos(\theta))sin(\varphi)$, $z=sin(\theta)$.\\
Where $a,c \in R$ and $\theta, \varphi \in \left[0 \quad 2\pi\right];$\\
(4)-Ellipsoid: $\frac{x^2}{a^2}+\frac{y^2}{b^2}+\frac{z^2}{c^2}=1$. Where $a, b, c \in R$.\\
Note that; it is  assumed that the mobile wireless sensor network knows the following desired three-dimensional geometric patterns a priori.
Now, we introduce a three-stage algorithm for the mobile wireless sensor  network so that drive the mobile sensors on the vertices of a truncated octahedral grid to form the desired three-dimensional geometric shapes as follows:\\
(\texttt{St1}): At first stage, we use control law (~\ref{eq:cons1}), (~\ref{eq:cons2}), (~\ref {eq: common})  to drive the mobile sensors  at vertices of a common truncated octahedral grid and to build a common coordinate system for all mobile sensors.\\
(\texttt{St2}): At second stage of the algorithm, the mobile sensors move to the closest vertices of three dimensional truncated octahedral grid located inside the given geometric shape.\\
(\texttt{St3}): To form the desired geometric pattern, at third stage we apply the control law (~\ref {eq: random coverage}) to randomly drive the mobile wireless sensors to cover the given three-dimensional pattern.\\
\begin{theorem} \label{TH2}: Suppose that assumption ~\ref{assump:connec} hold and the desired three dimensional geometric configuration is given by $f(x,y,z)\in R^3$. After applying the three stage algorithm \texttt{St1}, \texttt{St2} and \texttt{St3}, with probability 1 there exists a time $k_0 \geq 0$ such that for any $i = 1,2, . . . ,n$ there exists a vertex $\texttt{v}$ of the truncated octahedron grid set belonging to the f(x,y,z) such that $v = p_i(k)$ for all $k \geq k_0$.
\end{theorem}
Theorem ~\ref{TH2} immediately follows from Theorem ~\ref{TH1}.

\begin{figure*}[t!]
\begin{center}
\mbox{
\subfigure[Mobile sensors denoted by *, vertices of the grid represented by o, the desired formation given by the cuboid]{
{\includegraphics[width=0.47\textwidth,height=0.5\textwidth]{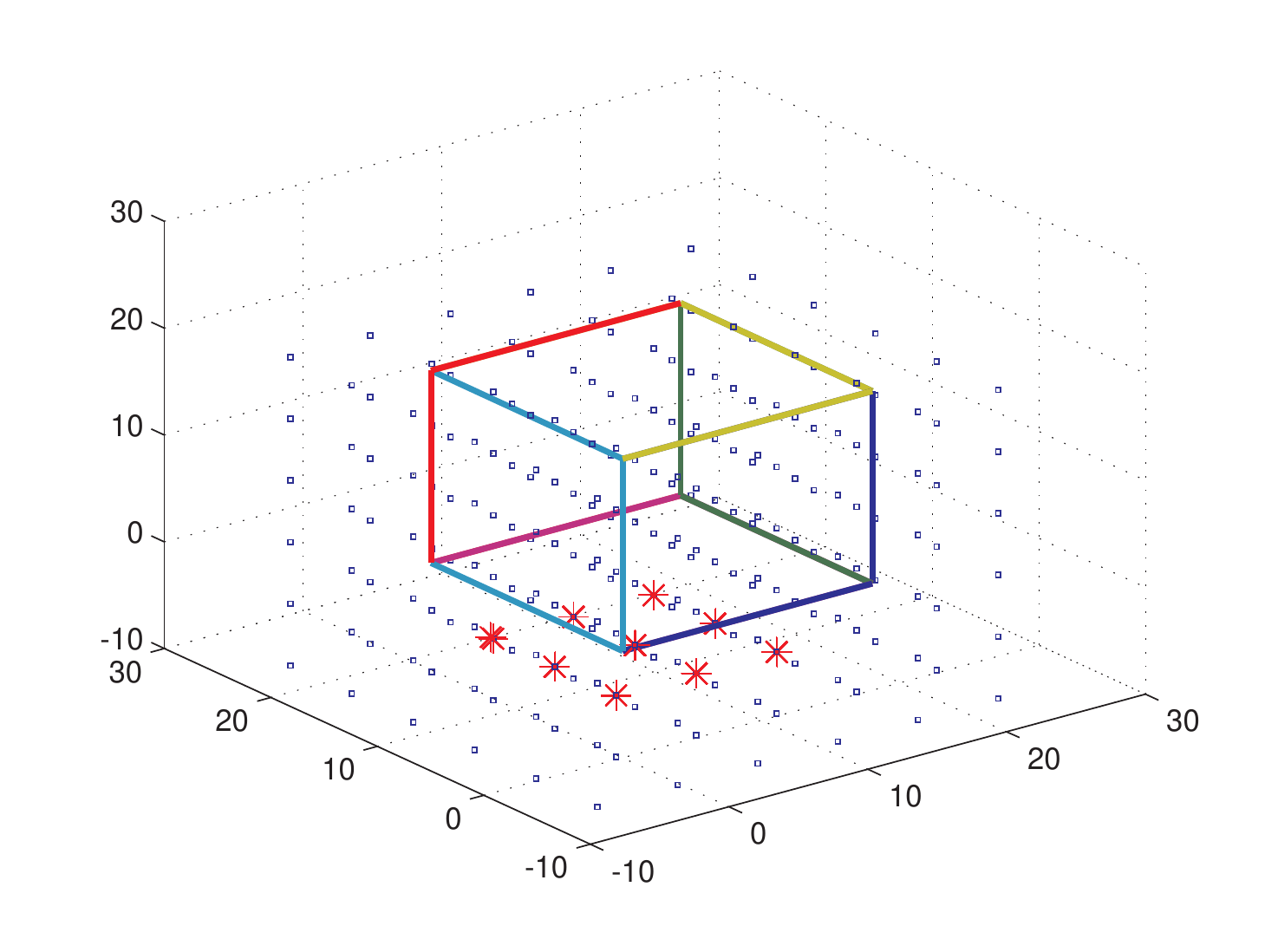}}\quad
\label{C1}
}
\subfigure[Mobile sensors  moved inside the cuboid]{
{\includegraphics[width=0.47\textwidth,height=0.5\textwidth]{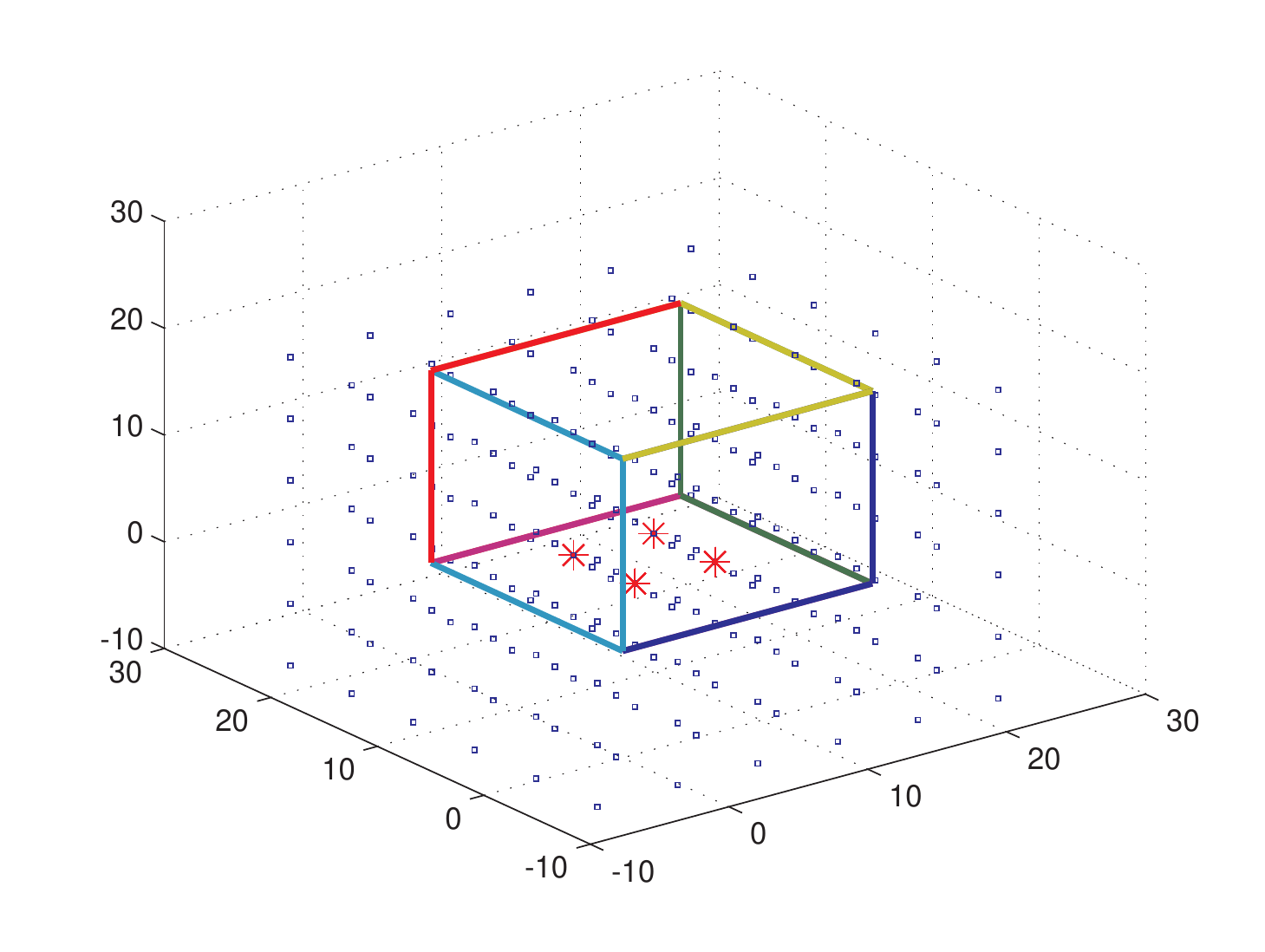}}\quad
\label{C2}
}
}
\mbox{
\subfigure[Mobile sensors' movement to cover vacant vertices inside the cuboid]{
{\includegraphics[width=0.47\textwidth,height=0.5\textwidth]{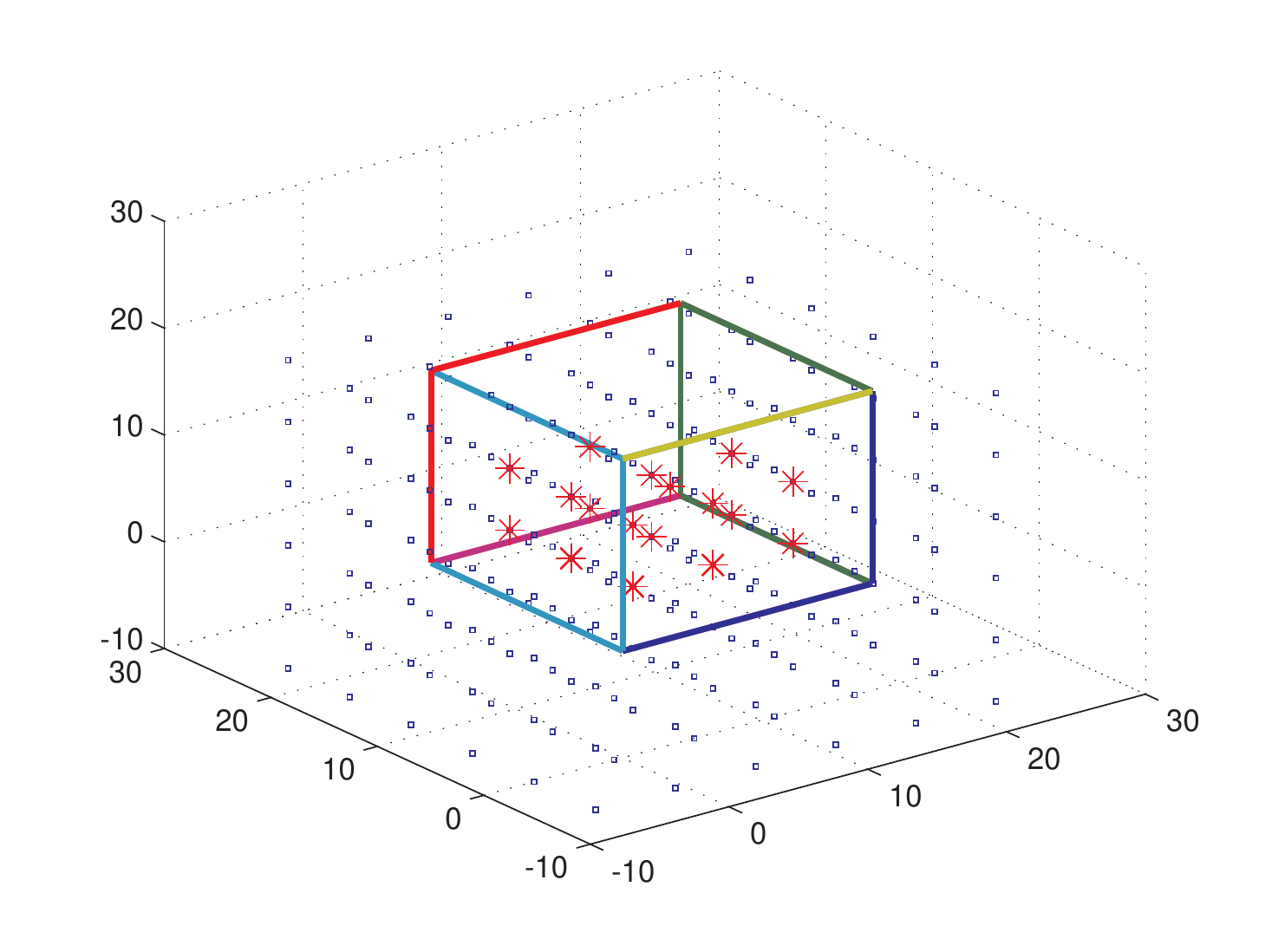}}\quad
\label{C3}
}
\subfigure[Complete sensing coverage]{
{\includegraphics[width=0.47\textwidth,height=0.5\textwidth]{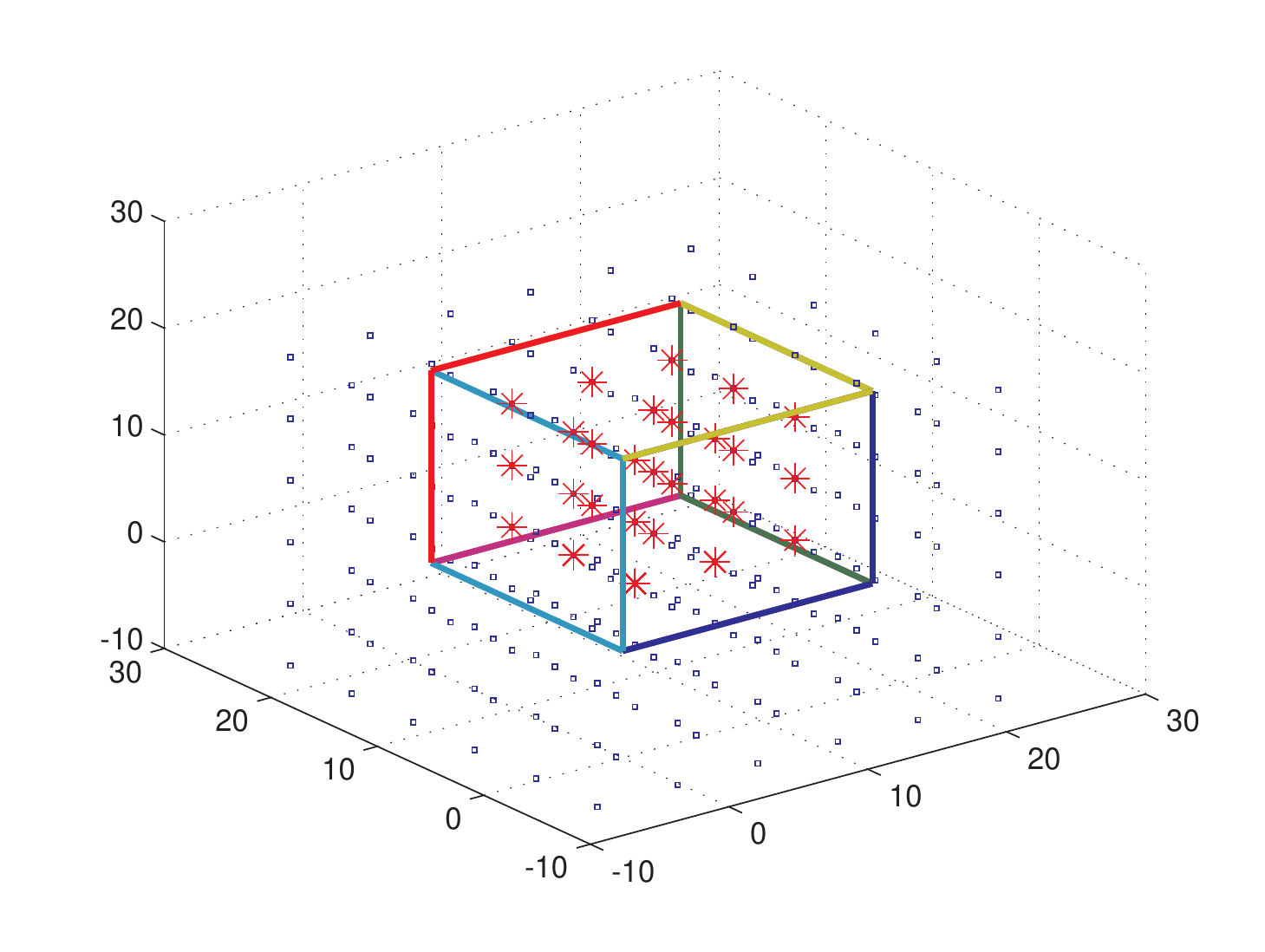}}\quad
\label{C4}
}
}
\caption{Mobile sensors’ movement to form the cuboid pattern}
\vspace{-6mm}
\end{center}
\end{figure*}

\begin{figure*}[t!]
\begin{center}
\mbox{
\subfigure[Sensors denoted by *, vertices of the grid represented by o, the desired formation given by the torus]{
{\includegraphics[width=0.47\textwidth,height=0.5\textwidth]{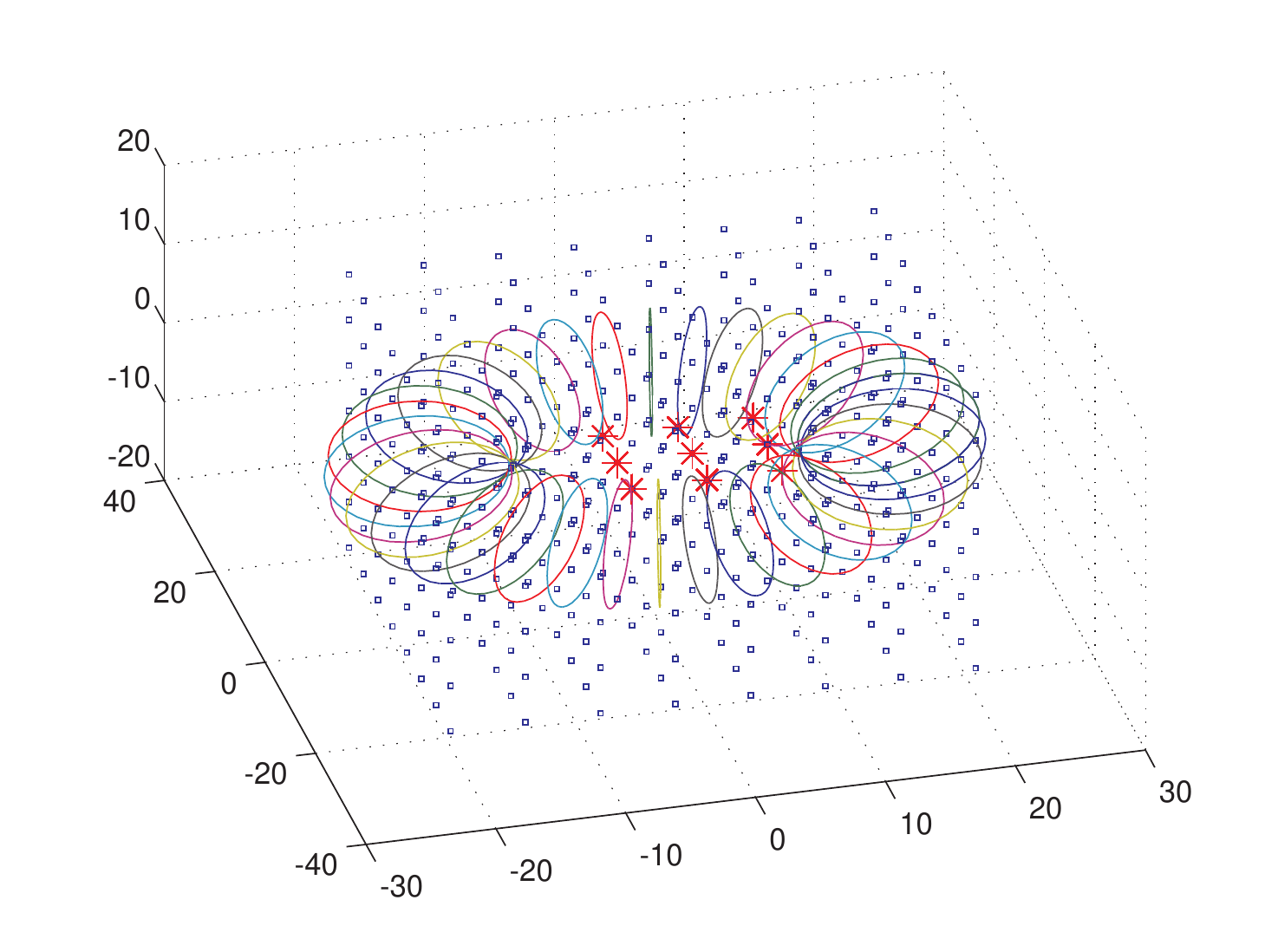}}\quad
\label{T1}
}
\subfigure[Sensor’s moved inside the torus]{
{\includegraphics[width=0.47\textwidth,height=0.5\textwidth]{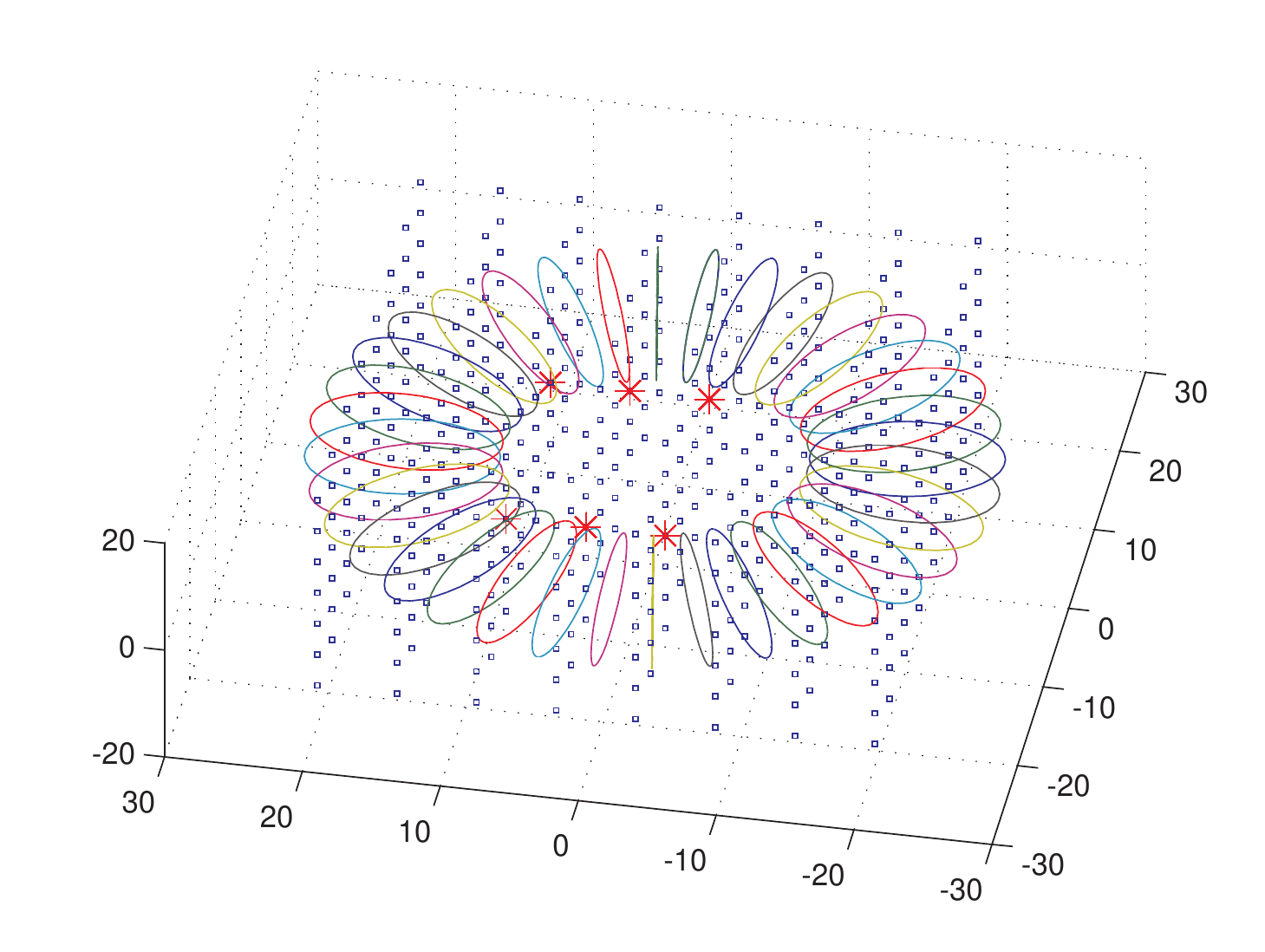}}\quad
\label{T2}
}
}
\mbox{
\subfigure[Sensors' movement to cover vacant vertices inside the torus]{
{\includegraphics[width=0.47\textwidth,height=0.5\textwidth]{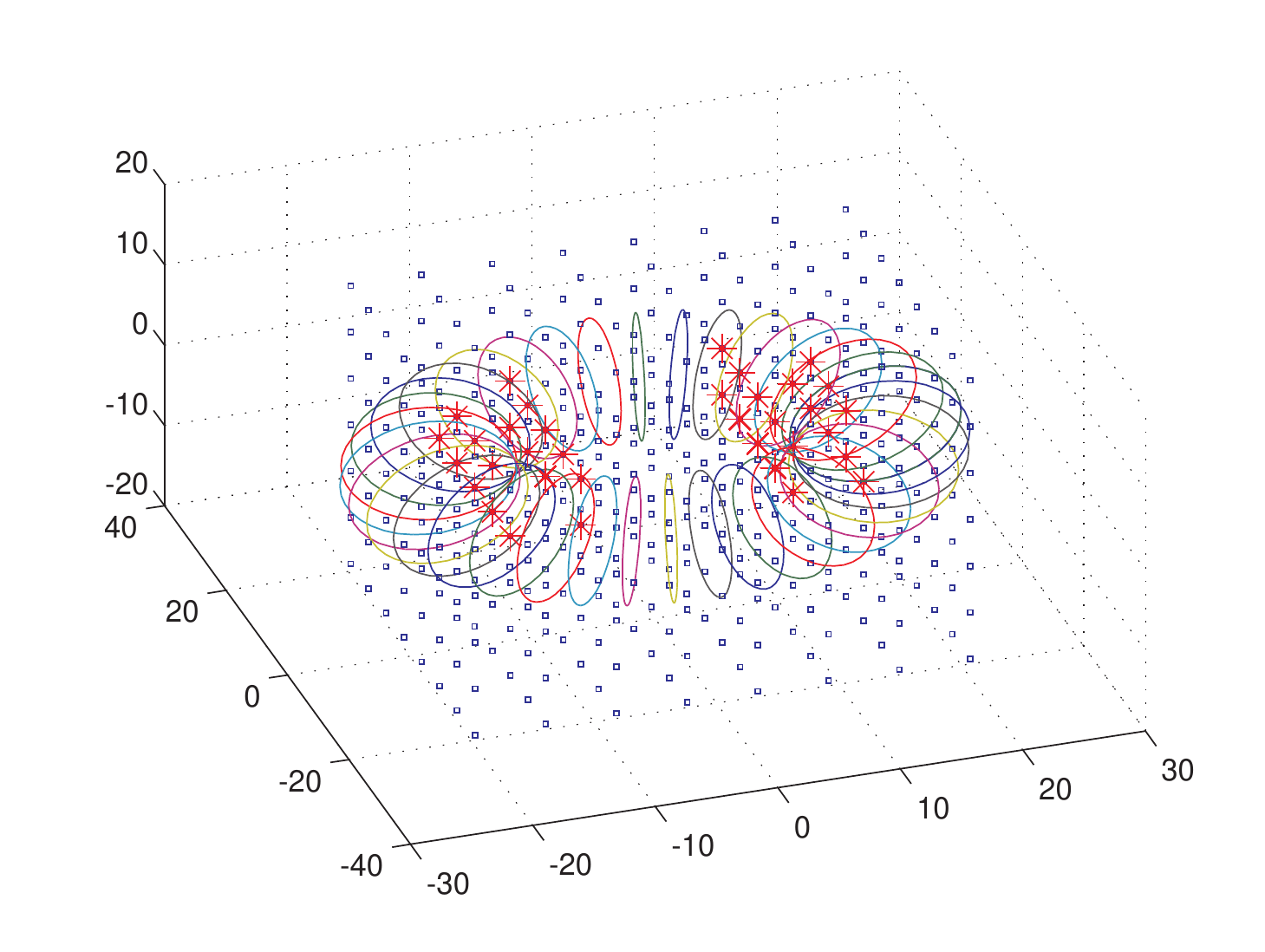}}\quad
\label{T3}
}
\subfigure[Complete sensing coverage]{
{\includegraphics[width=0.47\textwidth,height=0.5\textwidth]{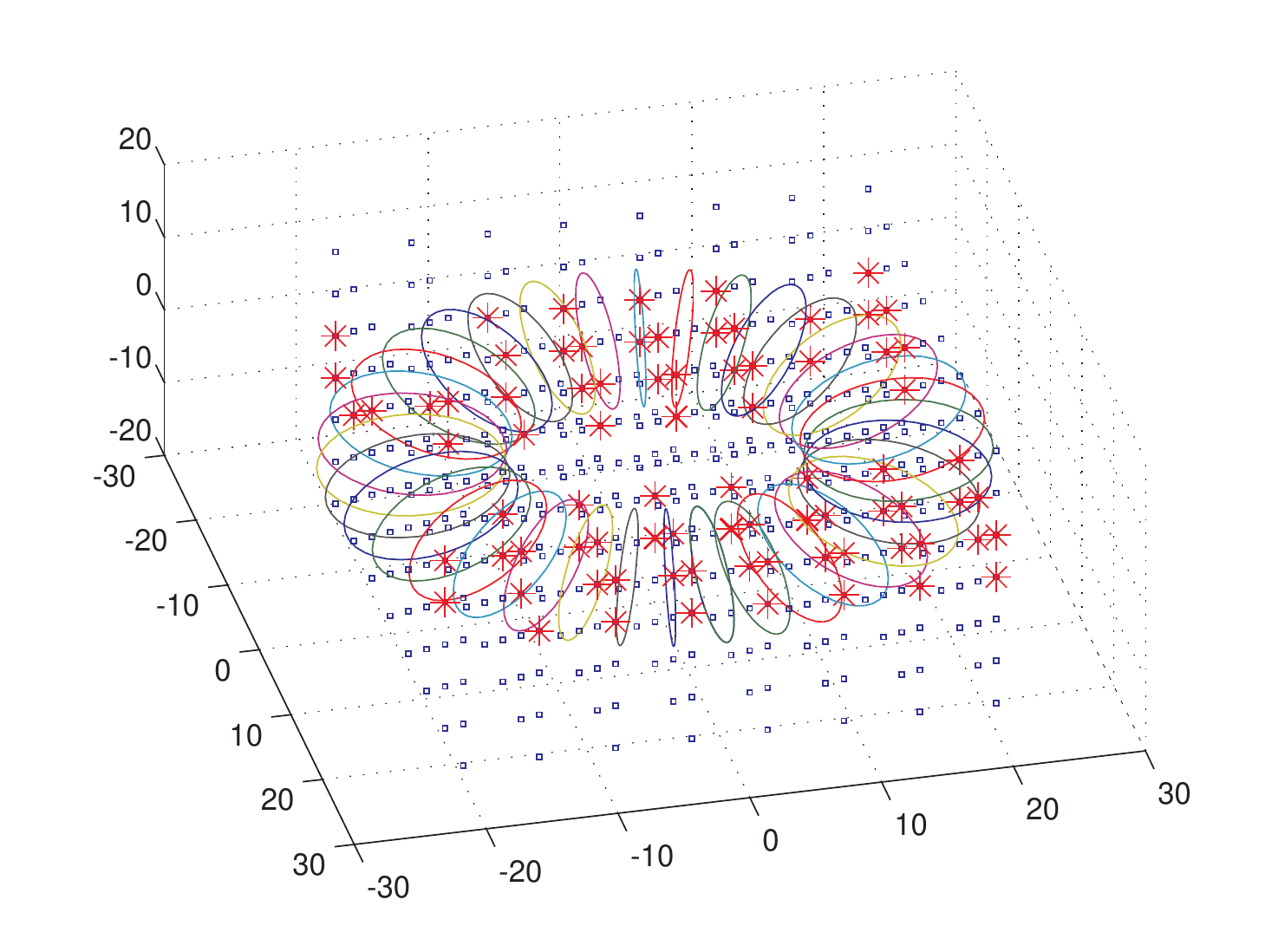}}\quad
\label{T4}
}
}
\caption{Mobile sensors’ movement to form the torus pattern}
\vspace{-6mm}
\end{center}
\end{figure*}

\begin{figure*}[t!]
\begin{center}
\mbox{
\subfigure[Sensors denoted by *, vertices of the grid represented by o, the desired formation given by the ellipsoid]{
{\includegraphics[width=0.47\textwidth,height=0.5\textwidth]{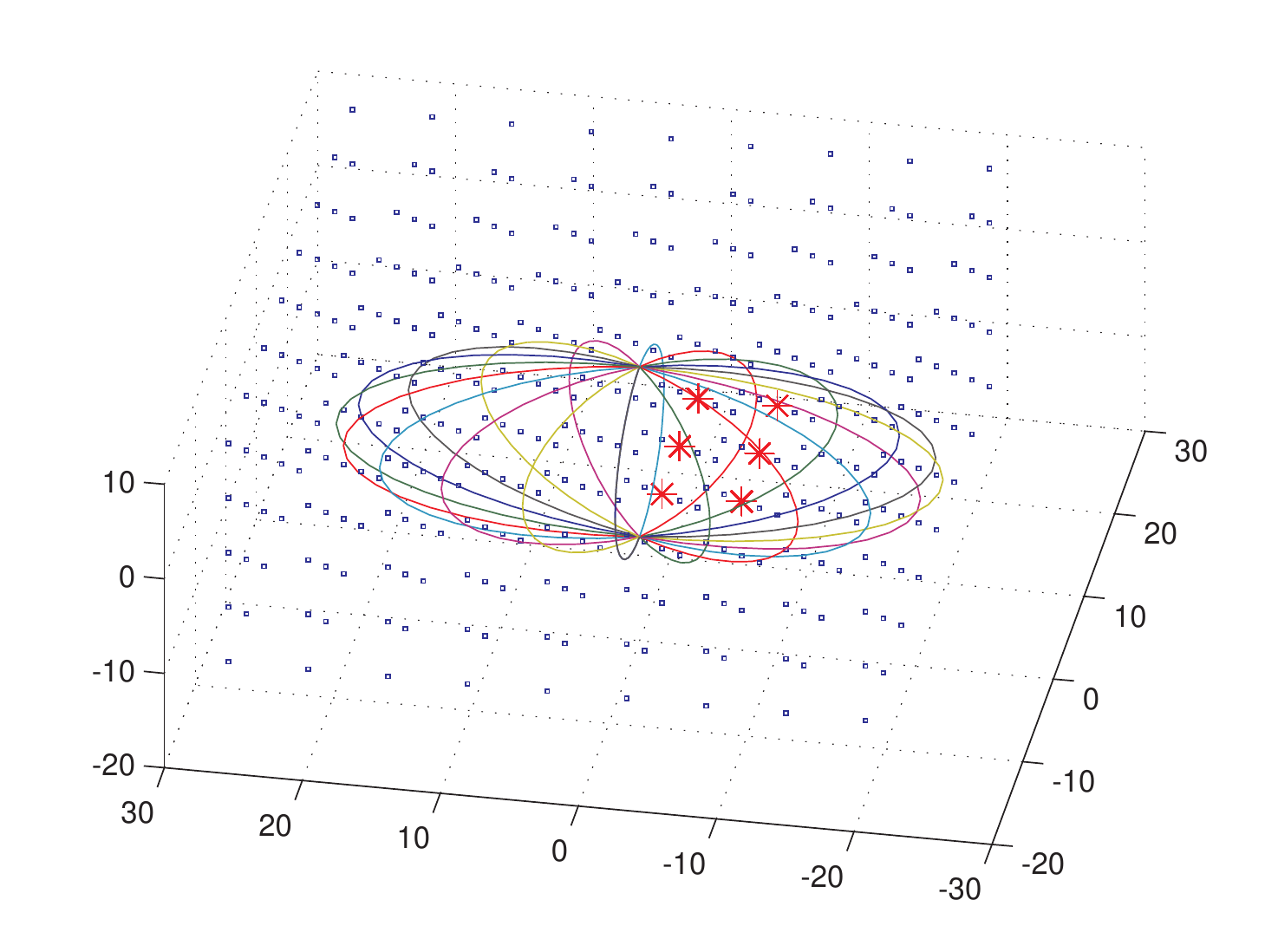}}\quad
\label{E1}
}
\subfigure[Sensor’s moved inside the ellipsoid]{
{\includegraphics[width=0.47\textwidth,height=0.5\textwidth]{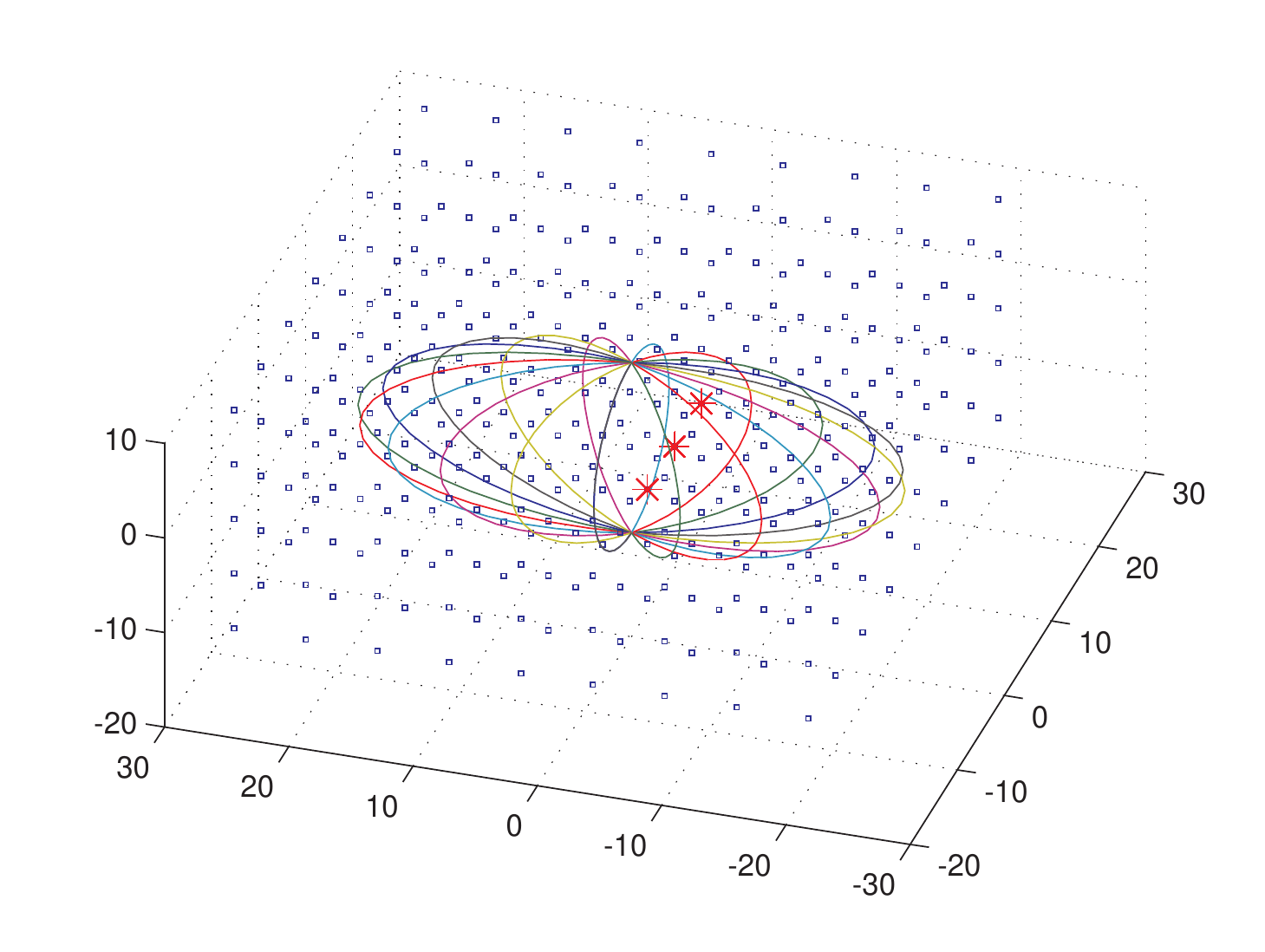}}\quad
\label{E2}
}
}
\mbox{
\subfigure[Sensors' movement to cover vacant vertices inside the ellipsoid]{
{\includegraphics[width=0.47\textwidth,height=0.5\textwidth]{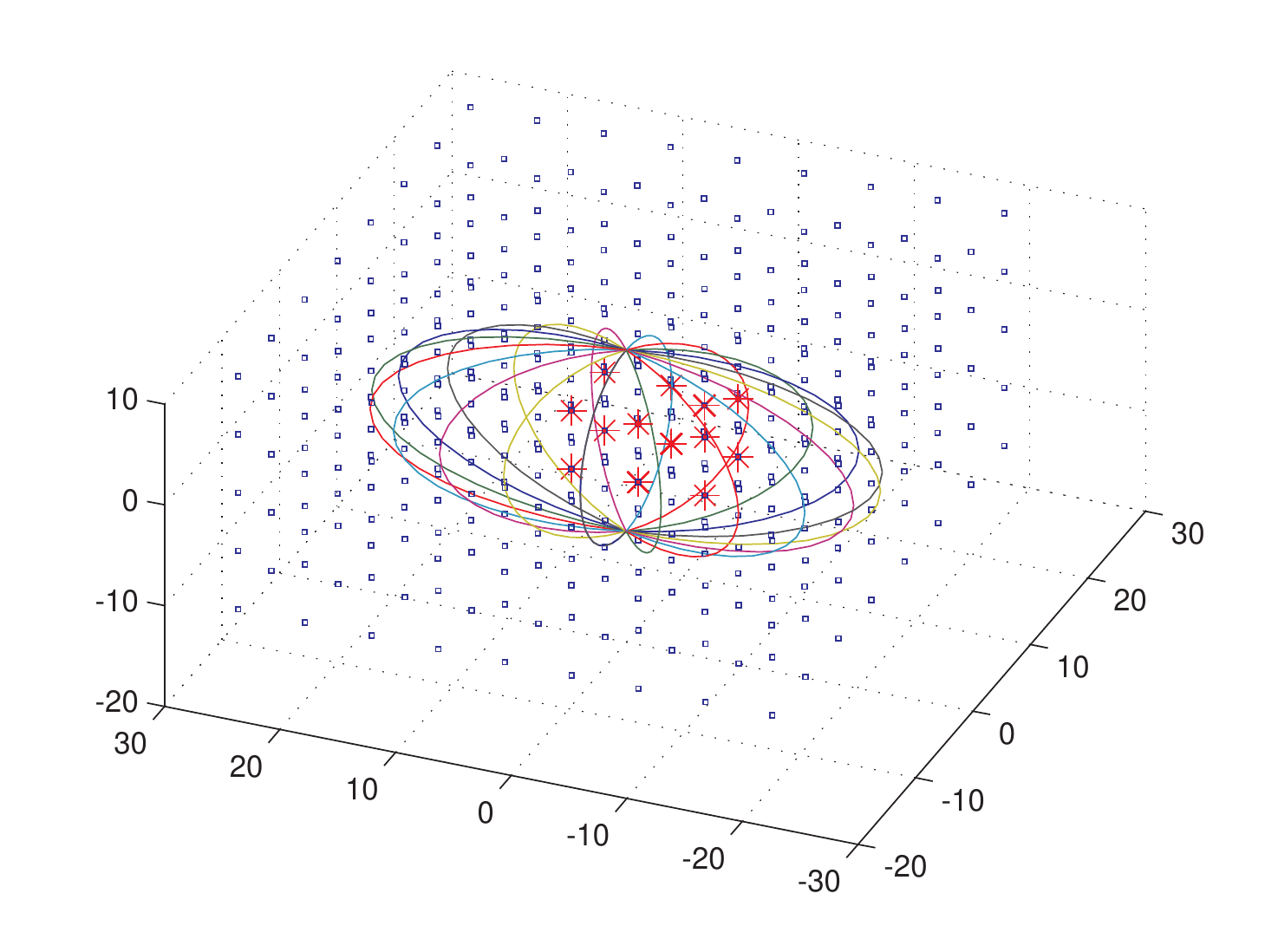}}\quad
\label{E3}
}
\subfigure[complete sensing coverage]{
{\includegraphics[width=0.47\textwidth,height=0.5\textwidth]{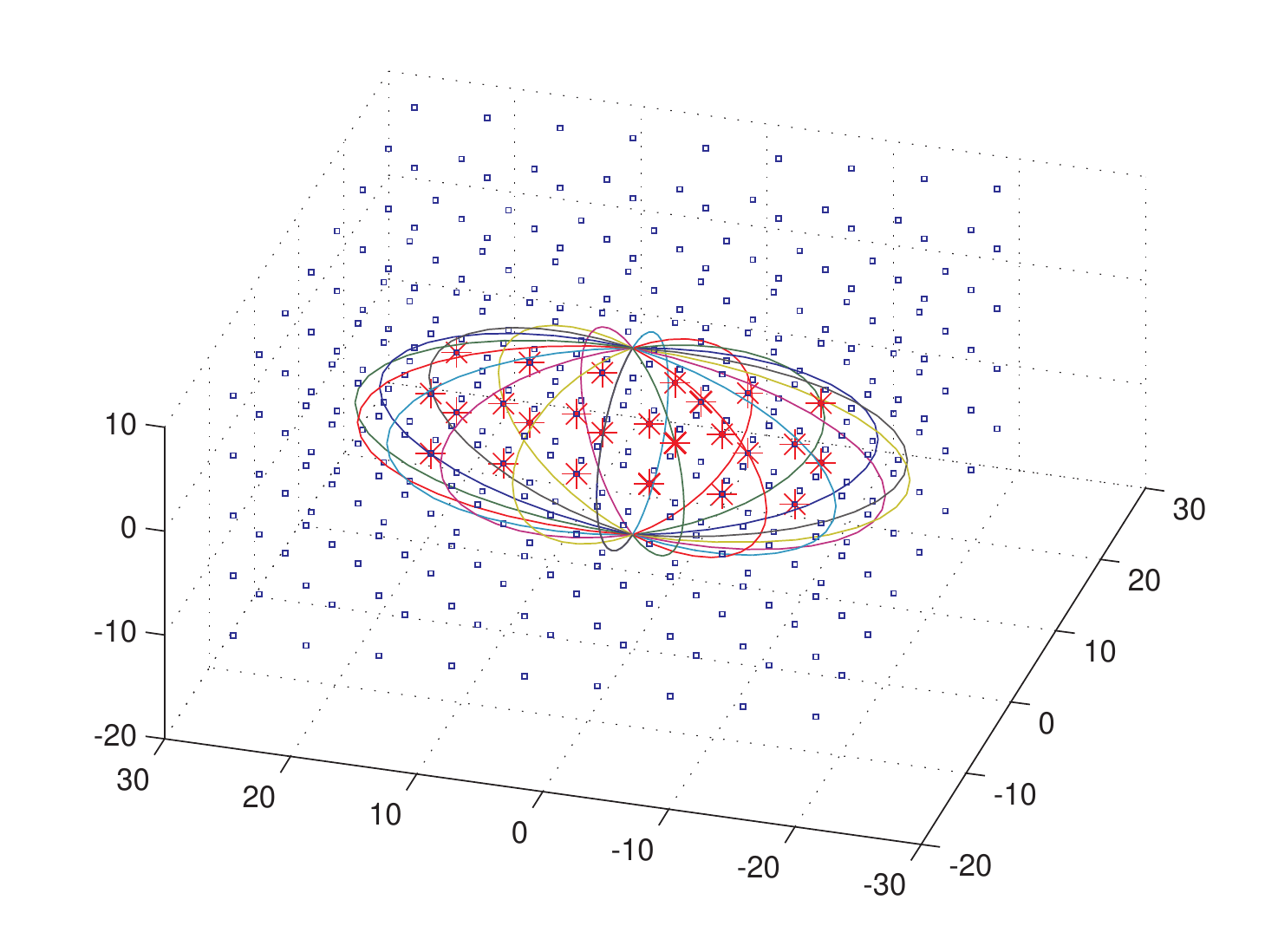}}\quad
\label{E4}
}
}
\caption{Mobile sensors’ movement to form the ellipsoid pattern}
\vspace{-6mm}
\end{center}
\end{figure*}

\subsection{Simulation Results} \label{2.5}
In this section, we illustrate simulation results for the problem of driving the mobile sensors to form a specific geometric pattern in three-dimensional spaces. We suppose that 80 mobile wireless sensors are distributed randomly in a three-dimensional area.	First, we consider a spherical region defined by $x^2+y^2+z^2<r$. Based on St1 the mobile sensors move according to the control law (~\ref{eq:cons1}),(~\ref{eq:cons2}) to the vertices of a common truncated octahedral grid as shown in  Fig.\ref{s1}. Then, all  mobile sensors located outside of the given sphere move to the inside of the given sphere as shown in Fig.\ref{s2}. Fig.\ref{s3} shows the sensors locations after 5 steps. The final formation of the mobile sensors is shown in Fig.\ref{s4} when all mobile sensors have spread inside the given sphere on different vertices of the truncated octahedral grid. At the next stage, the mobile sensors form a cuboid formation defined by $c_{x_1}<x<c_{x_2}, c_{y_1}<y<c_{y_2}, c_{z_1}<z<c_{z_2}$.  Fig.\ref{C1}- Fig.\ref{C4} shows different stages of the formation building inside the cuboid region.\\
 The next set of simulations are carried out to show the effectiveness of the proposed algorithm to form a complicated torus and ellipsoid patterns shown in (Fig.\ref{T1}- Fig.\ref{T4}) and (Fig.\ref{E1}- Fig.\ref{E4}), respectively.\\ As shown in the simulations, for all of the given geometric patterns the mobile wireless sensors build a covering truncated octahedral grid after few steps and then they converge to the vertices of the truncated octahedral grid. After that, the mobile wireless sensors move inside the given pattern on the vertices of the 3D grid and finally they spread inside the given geometric formation and build the given pattern. As a result, the presented simulation results verify the effectiveness of the proposed distributed formation building control law for various given patterns.

\section{Summary} \label{2.6}
In this chapter, a distributed control algorithm for self-deployment of a mobile sensor network for complete sensing coverage of a bounded three-dimensional space was presented. The proposed control laws drive the network of mobile wireless sensors to form a three-dimensional covering grid in the target region then they occupy the vertices of the grid. Also, we developed a distributed control law for coordination of the mobile sensors such that they form a given three-dimensional shape at vertices of a truncated octahedral grid from any initial positions. Simulation results showed that the three-dimensional truncated octahedral grid outperforms than other three-dimensional grids in terms of complete sensing coverage time and the minimum number of mobile sensors required to cover completely the given three-dimensional region.\\ The presented control algorithms are based on the consensus approach that is simply implemented and computationally effective. The control algorithms are distributed and the control action of each mobile wireless sensor  is based on the local information of its neighbouring mobile wireless sensors. Therefore, these algorithms could be applicable to the cases where the communication range is limited. The effectiveness of the proposed control algorithms were confirmed by simulations. Also, we give mathematically rigorous proof of convergence with probability 1 of the proposed algorithms.

%\singlespace
\singlespacing

\chapter{Distributed 3D Dynamic Search Coverage}\label{chap:1ValiCH2}
\minitoc
In the previous chapter, it was assumed that we have enough mobile sensors to deploy for the full sensing coverage in a 3D space. In this chapter, we assume the number of mobile sensors is  limited. Here, we introduce a distributed grid-based random algorithm for search in 3D environments, where the search environment is unknown to the mobile sensors a priori \cite{nazarzehi2015distributed}. Mobile sensors normally rely on absolute positioning or localisation to perform the search task. Absolute positioning needs GPS signal which is usually unfeasible inside or in underwater due to attenuated signals \cite{antonelli2014underwater}. Robot localization commonly entails odometer sensing and environment maps. Such approaches are not applicable on the low cost large swarms as environment maps may be unknown a priori \cite{grzonka2009towards, burgard2005coordinated}. To do the task of localization and positioning without global information, we employ the concept of distributed processing of local information through a wireless communication network \cite{stirling2010energy}. Our proposed algorithm relies on local sensing and communication as it uses the mobile robotic sensors  themselves as sensor nodes. Therefore, this method does not require absolute positioning and localization system. Based on this approach, the mobile sensors communicate with their neighbours to share their information while doing search task. As mentioned before, the proposed algorithm relies on the decentralized communication method which is a real option for the case where the mobile sensors have limited communication ranges.\\
To minimize the time of search and to avoid repeated exploration, each mobile sensor builds a map of the explored area and shares it with other sensors passing within its communication range. As a result, using this strategy the mobile sensors do the search task cooperatively and time efficiently \cite{nazarzehi2015distributed}.\\
In this chapter, the mobile sensor network accomplishes the search task based on a cubic and a truncated octahedral grid. The presented search algorithm uses the vertices of the mentioned 3D grids for the search of a given 3D environment. We demonstrate that the truncated octahedral grid provides better  3D covering than a cubic grid, as a result using this grid can minimize the time of the search. Moreover, simulation results show that the average search time by the proposed grid based random search method is less than the Levy flight random algorithm. Also, we give a mathematically rigorous proof of the convergence of the proposed algorithm with probability 1 for any number of sensors and any sensor initial conditions. A practical application of the proposed search algorithm is target searching in a three-dimensional aerial and aquatic environment. For example, it can be used for the search of sea mines, sources of pollution an black boxes from downed aircraft.
\section{Problem Formulation}
In this chapter, our objective is to design a distributed random  search algorithm to drive a network of mobile sensors for search in bounded 3D spaces. We assume the sensor network has no information about the search area ($\emph{M}  \subset \mathbb{R}^3$) and the positions of the targets $\emph{T}_1, \emph{T}_2, ... \emph{T}_l$, but it can detect the boundaries of the search area. Also, we assume the region $\emph{M} $ is bounded. Here, we adopt the binary sensing model as the sensing model of the mobile sensors. Using this model, the detection of a mobile sensor is defined by $1$ or $0$.  If $\left\|P_t-p_i \right\|\leq R_s$, outputs ‘1’, which means that target is within the sensing range of at least one mobile sensor and outputs ‘0’, otherwise.\\ We define two search scenarios follows:\\
In the first scenario, we assume that the mobile sensor network knows the total number of targets, but it is not aware of their positions. Consequently, the search should be stopped after finding all targets.\\
In the second case, it is assumed that the mobile sensors do not know the number of targets. Therefore, they should search entire area to find all possible targets. In other words, for the second case where the mobile sensor network is not aware of the number of neither targets nor their locations exploring the whole space can guarantee the accomplishment of the search task.\\
Here, our goal is to present a distributed control algorithm to drive the mobile sensors in the search area for identifying all targets in both search scenarios. We propose search algorithms that utilize two different grids (a cubic grid and a truncated octahedral grid) to carry out the search task.\\
Note that, vertices of a cubic (truncated octahedral) grid are  the centres of cubes (truncated octahedrons) of that grid. In 1887, Kelvin \cite{alam2008coverage,kelvin} proposed the truncated octahedron as the solution to the problem of finding space-filling arrangement of similar cells of equal volume with minimal surface area. Finding such optimal space filling polyhedra is still an important open mathematical problem. Different 3D grids can be compared using a criteria called volumetric quotient. The volumetric quotient is defined as: $V_q=\frac{V_p}{V_s}$, where $V_p$ is the volume of a polyhedron and $V_s$ is the volume of its circumsphere. The volumetric quotient of a truncated octahedron is about 1.8 times that of a cube. Intuitively, this gives the hope that using a truncated octahedral grid  for search leads to a fewer number of vertices to be searched and  minimizes the time of the search. In the following, we use the term 3D grid for expression of both cubic grid and truncated octahedral grid.
\begin{figure}
\centering
{\includegraphics[width=0.75\textwidth]{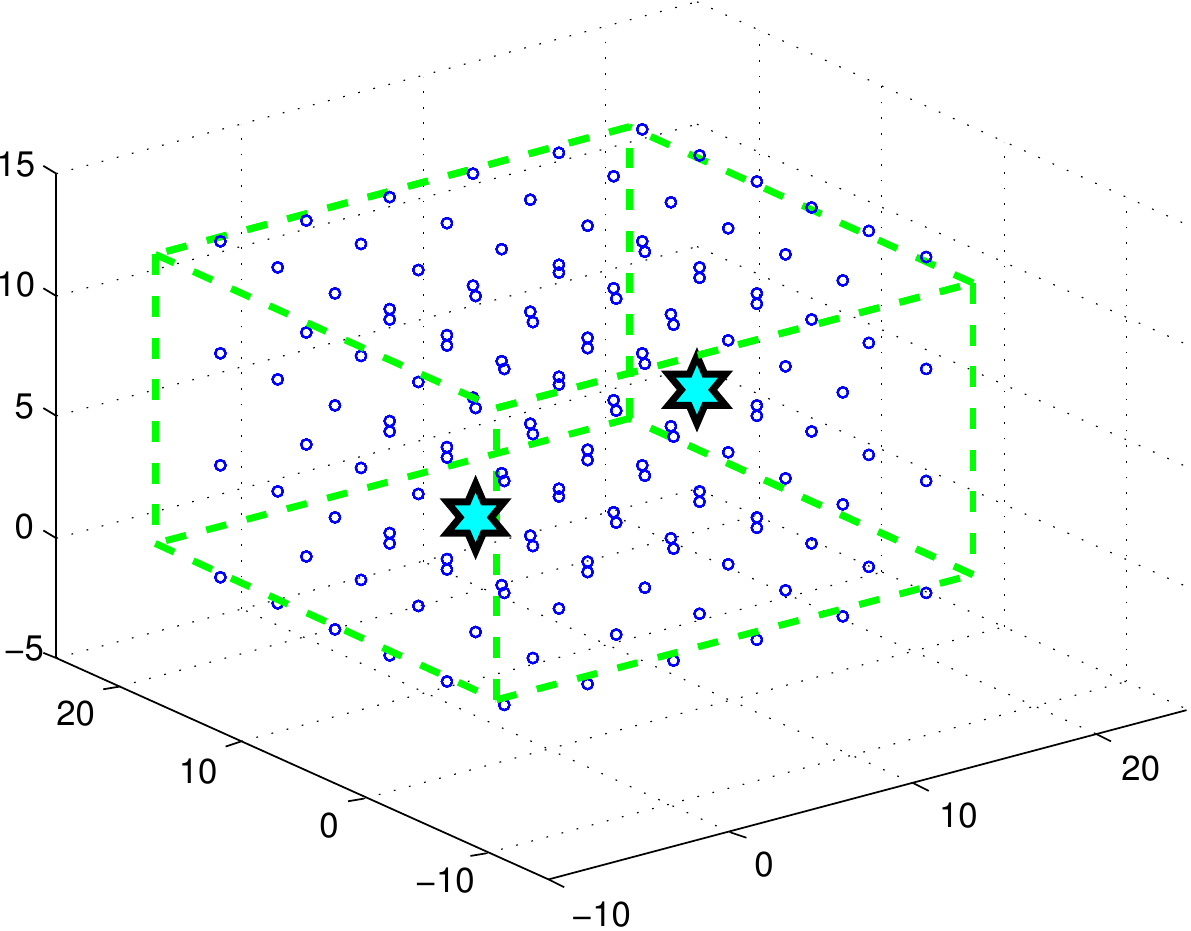}}
\caption{Boundaries of M denoted by –, vertices of the grid by o, targets by star}
\label{schematic2}
\end{figure}
\begin{figure}
\centering
{\includegraphics[width=0.75\textwidth]{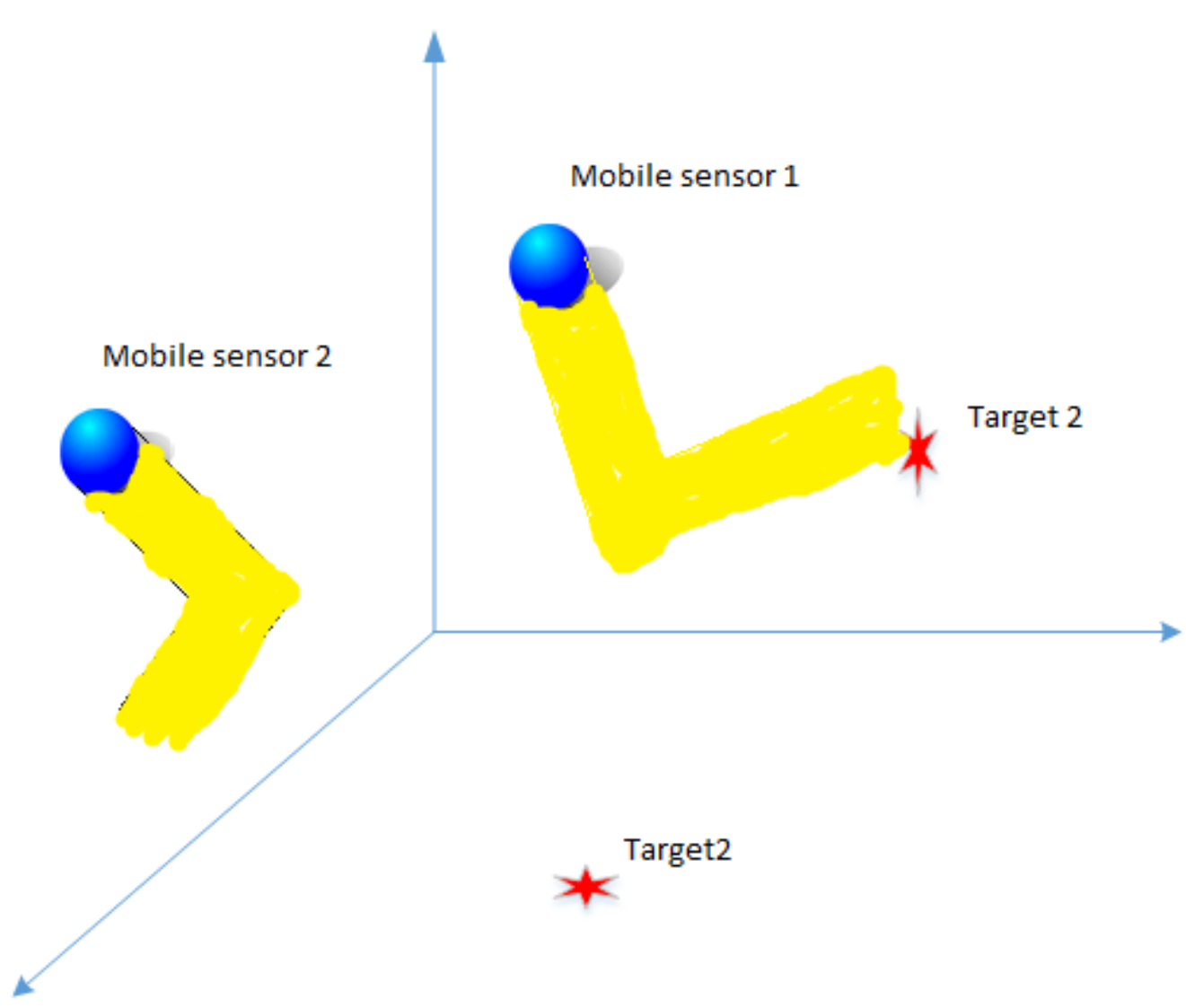}}
\caption{Mobile sensors denoted by sphere (sensing range) , targets denoted by stars, detected area shown by yellow tape}
\label{searchschematic}
\end{figure}
\begin{definition}
Consider a cubic (truncated octahedral) grid cutting $\emph{M} $ into equal cubes (truncated octahedrons) with the sides of $\frac{2\emph{R}_s}{\sqrt{3}}$($\frac{2\emph{R}_s}{\sqrt{10}}$). Let ${\mathcal{\hat{V}}}$ be the infinite set of centres of all the cubes (truncated octahedrons)  of this grid. The set ${\mathcal{V}}={\mathcal{\hat{V}}} \cap \emph{M} $ is called a cubic (truncated octahedral) covering set of the search area (Fig.\ref{schematic2}).
\end{definition}
At first, the mobile sensor network builds a grid which is common among all mobile sensors, also it covers the search area. Then, they perform the search task by moving randomly to the vertices of the grid from different initial locations. During the search, the mobile sensors keep the information of the previously visited vertices of the grid and the targets such as the position of the detected targets and the visited vertices of the grid, which are representative of the previously searched area and communicates with its neighbours in the communication range at the discrete sequence of times $\emph{k} = 0, 1, 2,...$, to exchange the search  information. The neighbouring set of mobile sensor $i$ is defined by
\begin{equation}
\emph{S}_
{i,\emph{R}_c}=\{p\in \mathbb{R}^3;\|(p-p_i (\emph{k}))\| \leq \emph{R}_c\}
\end{equation}
Where $\emph{R}_c$ is the mobile sensors' communication range and $ p_i(\emph{k})\in \mathbb{R}^3 $ is the Cartesian coordinates of the mobile sensor $i$ at time $\emph{k}$.
\begin{definition}
A target is said to be detected by a mobile sensor if it is located within the sensing region of the mobile sensor, which is the sphere of radius $\emph{R}_s>0$ centred at the mobile sensor.
\end{definition}
The schematic of two mobile sensors searching for two randomly locate targets is shown in Fig.\ref{searchschematic}.\\
Let $N_i(\emph{k})$ be the set of all  mobile sensors $j$, $j\neq  i$ and $j\in {1, 2, . . ., n}$ that belongs to the sphere $\emph{S}
_{i,\emph{R}_c} (\emph{k})$ also, we assume that  mobile sensor $i$ has $|N_i(\emph{k})|$ number of neighbours at time $\emph{k}$. We take advantage of the notion of graph to define the relationships between neighbouring mobile sensors. As a result, the vertices $i$ and $j$ of the graph $\emph{g}(\emph{k})$ are considered connected if the  mobile sensors $i$ and $j$ are neighbours at time $\emph{k}$. In the following, we impose a condition on the connectivity
of the graph.
\begin{assumption}\label{assump:connec2}
There exists non-empty, an infinite sequence of contiguous, bounded time-intervals $[\emph{k}_m, \emph{k}_m+1)$, $m = 0, 1, 2, . . .$, such that across each $[\emph{k}_m, \emph{k}_m+1)$, the graph $\emph{g}(\emph{k})$ is connected \cite{main}.
\end{assumption}
\section{Decentralized Random Search Algorithm}
In this section, we introduce a two-stage decentralized grid based random search algorithm for finding targets in a bounded 3D area. At the first phase, the mobile sensors use consensus variables to build a common 3D grid. At the end of this stage, sensors are located on the vertices of the 3D grid. In the second stage, the mobile sensors perform the search task by moving randomly on the vertices of the 3D grid. To avoid unnecessary motion and to save the time and the energy consumption, during the search the mobile sensors share the information about the visited vertices and the detected targets with other mobile sensors passing on their communication range.
\subsection{Building a Common Three-dimensional Covering Grid}
In this section, we propose a distributed consensus algorithm to enable the mobile sensors to reach a common value and to build a common covering grid. Using our proposed distributed consensus algorithms each mobile sensors updates its information based on the information states of its neighbours in such a way that the final information state of each mobile sensor converges to a common value. It is said that agreement or consensus is achieved if the corresponding states of all agents converge to the same value.
 To make a covering truncated octahedral or cubic grid, we use 3D node placement strategies proposed in \cite{alam2008coverage} as follows:
\begin{equation}
V_1=(\emph{x}+\alpha_1\frac{2\emph{R}_s}{\sqrt{3}}, \emph{y}+\alpha_2\frac{2\emph{R}_s}{\sqrt{3}},\emph{z}+\alpha_3\frac{2\emph{R}_s}{\sqrt{3}})
\end{equation}
\begin{equation}
V_2=(\emph{x}+(2\alpha_1+\alpha_3)\frac{2\emph{R}_s}{\sqrt{5}}, \emph{y}+(2\alpha_2+\alpha_3)\frac{2\emph{R}_s}{\sqrt{5}},\emph{z}+\alpha_3\frac{2\emph{R}_s}{\sqrt{5}})
\end{equation}
The vectors $V_1$, $V_2$ determine the position of the centres of the cubes and the truncated octahedrons which are used to tessellate the search area, respectively.
Here, $\alpha_1,\alpha_2$ and $\alpha_3 \in Z$ and $Z $ is the set of all integers. The inputs to our algorithm are the sensing range $\emph{R}_s$ and the coordinates of the point $(\emph{x}, \emph{y}, \emph{z})$, which act as a seed for the growing of the 3D grid.
 In \cite{alam2008coverage} , the point $(\emph{x},\emph{y},\emph{z})$ was considered  as an arbitrary known point. In this chapter, we assume that the mobile sensor network does not have information about this point a priori. Therefore, we use consensus variables to build a common coordinates $(\emph{x},\emph{y},\emph{z})$ for the mobile sensor network, as a result, we will create a common 3D grid for the mobile sensor network. Initially, the sensors  do not have a common coordinate system. Therefore, we assume that  each mobile sensor has consensus variables $\emph{x}_i(\emph{k})$,$\emph{y}_i(\emph{k})$ and $\emph{z}_i(\emph{k})$ in its coordinate system. It is obvious that, any 3D grid is uniquely defined by a point $q_i(\emph{k})=(\emph{x}_i(\emph{k}),\emph{y}_i(\emph{k}),\emph{z}_i(\emph{k}))$  and $\emph{R}_s$. Thus, any   $q_i(\emph{k})$ and $\emph{R}_s$  uniquely define a 3D grid in  $\emph{M} $, which will be denoted as $\mathcal{V}[q,\emph{R}_s]$. The scalar parameters  $ \alpha_1 $, $ \alpha_2 $,$ \alpha_3$ and $ \emph{R}_s $ along with the 3D consensus variables $q_i(\emph{k})=[\emph{x}_i(\emph{k})\quad  \emph{y}_i(\emph{k})\quad   \emph{z}_i(\emph{k})]$ characterize the coordinates of the vertices of the grid. The mobile sensors will start with different values of the coordination variables  $\emph{x}_i(0)$,$\emph{y}_i(0)$ and $\emph{z}_i(0)$, then eventually converge to some consensus values $\emph{x}_0$, $\emph{y}_0$ and $\emph{z}_0$ which define a common coordinate system for the mobile sensor network.\\
Let $p_i(\emph{k})$ be the coordinate of the mobile  sensor $i$ at time $\emph{k}$, and $\emph{C}[q](p)$ defines the closest vertices of the 3D grid $\mathcal{V}[q,\emph{R}_s]$ to p. The following rules for updating the consensus variables and the mobile sensors’ coordinates are proposed that drives all mobile sensors to the same constant values:
\begin{equation}\label{eq:cons3}
\emph{x}_i(\emph{k}+1)=\frac{\emph{x}_i(\emph{k})+\sum_{\substack{
   j\in N_i(\emph{k})
}}\emph{x}_j(\emph{k}) }
 {1+| N_i(\emph{k}) |}
\end{equation}
\begin{equation*}
\emph{y}_i(\emph{k}+1)=\frac{\emph{y}_i(\emph{k})+\sum_{\substack{j\in N_i(\emph{k})}}\emph{y}_j(\emph{k})}{1+|N_i(\emph{k})|}
\end{equation*}
\begin{equation*}
\emph{z}_i(\emph{k}+1)=\frac{\emph{z}_i(\emph{k})+\sum_{\substack{j\in N_i(\emph{k})}}\emph{z}_j(\emph{k})}{1+|N_i(\emph{k})|}
\end{equation*}
\begin{equation}\label{eq:cons4}
p_i(\emph{k}+1)=\emph{C}[q_i(\emph{k}),\emph{R}_s](p_i(\emph{k}))
\end{equation}
Based on the rules (~\ref{eq:cons3}) and (~\ref{eq:cons4}), each mobile sensor updates its consensus variables using the average of its own and the consensus variables values of its neighbours located within its communication range.
\begin{theorem} Suppose that Assumptions ~\ref{assump:connec2} hold and the mobile sensor network moves according to the laws (~\ref{eq:cons3}), (~\ref{eq:cons4}). Then, there exists a cubic(truncated octahedral) covering set $\mathcal{V}$ such that: $\forall$ $i = 1, 2, . . ., n$, $\exists$ $v\in\mathcal{V} $ ; $lim_{\emph{k}\rightarrow\infty} p_i\left(\emph{k}\right)= v$.
\end{theorem}
\begin{proof}
 Based on the update law (~\ref{eq:cons3}) and the assumption ~\ref{assump:connec2}, the consensus variables converges to some constant values as follows: $ \emph{x}_i\left(\emph{k}\right)\rightarrow \emph{x}_0$, $ \emph{y}_i\left(\emph{k}\right)\rightarrow \emph{y}_0$ and  $ \emph{z}_i\left(\emph{k}\right)\rightarrow \emph{z}_0$.\\(See \cite{jadbabaie2003coordination} for the proof of convergence of consensus variables to some constant values). Since $\emph{R}_s$ ,$\alpha_1$ ,$\alpha_2$ and $\alpha_3$ are common to all mobile sensors, therefore:
\begin{equation*}
(\emph{x}_i+\alpha_1\frac{2\emph{R}_s}{\sqrt{3}},\emph{y}_i+\alpha_2\frac{2\emph{R}_s}{\sqrt{3}},\emph{z}_i+\alpha_3\frac{2\emph{R}_s}{\sqrt{3}})\rightarrow
\end{equation*}
\begin{equation}
(\emph{x}_0+\alpha_1\frac{2\emph{R}_s}{\sqrt{3}},\emph{y}_0+\alpha_2\frac{2\emph{R}_s}{\sqrt{3}},\emph{z}_0+\alpha_3\frac{2\emph{R}_s}{\sqrt{3}}).
\end{equation}
\begin{equation*}
(\emph{x}_i+(2\alpha_1+\alpha_3)\frac{2\emph{R}_s}{\sqrt{5}}, \emph{y}_i+(2\alpha_2+\alpha_3)\frac{2\emph{R}_s}{\sqrt{5}},\emph{z}_i+\alpha_3\frac{2\emph{R}_s}{\sqrt{5}})\rightarrow
\end{equation*}
\begin{equation}
(\emph{x}_0+(2\alpha_1+\alpha_3)\frac{2\emph{R}_s}{\sqrt{5}}, \emph{y}_0+(2\alpha_2+\alpha_3)\frac{2\emph{R}_s}{\sqrt{5}},\emph{z}_0+\alpha_3\frac{2\emph{R}_s}{\sqrt{5}}).
\end{equation}
 It means that, after a while, all mobile sensors will have a common covering 3D grid and, the rule (~\ref{eq:cons4}) guaranties that all mobile sensors move to the vertices of the common 3D grid.
 \end{proof}
\subsection{Randomized Grid Search}
 In this section, we propose a randomized search algorithm to transfer the mobile sensors to visit the vertices of the 3D grid. The 3D grid covers the entire search area, as a result, transmission of the mobile sensors to the vertices of the 3D grid implies that the whole  area is searched by the mobile sensor network.
We introduce the Boolean variables $b_v(\emph{k})$ and $b_\emph{T}(\emph{k})$ which define the states of the vertices of the grid and the states of the targets at the time $\emph{k}$, respectively. If vertex $v$ has been visited before time $\emph{k}$ by any of the mobile sensors  then, $b_v(\emph{k})= 1$ and $b_v(\emph{k})= 0$ otherwise. If the target $\emph{T}$ has been detected by any of the mobile sensors before time $\emph{k}$ then, $b_\emph{T}(\emph{k})= 1$  and $b_\emph{T}(\emph{k})= 0$ otherwise.\\ Now the following random algorithm for relocating  of the mobile sensors is presented:
\begin{equation} \label{eq:search}
p_i(\emph{k}+1)=\begin{cases}
\hat{N}_{p_i(\emph{k})}\quad with \quad probability\quad\frac{1}{\left\|\mathrm{\hat{S}}(p_i(\emph{k}))\right\|}\\
if\quad \left\|\mathrm{\hat{S}}(p_i(\emph{k}))\right\|\neq0\\
N_{p_i(\emph{k})}\quad \quad with \quad probability \quad\frac{1}{\left\|\emph{S}
(p_i(\emph{k}))\right\|}\\
if\quad \left\|\mathrm{\hat{S}}(p_i(\emph{k}))\right\|=0.
\end{cases}
\end{equation}
Where, $\emph{S}
(p_i(\emph{k}))$  is the  set of  the closest vertices of the 3D grid to $p_i(\emph{k})$ and $N_{p_i(\emph{k})}$ is a randomly selected element of $\emph{S}
(p_i(\emph{k}))$. $\mathrm{\hat{S}}(p_i(\emph{k}))$ is considered as a set of the all elements of $\emph{S}
(p_i(\emph{k}))$ which have not been already visited by any of the sensors also
$\hat{N}_{p_i(\emph{k})}$ is its randomly selected element. $\left\|\emph{S}
(p_i(\emph{k}))\right\|$ and  $\left\|\mathrm{\hat{S}}(p_i(\emph{k}))\right\|$ are considered as the number of the elements in each set.\\
The algorithm (~\ref {eq:search}) implies that the mobile sensors perform the search by moving randomly to the vertices of the covering grid. During the search, they keep the information of the previously searched vertices and share it with their neighbours. The mobile sensors randomly select their next destination from the unvisited vertices in their neighbourhood. In the case where all of the neighbouring vertices are visited they randomly move to one of the neighbouring vertices of the grid.
Based on this algorithm, the mobile sensor network constantly moves in the search area.\\ To stop the search process after search of the entire area when the mobile sensor network is unaware of the number of the targets; the following condition should be satisfied:
\begin{equation} \label{eq:stop1}
 p_i(\emph{k}+1)= p_i(\emph{k})
\end{equation}
\begin{equation*}
if \quad \forall \quad v \in \mathcal{V};\quad \left\|\mathrm{\hat{S}}(p_v(\emph{k}))\right\| = 0
\end{equation*}
\begin{theorem}: Suppose that the mobile sensors move according to the law (~\ref {eq:search}),(~\ref {eq:stop1}). Then for any number of mobile sensors, with probability 1 there exists a time $\emph{k}_0$  such that:
\begin{equation*}
 \forall \quad v\in \mathcal{V} , \quad  b_v(\emph{k}_0)=1
\end{equation*}
\end{theorem}
\begin{proof}
The algorithm (~\ref {eq:search}),(~\ref {eq:stop1}) describe an absorbing Markov chain that consists of many absorbing states (that are impossible to leave) and many transient states. The vertices where the mobile sensors stop are considered as absorbing states also, the vertices that are visited during the search are considered as transient states. Algorithm (~\ref {eq:search}) implies that the mobile sensors move randomly to the neighbouring unoccupied vertices of the 3D covering grid; this will continue until all mobile sensors’ neighbours are visited (an absorbing state). It is obvious that, from any initial state with probability 1, one of the absorbing states will be reached.
\end{proof}
For the case where the mobile sensor network knows the number of targets, the search should be stopped after finding all targets. In other words:
\begin{equation} \label{eq:stop2}
 p_i(\emph{k}+1)= p_i(\emph{k})
\end{equation}
\begin{equation*}
if \quad \forall\quad T=\left\{\emph{T}_1,\emph{T}_2,...\emph{T}_l\right\}, \exists\quad \emph{k},i\quad;\left\|(P_\emph{T}-p_i (\emph{k}))\right\| \leq \emph{R}_s
\end{equation*}
Where, $T=\left\{\emph{T}_1,\emph{T}_2,...\emph{T}_l\right\}$ and $P_\emph{T}=\left\{P_{\emph{T}_1},P_{\emph{T}_2},...P_{\emph{T}_l}\right\}$, are the sets of targets and their positions, respectively.
\begin{theorem}: Suppose that the mobile sensors move according to the law (~\ref {eq:search}),(~\ref {eq:stop2}). Then for any number of mobile sensors and any number of targets, with probability 1 there exists a time $\emph{k}_0$  such that:
 \begin{equation*}
\forall \quad  \emph{t} \in T,\quad b_\emph{T}(\emph{k}_0)=1
\end{equation*}
\end{theorem}
\begin{proof}
Vertices where sensors stop are considered as absorbing states (~\ref {eq:stop2}), also the vertices of the grid which are visited during search are considered as transient states (~\ref {eq:search}). As a result, the algorithm (~\ref {eq:search}),(~\ref {eq:stop2}) describes an absorbing Markov chain. The mobile sensors randomly move to visit unoccupied vertices; this continues until all targets are detected. Now the proof of this theorem immediately follows from the proof of Theorem 3.2.
\end{proof}

\begin{figure*}[t!]
\begin{center}
\mbox{
\subfigure[Mobile sensor's initial position]{
{\includegraphics[width=0.8\textwidth,height=0.65\textwidth]{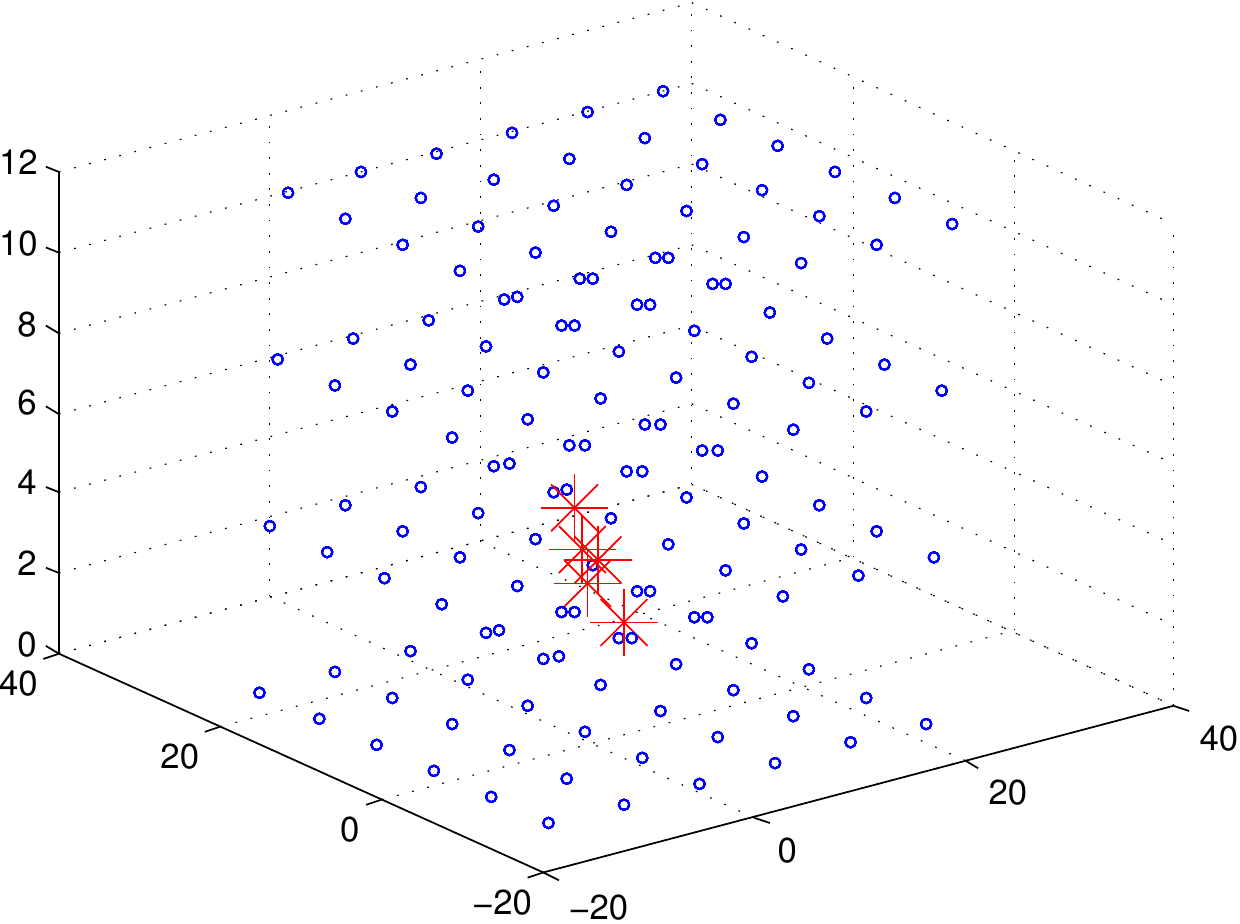}}\quad
\label{init2}
}
}
\mbox{
\subfigure[Consensus to the vertices of the truncated octahedral grid]{
{\includegraphics[width=0.8\textwidth,height=0.65\textwidth]{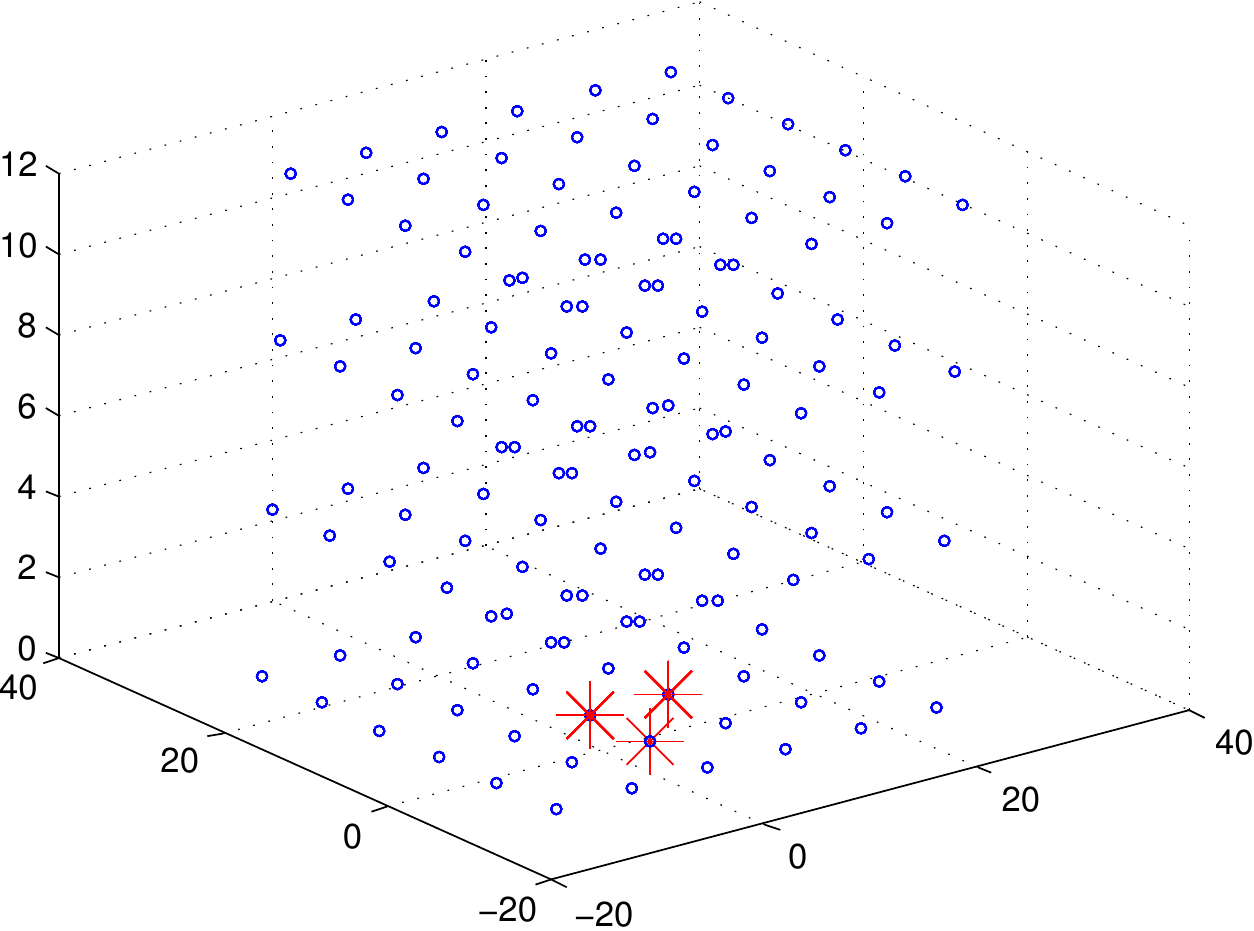}}\quad
\label{cons2}
}
}
\caption{Mobile sensor’s movement to the  vertices of the  common grid}
\label{search1}
\end{center}
\end{figure*}

\begin{figure*}[t!]
\begin{center}
\mbox{
\subfigure[The first scenario]{
{\includegraphics[width=0.8\textwidth,height=0.65\textwidth]{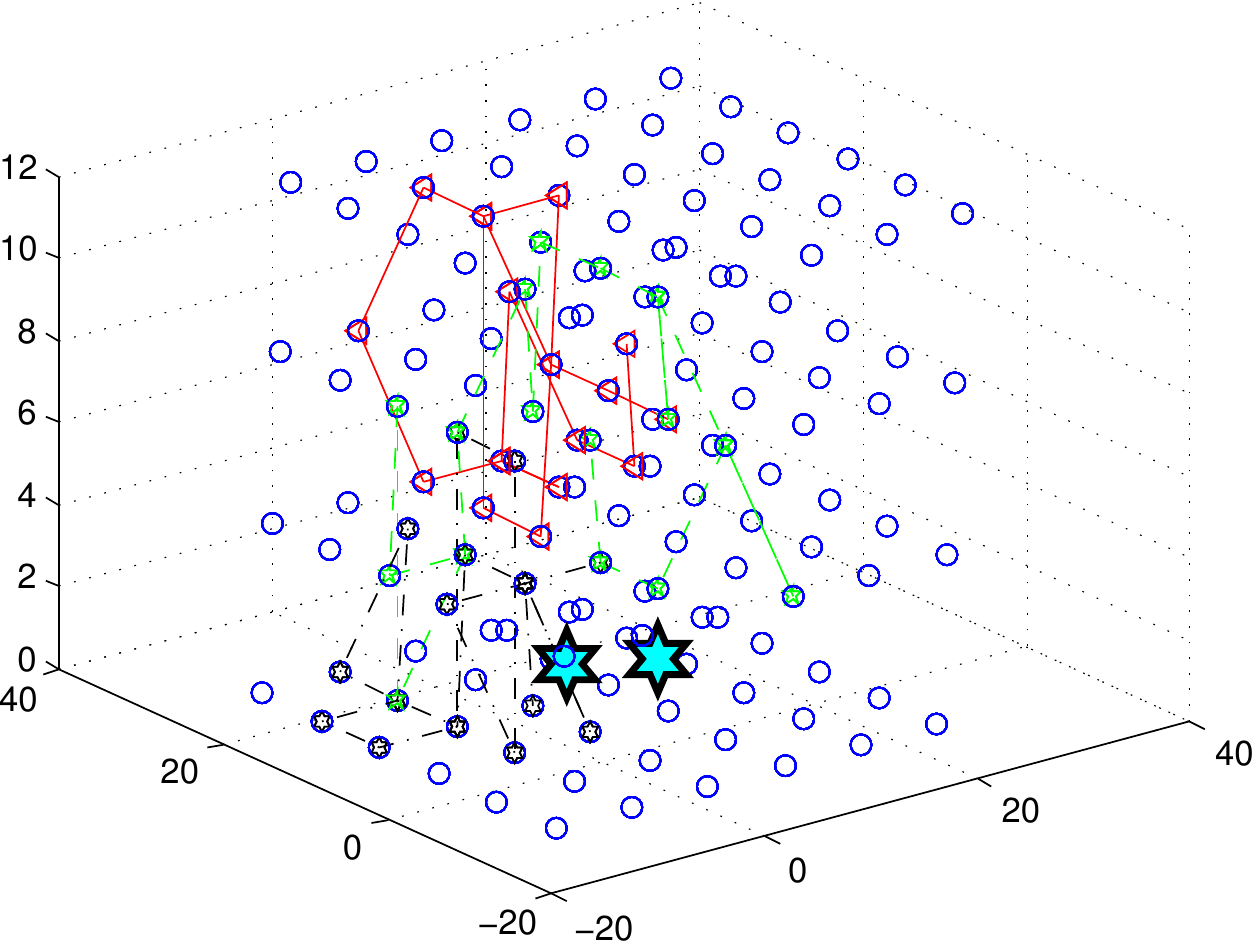}}\quad
\label{scenario1}
}
}
\mbox{
\subfigure[The second scenario]{
{\includegraphics[width=0.8\textwidth,height=0.65\textwidth]{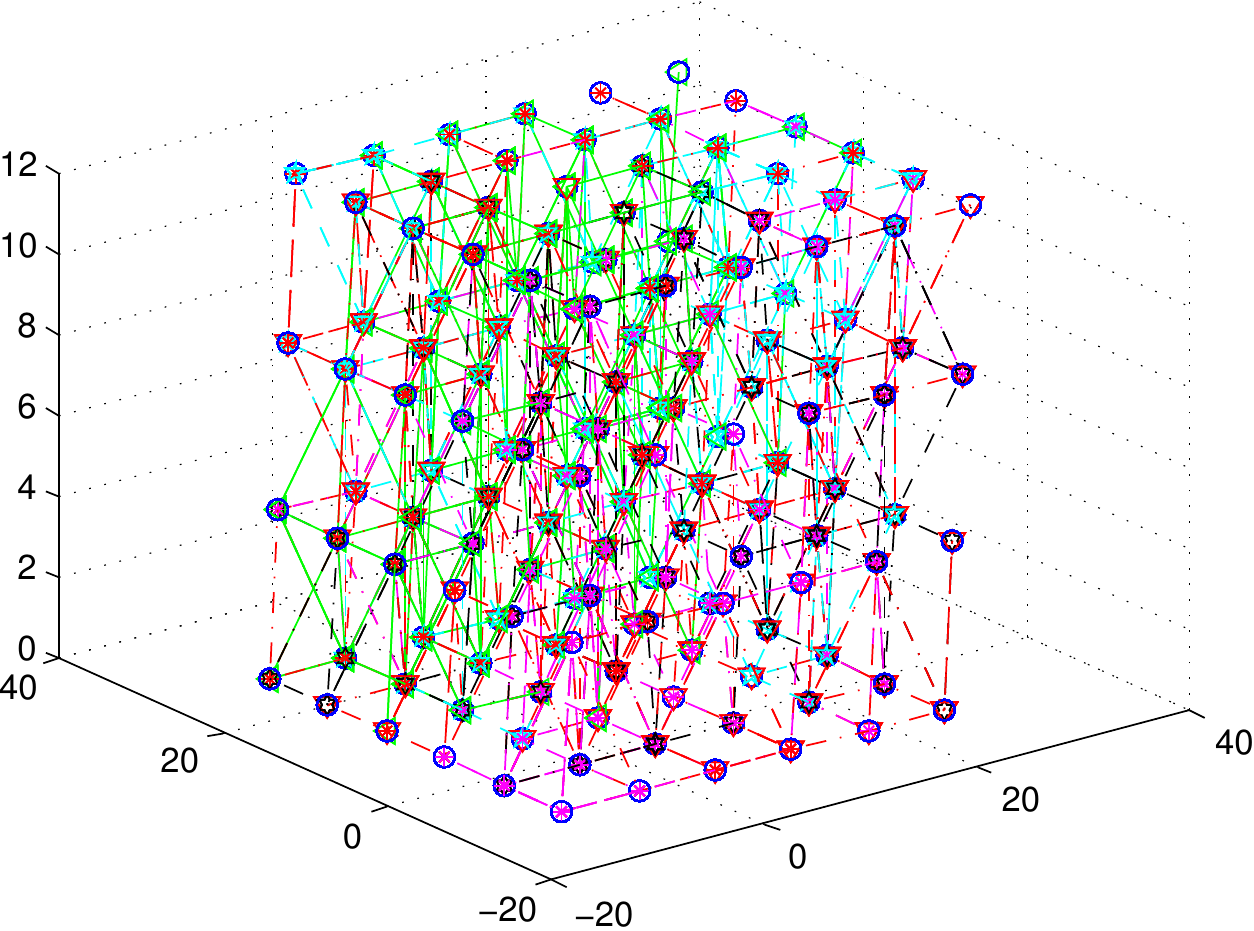}}\quad
\label{scenario2}
}
}
\caption{Mobile sensors’ movement to cover monitoring region (Truncated octahedron based grid)}
\label{search2}
\end{center}
\end{figure*}

%\begin{figure*}[t!]
%\begin{center}
%\mbox{
%\subfigure[Mobile sensor's initial position]{
%{\includegraphics[width=0.5\textwidth]{FigVali/search/nazar3}}\quad
%\label{init2}
%}
%\subfigure[Consensus to the vertices of the truncated octahedral grid]{
%{\includegraphics[width=0.5\textwidth]{FigVali/search/nazar4}}\quad
%\label{cons2}
%}
%}
%\mbox{
%\subfigure[The first scenario]{
%{\includegraphics[width=0.5\textwidth]{FigVali/search/nazar5}}\quad
%\label{scenario1}
%}
%\subfigure[The second scenario]{
%{\includegraphics[width=0.5\textwidth]{FigVali/search/nazar6}}\quad
%\label{scenario2}
%}
%}
%\caption{Mobile sensors’ positions during search, mobile sensors denoted by * , covering grid vertices by o, the targets represented by star}
%\label{search}
%\vspace{-6mm}
%\end{center}
%\end{figure*}

\section{Simulation Results}
In this section, we present simulation results for different scenarios to explain the proposed algorithm. We consider six mobile sensors with random initial positions for the search task shown in Fig.\ref{init2}. By using the truncated octahedral grid and the control law (~\ref{eq:cons3}) and (~\ref{eq:cons4}), the mobile sensors move to the vertices of the common 3D grid. At the end of this stage, all mobile sensors are located at the vertices of the 3D grid as shown in Fig.\ref{cons2}. Then, we apply the second stage of the algorithm to transfer the mobile sensors for the random search. Based on the decentralized random control law (~\ref{eq:search}), the mobile sensors move randomly to visit unoccupied vertices of the covering grid. Throughout the search, each mobile sensor keeps information about the visited vertices and communicates with its neighbours in the communication range; this optimizes the search task because the mobile sensors do not search the vertices which have been already visited by the others.\\
Fig.\ref{scenario1} shows the performance of the search algorithm for the first scenario. As mentioned before, for this scenario we assume the mobile sensor network knows nothing about the environment and the positions of targets except the number of targets. In this case, the search goal is to identify two targets which are randomly placed in the search area. In this case, we evaluate our proposed distributed search coverage algorithm based on the time taken to find all targets. As shown, the mobile sensor network has identified the targets after 8 steps. Trajectories of the mobile sensors after 8 steps have been shown in Fig.\ref{scenario1}. As shown, the control  law (~\ref {eq:stop2}) terminates the search process after detecting the given number of targets.\\
For the second scenario where the mobile sensor network is not aware of the number of targets, the whole area should be searched to identify all possible targets. In this case, the control law (~\ref {eq:stop1}) terminates the search process after complete search of the whole area as shown in Fig.\ref{scenario2}. It is obvious that, adding the number of mobile sensor means more vertices of the grid to be visited at the same time. As a result, the increase in the numbers of the mobile sensors makes the search more efficient because increasing the number of mobile sensors will decrease the time taken to the search task for both search scenarios. As it is seen from the Fig.\ref{comp2} the performance of the system consisting of 14 mobile sensors is better than that of 8 mobile sensors and similarly 2 mobile sensors search performance is better than 1 mobile sensor. Also, the figure demonstrates the algorithm becomes faster with  more mobile sensors. However, if the number of mobile sensors exceeds a limit, the effectiveness of the algorithm has less impact on the whole search performance.\\
To compare the performance of the grids, we did several simulations with different number of mobile sensors. Based on the simulations the average search length for the truncated octahedron grid was about 28 percent less than for the cubic grid. In the following, we carried out several simulations to compare the proposed grid based random search method with the randomized algorithms based on Levy flights proposed in \cite{keeter2012cooperative} for detecting targets in a bounded 3D environment. Simulations result show that the average search time of the proposed grid based search method for two mobile sensors performing a search is far less than that of the Levy flight approach.\\
\begin{figure}
\centering
{\includegraphics[width=0.7\textwidth]{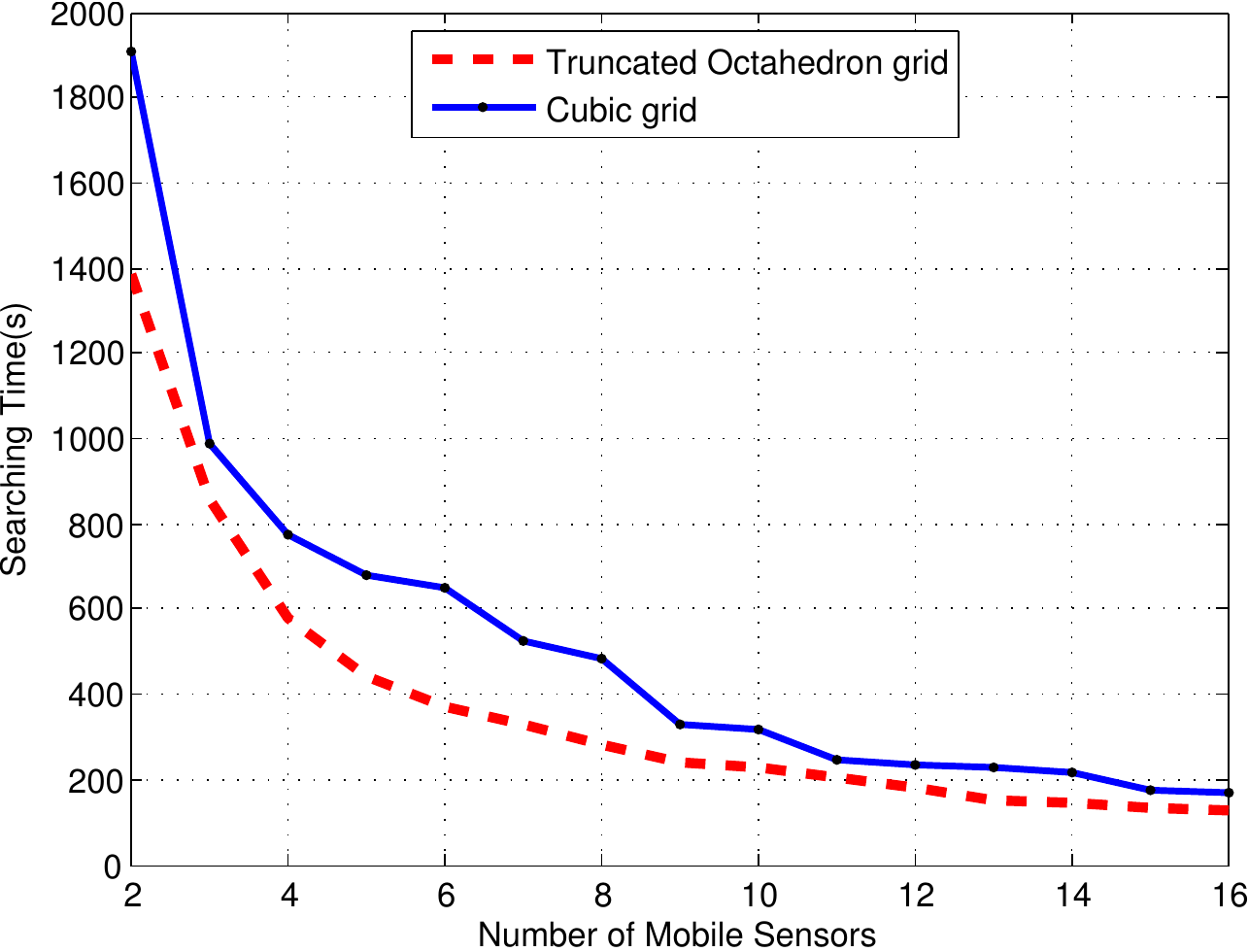}}
\caption{Searching time versus number of mobile sensors}
\label{comp2}
\end{figure}
\begin{figure}
\centering
{\includegraphics[width=0.8
\textwidth]{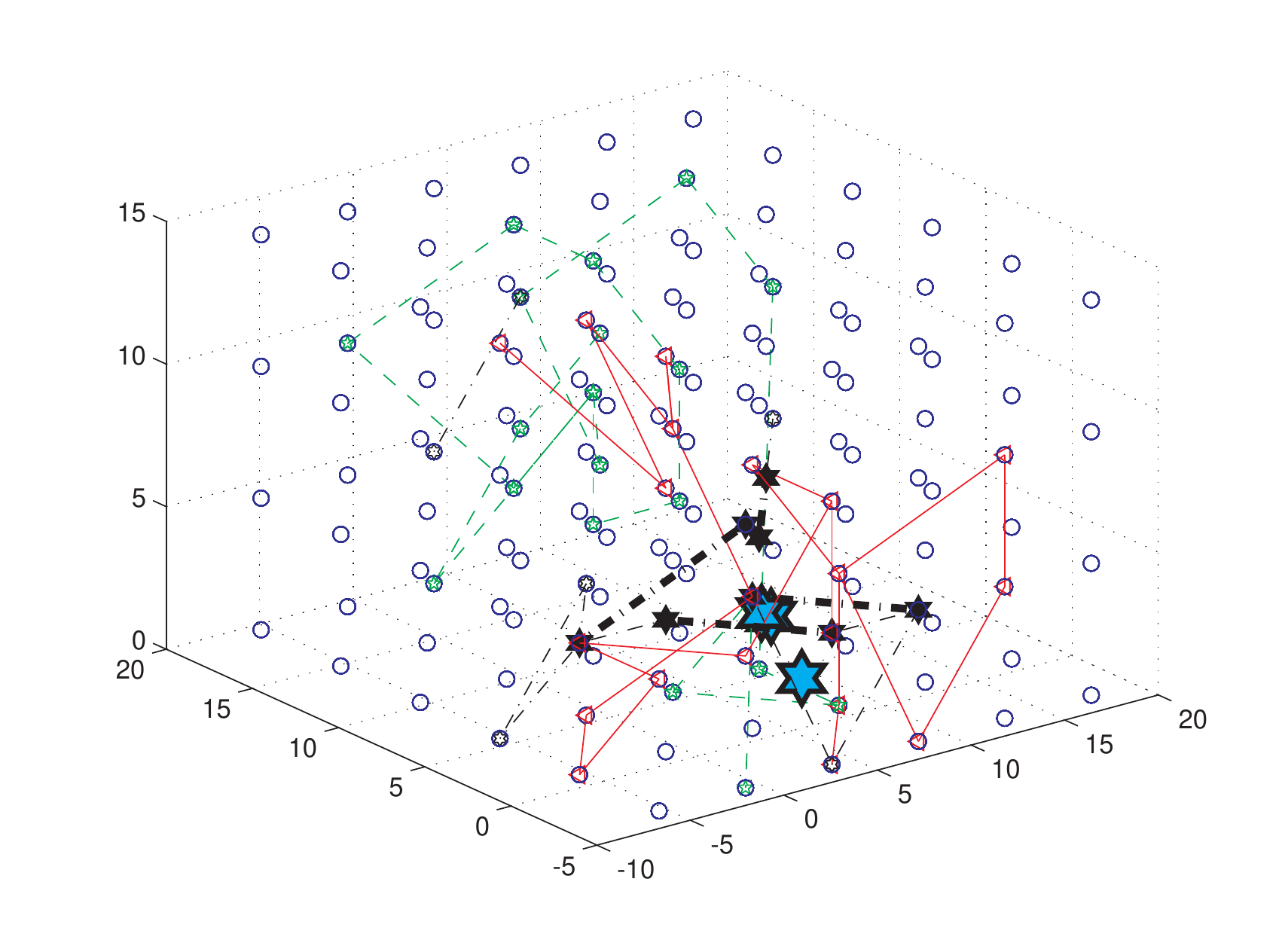}}
\caption{Robustness againt failure}
\label{rob}
\end{figure}
In the next set of simulations, we have evaluated the ability of the proposed search algorithm for handling some common uncertainties by assuming that one mobile sensor unexpectedly break during the search task. For this case, we consider six mobile sensors search for three targets. Fig.\ref{rob} shows that the other mobile sensors still be able to perform the search task in case of failure of one mobile sensors. It means that despite failure of one or more mobile sensor during the process the search task has been accomplished successfully.\\
Now, we investigate the effect of the size of the communication range on the performance of the proposed search algorithm. We have done three set of simulation. The communication ranges are chosen $R_c$, $\frac{R_c}{3}$ and $3R_C$. Simulation results show that, when the communication range is increased the search performance increases up to a certain point beyond which there is not much change in the performance of the system. As a result, it is not necessary to have the communication range to cover all the search area for better performance. Furthermore, shorter communication range means lower power consumption.
\section{Summary}
 In this chapter, we presented a set of distributed random laws for search in unknown bounded three dimensional environments. The mobile sensors utilize a truncated octahedral grid for the search process. Furthermore, they exchange information with their neighbouring sensors to minimize the search time. The performance of the proposed algorithm has been confirmed by extensive simulations and the convergence with probability 1 for any number of mobile sensors of the proposed algorithm has been mathematically rigorously proved.

%\singlespace
\singlespacing

%\input{privacyincrowd}

%%%%%%%%%%%%%%%%%%%%%%%%%%%%%%%%%%%%%%%%%%%%%%%%%%%%%%%%%%%%%%%%%%%%%%%%%%%%%%%%
\chapter{Distributed Bio-inspired Random Search for Locating Static Targets}\label{chap:1ValiCH3}
\minitoc
In many applications, it is not necessary to cover the whole area of interest. Monitoring only some specific points in the area of interest is sufficient and acceptable. For such cases, we assume only a limited number of discrete target points are to be monitored. The point of interest can be either static or mobile. Here we assume that the targets are stationary which are not able to move. In this chapter, we propose a novel decentralized bio-inspired random search algorithm for the search of randomly located static targets in 3D spaces \cite{vali}. Here, we combine the bio-inspired Levy flight random search mechanism \cite{benhamou2007many} with a 3D covering grid proposed in \cite{nazarzehi2015distributed} to optimize the search procedure. In previous works the Levy flight random walks took place on a continuous space but here the random walk occurs on a discrete grid. Unlike \cite{nazarzehi2015distributed}, here we assume the mobile sensors move randomly on the vertices of the covering grid with the length of the movement follow a Levy flight distribution. Based on this scheme, the Levy flight distribution will generate the length of the movement and the vertices of the covering truncated octahedral grid will improve the dispersion of the deployed mobile sensor. As a result, this method reduces the search task and optimizes the search procedure by dispersing efficiently the mobile sensors in the search area. In order to stop search after finding all targets and reduce the cost of operation, each mobile sensor communicates to other mobile sensors within its communication range to broadcast information regarding detected targets. Performance of this bio-inspired random search method is verified by simulations. Moreover, to evaluate the performance of the proposed grid-based Levy flight algorithm we compare it to other random search strategies for locating sparse and clustered distribution of targets. Furthermore, we give a mathematically rigorous proof of the convergence of the proposed algorithm with probability 1 for any number of mobile sensors and targets. \\
The remainder of the chapter is structured in the following way: Section \ref{4.1} describes the problem of bio-inspired random search in 3d environments. The proposed algorithm is explained in Section \ref{4.2}. Section \ref{4.3} is devoted to evaluate the proposed algorithm by simulations. Section \ref{4.4} concludes this work.
\section{Problem Formulation} \label{4.1}
The objective of this chapter is to design a decentralized bio-inspired random search algorithm to drive a network of mobile sensors for target search in a  bounded 3D area. We assume that the mobile sensors have no information about the the positions of the targets $\emph{T}_1, \emph{T}_2, ... \emph{T}_l$ and search area ($\emph{M}  \subset \mathbb{R}^3$)  , but they can detect the boundaries of the search area. Let $ p_i(.)\in \texttt{R}^3 $ be the Cartesian coordinates of the mobile sensor $i$ and each mobile sensor  has a sensing range of $\texttt{R}_s>0 $ and a communication range of  $\texttt{R}_c> \frac{4}{\sqrt{5}}\texttt{R}_s $.
%We assume a spherical sensing model such that each mobile sensor has a sensing range of $\texttt{R}_s>0 $ can reliably detect any object that is located within a distance of $\texttt{R}_s $ from the robot. In other words, sensor $i$ has the ability to identify objects in a sphere of radius $\texttt{R}_s $ defined by:
%\begin{equation}
%\texttt{S}_{i,\texttt{R}_s}=\{p\in \texttt{R}^3;\|(p-p_i (\texttt{K}))\| \leq \texttt{R}_s\}
%\end{equation}
%We assume a spherical communication model where each  mobile sensor has a communication range of $\texttt{R}_c>0 $ can reliably communicate with  any mobile sensors located within a distance of $\texttt{R}_c $ from the mobile sensor. It means that mobile sensor $i$ has the ability to obtain information on its neighbours in a sphere of radius  $\texttt{R}_c$ defined by:
%\begin{equation}
%\texttt{S}_{i,\texttt{R}_c}=\{p\in \texttt{R}^3;\|(p-p_i (\texttt{K}))\| \leq \texttt{R}_c\}
%\end{equation}
\begin{figure}
\centering
{\includegraphics[width=9cm,height=7.5cm]{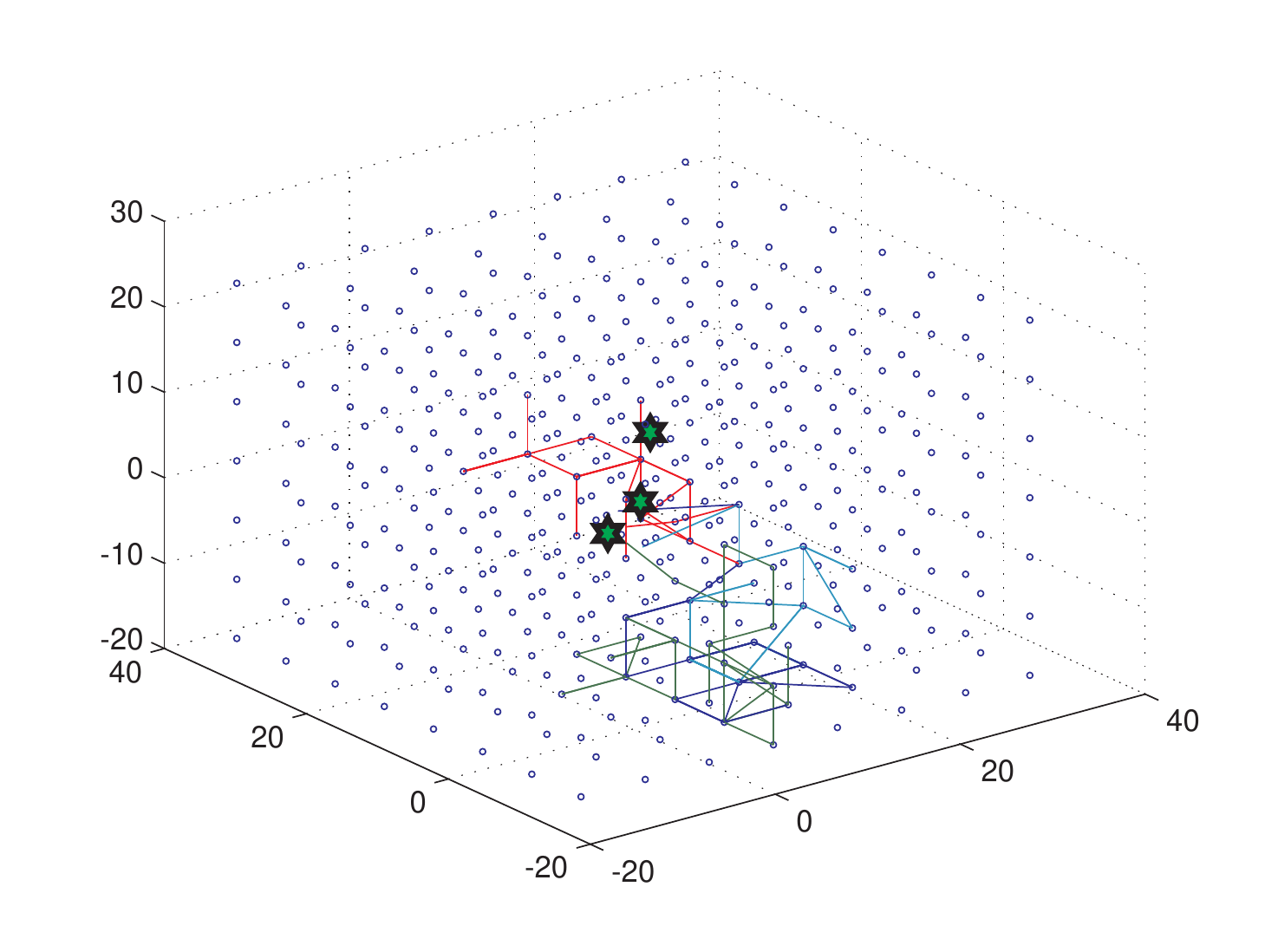}}\\
\caption{Vertices of the grid denoted by $’o’$, mobile sensors search trajectories represented by $’-’$ and targets by star }
\label{ch3}
\end{figure}
Our goal is to develop a distributed grid-based random search algorithm to drive the mobile sensors and to efficiently spread the mobile sensors in the search area for identifying all targets in a bounded 3D space. We take advantage of vertices of a 3D truncated octahedral grid proposed in \cite{alam2006coverage,alam2014coverage}.

 %Unlike \cite{alam2006coverage}, here we assume that there is not a common coordinate system among all mobile  robots, therefore we use consensus variables to build a common coordinate system for the mobile  robots. One advantage of using vertices of a truncated octahedral grid for the search is that using this grid for search leads to fewer number of vertices to be searched and this minimizes the time of search.\\
%\begin{definition} Consider a truncated octahedral grid cutting $\texttt{M} $ into equal truncated octahedrons with the sides of $\frac{2 \texttt{R}_s}{\sqrt{10}}$. Let $\texttt{V}$ be the infinite set of centres of all the truncated octahedrons of this grid. The set ${\hat{\texttt{V}}}={\texttt{V}} \cap \texttt{M} $ is called a truncated octahedral covering set of $\texttt{M} $ (Fig.\ref{ch3}).
%\end{definition}
At first, we assume there is not a common coordinate system among all mobile sensors. As a result, the mobile sensors use consensus variables to build a common coordinate system  and a common covering truncated octahedral grid. Then, they perform the search task by moving randomly with the random steps length drawn from a Levy flight probability distribution on the vertices of the grid from different initial positions to locate all randomly located targets in the 3D area. During the search, the mobile sensors keep the information of the previously detected targets and communicates with its neighbours in the communication range at the discrete sequence of times $\emph{k} = 0, 1, 2,...$, to exchange the search information.
\begin{definition}
A target is said to be detected by a mobile sensor if the Euclidean distance between the target and the mobile sensor is less than the sensing range. Mathematically speaking, the  target $T_i$ is said to detected by a mobile sensor $s_i$ , if and only if, $\left\|P_{T_i}-p_i (\texttt{K})\right\| \leq \texttt{R}_s$.
\end{definition}
Where, $P_{T_i}$ is position of target $T_i$.
%We take advantage of the notion of graph to define the relationships between neighbouring mobile  robots. As a result, the vertices $i$ and $j$ of the graph $\emph{g}(\emph{k})$ are considered connected if the  mobile  robots $i$ and $j$ are neighbours at time $\emph{k}$. In the following, we impose a condition on the connectivity
%of the graph.\\
%\begin{assumption} \label{assump10} There exists non-empty, an infinite sequence of contiguous, bounded time-intervals $[\emph{k}_m, \emph{k}_m+1)$, $m = 0, 1, 2, . . .$, such that across each $[\emph{k}_m, \emph{k}_m+1)$, the graph $\emph{g}(\emph{k})$ is connected \cite{savkin2012optimal,cheng2012self}.
%\end{assumption}
\section{Decentralized Bio-inspired Random Search Algorithm} \label{4.2}
In this section, we present a two-stage decentralized bio-inspired random search algorithm for finding targets in a bounded 3D area. First, the mobile sensors build a common 3D truncated octahedral grid as explained in the previous chapter,section 3.2.1. Then, they perform the search task by moving randomly on the vertices of the 3D grid using random walks with a Levy-flight probability distribution.\\ Levy flight is a renowned bio-inspired random search mechanism with step lengths taken from a heavy-tailed probability distribution \cite{reynolds2009levy}. Using this distribution the probability of returning to a previously visited site is smaller, therefore advantageous when target sites are sparsely and randomly distributed \cite{viswanathan2000levy,nurzaman2009yuragi}. To optimize the search task, we supplement the bio-inspired Levy flight random search strategy to a covering truncated octahedral grid. It means that, the mobile sensors do the search task by moving randomly to the vertices of the covering grid with steps length drawn form Levy-flight probability distribution. At first we assume the mobile sensors do not have a common coordinate system and a common covering grid therefore the robots use consensus variables to build a common covering grid 3D \cite{nazarzehi2015distributed}.
In our previous search algorithm, mobile sensors performed the search task by moving randomly on the unvisited vertices in their neighbourhood. To optimize the search task and to minimize the time of search  in this section, we propose a novel random search strategy for the mobile sensors. Using this method, the mobile sensors do the search task by moving randomly to the vertices of the covering grid based on the random walk generated by the Levy flight distribution. A Levy flight random walk pattern uses a Levy probability distribution as follows \cite{viswanathan2008levy}:
\begin{equation}
P_{\alpha,\gamma}(l)=\frac{1}{2 \pi}\int_{-\infty}^{+\infty} e^{-\gamma q^\alpha}cos(ql) dq
\end{equation}
Where $\gamma$ denotes the scaling factor and $\alpha$ represents the shape of the distribution. Also, the distribution is symmetrical around $l$. In this chapter,
 the Levy flight  provides a random walk with the random step length is drawn from an approximated  Levy distribution proposed in \cite {sutantyo2013collective} as follows:\\
\begin{equation}
P_\alpha(l)=(l)^{-\alpha}
\end{equation}
Where $l$ is a random number, $\alpha>0$ defines the length of walk and $1<\alpha<3$. Based on our method, a mobile sensor selects next vertices of the grid as follows:\\
Step 1: The mobile sensor's direction will be drawn from a uniform distribution of angles in a range of $[0, 2\pi]$. In other words, the mobile sensor chooses a direction uniformly from a sphere.\\
Step 2: The mobile sensor chooses a step length from a Levy flight distribution where the power law parameter $\alpha$ determines the distance travelled by the mobile sensors between direction changes.\\
 Step 3: The mobile sensor moves to the closest vertices of the grid to the point calculated in Step 2 and 3.\\
 Step 4: If a target is ever within the sensing rage of the mobile sensor, then the target is said to be detected. If no target is found, the mobile robotic sensor repeats its search according to steps mentioned above.\\
 Note that the mobile sensors move from one vertex of the grid to another until they detect all targets. During the search process if a new random walk places a mobile sensor out of the boundary of the search area it will be rejected and a new random walk will be drawn from Levy probability distribution.\\
 Let ${\hat{\texttt{V}}}$ be the  set of  all vertices of the 3D covering grid and $\hat{v} \in \hat{\texttt{V}}$ is a randomly selected element of $\hat{\texttt{V}}$ with Levy flight probability distribution. We propose the following random algorithm for the  mobile sensors to search the 3D area:
\begin{equation} \label{eq:law1}
p_i(\emph{k}+1)=
\hat{v}
\end{equation}
It means that the mobile sensors use steps 1,2,3 to select an element of ${\hat{\texttt{V}}}$ in each step and move on it.\\
Here we assume that the Boolean variables $b_\emph{T}(\emph{k})$ defines the states of the targets at the time $\emph{k}$. If target $\emph{T}$ has been detected by any of the mobile sensors before time $\emph{k}$ then, $b_\emph{T}(\emph{k})= 1$  and $b_\emph{T}(\emph{k})= 0$ otherwise.
Throughout the search, each mobile sensor communicates to other mobile sensors in the communication range at the discrete sequence of times $k = 0, 1, 2,…$ to exchange the information about detected targets. The search process should be stopped after finding all targets.
In other words,
\begin{equation} \label{eq:law2}
 p_i(\emph{k}+1)= p_i(\emph{k})
\end{equation}
\begin{equation*}
if \quad \forall\quad T=\left\{\emph{T}_1,\emph{T}_2,...\emph{T}_l\right\}, \exists\quad \emph{k},i\quad;\left\|(P_\emph{T}-p_i (\emph{k}))\right\| \leq \emph{R}_s
\end{equation*}
Where, $T=\left\{\emph{T}_1,\emph{T}_2,...\emph{T}_l\right\}$ and $P_\emph{T}=\left\{P_{\emph{T}_1},P_{\emph{T}_2},...P_{\emph{T}_l}\right\}$, are the sets of targets and their positions, respectively.
\begin{theorem} Suppose that the mobile sensors move according to the law (~\ref{eq:law1}),(~\ref{eq:law2}). Then for any number of mobile sensors and any number of targets, with probability 1 there exists a time $\emph{k}_0$  such that:
 \begin{equation*}
\forall \quad  \emph{t} \in T,\quad b_\emph{T}(\emph{k}_0)=1
\end{equation*}
\end{theorem}

\begin{figure*}[t!]
\begin{center}
\mbox{
\subfigure[Mobile sensors' positions after convergence to the vertices of the common covering truncated octahedral grid]{
{\includegraphics[width=0.75\textwidth,height=0.6\textwidth]{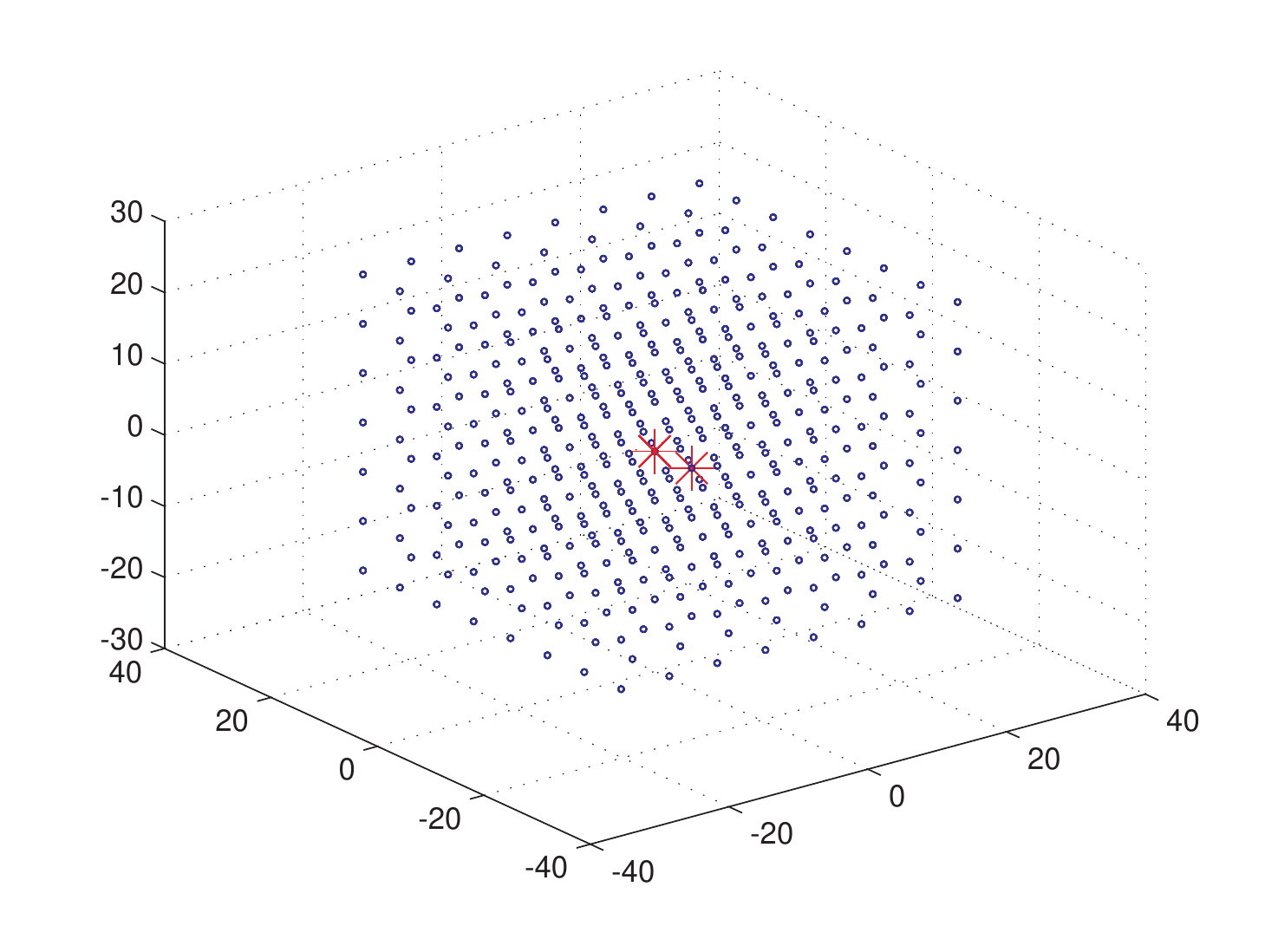}}\quad
\label{levy1}
}
}
\mbox{
\subfigure[Mobile sensors' trajectories after 15 steps]{
{\includegraphics[width=0.75\textwidth,height=0.6\textwidth]{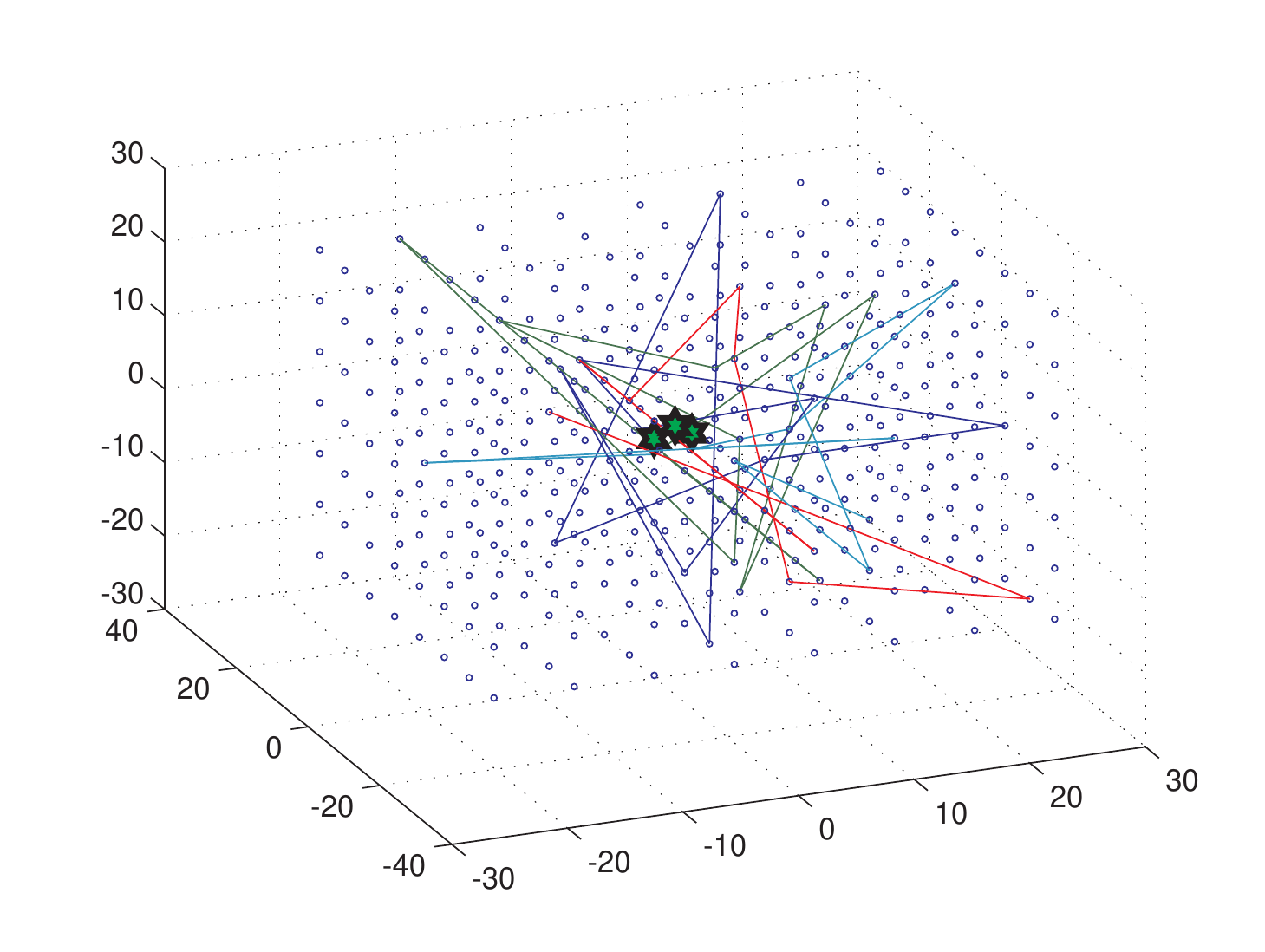}}\quad
\label{levy2}
}
}
\caption{Trajectories of four mobile sensors performing grid-based Levy flight random search for locating clustered  targets. mobile sensors' trajectories denoted by -, vertices of the common grid by o, Targets by star}
\label{levy}
\end{center}
\end{figure*}

\begin{figure}
\centering
{\includegraphics[width=8.5cm,height=8cm]{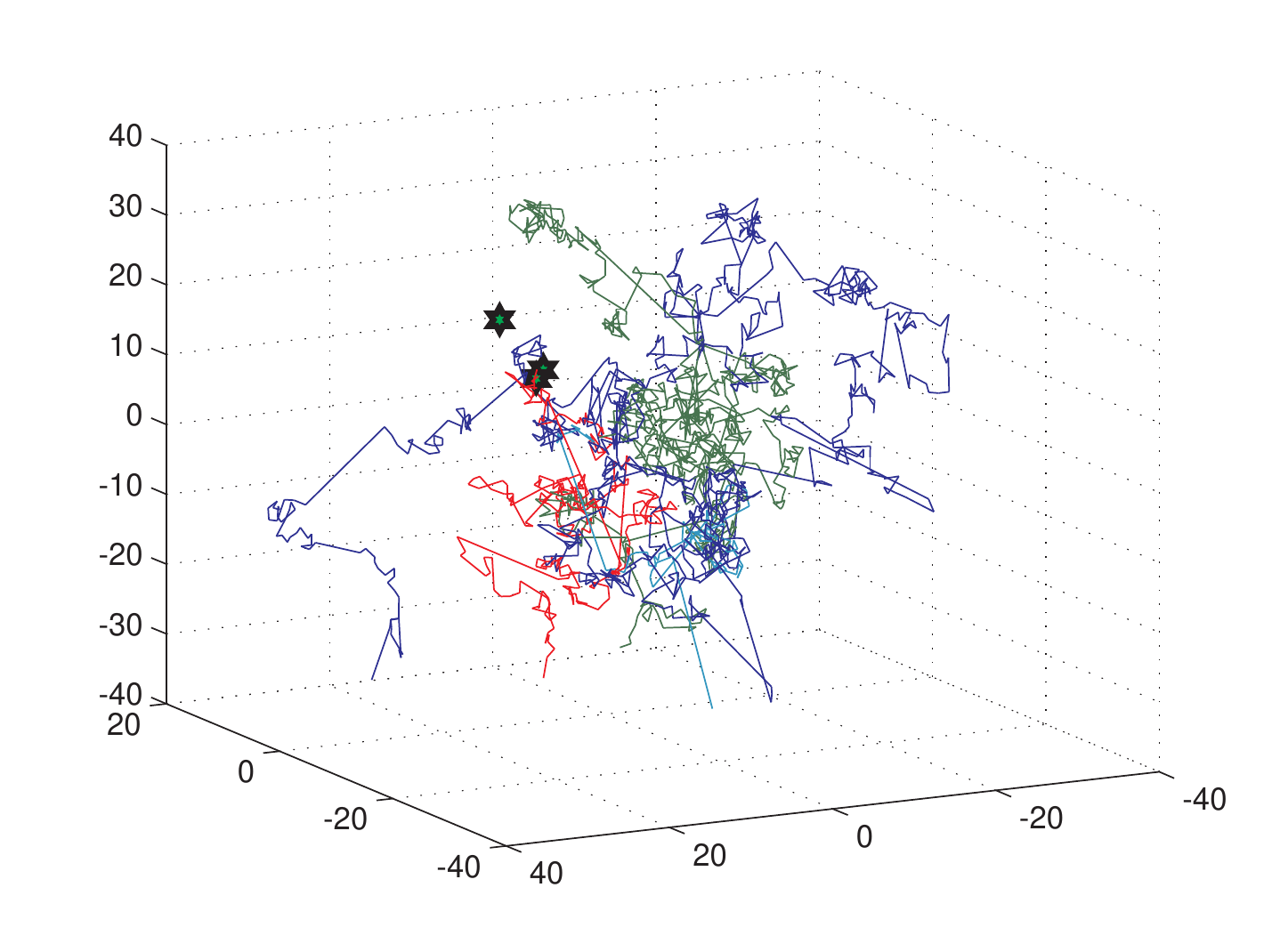}}\\
\caption{Levy flight without grid search: trajectories of four mobile sensors performing Levy flight search for three targets. mobile sensors' trajectories denoted by -, targets by star.}
\label{nogrid}
\end{figure}
%\begin{figure*}[t!]
%\begin{center}
%\mbox{
%\subfigure[Mobile sensors' positions after convergence to the vertices of the common covering truncated octahedral grid]{
%{\includegraphics[width=0.5\textwidth]{FigVali/optimum/sp1}}\quad
%\label{levy1}
%}
%\hspace{-2mm}
%\subfigure[Mobile sensors' trajectories after 15 steps]{
%{\includegraphics[width=0.5\textwidth]{FigVali/optimum/sp}}\quad
%\label{levy2}
%}
%}
%\vspace{-4mm}
%\caption{Trajectories of four mobile sensors performing grid-based Levy flight random search for locating clustered  targets. mobile sensors' trajectories denoted by -, vertices of the common grid by o, Targets by star}
%\vspace{-8mm}
%\label{levy}
%\end{center}
%\end{figure*}
\textit{Proof}:
The algorithm (~\ref{eq:law1}),(~\ref{eq:law2}) describes an absorbing Markov chain that consists of a number of absorbing states (that are impossible to leave) and many transient states. Vertices, where mobile sensors stop, are considered as absorbing states (~\ref{eq:law2}). Also, the vertices of the grid which are visited during the search are considered as transient states (~\ref{eq:law1}). Relation (~\ref{eq:law1}) implies that the mobile sensors randomly move to visit unvisited vertices, this continues until all targets are detected (an absorbing state). It is clear that from any initial state with probability 1, one of the absorbing states will be reached.
\section{Simulation Results} \label{4.3}
In this section, we conduct several simulations using MATLAB R2014b to evaluate the performance of the proposed bio-inspired Levy flight search algorithm for different number of mobile sensors and for clustered and sparsely distributed targets in bounded three dimensional environments. At the first set of simulations, we assume that three targets are randomly distributed within the search area, and they are shown as stars, see Fig.\ref{levy2}. The initial positions of all mobile sensors and targets are generated randomly within the search area. For this case, a team of four mobile sensors are  deployed in the area to locate clustered targets. Fig.\ref{levy2} indicates the results of searching for three randomly distributed  stationary targets within the search region. As shown in Fig.\ref{levy1}, using the update law proposed in the previous chapter, first the mobile sensors build a common covering truncated octahedral grid and after that all mobile sensors move to the vertices of this grid. Then, the mobile sensors perform the search task by moving randomly by the random walk taken by a Levy-flight probability distribution on the vertices of the grid. Fig.\ref{levy2} shows the trajectories of the mobile sensors during the search. For this case, the mobile sensors have detected the targets after 15 steps. We performed a simulation to compare the performance of our proposed Levy-flight grid-based search algorithm with Levy flight without grid. Our simulation (Fig.\ref{nogrid}) demonstrates that using Levy flight without a grid, the mobile sensors has detected the targets after 800 steps which is far more than the time taken by Levy flight grid based method.
\begin{figure}
\centering
{\includegraphics[width=8cm,height=7.5cm]{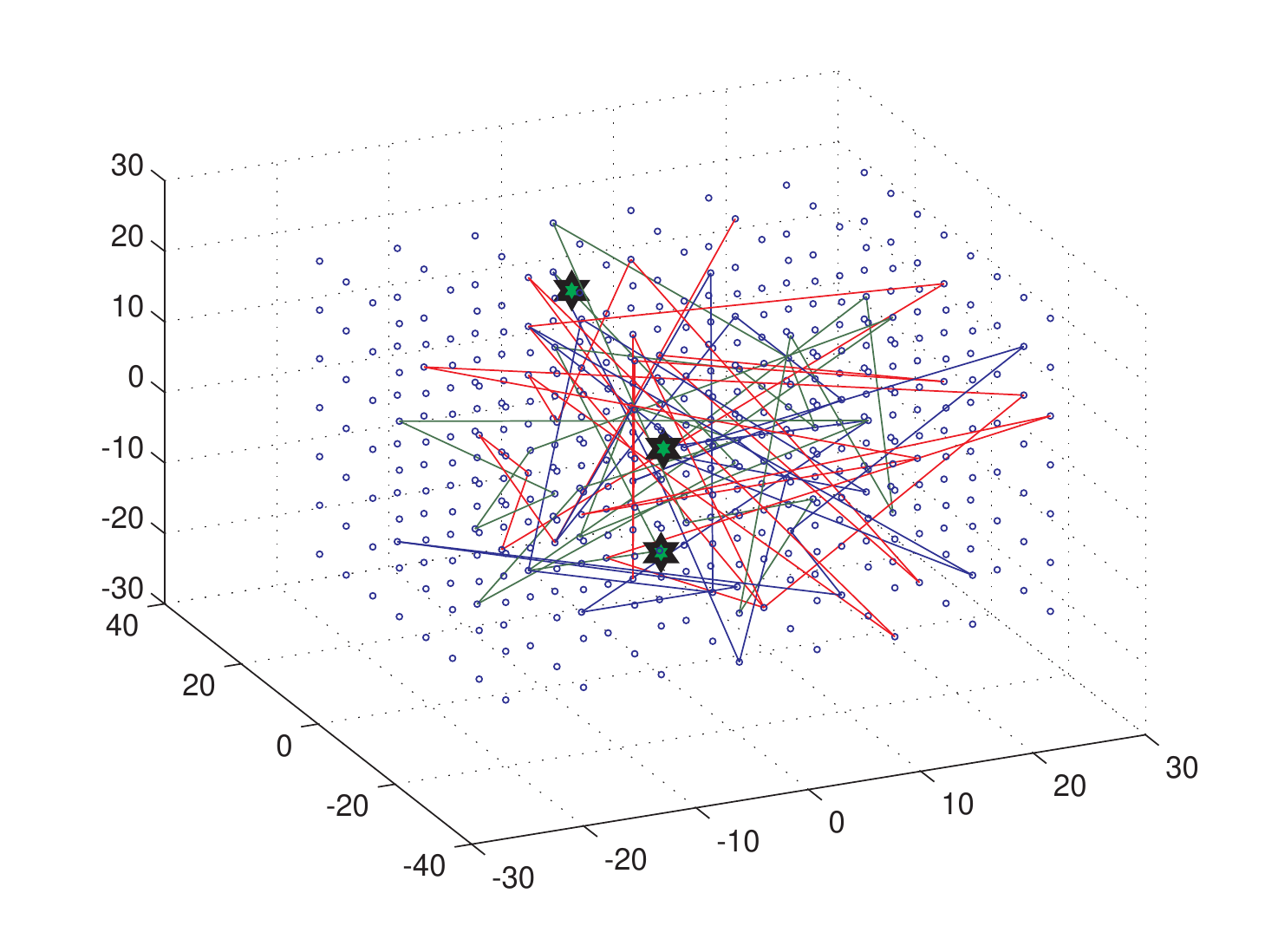}}\\
\caption{Trajectories of mobile sensors performing bio-inspired Levy flight random search for locating sparsely located  targets. mobile sensors' trajectories denoted by -, vertices of the common grid by o, targets by star.}
\label{levy3}
\end{figure}

\begin{figure*}[t!]
\begin{center}
\mbox{
\subfigure[Searching time versus number of mobile sensors]{
{\includegraphics[width=0.75\textwidth,height=0.6\textwidth]{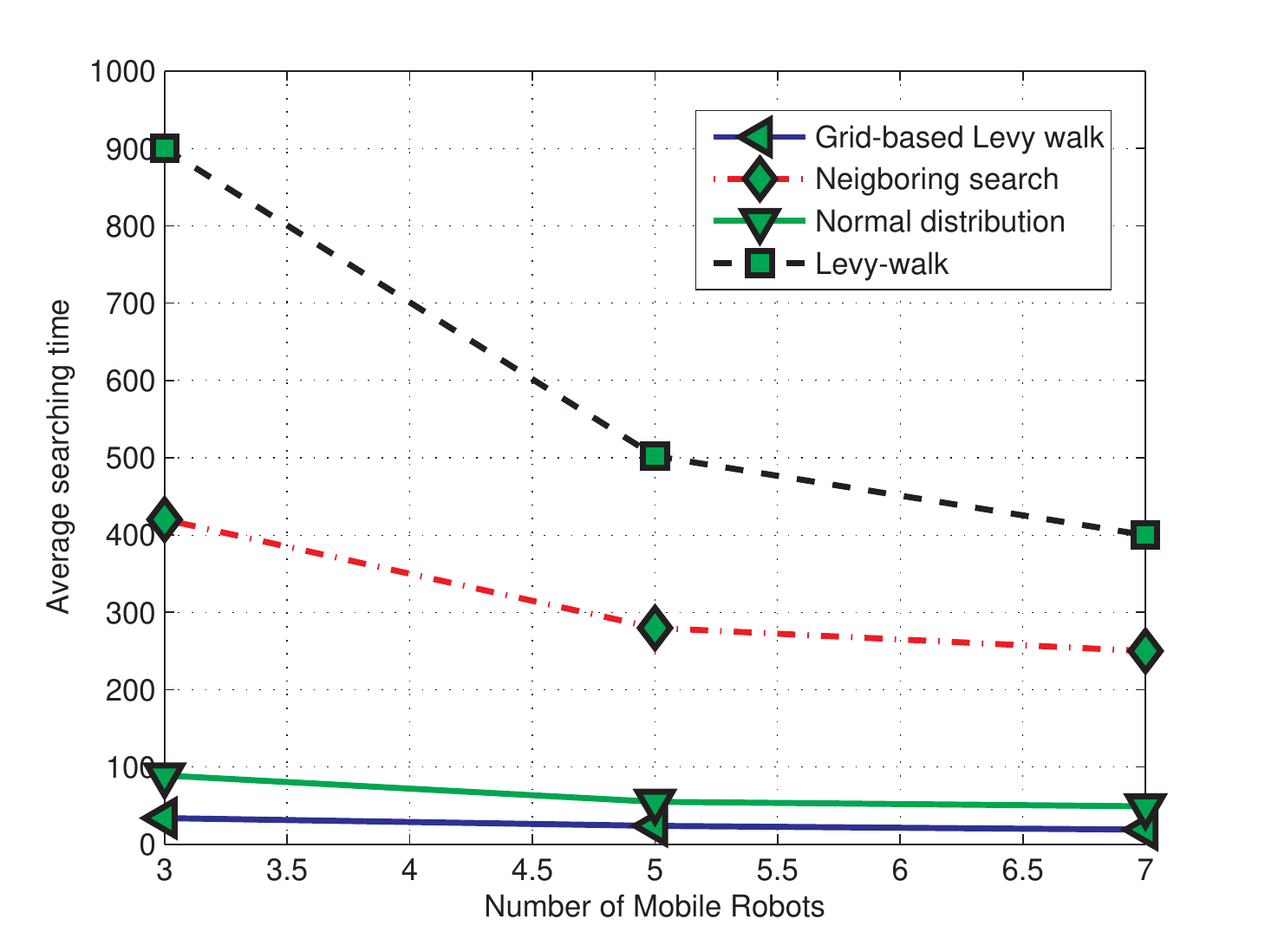}}\quad
\label{comp10}
}
}
\mbox{
\subfigure[Searching time versus number of targets]{
{\includegraphics[width=0.75\textwidth,height=0.6\textwidth]{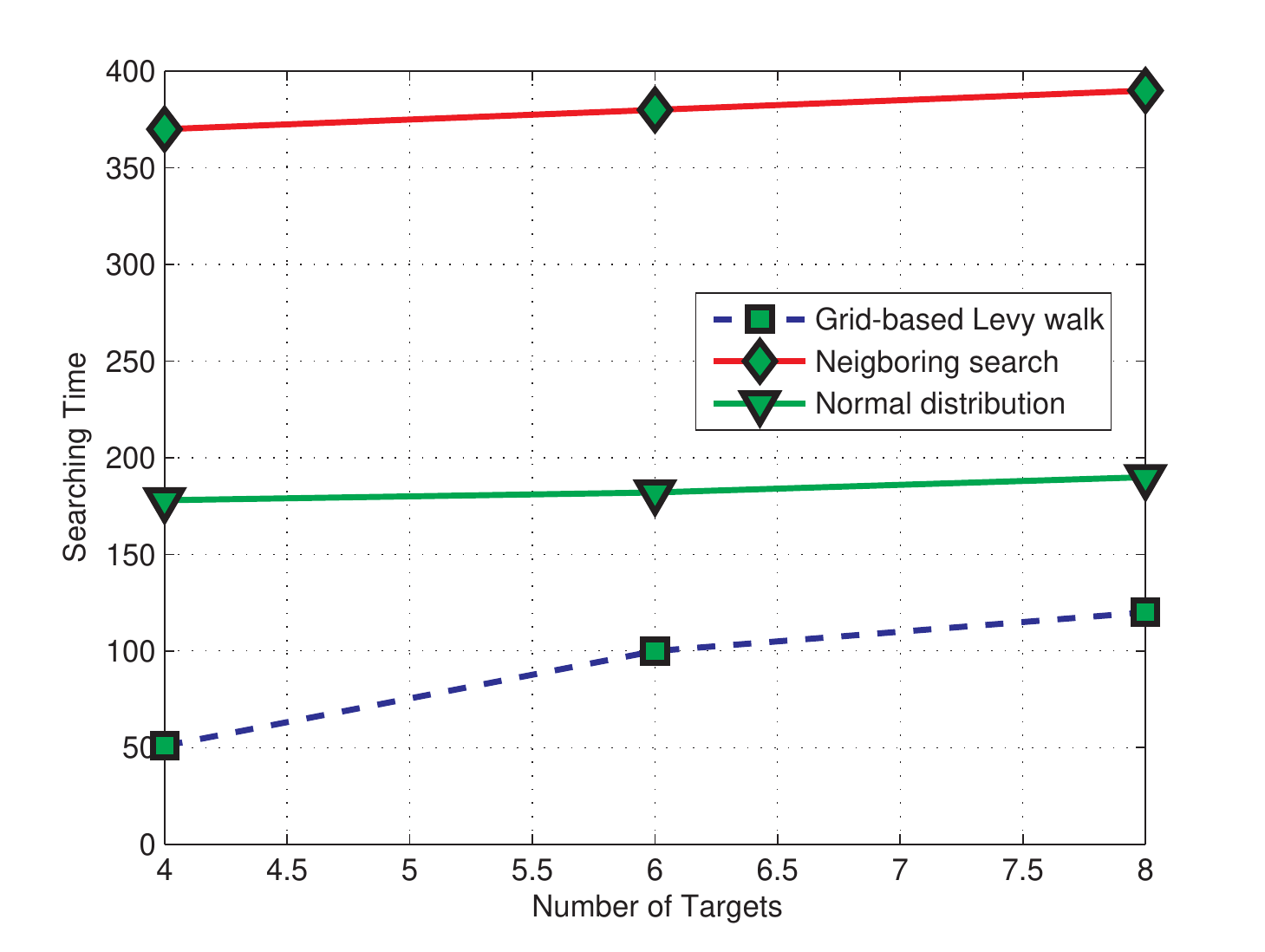}}\quad
\label{comp11}
}
}
\caption{Comparison of random search methods using different number of mobile sensors and targets}
\label{no}
\end{center}
\end{figure*}
%%%%%%%%%%%%%%%%%%%%%%%%%%%%%%%%%%%%%%%%%%%%%%%%%%%%%%%%%%%%%%%%%%%%%%%%%%%%%%%%%%%%%%%%%%%%%%%%%%%%%
%%\begin{figure*}[t!]
%\begin{center}
%\mbox{
%\subfigure[Searching time versus number of mobile sensors]{
%{\includegraphics[width=0.5\textwidth]{FigVali/optimum/rob}}\quad
%\label{comp10}
%}
%\hspace{-2mm}
%\subfigure[Searching time versus number of targets]{
%{\includegraphics[width=0.5\textwidth]{FigVali/optimum/tar}}\quad
%\label{comp11}
%}
%}
%\vspace{-4mm}
%\caption{Comparison of random search methods using different number of mobile sensors and targets}
%\vspace{-8mm}
%\label{copm12}
%\end{center}
%\end{figure*}

\begin{table}
 \caption{Comparison of different search styles for finding sparsely located targets}
 \label{no}
 \begin{center}
 {\small
\begin{tabular*}{\columnwidth}[t]{cp{100pt}}\hline\hline
Type of random walk & Average search time \\ \hline
$Levy-flight$ & {\raggedright 1300 } \\
$Grid-based-Levy-flight$ & {\raggedright 87 } \\
$Grid-based-normal-distribution$ & {\raggedright  125} \\
$Neighboring-random-search$ & {\raggedright 300 } \\
\hline\hline
  \end{tabular*}
 }%end of small
 \end{center}
\end{table}
For the second case, we assume that targets are sparsely distributed in the search environment. Fig.\ref{levy3} shows paths of the mobile sensors doing the search task by moving randomly to the vertices of the common truncated octahedral grid to find the targets. For this case, the mobile sensors have detected the targets after 47 steps. Simulation results in Fig.\ref{levy} and Fig.\ref{levy3} demonstrate that for both cases the mobile sensors have detected the given targets after a reasonable amount of time(15 and 47 steps) by moving randomly to the vertices of the common covering grid. It means that, simulations verify the effectiveness of the proposed decentralized bio-inspired random search algorithm for finding targets in both cases.\\
The next set of simulations are carried out to compare the proposed grid-based Levy flight search algorithm to the Levy flight search algorithm without using grid, our random search algorithm proposed in the last chapter and to the grid-based random search with the step length taken by a normal distribution for detecting targets in a bounded 3D environment. In the following figures and table, for notational convenience we call the proposed grid-based Levy flight search algorithm, the Levy flight search algorithm without using a grid, our random search algorithm proposed in \cite{nazarzehi2015distributed} and the grid-based random search with the step length taken by a normal distribution as grid-based Levy flight, Levy flight, neighboring search and normal distribution, respectively.\\
As these algorithms are stochastic, we run each simulation in 10 trials. Moreover, we compare these random search methods using different number of mobile sensors and targets with different initial positions. Performance of the simulations is measured by the total steps, which is indicative of the total time spent to locate targets. Fig.\ref{comp10} shows simulation results for finding four clustered targets in a bounded 3D area using 3,5 and 7 mobile sensors by four random search methods. As shown, the proposed grid-based Levy flight method  outperforms the other random search strategies for locating these targets. Fig.\ref{comp11} displays the simulation results for locating 4, 6 and 8 clustered targets in 3D area using three mobile sensors. This simulation demonstrates that using our proposed algorithm the mobile sensors detected the given number of targets in the least time.\\
To evaluate the proposed algorithm for locating sparsely located targets, we did several simulations with different number of targets and compared our proposed grid-based Levy walk random search strategy to the others. Table.\ref{no} demonstrates that the proposed grid-based Levy flight strategy out-performs other random walk strategies in locating targets.
%Moreover, we study the effect of the Levy search parameter $\alpha$ by conducting some simulation using different values of $\alpha$, (0.1, 0.4, 1.0 and 2.0). Fig., present the result corresponding to the case when search parameter is 0.25. In Figures 4-6 show, simulation result for the case when the movement step follows a Levy flight with $\alpha = .4, 1 and 2$. Our simulations show that the increase in the length of the movement decreases the total area searched by the mobile robotic sensors.
\section{Summary} \label{4.4}
In this chapter, we proposed an efficient bio-inspired random search algorithm to drive mobile sensors to find clustered and sparsely distributed targets in bounded 3D areas. The proposed method combines the bio-inspired Levy flight random search mechanism for determining the length of the walk with vertices of a covering truncated octahedral grid to optimize the search procedure.\\ We showed that moving randomly on the vertices of the covering grid can minimize  the time of locating randomly located objects in the search area by improving dispersion of the robots. Simulation results demonstrated that the proposed algorithm outperform to the other random search methods. Also, we gave mathematically rigorous proof of convergence with probability 1 of the proposed algorithm for any number of mobile sensors and targets.

%\singlespace
\singlespacing

\chapter{Distributed Bio-inspired Algorithm for Search of Moving Targets}\label{chap:1ValiCH5}
\minitoc
In the previous chapter, we introduced a decentralized random search algorithm for detecting static targets in three dimensional spaces. In this chapter, we study the problem of detecting  mobile targets in a bounded 3D environment by a network of mobile sensors. Compare to a single mobile sensor, a network of mobile sensors can complete the search task in cheaper, faster, and more efficient way \cite{ anderson2008implicit}. A network of mobile sensors can handle the uncertainties and accomplish the task even when some mobile sensor fail during the search \cite{bucsoniu2015handling}. As a result, it can significantly improve the efficiency and provide better robustness and adaptability. A mobile or dynamic target has the ability to change its location. A running vehicle can be represented as a dynamic target. The main focus of our proposed search method is the movement of the dynamic target which  is complicated compared to the static one \cite{yamashita2001searching}. In \cite{koopman1979search}, Koopman showed that the Lawnmower search strategy is optimal with stationary targets in 2D spaces, but the analysis does not apply to the problem of search for moving targets. In \cite{doi:10.1117/12.918719}, authors demonstrated the advantages of random Levy walk search over the lawnmower strategy especially for moving targets in 2D spaces. The work presented here addresses significant gaps in the deployment of a team of mobile robotic sensors for the search of mobile targets in three-dimensional spaces. We introduce a random search algorithm based on the Levy walk in 3D spaces. Furthermore, we take advantage of vertices of a covering grid to efficiently disperse the mobile robot for the search task.\\
Here, we use our algorithm presented in the previous chapter for identifying moving targets by combining the bio-inspired Levy flight random search mechanism with a 3D covering grid proposed in \cite{nazarzehi2015distributed}. Unlike the last chapter, in this chapter we assume targets are moving in the search space randomly.
%In previous works the Levy flight random walks took place on a continuous space but here the random walk occurs on a discrete grid.
Using proposed algorithm, the mobile robotic sensors randomly move on the vertices of the covering grid with the length of the movement follow a Levy flight distribution. The vertices of the covering truncated octahedral grid optimize the search task by efficiently spreading mobile sensors in the search area. To stop the search procedure, each sensor communicates to the other sensors in its communication range to broadcast information regarding detected targets. The performance of this grid-based random search method is verified by extensive simulations.\\ To evaluate the performance of the proposed grid-based Levy walk algorithm we compare it to the Levy walk random search strategy for detecting moving targets. Furthermore, we give a mathematically rigorous proof of the convergence of the proposed algorithm with probability 1 for any number of mobile sensors and moving targets. The proposed algorithm is distributed, and it relies only on local sensing and communication, and does not require absolute positioning system.\\ The rest of the chapter is structured in the following way. Section \ref{5.1} describes the problem of search by a network of mobile sensor sensors in 3D environments. The proposed algorithm is described in Section \ref{5.2}. Section \ref{5.3} presents simulation results and section \ref{5.4} concludes this work.
\begin{figure*}[t!]
\begin{center}
\mbox{
\subfigure[Trajectories of mobile senosrs/targets performing grid-based Levy flight random search: Mobile sensors' trajectories denoted by -, Mobile sensors' denoted by sphere, Mobile targets' trajectories represented by --, Mobile targets' by  star]{
{\includegraphics[width=0.75\textwidth,height=0.6\textwidth]{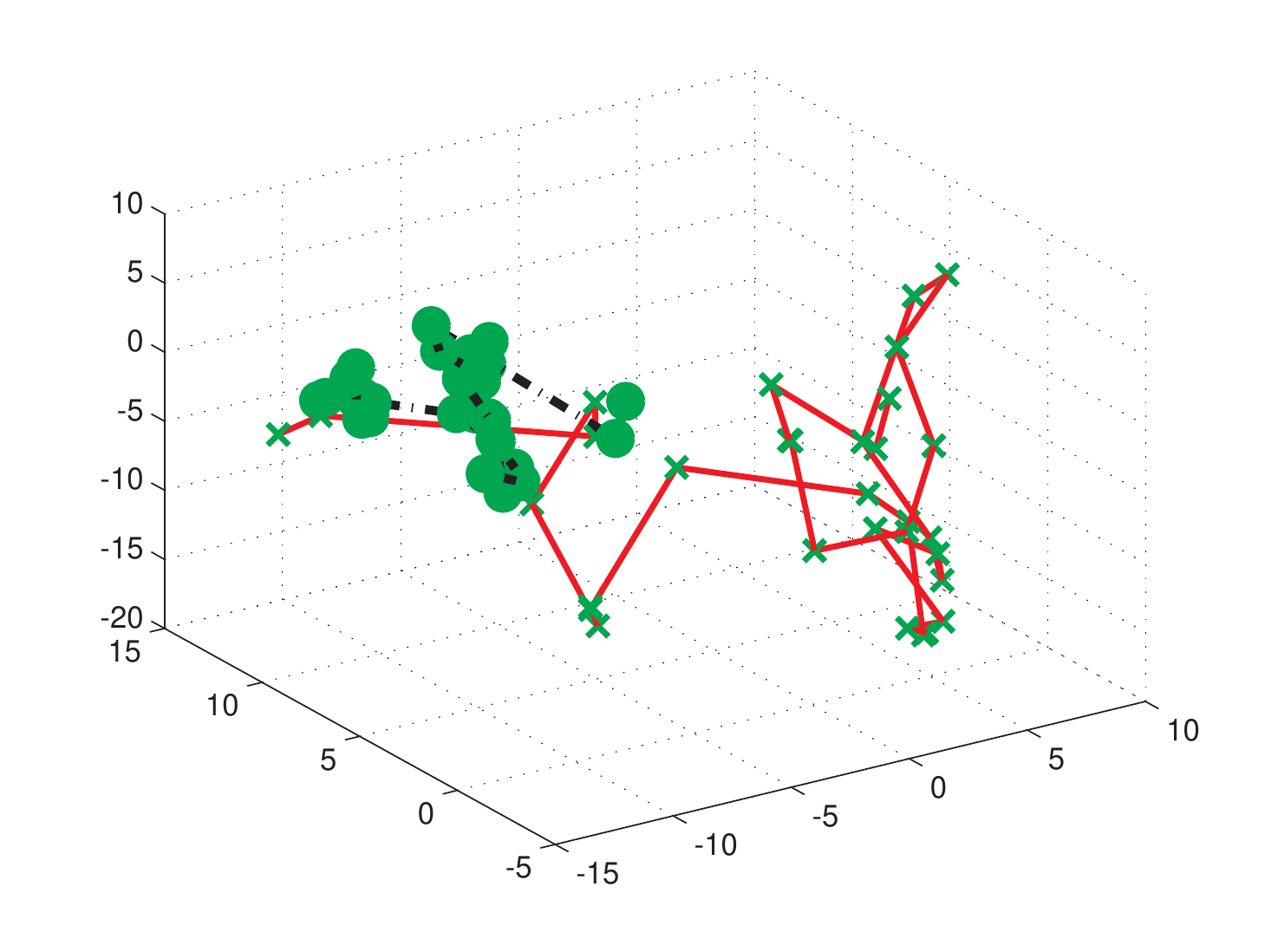}}\quad
\label{moving1}
}
}
\mbox{
\subfigure[Mobile sensors represented by their sensing sphere, communication between robots denoted by ---]{
{\includegraphics[width=0.75\textwidth,height=0.6\textwidth]{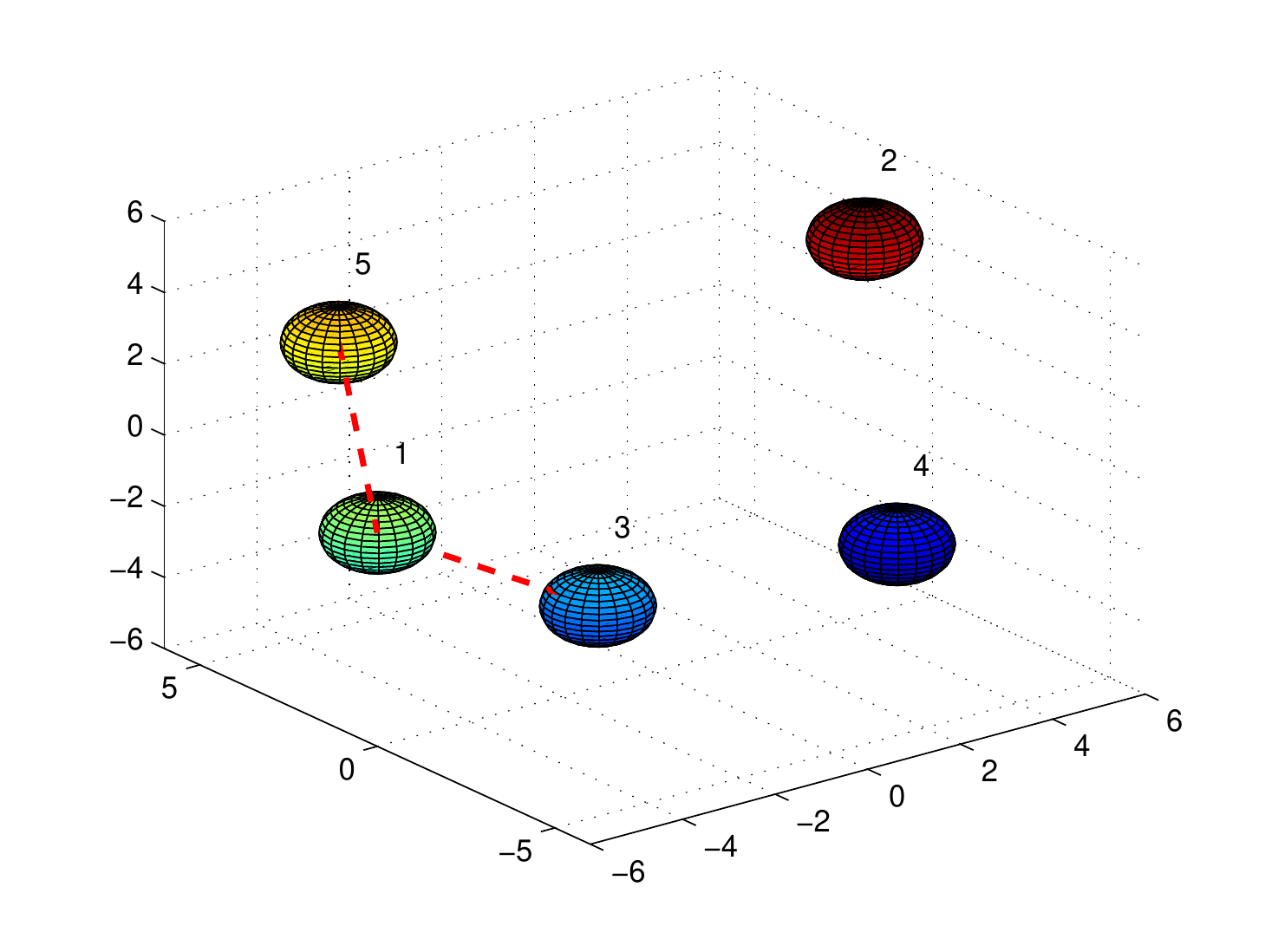}}\quad
\label{moving2}
}
}
\caption{A schematic of mobile sensors/targets}
\label{moving}
\end{center}
\end{figure*}
\section{Problem Statement} \label{5.1}
In this chapter, we study the problem of detecting mobile targets moving in a bounded 3D area using a network of mobile robotic sensors. Let $l$ mobile robotic sensors are deployed to detect $m$ moving targets in the search area. Also, let $ p_i(.)\in \texttt{R}^3 $ be the Cartesian coordinates of the mobile sensor $i$. Let {$\emph{T}_{m_1}, \emph{T}_{m_2}, ... \emph{T}_{m_l}$} be $l$ number of mobile targets in the search area that the mobile sensors have no information about their positions, the spatial distribution and the movement of the targets but they know the total number of mobile targets. Suppose that the search area ($\emph{M}  \subset \mathbb{R}^3$) is unknown to the mobile robotic network, but the mobile robotic sensors can detect the boundaries of $\emph{M}$. The problem of detecting moving targets depends not only on the mobility behavior of the robotic sensors but also on the movement of the mobile targets themselves.\\
In practice, mobile targets can have a wide variety of movement patterns. In this work, we consider the case where mobile robotic sensors and targets move randomly in the search area with Levy walk/normal probability distributions, respectively.
Fig.\ref{moving1} shows trajectories of mobile sensors/targets moving randomly in the search area. Here we assume the moving targets select a random location to travel to next location depending on a normal probability distribution $p(x)$. To find the next location for a mobile target, a  point is chosen uniformly at random on the surface of a unit sphere. This point establishes the direction for the step. Second, the step length $\lambda$ is determined from a probability density. We take advantage of simple kinematic model to describe the mobile targets' motions  in the three-dimensional environment as follows:
\begin{equation*}
x_t(k+1)=x_t(k)+\lambda\cos(\theta)cos(\psi)
\end{equation*}
\begin{equation*}
y_t(k+1)=y_t(k)+\lambda\cos(\theta)sin(\psi) \\
\end{equation*}
\begin{equation} \label{movingt}
z_t(k+1)=z_t(k)+\lambda\sin(\theta)\\
\end{equation}
Where $x_t(k)$, $y_t(k)$ and $z_t(k)$ are the Cartesian coordinates of the target. As mentioned earlier, the  mobile targets chooses randomly their directions ( $\theta$ and $\psi$ )  between $\left[0 \quad 2\pi\right)$.\\
We assume the mobile robotic sensors are equipped with short range distance sensors. These sensors have a maximum sensing range $\texttt{R}_s>0$. It means that the mobile robotic sensors can only sense the environment and detect mobile targets within their sensing area, which is the sphere of radius $\texttt{R}_s$ cantered at the sensor. Moreover, we assume that a mobile target will be identified when it falls into the sensing range of a mobile robotic sensor. Schematic of two mobile sensors moving randomly to find two moving targets is shown in Fig.\ref{mobiletarget2}. To back a moving mobile sensor to the search area when a mobile sensor reaches the boundary of the search area, it is given a new random heading that would direct it back to the search area. Here, our search goal is, to detect every moving target by at least one of the mobile robotic sensors and the target searching task is accomplished when all moving targets are detected.

%We assume a spherical communication model where each  mobile robot has a communication range of $\texttt{R}_c>0 $ can reliably communicate with  any mobile robots located within a distance of $\texttt{R}_c $ from the mobile robot. It means that mobile robot $i$ has the ability to obtain information on its neighbors in a sphere of radius  $\texttt{R}_c$ defined by:
%\begin{equation}
%\texttt{S}_{i,\texttt{R}_c}=\{p\in \texttt{R}^3;\|(p-p_i (\texttt{K}))\| \leq \texttt{R}_c\}
%\end{equation}
We use a bio-inspired distributed random search algorithm proposed in the last chapter to efficiently spread the mobile sensors in the search area for identifying moving targets in a bounded 3D space. To disperse the mobile sensors in the search area, we take advantage of vertices of a 3D truncated octahedral grid proposed in \cite{alam2006coverage}.
%One advantage of using vertices of a truncated octahedral grid for the search is that using this grid for search leads to fewer number of vertices to be searched and this minimizes the time of search.\\
%\textit{Definition 2.1}: Consider a truncated octahedral grid cutting $\texttt{M} $ into equal truncated octahedrons with the sides of $\frac{2 \texttt{R}_s}{\sqrt{10}}$. Let $\texttt{V}$ be the infinite set of centres of all the truncated octahedrons of this grid. The set ${\hat{\texttt{V}}}={\texttt{V}} \cap \texttt{M} $ is called a truncated octahedral covering set of $\texttt{M} $ \\
\begin{figure}
\centering
{\includegraphics[width=0.7\textwidth,height=0.7\textwidth]{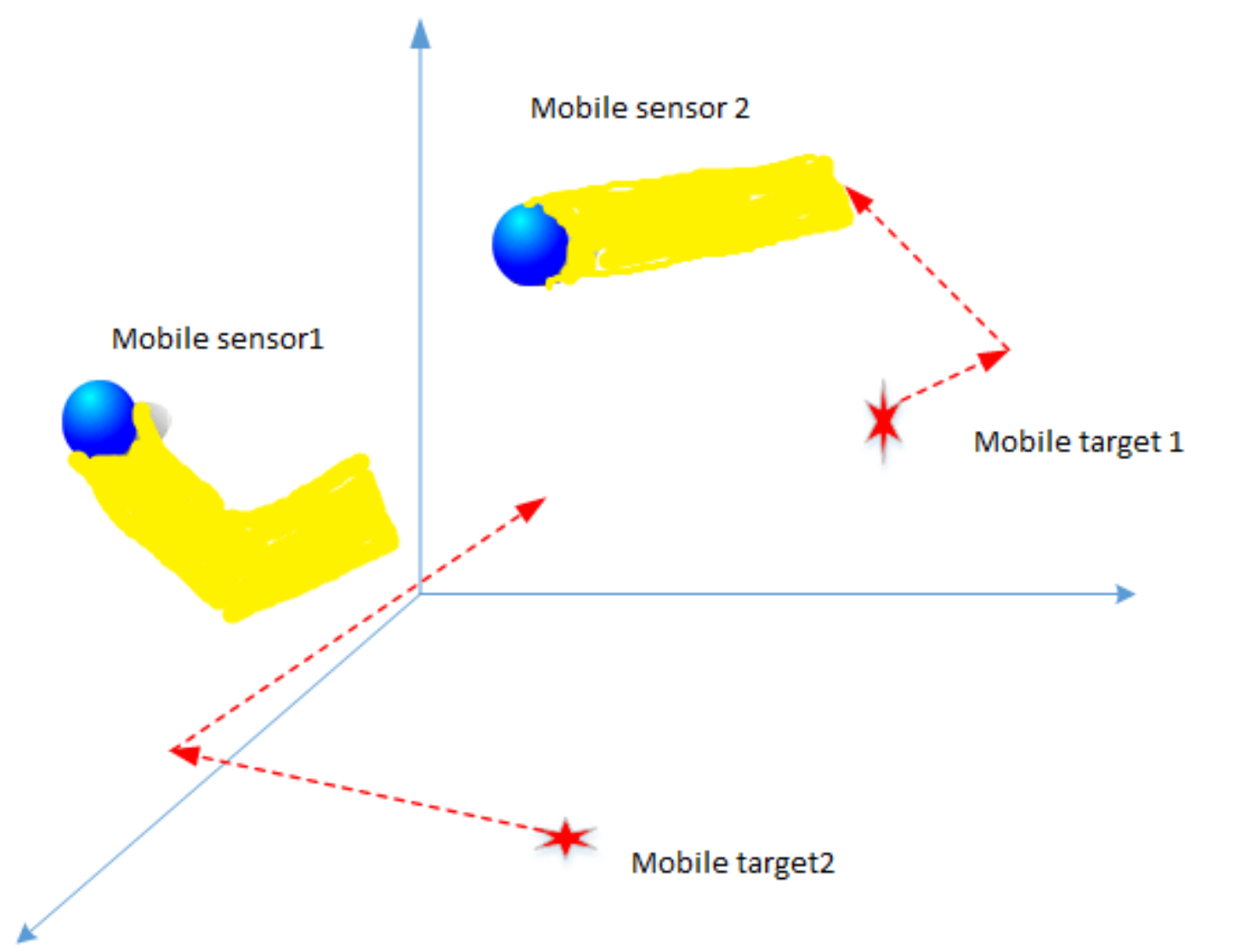}}
\caption{Mobile sensors denoted by sphere (sensing range) , moving targets denoted by stars, detected area shown by yellow tape}
 \label{mobiletarget2}
\end{figure}
To detect moving targets, the mobile  sensors build a common covering truncated octahedral grid. Then, they perform the search task by moving randomly with the step length taken by a Levy flight probability distribution on the vertices of the grid from different initial locations. As shown in Fig.\ref{moving2}, the mobile sensors are modeled as spheres having a sensing radius of $R_s$ that are randomly moving within the search area. During the search procedure, the mobile sensors keep the information of the detected mobile targets in their memory and share it with other sensors passing in their communication range. As shown, mobile sensors 1 and  5 as well as mobile sensors 1 and 3 are in their communication range; as a result, they update each other by exchanging the search information. Therefore, the mobile sensors 3 and 5 communicate indirectly through mobile sensor 1 and this will optimize the search task, because the mobile sensors do not search for the mobile targets that have been previously detected by other sensors.
%In the following, we take advantage of the notion of graph to define the relationships between neighboring mobile  robots. As a result, the vertices $i$ and $j$ of the graph $\emph{g}(\emph{k})$ are considered connected if the  mobile  robots $i$ and $j$ are neighbors at time $\emph{k}$. In the following, we impose a condition on the connectivity
%of the graph.\\
%\textit{Assumption 2.1:} There exists non-empty, an infinite sequence of contiguous, bounded time-intervals $[\emph{k}_m, \emph{k}_m+1)$, $m = 0, 1, 2, . . .$, such that across each $[\emph{k}_m, \emph{k}_m+1)$, the graph $\emph{g}(\emph{k})$ is connected \cite{savkin2012optimal}.
\section{Distributed Bio-inspired Levy Walk Random Search Algorithm} \label{5.2}
In this section, we present a distributed grid-based random search algorithm for detecting moving targets in bounded 3D spaces. Based on this algorithm, the mobile robotic sensors build a common 3D truncated octahedral grid. Then, they perform the search task by moving randomly on the vertices of the 3D grid using random walks with a Levy-flight probability distribution. Here, the mobile sensors are constrained to operate within a bounded search region, constructed as the cuboid area shown in Fig.\ref{moving2.1}. To keep the mobile sensors/targets in the search area at all times at each step, the length of the next Levy walk is calculated to ensure that the travel from the current position to the next position will not place the mobile sensors/targets outside of the search area. If this new length places the mobile sensor out of the bounded search area, that value is rejected, and a new length is drawn. The process is repeated for each step until an acceptable length is provided.\\
\begin{figure}
\centering
{\includegraphics[width=0.8\textwidth,height=0.6\textwidth]{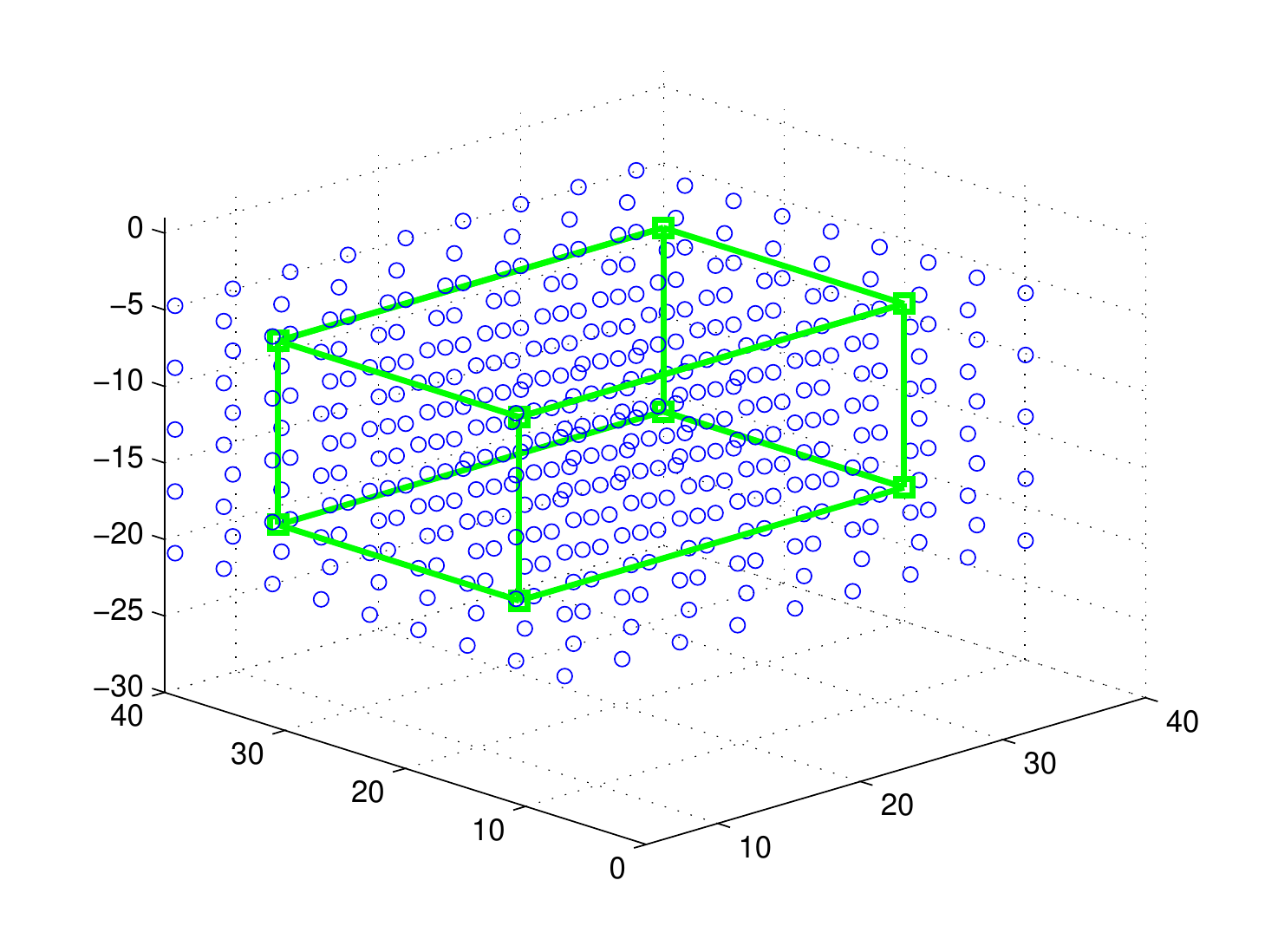}}
\caption{Boundaries of the search area ($\emph{M} $) represented by --, vertices of a covering 3D grid denoted by $’o’$}
 \label{moving2.1}
\end{figure}
Levy flight is a renowned bio-inspired random search mechanism with step lengths taken from a heavy-tailed probability distribution \cite{plank2008optimal}. Using this distribution the probability of returning to a previously visited site is smaller. Therefore, advantageous when target sites are sparsely and randomly distributed \cite{viswanathan2000levy}. To optimize the search task, we supplement the bio-inspired Levy flight random search strategy to a covering truncated octahedral grid. It means that mobile sensors do the search task by moving randomly to the vertices of the covering grid with random movements with Levy-flight probability distribution. As mentioned earlier, at first we assume that the mobile sensors do not have a common coordinate system and a common covering grid therefore the sensors use consensus variables to build a common covering grid 3D as explained in chapter 3, section 3.2.1 \cite{nazarzehi2015distributed}.
In our previous algorithm sensors performed the search task by moving randomly on the unvisited vertices in their neighbourhood. To optimize the search task and to minimize the time of the search, in this section, we propose a novel random search strategy for the mobile sensors. Using this method, the sensors do the search task by moving randomly to the vertices of the covering grid based on the random walk generated by the Levy flight distribution. A Levy flight random search uses a Levy probability distribution as follows \cite{viswanathan2008levy}:
\begin{equation}
P_{\alpha,\gamma}(l)=\frac{1}{2 \pi}\int_{-\infty}^{+\infty} e^{-\gamma q^\alpha}cos(ql) dq
\end{equation}
Where $\gamma$ denotes the scaling factor and $\alpha$ represents the shape of the distribution. Also, the distribution is symmetrical around $l$. In this chapter, the Levy flight  provides a random walk with the random step length is drawn from an approximated Levy distribution proposed in \cite {sutantyo2013collective} as follows:\\
\begin{equation}
P_\alpha(l)=(l)^{-\alpha}
\end{equation}
Where $l$ is a random number, $\alpha>0$ defines the length of walk and $1<\alpha<3$.\\ Let  ${\hat{\texttt{V}}}$ be the  set of  all vertices of the 3D covering grid and $\hat{v} \in \hat{\texttt{V}}$ is a randomly selected element of $\hat{\texttt{V}}$ with Levy flight probability distribution. We use the following random algorithm proposed in chapter 4 section 4.2, for the mobile sensors to search the 3D area:
\begin{equation} \label{movingr}
p_i(\emph{k}+1)=
\hat{v}
\end{equation}
Here we assume that the Boolean variables $b_\emph{T}(\emph{k})$ defines the states of the moving targets at the time $\emph{k}$. If a moving target $\emph{T}$ has been detected by any of the mobile sensors before time $\emph{k}$ then, $b_\emph{T}(\emph{k})= 1$  and $b_\emph{T}(\emph{k})= 0$ otherwise.
Throughout the search, each mobile sensor communicates to other sensors in the communication range at the discrete sequence of times $k = 0, 1, 2,…$ to exchange the information about detected mobile targets. The search process should be stopped after finding all targets.

\begin{figure*}[t!]
\begin{center}
\mbox{
\subfigure[Mobile sensors' positions after convergence to the vertices of the common covering truncated octahedral grid]{
{\includegraphics[width=0.75\textwidth,height=0.6\textwidth]{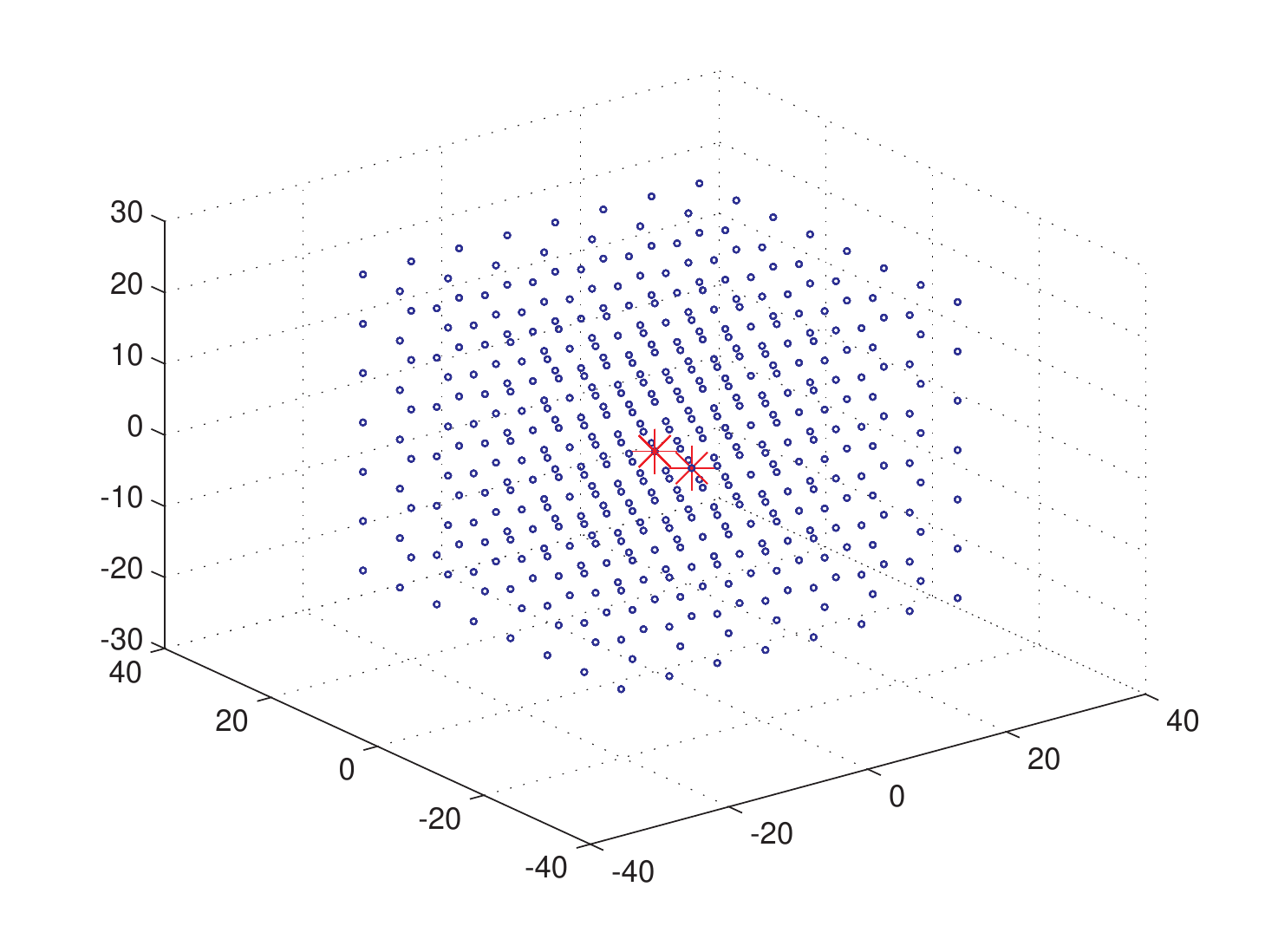}}\quad
\label{moving3}
}
}
\mbox{
\subfigure[Mobile sensors/targets' trajectories after 27 steps]{
{\includegraphics[width=0.7\textwidth,height=0.6\textwidth]{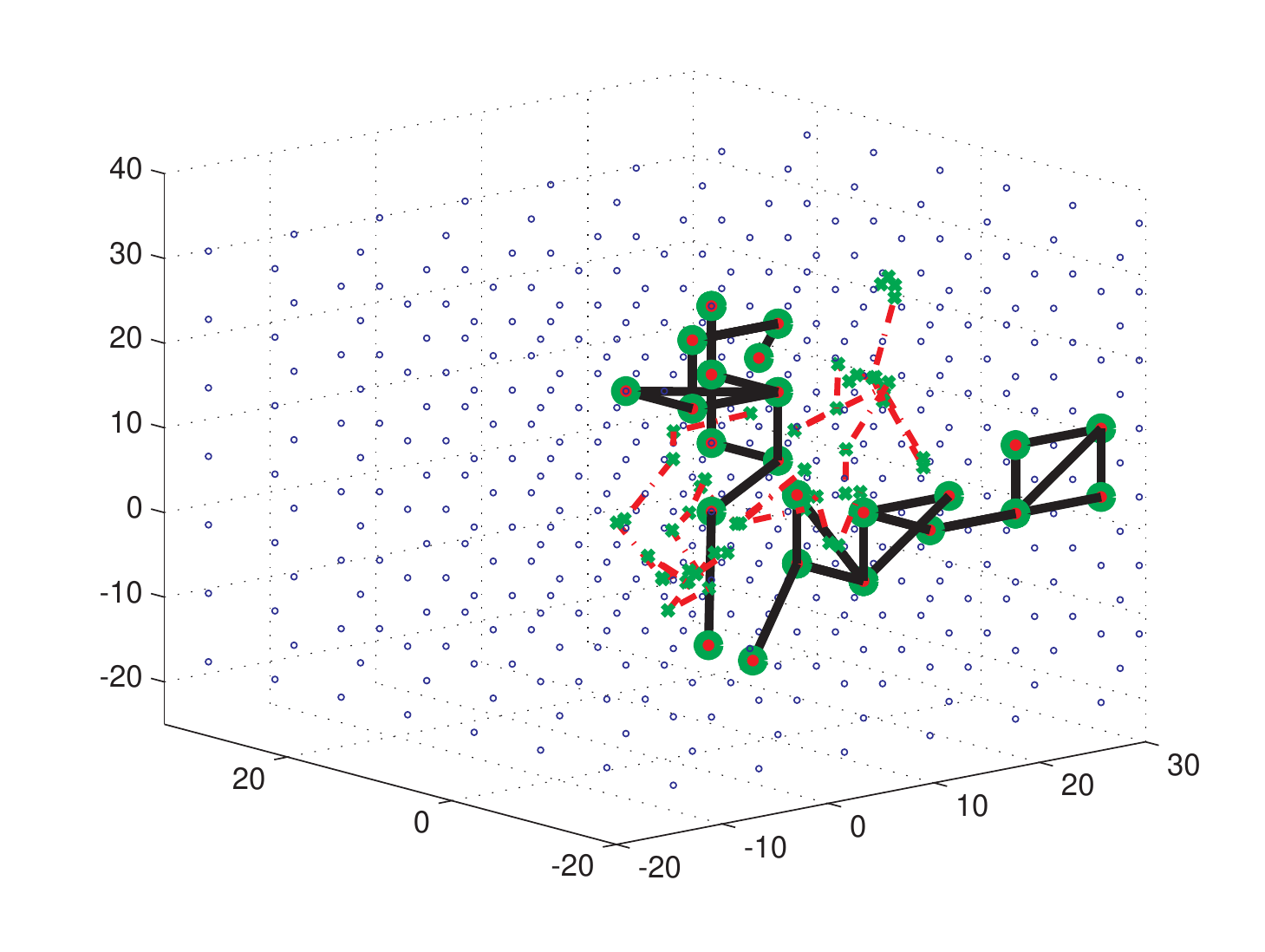}}\quad
\label{moving4}
}
}
\caption{Trajectories of two mobile sensors performing grid-based random search for detecting two moving targets: Mobile sensors' trajectories denoted by solid line ; moving targets trajectories represented by dashed line, vertices of the grid by o}
\label{moving5}
\end{center}
\end{figure*}

In other words,
\begin{equation} \label{stopping}
 p_i(\emph{k}+1)= p_i(\emph{k})
\end{equation}
\begin{equation*}
if \quad \forall\quad T=\left\{\emph{T}_1,\emph{T}_2,...\emph{T}_m\right\}, \exists\quad \emph{k},i\quad;\left\|(P_\emph{T}-p_i (\emph{k}))\right\| \leq \emph{R}_s
\end{equation*}
Where, $T=\left\{\emph{T}_1,\emph{T}_2,...\emph{T}_m\right\}$ and $P_\emph{T}=\left\{P_{\emph{T}_1},P_{\emph{T}_2},...P_{\emph{T}_m}\right\}$, are the sets of mobile targets and their positions, respectively.
\begin{theorem}
Suppose that the mobile sensors move according to the law (~\ref{movingr}), (~\ref{stopping}), and the mobile targets move based on the law(~\ref{movingt}). Then for any number of mobile sensors and any number of moving targets, with probability 1 there exists a time $\emph{k}_0$  such that:
\begin{equation*}
\forall \quad  \emph{t} \in T,\quad b_\emph{T}(\emph{k}_0)=1
\end{equation*}
\end{theorem}
\textit{Proof}: The algorithm (~\ref{movingr}) describes an absorbing Markov chain that consists of a number of absorbing states (that are impossible to leave) and many transient states. Vertices where mobile sensors stop are considered as absorbing states (~\ref{stopping}). Also, the vertices of the grid which are visited during search are considered as transient states (~\ref{movingr}). Relation (~\ref{movingr}) implies that the mobile sensors randomly move to visit unvisited vertices, this continues until all mobile targets are detected (an absorbing state). It is obvious that, from any initial state with probability 1, one of the absorbing states will be reached.
\section{Simulation Results} \label{5.3}
In this section, we carry out a simulation study to evaluate the performance of the proposed bio-inspired distributed laws for search of mobile targets in bounded three dimensional environments. As mentioned before, the mobile targets can adopt a wide variety of movement patterns. In this work, we consider the mobile targets choose a direction uniformly from a sphere. Then, it chooses a step length from a normal distribution. The initial positions of all mobile sensors and targets are randomly generated  within the search area. As these algorithms are stochastic, we run each simulation in 20 trials.\\
In the first set of simulations, we study the detection time of the mobile targets. It is obvious that, detection time depends on the mobility behaviour of the mobile sensors and on the movement of the moving targets themselves. We assume that a team of two mobile sensors are  deployed in the area to detect two mobile targets moving randomly in the search area. As mentioned earlier, a mobile target is detected when it falls into the sensing range of a mobile sensor. As shown in Fig.\ref{moving3}, using the updating law proposed in 3.2.1, the mobile sensors build a common covering truncated octahedral grid then all mobile sensors move to the vertices of this grid. Subsequently, the mobile sensors perform the search task by moving randomly with the random walk taken from a Levy-flight probability distribution on the vertices of the grid.\\ Fig.\ref{moving4} shows trajectories of the mobile sensors doing the search task by moving randomly to the vertices of the common truncated octahedral grid to find the moving targets. As shown in Fig.\ref{moving4}, the mobile sensors have detected the targets after 27 steps. \\
Note that, to back a moving target to the search area when a moving target reaches the boundary of the search area, it is given a new random heading that would direct it back into the search area.\\
The second set of simulations are carried out to compare the proposed grid based Levy walk method to the Levy walk method without grid for detecting two moving targets.
\begin{figure}
\centering
{\includegraphics[width=0.8\textwidth,height=0.75\textwidth]{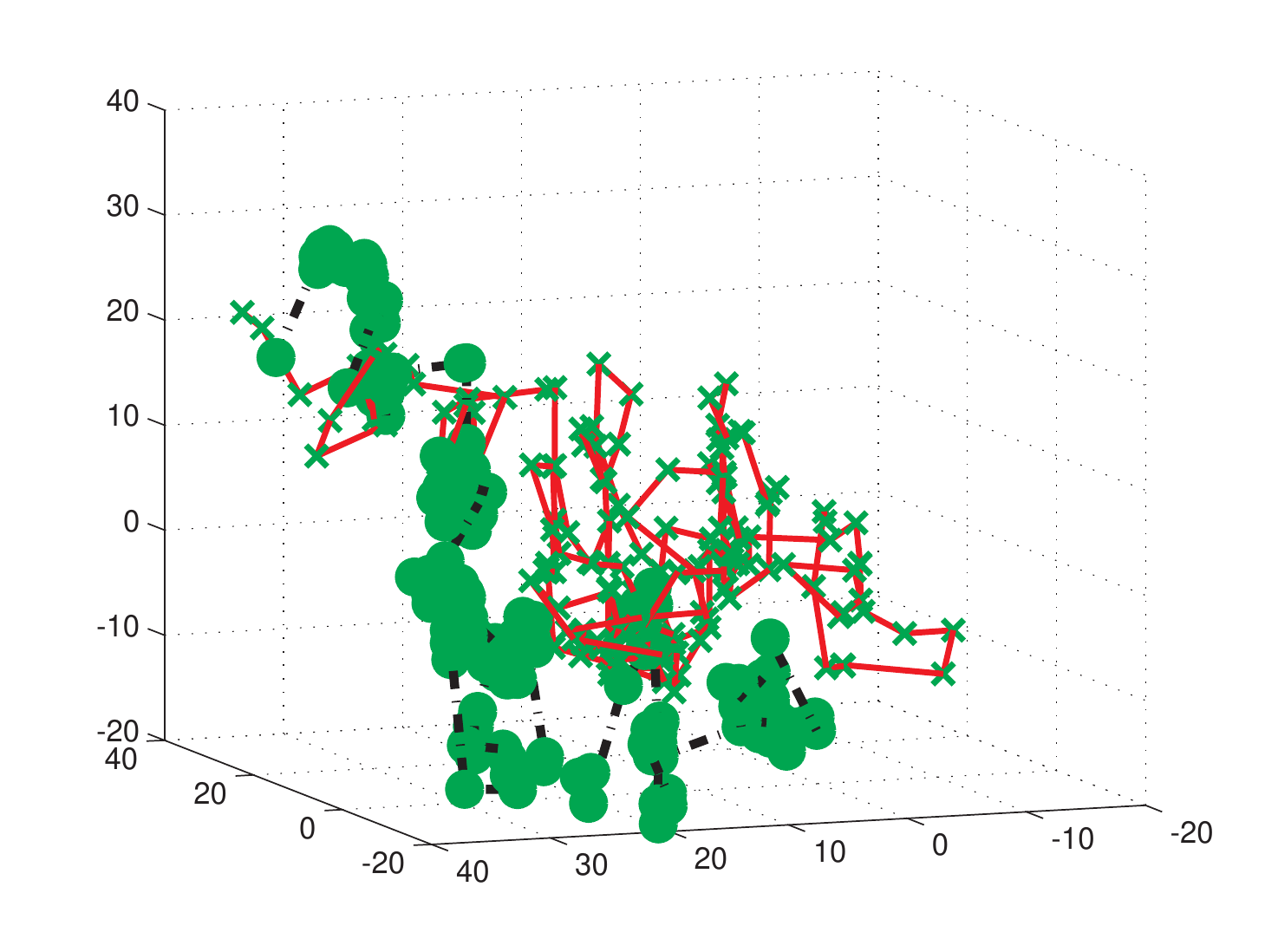}}
\caption{Trajectories of two mobile sensors performing Levy walk random search for detecting two moving targets: Mobile sensors' trajectories denoted by solid line, Mobile sensors by sphere, Mobile targets' trajectories represented by dashed line, Mobile targets by +}
 \label{moving6}
\end{figure}
Fig.\ref{moving6} shows the simulation results for detecting two moving targets in a 3D area using two mobile sensors using Levy walk random search method. Performance of the simulations are measured by the total steps, which is indicative of the total time spent to detect moving targets. Simulation results demonstrate that using the proposed grid-based random algorithm the mobile sensors detect the given number of targets in less time compare to the Levy flight method without grid as using the common grid improves mobile sensors dispersion in the search area.\\
In the next set of simulations, we have evaluated the ability of the proposed search algorithm for handling some uncertainties by assuming that two
mobile sensors unexpectedly fail during the search task. For this case, we consider four  mobile sensors search for two targets. Our simulations show that the other
mobile sensors continue to perform the search task despite failure on some mobile sensors. 
\section{Summary} \label{5.4}
In this chapter, we considered the problem of detecting mobile targets moving randomly in a bounded 3D space by a mobile sensor network. We presented a bio-inspired grid-based random search algorithm to drive mobile sensors for search of moving targets in the search areas. The proposed search method combines the bio-inspired Levy flight random search mechanism for determining the length of the walk with a covering truncated octahedral grid to optimize the search procedure. The presented algorithm is decentralized, and the control action of each sensor is based on the local information of its neighbors. Simulation results demonstrated that the proposed grid based algorithm outperforms to the Levy walk search method. Also, we gave mathematically rigorous proof of convergence with probability 1 of the proposed algorithm for any number of mobile sensors and targets.\\
%In the future works, we compare the proposed grid-based Levy flight search algorithm to our random search algorithm proposed in \cite{nazarzehi2015distributed} and to a grid-based random search with the step length taken by a normal distribution for detecting moving targets in a bounded 3D environment. Moreover, we will consider the effect of the mobility strategies of both mobile robotic sensors and mobile targets in the detection time.

%\singlespace
\singlespacing

\chapter{Decentralized Three Dimensional Formation Building}\label{chap:1ValiCH4}
\minitoc
In this chapter, we study the problem of 3D formation building in 3D environments \cite{robio}. Our goal is to design a distributed control algorithm to coordinate a group of mobile robotic sensors such that they achieve a particular geometric pattern \cite{ cao2013overview}. In particular, here, we introduce a random formation building algorithm for anonymous robots in 3D spaces. First, we present a decentralized consensus-based control algorithm for the problem of 3D formation building for nonholonomic mobile robotic sensors in 3D environments. Based on the proposed algorithm, the mobile robotic sensors only communicate with their neighbors located in their communication range to build and maintain a desired 3D configuration. As a result, this method is a good option for the cases where the mobile robotic sensors have  limited communication range due to the economic reasons or because of physical constraints (underwater environments). Unlike previous works \cite{roussos20103d}, using this algorithm, the mobile robotic sensors move in not only the same direction in 3D spaces but also they finally build a given 3D geometric configuration and move with the same speed \cite{new1}. Furthermore, in this study, the motion of each robot is modeled by a nonlinear model also we take into account the standard constraints on mobile robotic sensors' angular acceleration and linear velocity \cite{matveev20143d}.\\
At second part, we study the problem of formation building for the case when the mobile robotic sensors are unaware of their positions in the configuration. We present a decentralized random motion coordination law for the multi-robot system so that the mobile robotic sensors reach consensus on their positions and form a desired three dimensional geometric pattern from any initial positions. The performance of the proposed three dimensional formation building algorithm is validated  by numerical simulations. Also, we give mathematically rigorous proof of convergence of the proposed algorithms to the given geometric configurations.\\
The rest of the chapter is organized as follows: In Section \ref{6.1}, we introduce a kinematic model for a group of mobile robotic sensors moving in a three-dimensional space. In Section \ref{6.2}, we formulate the problem. In Section \ref{6.3}, a novel consensus-based algorithm for formation building in 3D environments is presented. In Section \ref{6.4}, we introduce a 3D decentralized random formation building algorithm for anonymous mobile robotic sensors. Numerical simulation results that validate the presented algorithms are presented in Section \ref{6.5}. Finally, in Section \ref{6.6} conclusion is summarized.

%\section{Section title}
%\label{sec:1}
\section{Problem Formulation} \label{6.1}
In this section, we define the kinematics equations of the motions of the multi-robot system in 3D environments and we formulate the problem of formation control. In this chapter, the term mobile robotic sensor is used for expression of both unmanned aerial vehicles and autonomous underwater vehicles. Moreover, the standard inner product is denoted by $< ¢; ¢ >$.
\subsection{Multi-Robot Model}
\label{sec:2}
The system under consideration consists of n  three degree-of-freedom nonholonomic mobile robotic sensors labelled 1 through $n$. All these robots move in the three dimensional  workspace in continuous time, and they are  controlled by bounded $v_i$, pitch ($q_i$) and yaw ($r_i$) rates. Let $\xi_i=col(x_i(t), y_i(t), z_i(t))$ denotes the column vector of the Cartesian coordinates of the mobile robotic sensor $i$. Also, let $\theta_i(t)$ and $\psi_i(t)$ denote the pitch and the yaw angle, respectively. Moreover, let $v_i(t)$, $q_i(t)$ and $r_i(t)$ be the linear velocity, the pitch and the yaw angular velocities of the robot $i$, respectively. Note that, the mobile robotic sensors considered in this chapter do not have roll motion. Then, the mobile robotic sensors kinematic equations of motion are given by the following equations:
  \begin{equation*}
\dot{\vec{\xi_i(t)}}=v_i(t)\vec{c_i(t)}
\end{equation*}
\begin{equation}\label{eq:1}
\dot{\vec{c_i(t)}}=\vec{u_i(t)}
\end{equation}
for all $i = 1,2,...n$. Here, $\vec{c_i(t)}\in R^3$, $\left\|\vec{c_i(t)}\right\|=1$, is the unit vector along the robot’s centerline described by:\\
\begin{equation}\label{eq:2}
\vec{c_i}=\left(
\begin{array}{c}
\cos(\theta_i(t))\cos(\psi_i(t)\\
\cos(\theta_i(t))\sin(\psi_i(t))\\
-\sin(\theta_i(t)
\end{array}
\right)
\end{equation}
Also, introduce three-dimensional vectors $V_i(t)$ of the robots’ velocities by $V_i(t)=v_i(t)c_i(t)$.\\
%\begin{equation*}
	%x^._i(t)=v_i(t)\cos(\theta_i(t))\cos(\psi_i(t))
%\end{equation*}
%\begin{equation}
	%y^._i(t)=v_i(t)\cos(\theta_i(t))\sin(\psi_i(t))
%\end{equation}
%\begin{equation*}
	%z^._i(t)=-v_i(t)\sin(\theta_i(t))
%\end{equation*}
%\begin{equation*}
	%\theta^._i(t)=q_i(t)
%\end{equation*}
%\begin{equation*}
	%\psi^._i(t)=\frac{r_i(t)}{\cos(\theta(t))}
%\end{equation*}
$\vec{u_i}(t) \in R^3$ is the two-degree-of-freedom control. Furthermore, $\vec{u_i}(t)$ and $v_i(t)\in R $ are the control inputs for all robots. We assume the following  constraints on the control inputs:
\begin{equation}\label{eq:3}
\left\langle \vec{u_i(t)},\vec{c_i(t)}\right\rangle=0
\end{equation}
\begin{equation}\label{eq:4}
\left\|\vec{u_i(t)}\right\|\leq u_{max}
\end{equation}
\begin{equation} \label{eq:5}
V_{min} \leq v_i(t) \leq V_{max}
\end{equation}

\begin{figure}
\centering
{\includegraphics[scale=0.85]{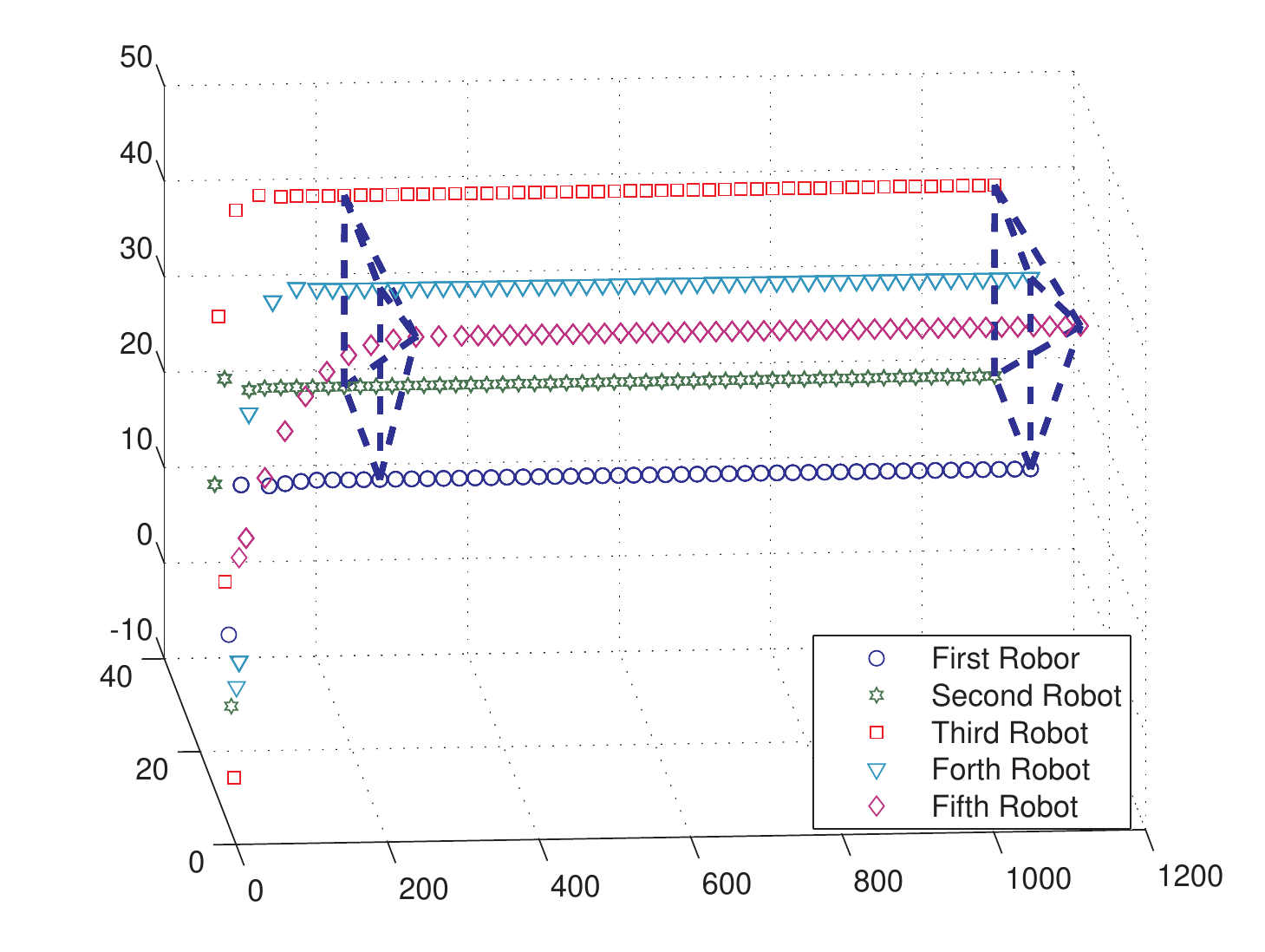}}\\
\caption{Formation Building Schematic }
\label{schematicnew}
\end{figure}

Here $u_{max}>0$ and also $0<V_{min}<V_{max}$ are given constants.
\subsection{Formation Building Problem} \label{6.2}
In this section, we explain the problem of formation building by nonholonomic mobile robotic sensors. Especially, we study the decentralized 3D formation building problem for $n$ mobile robotic sensors moving in three-dimensional environments such that, after transient state, the mobile robotic sensors constitute the desired formation and move collectively along a desired geometric pattern from any initial position. The desired geometric configuration of the mobile robotic sensors in the formation are characterized by a set of given numbers $X_1,X_2...X_n$, $Y_1, Y_2, ... Y_n$ and $Z_1,Z_2,....Z_n$ . For all $i = 1,...n$. , $Y_i-Y_j$ and $Z_i-Z_j$ are the desired horizontal and vertical distances between the $ith$' and the $jth$' robots respectively. Also, $X_i-X_j$ is the desired distance  between the $ith$ and the $jth$ mobile robotic sensor along $X$ axis.\\
\begin{definition} A navigation law with the given configuration $C=\left\{X_1,X_2...X_n, Y_1, Y_2, ... Y_n, Z_1,Z_2,....Z_n\right\}$ and initial conditions $\theta_i(0)$, $\psi_i(0)$, $x_i(0)$, $y_i(0)$,
$z_i(0)$ and $v_i(0)$ , $i = 1,...n$  is said to be globally stabilizing, if there exist a Cartesian coordinate system and $v_0$ such that the solution of the closed-loop system (\ref{eq:1}) with these initial conditions and the proposed navigation law in this Cartesian coordinate system satisfies:
\begin{equation}\label{eq:6}
\lim_{t\rightarrow\infty} x_i(t)-x_j(t) = X_i-X_j
\end{equation}
\begin{equation*}
\lim_{t\rightarrow\infty} y_i(t)-y_j(t) = Y_i-Y_j
\end{equation*}
\begin{equation*}
\lim_{t\rightarrow\infty} z_i(t)-z_j(t) = Z_i-Z_j
\end{equation*}
\begin{equation*}
\lim_{t\rightarrow\infty} \psi_i(t)=0
\end{equation*}
\begin{equation}\label{eq:7}
\lim_{t\rightarrow\infty} \theta_i(t)=0
\end{equation}
\begin{equation} \label{eq:8}
\lim_{t\rightarrow\infty} v_i(t) = v_0
\end{equation}
\end{definition}
for $i,j=1,2,.....n$, $i\neq j$.
Thus, the control objective is to design a navigation law to satisfy the conditions mentioned above for any initial conditions. The schematic of the mobile robotic sensors during formation building in  a given configuration is shown in Fig.~\ref{schematicnew}.
\section{Decentralized Consensus Based Formation Building Algorithm} \label{6.3}

For a given time $k > 0$, each mobile robotic sensor communicates with its surrounding neighbours in the communication range $R_c$ at the discrete time instances $k = 0,1,2, . .$ . We assume a spherical communication model where each robot has the ability to obtain information of its neighbours within a sphere of radius $R_c > 0$ defined by:
\begin{equation} \label{eq:9}
S_{i,R_c}=\{p\in R^3;\|(p-\zeta_i (k))\| \leq R_c\}
\end{equation}
Let $N_i(k)$ be the set of mobile robotic sensors  $j$, $j\neq  i$ and $j\in {1, 2, . . ., n}$ that at time $k$ belong to the sphere $S_{i,R_c}$,  also we assume that robot $i$ has $|N_i(k)|$ number of neighbours at time $k$. We consider a simple undirected graph $\mathcal{G}(k)$ with vertex set ${1, 2, . . ., n}$ to describe the relationships between the neighbouring mobile robotic sensors for any time $k$. The different vertices $i$ and $j$ of the graph $\mathcal{G}(k)$ are connected by an edge if and only if, the robots $i$ and $j$ are neighbours at time $k$. In the following, we impose a condition on the connectivity of the graph.\\
\begin{assumption} \label{assump3}There exists an infinite sequence of non-empty, contiguous, bounded time-intervals $[k_j, k_{j+1})$, $ j = 0, 1, 2, . . .$  beginning at $k_0 = 0$ such that across each $[k_j, k_{j+1})$, the union of the collection $\{\mathcal{G}(k) : k \in [k_j, k_{j+1})\} $ is a connected graph \cite{main}.
\end{assumption}
We use consensus variables $\tilde{v_i}(k)$ to reach a common speed of the formation, and the consensus variables $\tilde{x_i}(k)$,$\tilde{y_i}(k)$ and $\tilde{z_i}(k)$ are used to build a common coordinates of the formation, and also the consensus variables $\tilde{\theta_i}(k)$ and $\tilde{\psi_i}(k)$ are utilized  to achieve the common heading of the formation. The robots will start with different values of $\tilde{v_i}(0)$, $\tilde{x_i}(0)$, $\tilde{y_i}(0)$, $\tilde{z_i}(0)$, $\tilde{\theta_i}(0)$, and $\tilde{\psi_i}(0)$, and finally converge to some constant values which define a common formation orientation and speed for all robots.\\
\begin{assumption} \label{assump4} The initial values of the consensus variables: $\tilde{\theta_i}(0) \in\left[0 \quad \pi\right)$ and $\tilde{\psi_i}(0) \in\left[0 \quad \pi\right) $  for all $i = 1, 2, ... n$.
\end{assumption}
We propose the following rules for updating the consensus variables:
\begin{equation*}
\tilde{\psi_i}(k+1)=\frac{\tilde{\psi_i}(k)+\sum_{\substack{
   j\in N_i(k)
}}\tilde{\psi}_j(k) }
 {1+| N_i(k) |}
\end{equation*}
\begin{equation*}
\tilde{\theta_i}(k+1)=\frac{\tilde{\theta_i}(k)+\sum_{\substack{
   j\in N_i(k)
}}\tilde{\theta}_j(k) }
 {1+| N_i(k) |}
\end{equation*}
\begin{equation*}
\tilde{x_i}(k+1)=\frac{x_i(k)+\tilde{x_i}(k)+\sum_{\substack{j\in N_i(k)}}x_j(k)+\tilde{x_j}(k)}{1+|N_i(k)|} -x_i(k+1)
\end{equation*}
\begin{equation*}
\tilde{y_i}(k+1)=\frac{y_i(k)+\tilde{y_i}(k)+\sum_{\substack{j\in N_i(k)}}y_j(k)+\tilde{y_j}(k)}{1+|N_i(k)|}-y_i(k+1)
\end{equation*}
\begin{equation*}
\tilde{z_i}(k+1)=\frac{z_i(k)+\tilde{z_i}(k)+\sum_{\substack{j\in N_i(k)}}z_j(k)+\tilde{z_j}(k)}{1+|N_i(k)|}-z_i(k+1)
\end{equation*}
\begin{equation} \label{eq:10}
\tilde{v}_i(k+1)=\frac{\tilde{v}_i(k)+\sum_{\substack{
   j\in N_i(k)
}}\tilde{v}_j(k) }
 {1+| N_i(k) |}
\end{equation}
Based on (~\ref{eq:10}), the mobile robotic sensors use the consensus variables to achieve a consensus on the heading, speed and mass centre of the formation.\\
\begin{lemma} \label{lemma8} Suppose that assumptions ~\ref{assump3} and ~\ref{assump4} hold and the consensus variables are updated according to the decentralized control algorithm (~\ref{eq:10}), then there exist constants $\theta_0$, $\psi_0$, $x_0$, $y_0$,
$z_0$ and $v_0$ such that:
\begin{equation*}
\lim_{k\rightarrow\infty} \tilde{\theta_i}(k) = \theta_0
\end{equation*}
\begin{equation*}
\lim_{k\rightarrow\infty} \tilde{\psi_i}(k) = \psi_0
\end{equation*}
\begin{equation*}
\lim_{k\rightarrow\infty} x_i(k)+\tilde{x_i}(k) = x_0
\end{equation*}
\begin{equation*}
\lim_{k\rightarrow\infty} y_i(k)+\tilde{y_i}(k) = y_0
\end{equation*}
\begin{equation*}
\lim_{k\rightarrow\infty} z_i(k)+\tilde{z_i}(k) = z_0
\end{equation*}
\begin{equation}\label{eq:11}
\lim_{k\rightarrow\infty} \tilde{v_i}(k) = v_0
\end{equation}
\end{lemma}
(See \cite{jadbabaie2003coordination,savkin2004coordinated} for the proof of convergence of consensus variables to some constant values).\\
Notice that, for all $i=1,2,...n.$, we consider a cartesian coordinate system with the $x-axis$ in the direction of $\tilde{\theta_i}(k)$. Thus, in this coordinate system $\tilde{\theta_i}(k) = 0$.\\
In the following, we present a decentralized formation building law in 3D spaces for the multi-robot system such that the robots will eventually move in the same direction of the x-axis in this Cartesian coordinate system, with the same speed and with the given  geometric configuration.\\
We define the functions $h_i(t)$ as:\\ $h_i(t) = x_i(t) + \tilde{x_i}(t) + X_i + t\tilde{v_i}(t)$. \\Next, we introduce $Tx_i$, $Ty_i(t)$ and $Tz_i(t)$ as:
%\begin{equation*}
%
%Tx_i(t)=
%\begin{cases}
%h_i(t)+c_0 \quad  if \quad x_i(t) \leq h_i(t)\\
%x_i(t)+c_0 \quad  if \quad x_i(t) \succ h_i(t)\\.
%\end{cases}
%\end{equation*}
\begin{equation*}
Tx_i(t)=\begin{cases}
h_i(t)+c_0 \quad  if \quad x_i(t) \leq h_i(t)\\
x_i(t)+c_0 \quad  if \quad x_i(t) \succ h_i(t)\\
\end{cases}
\end{equation*}

Note that  $c_0 > 0$ is a constant such that
\begin{equation} \label{eq:12}
 c_0>\frac{2V_{max}}{u_{max}}
\end{equation}
\begin{equation} \label{eq:13}
Ty_i(t)=y_i(t) + \tilde{y_i}(t)+Y_i
\end{equation}
\begin{equation*}
Tz_i(t)=z_i(t) + \tilde{z_i}(t)+Z_i.
\end{equation*}
Also, for $i=1,2,.....n$ , we consider $T_i$ as fictitious targets  moving  with coordinates defined by:
\begin{equation} \label{eq:14}
T_i=\left(
\begin{array}{c}
Tx_i(t)\\
Ty_i(t)\\
Tz_i(t)\\
\end{array}
\right)
\end{equation}
%\begin{figure}
%\centering
%\unitlength 1mm % = 2.845pt
%\linethickness{0.4pt}
%\ifx\plotpoint\undefined\newsavebox{\plotpoint}\fi % GNUPLOT compatibility
%\begin{picture}(47,35.25)(0,0)
%\put(0,0){\vector(1,1){21}}
%%\vector(0,0)(47,14.5)
%\put(47,14.5){\vector(3,1){.07}}\multiput(0,0)(.1093023256,.0337209302){430}{\line(1,0){.1093023256}}
%%\end
%\put(20.75,20.75){\vector(-1,1){14.5}}
%%\emline(17.5,23.75)(14.75,21)
%\multiput(17.5,23.75)(-.03353659,-.03353659){82}{\line(0,-1){.03353659}}
%%\end
%%\emline(15,21.5)(18.25,18.25)
%\multiput(15,21.5)(.033505155,-.033505155){97}{\line(0,-1){.033505155}}
%%\end
%\put(18.25,18.25){\line(0,1){0}}
%\put(23.75,3.75){\makebox(0,0)[cc]{$d_i$}}
%\put(16.75,32){\makebox(0,0)[cc]{$u_i$}}
%\put(8.25,12.5){\makebox(0,0)[cc]{$c_i$}}
%\end{picture}\\
\begin{figure}
\centering
{\includegraphics[trim={5cm 0 0 0}, scale=0.8]{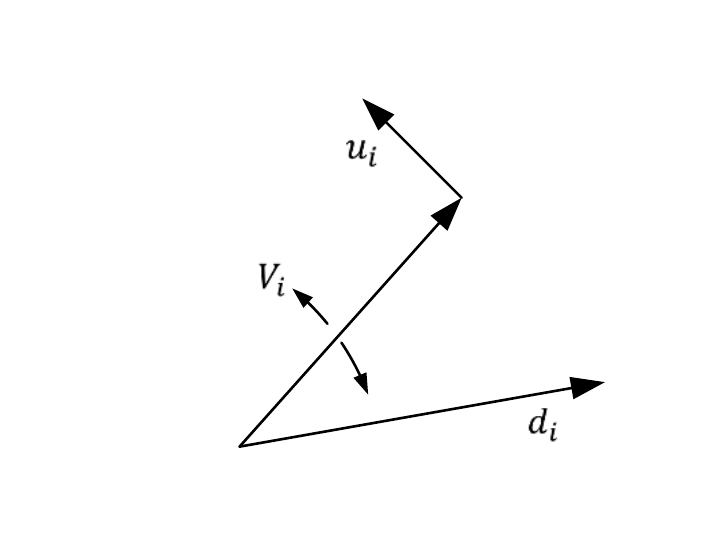}}
\caption{Vectors Setting}
 \label{vec}
\end{figure}

Furthermore, let $d_i$ be a three-dimensional vector as:
\begin{equation} \label{eq:15}
d_i=\left(
\begin{array}{c}
Tx_i(t)-x_i(t)\\
Ty_i(t)-y_i(t)\\
Tz_i(t)-z_i(t)\\
\end{array}
\right)
\end{equation}

 Which is the distance vector between robot $i$ and the fictitious target.\\
\begin{remark} We assume that each mobile robotic sensor knows the coordination $(x_i,y_i,z_i)$ and the consensus variables of all of its neighboring mobile robotic sensors, also the configuration $C=\left\{X_1,X_2,....X_n,Y_1,Y_2,...Y_n,Z_1,Z_2,...Z_n\right\}$ and the constant $c_0$  are known to the multi-robot system a priori.
\end{remark}
Since $\dot{c}_i=u_i(t)$ defined by (~\ref{eq:1}), then the following correspondence between the control vector $u_i$ from (~\ref{eq:1}) and the pitch ($q_i$) and yaw ($r_i$) rates can be defined:
\begin{equation} \label{eq:16}
\vec{u_i}=\dot{\theta}_i\vec{A_i}+\dot{\psi}_i\vec{B_i}
\end{equation}

where, $\dot{\theta}_i=q_i$, $\dot{\psi}_i=r_i$ and also $A_i$ and $B_i$ are considered as three-dimensional vectors defined by:
\begin{equation} \label{eq:17}
\vec{A_i}=\left(
\begin{array}{c}
-\sin(\theta_i(t))\cos(\psi_i(t)\\
-\sin(\theta_i(t))\sin(\psi_i(t))\\
-\cos(\theta_i(t)
\end{array}
\right)
\end{equation}
\begin{equation} \label{eq:18}
\vec{B_i}=\left(
\begin{array}{c}
-\cos(\theta_i(t))\sin(\psi_i(t)\\
\cos(\theta_i(t))\cos(\psi_i(t))\\
0\\
\end{array}
\right)
\end{equation}

for all $i=1,2,....n$.\\
We know that the direction of $\vec{u_i}$  is orthogonal to $\vec{V_i}$ and it can cause acceleration perpendicular to the mobile robotic sensors velocity vector($V_i$). The vector $u_i$ causes  a turning motion to parallel the mobile robotic sensors velocity vector to the vector $d_i$. Furthermore, we assume that the control signal($u_i$) is in the same plane with $\vec{d_i}$ and the velocity vector $V_i$,  therefore we can write the direction of $\vec{u}_i$ as(see ~\ref{vec}):
\begin{equation} \label{eq:19}
\vec{d_{u_i}}= \vec{d_i}-\frac{\left\langle d_i(t),V_i(t)\right\rangle}{\left\|V_i\right\|^2}\vec{V_i};
\end{equation}
We define $\tilde{u_i}=\frac{\vec{d_{u_i}}}{\vec{\left\|d_{u_i}\right\|}}$ as the  unit vector of $\vec{d_{u_i}}$.
Now, we introduce the following decentralized control law for formation building in 3D environments:
\begin{equation} \label{eq:20}
v_i(t)=\begin{cases}
V_{max}\quad  if \quad x_i(t) \leq h_i(t)\\
V_{min} \quad  if \quad x_i(t) > h_i(t)\\
\end{cases}
\end{equation}
\begin{equation} \label{eq:21}
\vec{u_i}(t)=\bar{u_i}(t)\vec{\tilde{u_i}}
\end{equation}
Where $\bar{u_i}=u_{max}sgn(V_i,d_i)$. $sgn(,)$ defines the sign of angle between vectors $V_i$ and $d_i$.\\ The controller ~\ref{eq:19}-~\ref{eq:21} belongs to the class of switched controller systems, see e.g \cite{savkin1996hybrid,skafidas1999stability,Matveev:2000:QTH:555969,savkin2002hybrid}
Also, (~\ref{eq:16}) implies that we can calculate the angular velocities as follows:
\begin{equation} \label{eq:22}
q_i(t)=\left\langle u_i(t),A_i(t)\right\rangle.
\end{equation}
\begin{equation} \label{eq:23}
r_i(t)=\left\langle u_i(t),B_i(t)\right\rangle
\end{equation}
for all $i=1,2,....n$.\\
\begin{lemma} Suppose that $q_i(t)$ and $r_i(t)$ are updated according to the decentralized control  law (~\ref{eq:22}),(~\ref{eq:23}), then both angular velocities are bounded.
\end{lemma}
 By using (~\ref{eq:21}, ~\ref{eq:22} ,~\ref{eq:23})  we can say that $\left\|q_i(t)\right\|\leq \left\|\vec{u_i}(t)\right\|\left\|\vec{A_i}\right\|$ and $\left\|r_i(t)\right\|\leq \left\|\vec{u_i}(t)\right\|\left\|\vec{B_i}\right\|$. Also, it is clear that $\left\|\vec{u_i}(t)\right\|=u_{max}$, $\left\|\vec{A_i}\right\|=1$ and $\left\|\vec{B_i}\right\|=\cos(\theta_i)$. This means that $q_i(t)$, $r_i(t)$ are bounded, and $u_{max}$ is an upper bound to them. Now, we are in a position to present the main result of the chapter.\\
\begin{theorem} \label{theorem:formation} Consider the mobile robotic sensors described by the equations (~\ref{eq:1}), (~\ref{eq:2}) and the constraints (~\ref{eq:3}), (~\ref{eq:4}), (~\ref{eq:5}). Let $C=\left\{X_1,X_2,....X_n,Y_1,Y_2,...Y_n,Z_1,Z_2,...Z_n\right\}$ be a given configuration. Suppose that Assumptions ~\ref{assump3} and ~\ref{assump4} hold, and $c_0$ is a constant satisfying (~\ref{eq:11}). Then, the decentralized control law (~\ref{eq:19}), (~\ref{eq:20}) and (~\ref{eq:21}) is globally stabilizing with any initial conditions and the configuration C.
\end{theorem}
\textit{Proof}: Consider a fictitious target $T_i(t)$ with coordinates defined by $T_i(t)$ moving in the 3D space. Moreover, let
\begin{equation}
\tilde{T}_i(t)=\left(
\begin{array}{c}
\tilde{T}x_i(t)\\
\tilde{T}y_i(t)\\
\tilde{T}z_i(t)\\
\end{array}
\right)
\end{equation}
be another fictitious target moving in the 3D space with coordinates $\tilde{T}_i(t)$ as follows:
\begin{equation*} \label{eq:24}
\tilde{T}x_i(t)=\begin{cases}
X_i+X_0+t\tilde{v_0}+c_0 \quad  if \quad x_i(t) \leq X_i+X_0+t\tilde{v_0} \\
x_i(t)+c_0 \quad       \quad if   \quad x_i(t) \geq X_i+X_0+t\tilde{v_0}\\
\end{cases}
\end{equation*}
 \begin{equation} \label{eq:25}
\tilde{T}y_i(t)=y_i(t) + \tilde{y_i}(t)+Y_i
\end{equation}
\begin{equation*}
\tilde{T}z_i(t)=z_i(t) + \tilde{z_i}(t)+Z_i.
\end{equation*}

Based on Lemma ~\ref{lemma8},
 \begin{equation*}
\lim_{t\rightarrow\infty} \tilde{T}_i(t) = T_i(t)
\end{equation*}
The system defined by the equations (~\ref{eq:1}),(~\ref{eq:20}),(~\ref{eq:21}), (~\ref{eq:22}), (~\ref{eq:23}) represents a closed-loop system  with discontinuous right-hand sides. Suppose $\lambda _{i}(t)$ is the angle between the velocity vector $V_{i}(t)$ of the robot $i$ and the line-of-sight between the robot and $\tilde{T}_{i}$. The equation $\lambda_{i}=0$ defines a switching surface of this system. In order to prove convergence to the sliding mode, we show that with this control, $\lambda_{i}$ will converge to 0 in finite time. In \cite{teimoori2010equiangular}, it has been shown that:
\begin{eqnarray}  \label{eq:26}
\dot{\lambda}_{i}(t)= a(t)-\bar{u}_i(t)- b(t),
\end{eqnarray}
where
\begin{equation}
\label{eq:27}
|a(t)|\leq \frac{|\sin(\lambda_{i}(t))|\|V_i(t)\|}
{\|\tilde{d}_{i}(t)\|};
\end{equation}
\begin{equation}
\label{eq:28}
|b(t)|\leq \frac{|\sin(\beta_{i}(t))|\|V_i^T(t)\|}
{\|\tilde{d}_{i}(t)\|}.
\end{equation}
$\beta _{i}(t)$ is the angle between the velocity vector $V_{i}^{T}(t)$ of $\tilde{T}_{i}$ and the line-of-sight from the robot $i$ to $\tilde{T}_{i}$. Also, $\tilde{d}_{i}(t)$ is defined as:
\begin{eqnarray}  \label{eq:29}
\tilde{d}_{i}(t):=\tilde{T}_{i}(t)-\xi_i(t)
\end{eqnarray}
(~\ref{eq:24}), (~\ref{eq:29}) imply that
\begin{eqnarray}  \label{eq:30}
\|\tilde{d}_{i}(t)\|\geq c~~~~\forall \quad t\geq 0.
\end{eqnarray}
It  follows from (~\ref{eq:24}) that
\begin{equation} \label{eq:31}
V_{i}^T(t)=\left(
\begin{array}{c}
V_{i}^{Tx}(t)\\
V_{i}^{Ty}(t)\\
V_{i}^{Tz}(t)\\
\end{array}
\right)
\end{equation}
\begin{equation*}
V_{i}^{Tx}(t)=\begin{cases}
 \tilde{v}_0 & if~x_i(t)\leq X_0+X_i+t\tilde{v}_0 \\
v_i(t) & if~x_i(t)> X_0+X_i+t\tilde{v}_0\\
\end{cases}
\end{equation*}
 \begin{equation} \label{eq:32}
V_{i}^{Ty}(t)= 0
\end{equation}
\begin{equation*}
V_{i}^{Tz}(t)=0
\end{equation*}

(~\ref{eq:32}) and (~\ref{eq:5}) imply that
\begin{eqnarray}  \label{eq:33}
\|V_{i}^T(t)\|\leq V^M.
\end{eqnarray}
Suppose that in control law (~\ref{eq:21}) the distance vector $d_{i}$ replaced by $\tilde{d}_{i}$. The mathematical relations (~\ref{eq:12}), (~\ref{eq:26}), (~\ref{eq:33}), (~\ref{eq:30}) implies that using this control law , there exists a
constant $\epsilon>0$ such that
\begin{eqnarray}  \label{eq:34}
\dot{\lambda}_{i}(t)<-\epsilon ~~~~if~~~\lambda_{i}(t)>0  \nonumber \\
\dot{\lambda}_{i}(t)>\epsilon ~~~~if~~~\lambda_{i}(t)<0.
\end{eqnarray}
It means that $\dot{\lambda}_{i}(t)\lambda_{i}(t)< 0$.  Therefore, there exists a time $t_0>0$ such that
\begin{eqnarray}  \label{eq:35}
\lambda_{i}(t)=0~~~~~\forall \quad t\geq t_0.
\end{eqnarray}
\begin{figure}
\centering
{\includegraphics[trim={7cm 0 0 0},scale=0.5]{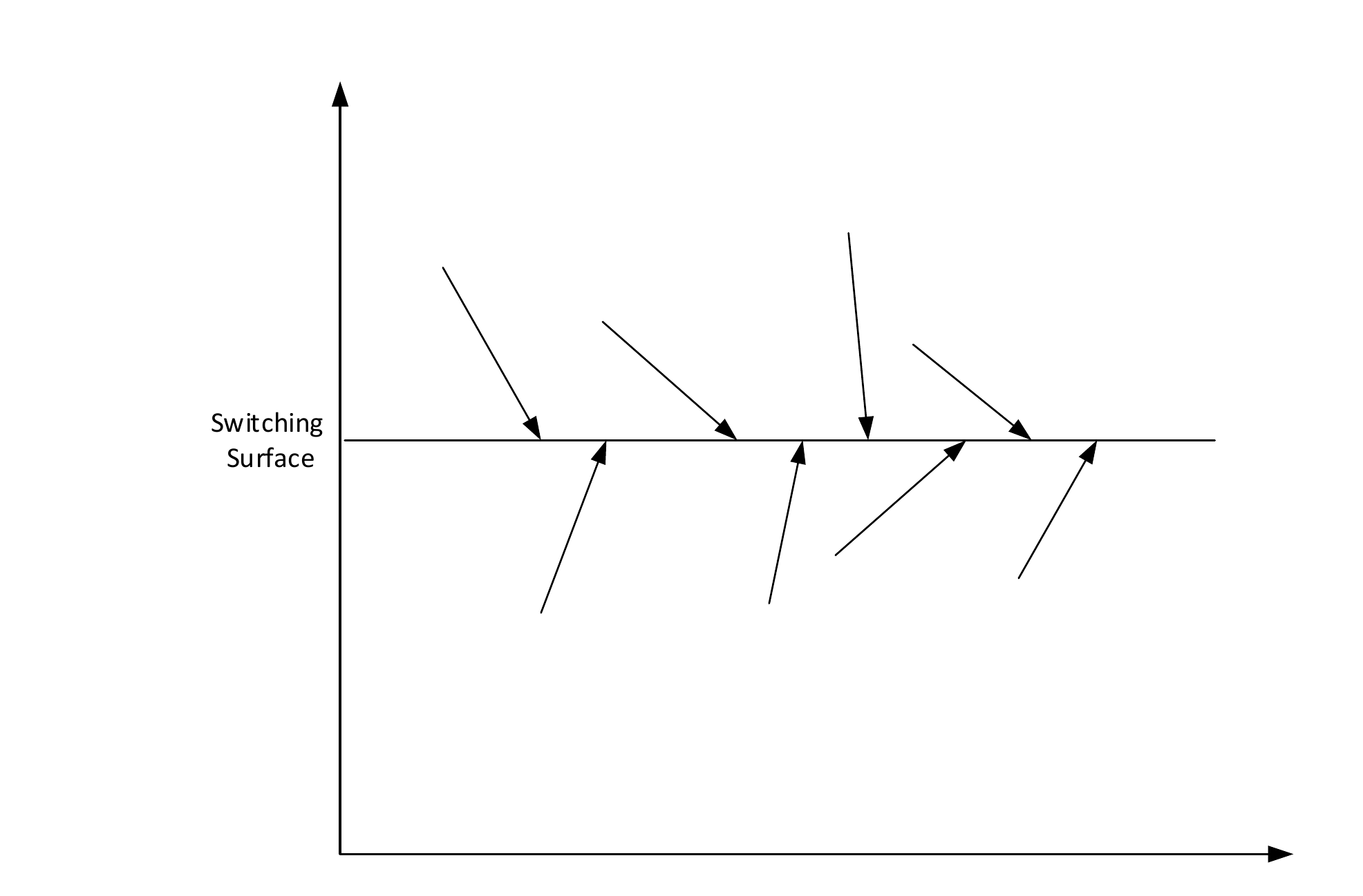}}\\
\caption{Vector field around the switching surface}
\label{switching}
\end{figure}
Fig.~\ref{switching} shows the vector field of the closed-loop system around the switching surface $(\lambda_{i}(t)=0)$.
(~\ref{eq:34}) implies that
\begin{eqnarray}  \label{eq:36}
\dot{\lambda}_{i}(t)=0~~~~~\forall \quad t\geq t_0.
\end{eqnarray}
It clearly sollows from (~\ref{eq:26}), (~\ref{eq:27}), (~\ref{eq:28}) and (~\ref{eq:35}) that
\begin{eqnarray}  \label{eq:37}
\bar{u}_i(t)= -b(t)
\end{eqnarray}
where $b(t)$ satisfies the inequality (~\ref{eq:28}) for all sliding mode solutions. As a result, for any initial condition, the sliding mode solution is unique and well-defined. Moreover, (~\ref{eq:37}), (~\ref{eq:33}), (~\ref{eq:12}) and (~\ref{eq:30}) indicate that the constraint (~\ref{eq:4}) holds for any sliding mode solution satisfying (~\ref{eq:34}). It follows from (~\ref{eq:35}) that for all $t\geq t_0$ the vector $\tilde{d}_{i}(t)$ is parallel to the velocity vector $V_i (t)$
as a result, $y_i(t)=Y_0+Y_i$ and $z_i(t)=Z_0+Z_i$ for all $t\geq t_0$. Therefore, we can say that the second and third parts of the condition (6) follow from this. The fact that the velocity vector $V_i (t)$ is parallel to the vector $%
\tilde{d}_{i}(t)$ for all $t\geq t_0$ and $d_{i}$ replaced by $\tilde{d}_{i}$ in the control law (~\ref{eq:21})  mean that
\[
\tilde{d}_{i}(t)= \left(
\begin{array}{l}
c \\
0\\
0%
\end{array}
\right)
\]
The first part of the conditions ~\ref{eq:6} follows from this. This completes the proof of Theorem ~\ref{theorem:formation}.

%\begin{figure}
%\centering
%\mbox
%
%\subfigure[Truncated Octahedron]{\includegraphics[scale=0.55]{octahedron}}
%\subfigure[Filling a given 3$\mathcal{D}$ space by circumspheres of a truncated octahedron based grid ]{\includegraphics[scale=0.55]{truncated} }
%
%\caption{  Truncated-Octahedron based grid}
 %\label{figor11}
%\end{figure}
%

\section{Decentralized 3D Formation Building with Anonymous  Robots} \label{6.4}

This section describes the problem of formation building in 3D environments for the case where the mobile robotic sensors are unaware of their positions in the configuration. In the following, we propose a random decentralized motion coordination law for the mobile robotic sensors so that they reach consensus on their positions and form a desired geometric pattern from any initial position in 3D spaces.
\begin{definition} A navigation law with the given configuration $C=\left\{X_1,X_2...X_n, Y_1, Y_2, ... Y_n, Z_1,Z_2,....Z_n\right\}$ and anonymous  robots is said to be globally stabilizing if for any initial conditions $\theta_i(0)$, $\psi_i(0)$, $x_i(0)$, $y_i(0)$, $z_i(0)$ and $v_i(0)$ , there exists a permutation  \textit{p(i)} of the index set  $ {1,2,...n}$ such that for any of them  there exists a Cartesian coordinate system and $v_0$ such that the solution of the closed-loop system (~\ref{eq:1}) with these initial conditions and the proposed navigation law in this Cartesian coordinate system satisfies:
\begin{equation*}
\lim_{t\rightarrow\infty} x_i(t)-x_j(t) = X_{\textit{p(i)}}-X_{\textit{p(j)}},
\end{equation*}
\begin{equation*}
\lim_{t\rightarrow\infty} y_i(t)-y_j(t) = Y_{\textit{p(i)}}-Y_{\textit{p(j)}},
\end{equation*}
\begin{equation} \label{38}
\lim_{t\rightarrow\infty} z_i(t)-z_j(t) = Z_{\textit{p(i)}}-Z_{\textit{p(j)}}.
\end{equation}
\end{definition}
Let $0<\lambda <\frac{R_c}{2}$ be a given constant. Also, let $N>1$ be a given integer. We consider a undirected graph $g$ as a representation of the set of vertices of the given configuration $C=\left\{X_1,X_2...X_n, Y_1, Y_2, ... Y_n, Z_1,Z_2,....Z_n\right\}$. We assume that graph $g$ is connected. The index permutation function $\textit{p(i)}$ defines the position of the mobile robotic sensors in the configuration. As a result, to determine the position of the mobile robotic sensors in the configuration we need to build the index permutation function. In the following, we propose a random algorithm to build the index permutation function $\textit{p(i)}$.
\begin{definition} a vertex j of the graph $g$ is vacant at time $kN$ for robot $i$ if there is not any robot inside the sphere of radius $\lambda$ centered at the following  point:
\begin{equation} \label{39}
\left(
\begin{array}{c}
x_i(kN) + \tilde{x_i}(kN) + X_j + kN\tilde{v_i}(kN)\\
y_i(kN) + \tilde{y_i}(kN)+Y_j\\
z_i(kN) + \tilde{z_i}(kN)+Z_j
\end{array}
\right)
\end{equation}
\end{definition}
Let $b_i(k)$ be a boolean variable such that $b_i(k)=1$ if there exists another robot $j \neq i$ within the sphere of radius $\lambda$ centred at the following point:

\begin{equation} \label{40}
\left(
\begin{array}{c}
x_i(kN) + \tilde{x_i}(kN) + X_i + kN\tilde{v_i}(kN)\\
y_i(kN) + \tilde{y_i}(kN)+Y_i\\
z_i(kN) + \tilde{z_i}(kN)+Z_i
\end{array}
\right)
\end{equation}
at time $kN$, and $b_i(k)=0$ otherwise. We introduce the following algorithm for updating permution $\textit{p(i)}$ as follows:
\begin{equation} \label{eq:41}
\textit{p}((k+1)N,i)=\begin{cases}
\textit{p}(kN,i)\\  \quad  if \quad b_i(kN)=0\\
j\quad with\quad probability\quad\frac{1}{\left\|S(kN,i)\right\|},\forall j\in S(kN,i)\\
\quad if\quad \quad b_i(kN)=1 \quad and \left\|S(kN,i)\right\|>1.
\end{cases}
\end{equation}
Where $S(kN,i)$ is a set consisting of $\textit{p}(kN,i)$ and those vertices of g that are vacant for robot $i$ at time $kN$ and connected to  $\textit{p}(kN,i)$.
\begin{figure*}[t!]
\begin{center}
\mbox{
\subfigure[Traveling in an edge formation with six robots]{
{\includegraphics[width=0.8\textwidth,height=0.65\textwidth]{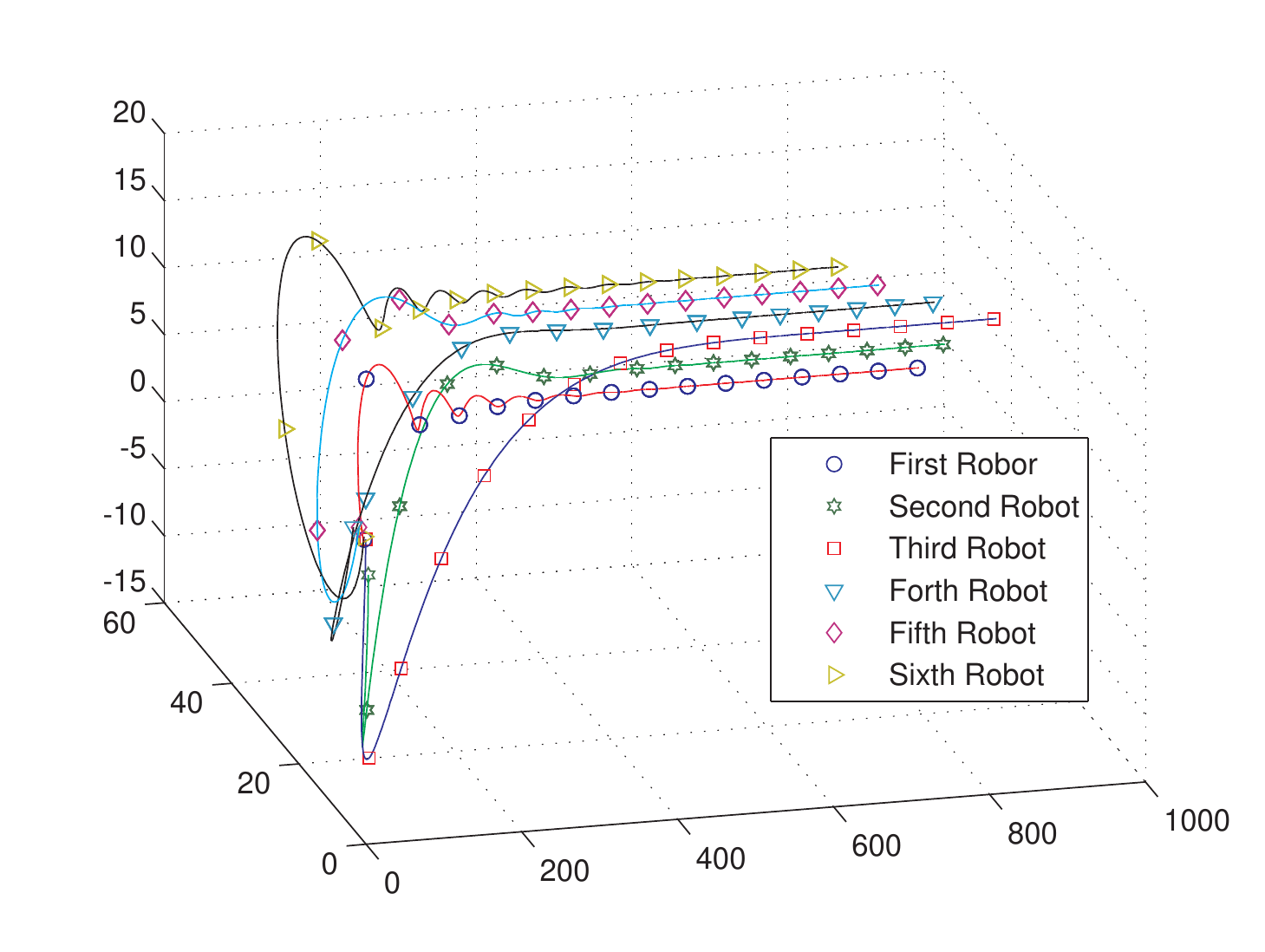}}\quad
\label{f1}
}
}
\mbox{
\subfigure[Traveling in a random formation with six robots]{
{\includegraphics[width=0.8\textwidth,height=0.65\textwidth]{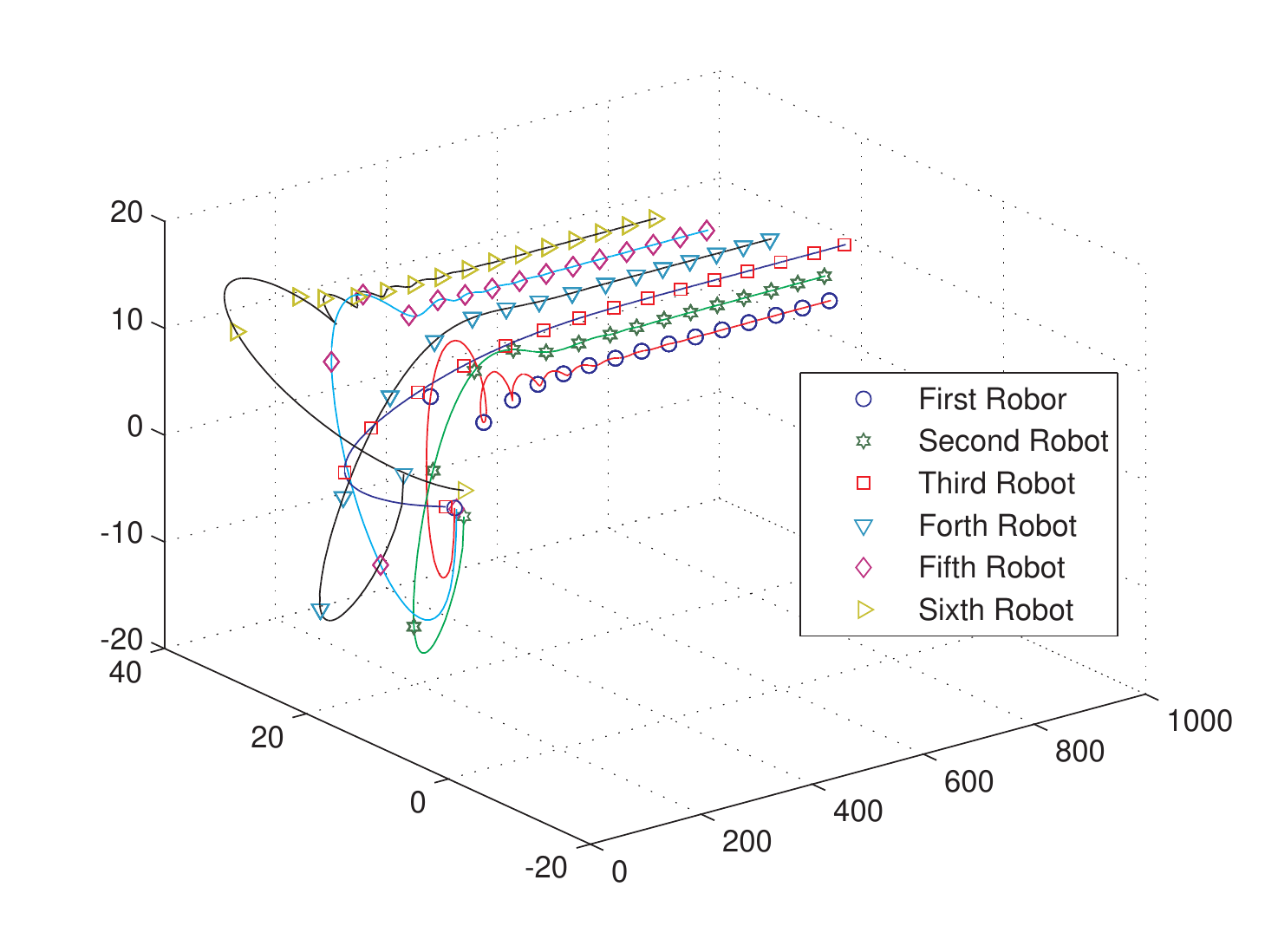}}\quad
\label{f2}
}
}
\caption{Trajectories of six mobile robotic sensors during formation  building}
\label{yz}
\end{center}
\end{figure*}
\begin{theorem} Consider the mobile robotic sensors and their constraints described by the equations (~\ref{eq:1}), (~\ref{eq:2}), (~\ref{eq:3}), (~\ref{eq:4}) and (~\ref{eq:5}). Let $C=\left\{X_1,X_2,....X_n,Y_1,Y_2,...Y_n,Z_1,Z_2,...Z_n\right\}$ be a given configuration. Suppose that assumptions ~\ref{assump3} and ~\ref{assump4} hold, and $c_0$ is a constant satisfying (~\ref{eq:12}). Also, assume that graph g is connected. Then, for any initial conditions there exists an integer $N_0>0$ such that for any $N>N_0$, the decentralized control law (~\ref{eq:10}),(~\ref{eq:20}), (~\ref{eq:21}), (~\ref{eq:22}) and (~\ref{eq:23}) with probability 1 is globally stabilizing with the given configuration.
\end{theorem}

%\begin{figure*}[t!]
%\begin{center}
%\mbox{
%\subfigure[]{
%{\includegraphics[width=0.5\textwidth]{FigVali/formation/sc1}}\quad
%\label{sc1}
%}
%\hspace{-2mm}
%\subfigure[]{
%{\includegraphics[width=0.5\textwidth]{FigVali/formation/sc2}}\quad
%\label{sc2}
%}
%}
%\vspace{-4mm}
%\caption{Tethrahedron Configuration}
%\vspace{-8mm}
%\label{sc}
%\end{center}
%\end{figure*}
%\end{theorem}
\textit{Proof}: Algorithm (~\ref{eq:41}) implies that the mobile robotic sensors move randomly to the neighbouring unoccupied vertices of the given configuration’s graph; this will continue until all mobile robotic sensors are positioned in different vertices of the graph g. These states are considered as absorbing states and they are impossible to leave. As a result, the algorithm (~\ref{eq:41}) describes an absorbing Markov chain. It is obvious that, from any initial state with probability 1, one of the absorbing states will be reached.

%%%%%%%%%%%%%%%%%%%%%%%%%%%%%%%%%%%%%%%%%%%%%%%%%%%%%%%%%%%%%%%%%%%%%%%%%%%%%%%%%%%%%%%%%%

%\begin{figure*}[t!]
%\begin{center}
%\mbox{
%\subfigure[Traveling in an edge formation with six robots]{
%{\includegraphics[width=0.5\textwidth]{FigVali/formation/11}}\quad
%\label{f1}
%}
%\subfigure[Traveling in a regular tetrahedron  formation with four robots]{
%{\includegraphics[width=0.5\textwidth]{FigVali/formation/unregul}}\quad
%\label{f2}
%}
%}
%\mbox{
%\subfigure[Traveling in a 3D random formation with five robots]{
%{\includegraphics[width=0.5\textwidth]{FigVali/formation/random}}\quad
%\label{f3}
%}
%\subfigure[Traveling in a 3D formation with five robots]{
%{\includegraphics[width=0.5\textwidth]{FigVali/formation/unregular2}}\quad
%\label{f4}
%}
%}
%\caption{The locations of the mobile robotic sensors during formation  building}
%\label{Formation}
%\vspace{-6mm}
%\end{center}
%\end{figure*}

%\begin{figure*}
%\centering
%\mbox{\subfigure[Traveling in an edge formation with six robots ]{\includegraphics[scale=0.38]{FigVali/formation/11}}}\quad
%\subfigure[Traveling in a regular tetrahedron  formation with four robots]{\includegraphics[scale=0.38]{FigVali/formation/unregul} } \quad
%\mbox{\subfigure[Traveling in a 3D random formation with five robots]{\includegraphics[scale=0.38]{FigVali/formation/random}}}\quad
%\subfigure[Traveling in a 3D formation with five robots ]{\includegraphics[scale=0.38]{FigVali/formation/unregular2} }
%\caption{The locations of the mobile robotic sensors during formation  building}
 %\label{figor10}
%\end{figure*}
\begin{table}
 \caption{Simulations Parameters}
 \label{parameter}
 \begin{center}
 {\small
  \begin{tabular*}{\columnwidth}[t]{cp{72pt}p{96pt}}\hline\hline
Parameter & {\raggedright Value } & {\raggedright Comment}\\ \hline
$T_s$ & {\raggedright .01 } & {\raggedright Sampling time }\\
$R_c$ & {\raggedright 100 } & {\raggedright Communication range}\\
$\lambda$ & {\raggedright  20} & {\raggedright $0<\lambda <\frac{R_c}{2}$}\\
$N$ & {\raggedright 10 } & {\raggedright The given integer }\\
$V_{max}$ & {\raggedright 8 m/s } & {\raggedright Maximum linear velocity for the mobile robotic sensors.}\\
$V_{min}$ & {\raggedright 2 m/s } & {\raggedright Minimum linear velocity for the mobile robotic sensors.}\\
$U_{max}$ & {\raggedright 2 } & {\raggedright Maximum angular velocity.}\\
\hline\hline
  \end{tabular*}
 }%end of small
 \end{center}
\end{table}
\begin{figure*}[t!]
\begin{center}
\mbox{
\subfigure[Traveling in a regular tetrahedron  formation with four robots]{
{\includegraphics[width=0.8\textwidth,height=0.65\textwidth]{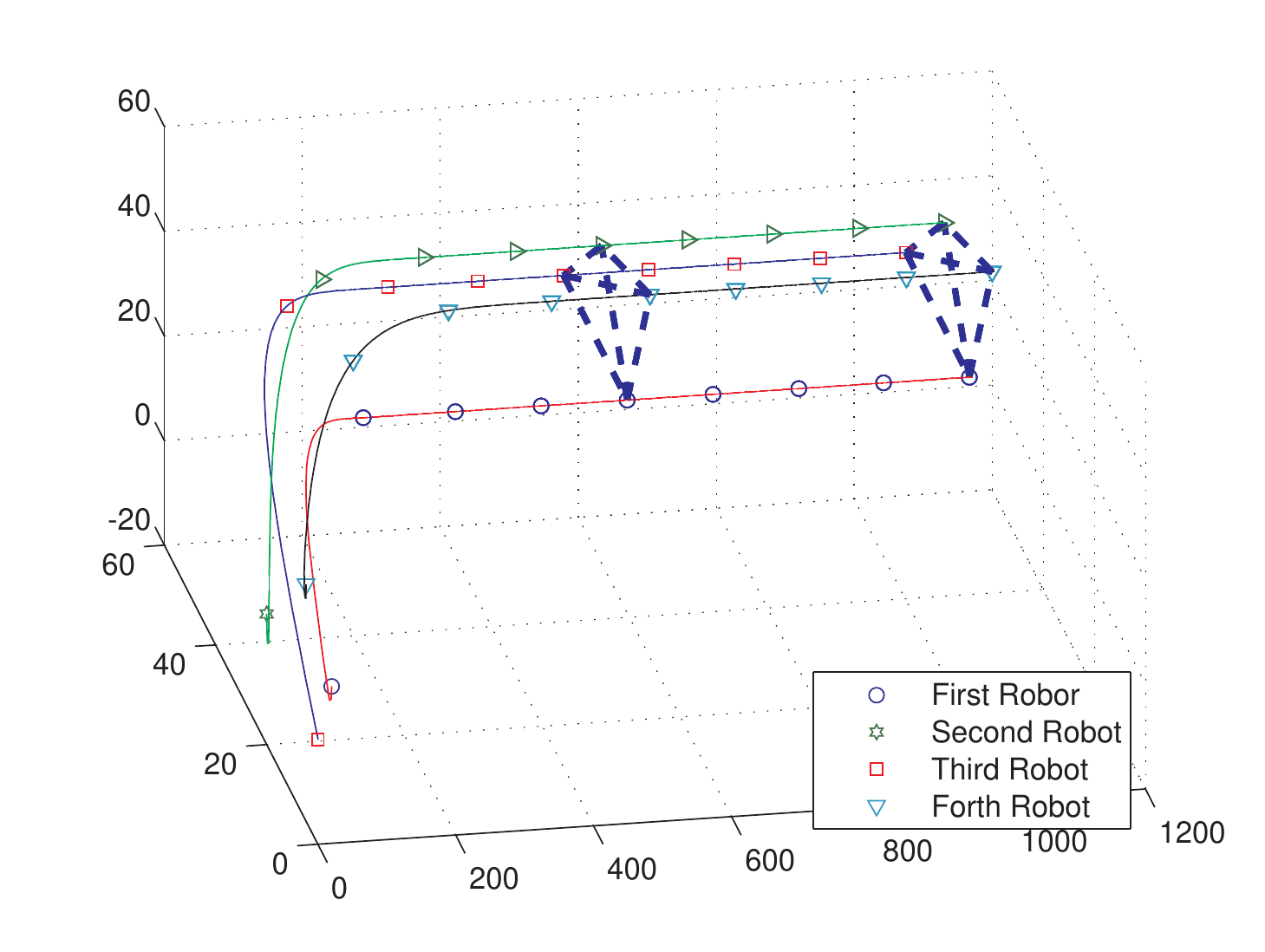}}\quad
\label{f3}
}
}
\mbox{
\subfigure[Traveling in a 3D formation with five robots]{
{\includegraphics[width=0.8\textwidth,height=0.65\textwidth]{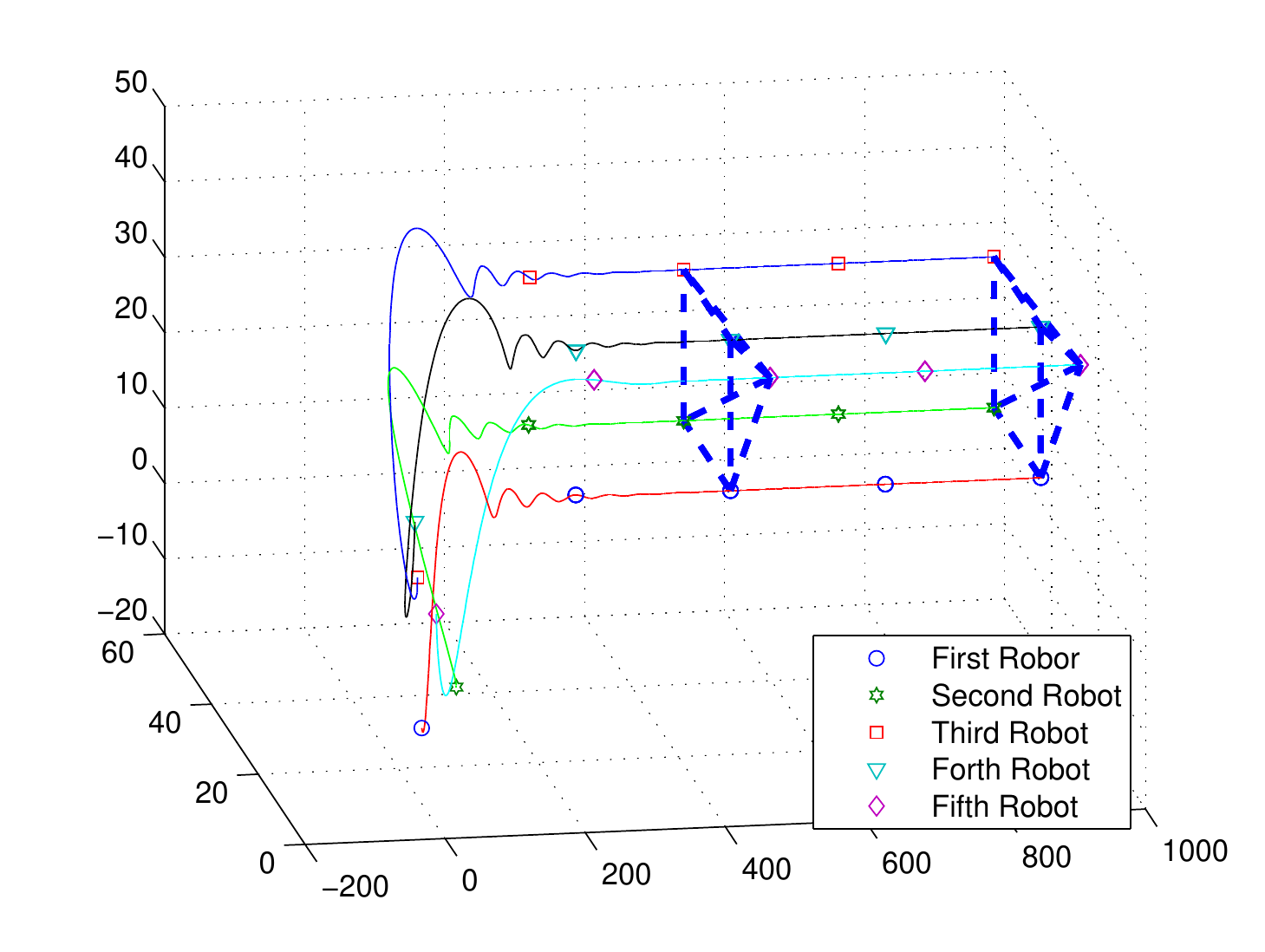}}\quad
\label{f4}
}
}
\caption{Trajectories of mobile robotic sensors during 3D formation  building}
\label{yz1}
\end{center}
\end{figure*}
 \begin{figure}
\centering
{\includegraphics[scale=0.75]{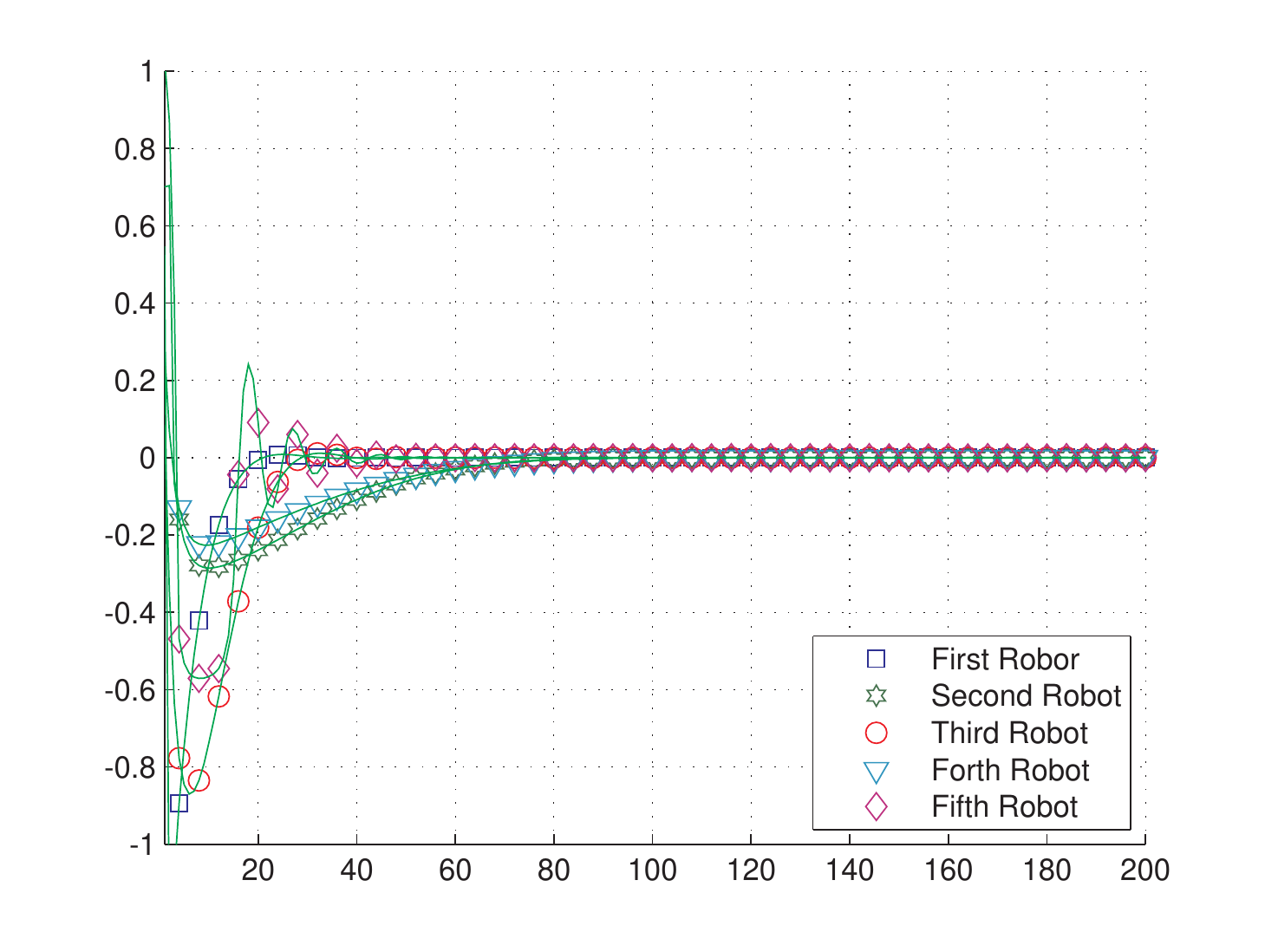}}
\caption{Convergence of the robot's headings' to the same directions}
 \label{heading}
\end{figure}
\begin{figure*}[t!]
\begin{center}
\mbox{
\subfigure[]{
{\includegraphics[width=0.75\textwidth,height=0.6\textwidth]{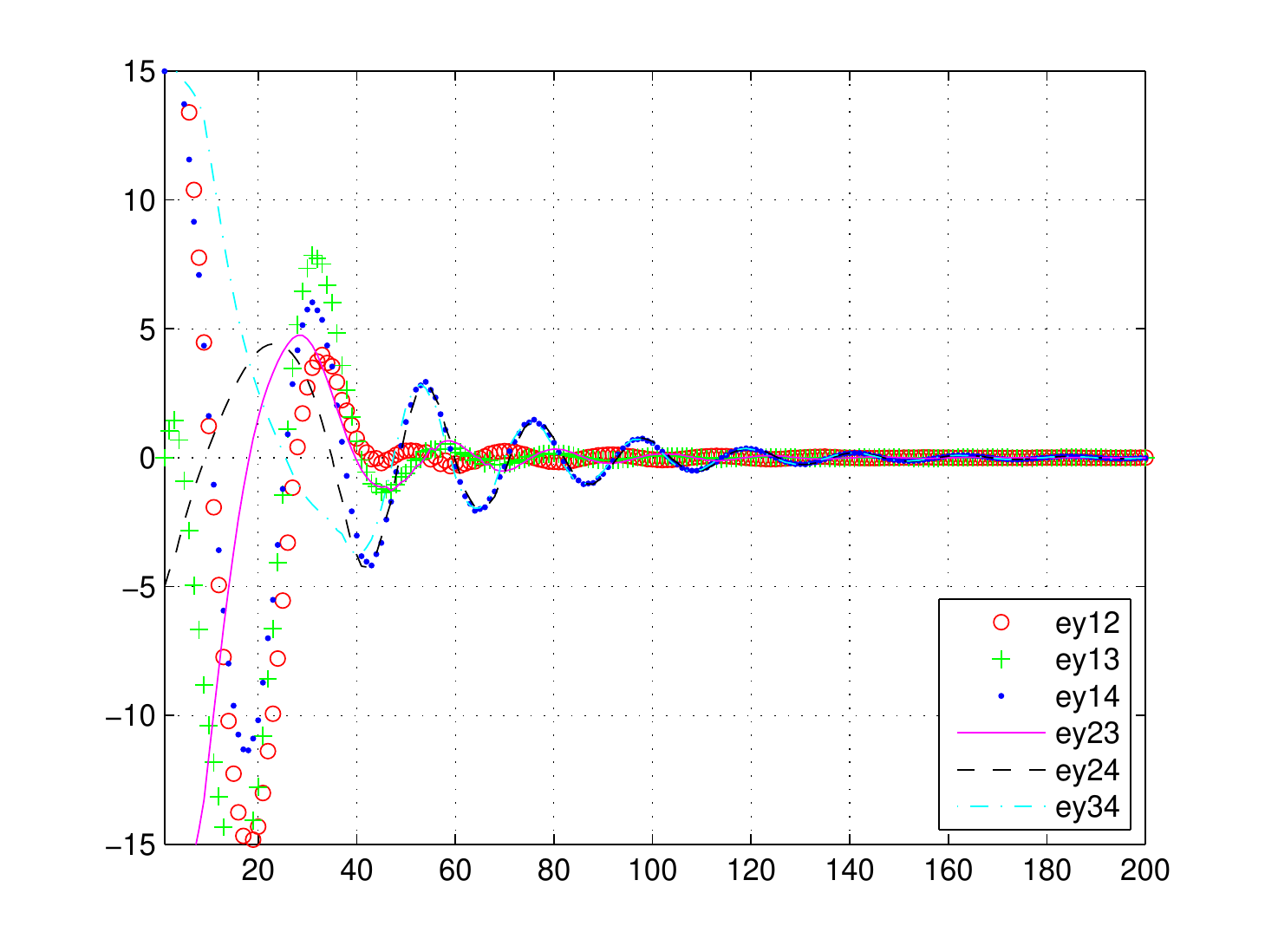}}\quad
\label{y}
}
}
\mbox{
\subfigure[]{
{\includegraphics[width=0.75\textwidth,height=0.6\textwidth]{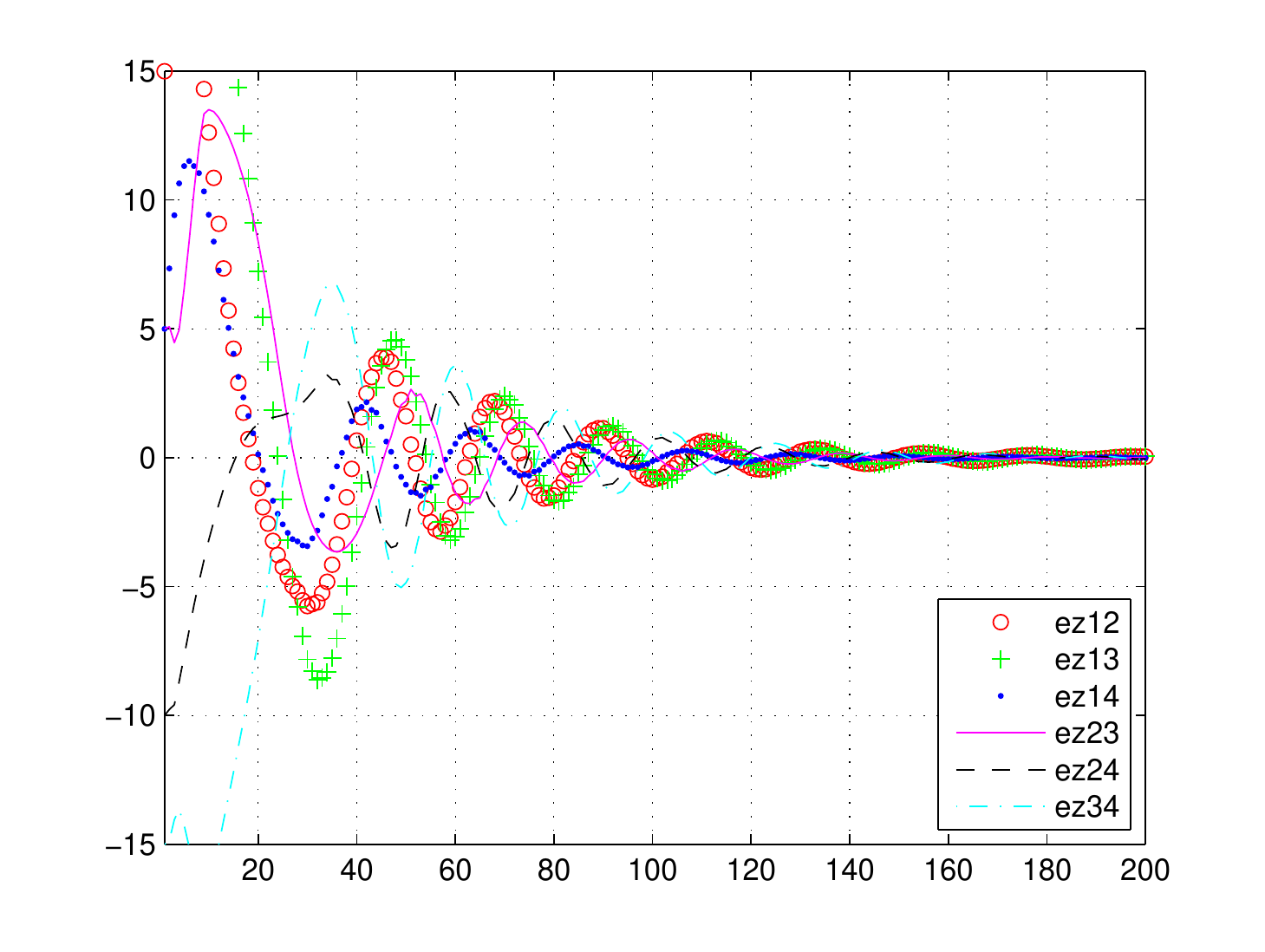}}\quad
\label{z}
}
}
\caption{Convergence of the mobile robotic sensors to the given configuration}
\label{yz}
\end{center}
\end{figure*}
\begin{figure*}[t!]
\begin{center}
\mbox{
\subfigure[]{
{\includegraphics[trim={2cm 10cm 0cm  9cm}, width=0.8\textwidth,height=0.6\textwidth]{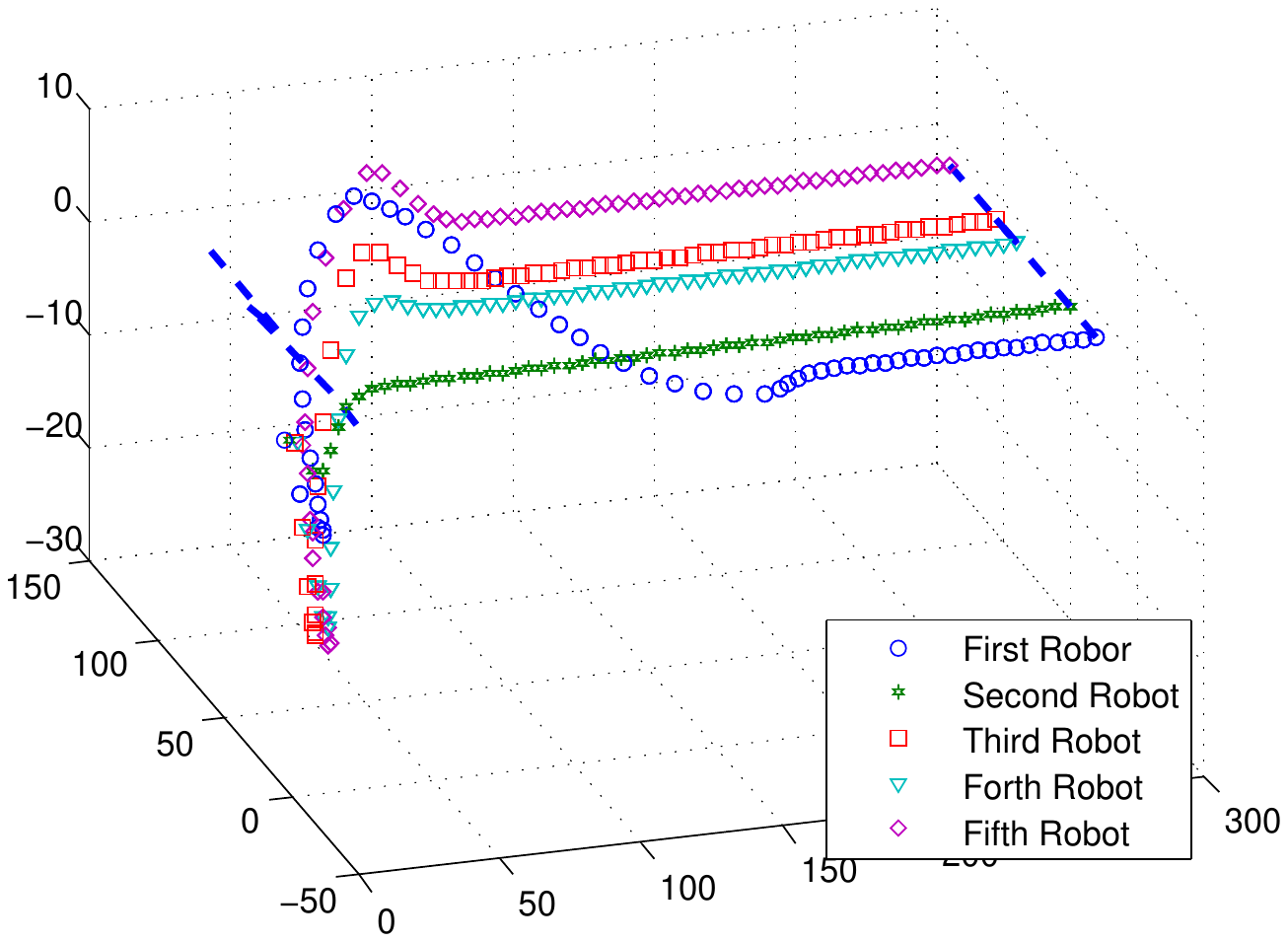}}\quad
\label{unknown1}
}
}
\mbox{
\subfigure[]{
{\includegraphics[trim={2cm 10cm 0cm  9cm}, width=0.8\textwidth,height=0.6\textwidth]{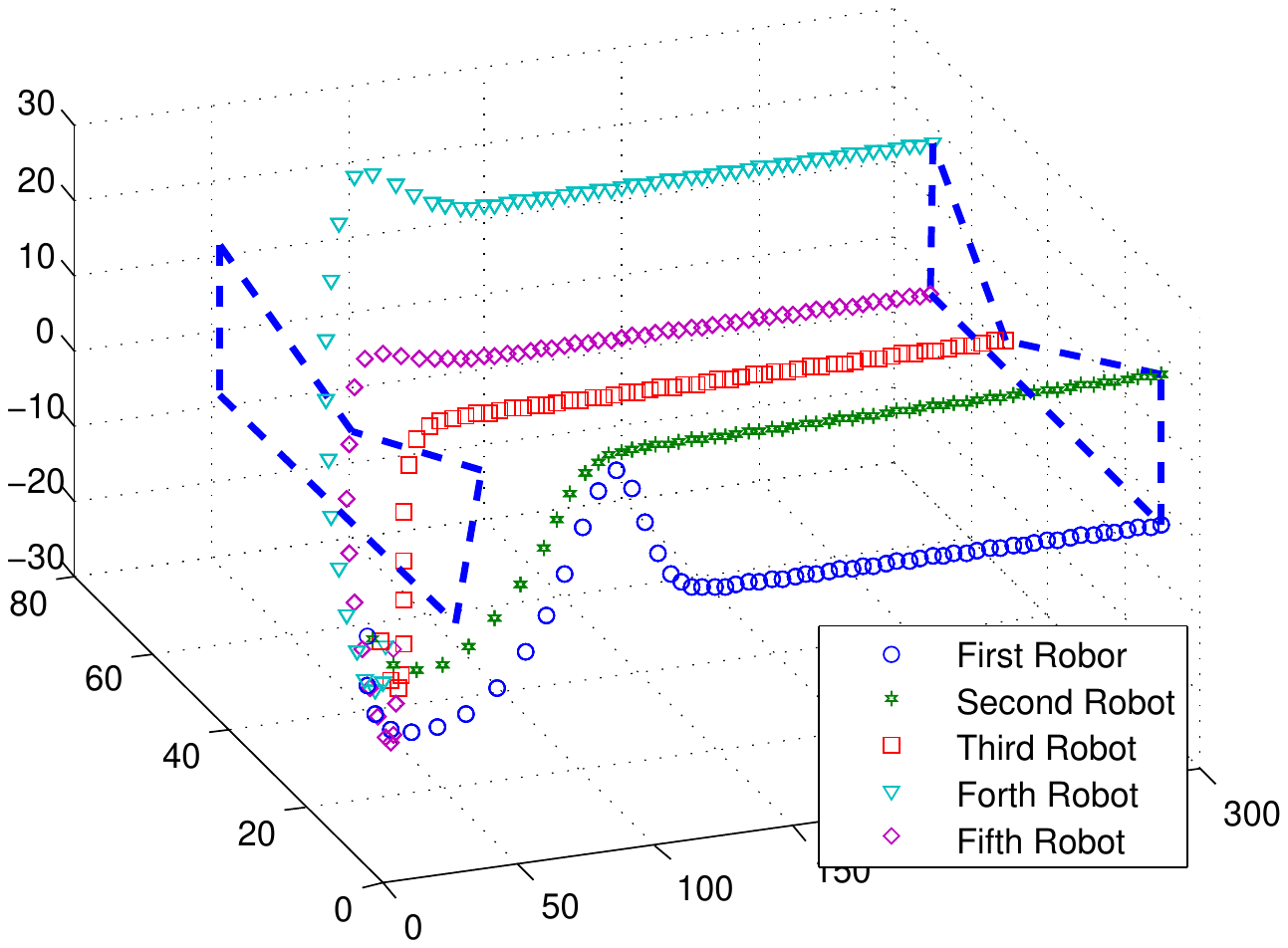}}\quad
\label{unknown2}
}
}
\caption{Formation building with anonymous robots}
\label{unknown}
\end{center}
\end{figure*}

\section{Simulation Results} \label{6.5}
In this section, a simulation study is carried out to evaluate the performance of the proposed consensus-based decentralized navigation laws for formation building in 3D environments. For simulation, we used MATLAB R2014b with simulation parameters shown in Table.~\ref{parameter} for all configurations. In the first set of simulations, we consider a fleet of six mobile robotic sensors defined by (~\ref{eq:1}) with the constraints (~\ref{eq:3}) and (~\ref{eq:4}) for an edge and a random formation building in a 3D environment. The initial headings and positions of all mobile robotic sensors are generated randomly within a pre-specified area. As shown in Fig.~\ref{f1} and Fig.~\ref{f2}, after the transient state the mobile robotic sensors constitute the desired formation and they preserve this formation as they move along the x-axis. It is obvious that the communication range should be chosen large enough so that the mobile robotic sensors form a connected graph during the simulations.\\
The next set of simulations show the ability of the proposed formation building strategy to work at different conditions. Fig.~\ref{f3} and Fig.~\ref{f4} demonstrate the process of formation building for two different configurations with five and four mobile robotic sensors. As shown in Fig.~\ref{f3} and Fig.~\ref{f4}, for both cases the mobile robotic sensors adjust their headings after a reasonable amount of time, and then they converge to the vertices of the given 3D configurations. The presented simulation results verify the effectiveness of the proposed decentralized formation building control law for various initial conditions, the number of mobile robotic sensors and configurations.\\
The second set of simulations is carried out to show the headings of the mobile robotic sensors in the whole formation building process. By application of the proposed decentralized navigation law, the robots heading angles converge to zero, because eventually all robots move in the same direction with the X-axis, as shown in Fig.~\ref{heading}. To illustrate the convergence of the mobile robotic sensors to the given configuration, in the following we introduce the errors term defined by: $ey_{ij}(t)= (y_i(t)-y_j(t)) -(Y_i-Y_j)$ and $ez_{ij}=(z_i(t)-z_j(t)) -(Z_i-Z_j)$. Fig.~\ref{y} and Fig.~\ref{z} show that  $ey_{ij}(t)\rightarrow 0$ and $ez_{ij}(t)\rightarrow 0$, for all $i=1,2,..n$. It means that, eventually all mobile robotic sensors converge to the given configuration.\\
In the following, we simulate the proposed algorithm for the problem of formation building with anonymous robots. Fig.~\ref{unknown1} and Fig.~\ref{unknown2} show the performance of this  algorithm using different initial conditions and configurations. As shown, in both cases the mobile robotic sensors converge to the given configuration.

%\begin{figure*}[t!]
%\begin{center}
%\mbox{
%\subfigure[]{
%{\includegraphics[width=0.5\textwidth]{FigVali/formation/y}}\quad
%\label{y}
%}
%\hspace{-2mm}
%\subfigure[]{
%{\includegraphics[width=0.5\textwidth]{FigVali/formation/z}}\quad
%\label{z}
%}
%}
%\vspace{-4mm}
%\caption{Convergence of the mobile robotic sensors to the given configuration}
%\vspace{-8mm}
%\label{yz}
%\end{center}
%\end{figure*}
 %\begin{figure*}
%\centering
%\subfigure[]{\includegraphics[scale=0.4]{FigVali/formation/y}}
%\subfigure[]{\includegraphics[scale=0.4]{FigVali/formation/z} }
%\caption{Convergence of the mobile robotic sensors to the given configuration}
 %\label{figor16}
%\end{figure*}

%\begin{figure*}[t!]
%\begin{center}
%\mbox{
%\subfigure[]{
%{\includegraphics[trim={3cm 10cm 0cm  9cm},scale=0.57]{FigVali/formation/index1}}\quad
%\label{unknown1}
%}
%\hspace{-2mm} \quad
%\subfigure[]{
%{\includegraphics[trim={3cm 10cm 0cm  9cm},scale=0.57]{FigVali/formation/index2}}\quad
%\label{unknown2}
%}
%}
%\vspace{-4mm}
%\caption{Formation building with anonymous robots}
%\vspace{-8mm}
%\label{unknown}
%\end{center}
%\end{figure*}

%\begin{figure*}
%\centering
%\subfigure[]{\includegraphics[trim={3cm 10cm 0cm  9cm},scale=0.57]{FigVali/formation/index1}}
%\label{unknown1}
%\subfigure[]{\includegraphics[trim={3cm 10cm 0cm  9cm},scale=0.57]{FigVali/formation/index2} }
%\label{unknown2}
%\caption{Formation building with anonymous robots}
 %\label{unknown}
%\end{figure*}
\section{Summary} \label{6.6}
At first part of this chapter, a decentralized control law for formation building in three dimensional environments was proposed. We used nonlinear standard kinematics equations with hard constraints on the robot angular and linear velocity to describe the robot motion in 3D spaces. Then, we presented a random formation building algorithm for the problem of formation building with anonymous robots in 3D environments. This algorithm is an appropriate  option for the case where  robots do not know a priori its position in the configuration. The proposed control laws are based on the consensus approach that is simply implemented and computationally effective. The proposed control algorithms are decentralized, and the control action of each robot is based on the local information of its neighbouring robots. Based on this algorithms, all mobile robotic sensors eventually converge to the desired geometric configuration with the same direction and the same speed. The performance of the proposed decentralized control laws have been confirmed by extensive simulations. Furthermore, we give a mathematically rigorous analysis of the convergence of the mobile robotic sensors to the given configurations.

%\singlespace
\singlespacing
\chapter{Conclusions and Future Work}
\label{chap:conclusion}
\minitoc
This report considers the problem of coverage, search and formation building by a network of mobile robotic sensors in three dimensional spaces. Unlike most of the existing algorithms which focus on two dimensional environments, our algorithms are designed for three dimensional environments. In this report, we present some novel decentralized control algorithms for search, coverage and formation building in 3D spaces. The decentralized control of mobile robotic sensor networks is a relatively recent area of research, and this approach is believed more promising due to many inevitable physical constraints such as limited resources and energy, short wireless communication ranges, narrow bandwidths, and large sizes of robots to manage and control. Furthermore, the distributed method has many advantages in realizing cooperative group performances, scalability, and robustness.\\
In the following, we summarize the important contributions of this report for self deployment of mobile robotic senors networks for coverage, search and formation building in 3D environments.
%\section{Contributions}

\begin{itemize}
\item We introduced a distributed motion coordination algorithm for a network of mobile robotic sensors for complete sensing coverage of a bounded three-dimensional space. The algorithm was developed based on some simple consensus laws that only require information about the closest neighbours of each mobile robotic sensor. The proposed control laws drive the network of mobile robotic sensors to form a 3D covering grid in the target region and then they occupy the vertices of the grid. Also, we proposed a distributed control law for coordination of the mobile robotic sensors such that they form a given 3D shape at vertices of a truncated octahedral grid from any initial positions. Simulation results showed that the 3D truncated octahedral grid outperforms other 3D grids in terms of complete sensing coverage time and a minimum number of mobile robotic sensors required to cover completely the given 3D region.
%\vspace{-3mm}
\item We presented a set of random distributed laws for search in unknown bounded three dimensional environments. The mobile sensors utilize a truncated octahedral grid for the search process. Furthermore, they exchange information with their neighbouring sensors to minimize the search time.
%\vspace{-3mm}
\item We introduced a novel random bio-inspired search algorithm to drive mobile robotic sensors for locating clustered and sparsely located targets in bounded 3D areas. The proposed method combines the bio-inspired Levy flight random search mechanism for determining the length of the walk with vertices of a covering truncated octahedral grid to optimize the search procedure. The proposed approach has the advantage that it does not need centralized control system, also it is scalable.
%\vspace{-3mm}
\item We extended our grid-based random search algorithm for the problem of detecting mobile targets moving randomly in a bounded 3D space by a mobile robotic sensor network.
%\vspace{-3mm}
\item We presented a decentralized control law for formation building in 3D environments. We used nonlinear standard kinematics equations with hard constraints on the robot angular and linear velocity to describe the robots' motion in 3D spaces. Then, we proposed a random formation building algorithm for the problem of formation building with anonymous robots in 3D environments. This algorithm is an appropriate option for the case where robots do not know a priori its position in the configuration.

\item  The proposed algorithms are distributed. As a result, these algorithms can be used to the cases where the communication and the sensing range are limited. Moreover, The proposed control laws are based on the consensus approach that is simply implemented and computationally effective. The control algorithms are decentralized and the control action of each robot is based on the local information of its neighbouring robots.
\item The proposed methods can deal with new tasks in various environments, such as different number of targets or mobile sensors. In other words, these algorithms are flexible in various scenarios.
\item Finally, the effectiveness of the proposed control algorithms have been confirmed by extensive simulations. Also, we gave mathematically rigorous proof of convergence with probability 1 of the proposed algorithms.
\end{itemize}

\vspace{-8mm}
\section{Future Work}
\vspace{-4mm}
The work in this report opens up some future research problems and the practical realization of the proposed algorithms. The following suggestions are examples of possible future research based on the results presented in this report:
\begin{itemize}
\item In the future studies, the problem of navigating mobile robotic sensors with complicated models can be considered. In such cases, it is not sufficient to solve the navigation problem by only considering the kinematic models of the mobile robotic sensors, the dynamics of the robots should also be studied. The model proposed in the current works is simple. Future research could consider a more complex nonlinear version of the model \cite{krstic1995nonlinear,savkin1995minimax,savkin1998robust}. In such a model, dynamic equations for mobile robots should be applied. Moreover, we will combine  our proposed navigation method with methods of robust nonlinear control and robust state estimation proposed in \cite{savkin2000robust,savkin1996model,moheimani1998robust,petersen1999robust}.
%\vspace{-3mm}
\item In the future search and coverage research, we will consider more realistic scenarios with robotic sensors motion described by non-holonomic kinematic models with the problem of collision avoidance \cite{hoy2015algorithms,stastny2015collision,yang20133d,savkin2013simple}
\item Current works assumed homogeneous sensors as well as spherical detection ranges of those sensors. These assumptions are simplistic. To extend this work, we consider nodes with more realistic features.
\item  Another exciting direction is to combine the proposed formation building algorithms with a local obstacle avoidance technique to generate safe and applicable strategies in the real world so that the mobile robots will not collide with each other and to obstacles during formation building. See e.g. \cite{chang2015coordinated,teimoori2010biologically,matveev2011method,matveev2012real,savkin2014seeking,savkin2015safe,subramanian2014real,munoz2015uas,brockers2014micro,
    hrabar2012evaluation,yang20133d,shim2007evasive}.
\item In the last chapter, we proposed a decentralized formation building control law for a group of mobile robots with hard constraints on the linear and angular velocities, and we presented their computer simulations. It would be interesting to investigate and verify the effectiveness of the proposed algorithms through the experiments with a group of real mobile robots.
\item In the future works, we can evaluate the proposed distributed algorithms based on the factors such as cost, power consumption, etc.
\item To use mobile robotic sensors in search, surveillance, monitoring and exploration in future applications, it is definitely necessary to move on real mobile robotic sensors and implement solutions using realistic scenarios. Moreover, we can use the 3D mobile robotic sensor networks for environmental exploration of 2D and 3D environmental fields and surfaces. See, e.g. \cite{matveev2011navigation,matveev2012method,clark2007mobile}
\item Further research is needed to optimize the travelling distance of the mobile robotic sensor networks as it can effectively reduce the energy consumption and prolong the network lifetime.
\item In chapter 2,3,4,5, we consider the problem of decentralized search and coverage in 3D spaces. It is possible to combine the proposed algorithms with a collision avoidance technique in order to avoid any collision with nearby moving agents. See e.g.\cite{breitenmoser2016combining,stipanovic2012collision}
 \item In our proposed algorithms, we didn't take into account limitations of communication channels. Future research is needed to combine our proposed navigation method with the tools for control and estimation via limited capacity channels developed in \cite{matveev2001optimal,matveev2009estimation,matveev2003problem,matveev2005comments,savkin2006analysis}.
\end{itemize}
%\vspace{-3mm}

%\singlespace
\singlespacing

%=========================GLOSSARY====================================================================================
% added by Upendra, for page header reformatting, 10 Mar 2010
\renewcommand{\chaptername}{}
\renewcommand{\thechapter}{}
\renewcommand{\chaptermark}[1]{%
\markboth{\MakeUppercase{%
#1}}{}}
% end addition by Upendra, for page header reformatting
%\include{glossary}

% added by Upendra, 10 Mar 2010

% end addition by Upendra, 10 Mar 2010

%==========================REFERENCES====================================================================================

\bibliographystyle{abbrv}
\bibliography{ref}
%______________________________________________________________________________________

\end{document}